\documentclass{article}
\usepackage{amsmath, amssymb, amsfonts}
\usepackage{amsthm}
\usepackage{geometry}
\usepackage{graphicx}
\usepackage[titletoc]{appendix}
\usepackage{natbib}

\newtheorem{theorem}{Theorem}[section]
\newtheorem{lemma}[theorem]{Lemma}
\newtheorem{proposition}[theorem]{Proposition}
\newtheorem{corollary}[theorem]{Corollary}
\newtheorem{assumption}{Assumption}
\theoremstyle{definition}
\newtheorem{definition}{Definition}[section]
\newtheorem{example}{Example}[section]
\newtheorem{remark}[theorem]{Remark}

\newcommand{\secret}[1]{}
\input{tcilatex}

\begin{document}

\title{On Size and Power of Heteroskedasticity and Autocorrelation Robust
Tests \thanks{%
Parts of the results in the paper have been presented as the Econometric
Theory Lecture at the International Symposium on Econometric Theory and
Applications, Shanghai, May 19-21, 2012. We are grateful to the Editor and
three referees for helpful comments.}}
\author{David Preinerstorfer and Benedikt M. P\"{o}tscher\thanks{%
Department of Statistics, University of Vienna, Oskar-Morgenstern-Platz 1,
A-1090 Vienna, Austria. E-mail: \{david.preinerstorfer,
benedikt.poetscher\}@univie.ac.at} \\
%EndAName
Department of Statistics, University of Vienna}
\date{ Preliminary version: April 2012\\
First version: January 2013\\
This version: June 2014}
\maketitle

\begin{abstract}
Testing restrictions on regression coefficients in linear models often
requires correcting the conventional F-test for potential heteroskedasticity
or autocorrelation amongst the disturbances, leading to so-called
heteroskedasticity and autocorrelation robust test procedures. These
procedures have been developed with the purpose of attenuating size
distortions and power deficiencies present for the uncorrected F-test. We
develop a general theory to establish positive as well as negative
finite-sample results concerning the size and power properties of a large
class of heteroskedasticity and autocorrelation robust tests. Using these
results we show that nonparametrically as well as parametrically corrected
F-type tests in time series regression models with stationary disturbances
have either size equal to one or nuisance-infimal power equal to zero under
very weak assumptions on the covariance model and under generic conditions
on the design matrix. In addition we suggest an adjustment procedure based
on artificial regressors. This adjustment resolves the problem in many cases
in that the so-adjusted tests do not suffer from size distortions. At the
same time their power function is bounded away from zero. As a second
application we discuss the case of heteroskedastic disturbances.

AMS Mathematics Subject Classification 2010: 62F03, 62J05, 62F35, 62M10,
62M15

Keywords: Size distortion, power deficiency, invariance, robustness,
autocorrelation, heteroskedasticity, HAC, fixed-bandwidth,
long-run-variance, feasible GLS
\end{abstract}

\section{Introduction}

So-called autocorrelation robust tests have received considerable attention
in the econometrics literature in the last two and a half decades. These
tests are Wald-type tests which make use of an appropriate nonparametric
variance estimator that tries to take into account the autocorrelation in
the data. The early papers on such nonparametric variance estimators in
econometrics date from the late 1980s and early 1990s (see, e.g., \cite%
{NW87,NW94}, \cite{A91}, and \cite{A92}) and typically consider consistent
variance estimators. The ideas and techniques underlying this literature
derive from the much earlier literature on spectral estimation and can be
traced back to work by \cite{Bartlett1950}, \cite{Jow1955}, \cite{Hannan1957}%
, and \cite{GR57}, the latter explicitly discussing what would now be called
autocorrelation robust tests and confidence intervals (Section 7.9 of \cite%
{GR57}). For book-length treatments of spectral estimation see the classics 
\cite{Hannan70} or \cite{Anderson71}. Autocorrelation robust tests for the
location parameter also play an important r\^{o}le in the field of
simulation, see, e.g., \cite{Heidel1981} or \cite{FJ2010}. In a similar
vein, so-called heteroskedasticity robust variance estimators and associated
tests have been invented by \cite{E63, E67} and have later been introduced
into the econometrics literature. As mentioned before, the autocorrelation
robust test statistics considered in the above cited econometrics literature
employ consistent variance estimators leading to an asymptotic chi-square
distribution under the null. It soon transpired from Monte Carlo studies
that these tests (using as critical values the quantiles of the asymptotic
chi-square distribution) are often severely oversized in finite samples.
This has led to the proposal to use a test statistic of the same form, but
to obtain the critical values from another (nuisance parameter-free)
distribution which arises as the limiting distribution in an alternative
asymptotic framework ("fixed bandwidth asymptotics") in which the variance
estimator is no longer consistent, see \cite{KVB2000}, \cite{KiefVogl2002,
KV2002,KV2005}. The idea of using "fixed bandwidth asymptotics" can be
traced back to earlier work by \cite{Neave1970}. Monte Carlo studies have
shown that these tests typically are also oversized, albeit less so than the
tests mentioned earlier.\footnote{\label{FN_1}Some of \ the Monte Carlo
studies in the literature initialize the disturbance process with its
stationary distribution, while others use a fixed starting value for
initialization. In both cases size distortions are found for both classes of
tests referred to in the text.} This improvement, however, is often achieved
at the expense of some loss of power. In an attempt to better understand
size and power properties of autocorrelation robust tests, higher-order
asymptotic properties of these tests have been studied (\cite{RV2001}, \cite%
{J04}, \cite{SPJ08,SPJ11}, \cite{ZhangShao2013}).

The first-order as well as the higher-order asymptotic results in the
literature cited above are all \emph{pointwise }asymptotic results in the
sense that they are derived under the assumption of a \emph{fixed}
underlying data-generating process (DGP). Therefore, while these results
tell us something about the limit of the rejection probability, or the rate
of convergence to this limit, for a \emph{fixed} underlying DGP, they do not
necessarily inform us about the \emph{size }of the test or its asymptotic
behavior (e.g., limit of the size as sample size increases) nor about the 
\emph{power function} or its asymptotic behavior. The reason is that the
asymptotic results do not hold uniformly in the underlying DGP under the
typical assumptions on the feasible set of DGPs in this literature. Of
course, one could restrict the set of feasible DGPs in such a way that the
asymptotic results hold uniformly, but this would require the imposition of
unnatural and untenable assumptions on the set of feasible DGPs as will
transpire from the subsequent discussion; cf. also Subsection \ref{Disc}.

In Section \ref{tsr} of the present paper we provide a theoretical
finite-sample analysis of the size and power properties of autocorrelation
robust tests for linear restrictions on the parameters in a linear
regression model with autocorrelated errors. Being finite-sample results,
the findings of the paper apply equally well regardless of whether we fancy
that the variance estimator being used would be consistent or not would
sample size go to infinity. Under a mild assumption on the richeness of the
set of allowed autocorrelation structures in the maintained model, the
results in Section \ref{tsr} imply that in most cases the size of common
autocorrelation robust tests is $1$ or that the worst case power is $0$ (or
both). The richness assumption just mentioned only amounts to requiring that
all correlation structures corresponding to stationary Gaussian
autoregressive processes of order $1$ are allowed for in the model. Compared
to the much wider assumptions on the DGP appearing in the literature on
autocorrelation robust tests cited above, this certainly is a very mild
assumption. [Not including all stationary Gaussian autoregressive models of
order $1$ into the set of feasible disturbance processes appears to be an
unnatural restriction in a theory of autocorrelation robust tests, cf. also
the discussion in Subsection \ref{Disc}.] A similar negative result is
derived for tests that do not use a nonparametric variance estimator but use
a variance estimator derived from a parametric model as well as for tests
based on a feasible generalized least squares estimator (Subsection \ref{GLS}%
). We also show that the just mentioned negative results hold generically in
the sense that, given the linear restrictions to be tested, the set of
design matrices such that the negative results do \emph{not} apply is a
negligible set (Propositions \ref{generic} and \ref{generic_2}).
Furthermore, we provide a positive result in that we isolate conditions (on
the design matrix and on the restrictions to be tested) such that the size
of the test can be controlled. While this result is obtained under the
strong assumption that the set of feasible correlation structures coincides
with the correlation structures of all stationary autoregressive process of
order $1$, it should be noted that the negative results equally well hold
under this parametric correlation model. The positive result just mentioned
is then used to show how for the majority of testing problems
autocorrelation robust tests can be adjusted in such a way that they do not
suffer from the "size equals $1$" and the "worst case power equals $0$"
problem. In Section \ref{Het} we provide an analogous negative result for
heteroskedasticity robust tests and discuss why a (nontrivial) positive
result is not possible.

The above mentioned results for autocorrelation/heteroskedasticity robust
tests can of course also be phrased in terms of properties of the confidence
sets that are obtained from these tests via inversion. For example, the
"size equals one" results for the tests translate into "infimal coverage
probability equals zero" results for the corresponding confidence sets.

We next discuss some related literature. Problems with tests and confidence
sets for the intercept in a linear regression model with autoregressive
disturbances have been pointed out in Section 5.3 of \cite{Dufour1997} (in a
somewhat different setup). These results are specific to testing the
intercept and do not apply to other linear restrictions. This is, in
particular witnessed by our positive results for certain testing problems.
Furthermore, there is a considerable body of literature concerned with the
properties of the \emph{standard} $F$-test (i.e., the $F$-test constructed
without any correction for autocorrelation) in the presence of
autocorrelation, see the references cited in \cite{KKB90} and \cite{Ban00}.
Much of this literature concentrates on the case where the errors follow a
stationary autoregressive process of order $1$. As the correlation in the
errors is not accounted for when considering the standard $F$-test, it is
not too surprising that the standard $F$-test typically shows deplorable
performance for large values of the autocorrelation coefficient $\rho $, see 
\cite{Kr89}, \cite{KKB90}, \cite{Ban00}, and Subsection \ref{Krae} for more
discussion. Section \ref{tsr} of the present paper shows that
autocorrelation robust tests, which despite having built into them a
correction for autocorrelation, exhibit a similarly bad behavior. Finally,
in a different testing problem (the leading case being testing the
correlation of the errors in a spatial regression model) \cite{Mart10} has
studied the power of a class of invariant tests including standard tests
like the Cliff-Ord test and observed somewhat similar results in that the
power of the tests considered typically approaches (as the strength of the
correlation increases) either $0$ or $1$. While his results are similar in
spirit to some of our results, his arguments are unfortunately fraught with
a host of problems. See \cite{Prein2013} for discussion, corrections, and
extensions.

The results in Section \ref{tsr} for autocorrelation robust tests and in
Section \ref{Het} for heteroskedasticity robust tests are derived as special
cases of a more general theory for size and power properties of a larger
class of tests that are invariant under a particular group of affine
transformations. This theory is provided in Section \ref{General}. One of
the mechanisms behind the negative results in the present paper is a
concentration mechanism explained subsequent to Theorem \ref{thmlrv} and in
more detail in Subsection \ref{sec_neg}, cf. also Corollary \ref{CW}. A
second mechanism generating negative results is described in Theorem \ref%
{prop_101}. The theory underlying the positive results mentioned above is
provided in Subsection \ref{sec_pos} and in Theorem \ref{TU_1} as well as
Proposition \ref{enforce_inv}. Furthermore, the results in Section \ref%
{General} allow for covariance structures more general than the ones
discussed in Sections \ref{tsr} and \ref{Het}. For example, from the results
in Section \ref{General} results similar to the ones in Section \ref{tsr}
could be derived for heteroskedasticity/autocorrelation robust tests of
regression coefficients in spatial regression models or in panel data
models; for an overview of heteroskedasticity/autocorrelation robust tests
in these models see \cite{KelPru2007, KelPru2010}, and \cite{Vog2012}. We do
not provide any such results for lack of space. We note that for the
uncorrected standard $F$-test in this setting negative results have been
derived in \cite{K03} and \cite{KH2009}.

\section{The Hypothesis Testing Framework\label{HTFramework}}

Consider the linear regression model 
\begin{equation}
\mathbf{Y}=X\beta +\mathbf{U},  \label{lm}
\end{equation}%
where $X$ is a (real) nonstochastic regressor (design) matrix of dimension $%
n\times k$ and $\beta \in \mathbb{R}^{k}$ denotes the unknown regression
parameter vector. We assume $\limfunc{rank}(X)=k$ and $1\leq k<n$. The $%
n\times 1$ disturbance vector $\mathbf{U}=(\mathbf{u}_{1},\ldots ,\mathbf{u}%
_{n})^{\prime }$ is normally distributed with mean zero and unknown
covariance matrix $\sigma ^{2}\Sigma $, where $0<\sigma ^{2}<\infty $ holds
(and $\sigma $ always denotes the positive square root). The matrix $\Sigma $
varies in a prescribed (nonempty) set $\mathfrak{C}$ of symmetric and
positive definite $n\times n$ matrices.\footnote{%
Although not expressed in the notation, the elements of $\mathbf{Y}$, $X$,
and $\mathbf{U}$ (and even the probability space supporting $\mathbf{Y}$ and 
$\mathbf{U}$) may depend on sample size $n$. Furthermore, the obvious
dependence of$\ \mathfrak{C}$ on $n$ will also not be shown in the notation.
[Note that $\mathfrak{C}$ depends on $n$ even if it is induced by a
covariance model for the entire process $(\mathbf{u}_{t})_{t\in \mathbb{N}}$
that does not depend on $n$.]} Throughout the paper we make the assumption
that $\mathfrak{C}$ is such that $\sigma ^{2}$ and $\Sigma \in \mathfrak{C}$
can be uniquely determined from $\sigma ^{2}\Sigma $. [For example, if the
first diagonal element of each $\Sigma \in \mathfrak{C}$ equals $1$ this is
satisfied; alternatively, if the largest diagonal element or the trace of
each $\Sigma \in \mathfrak{C}$ is normalized to a fixed constant, $\mathfrak{%
C}$ has this property.] Of course, this assumption entails little loss of
generality and can, if necessary, always be achieved by a suitable
reparameterization of $\sigma ^{2}\Sigma $.

The linear model described above induces a collection of distributions on $%
\mathbb{R}^{n}$, the sample space of $\mathbf{Y}$. Denoting a Gaussian
probability measure with mean $\mu \in \mathbb{R}^{n}$ and (possibly
singular) covariance matrix $\Phi $ by $P_{\mu ,\Phi }$ and setting $%
\mathfrak{M}=\text{span}(X)$, the induced collection of distributions is
given by 
\begin{equation}
\left\{ P_{\mu ,\sigma ^{2}\Sigma }:\mu \in \mathfrak{M},0<\sigma
^{2}<\infty ,\Sigma \in \mathfrak{C}\right\} .  \label{lm2}
\end{equation}%
Note that each $P_{\mu ,\sigma ^{2}\Sigma }$ in (\ref{lm2}) is absolutely
continuous with respect to (w.r.t.) Lebesgue measure on $\mathbb{R}^{n}$,
since every $\Sigma \in \mathfrak{C}$ is positive definite by assumption. We
consider the problem of testing a linear (better: affine) restriction on the
parameter vector $\beta \in \mathbb{R}^{k}$, namely the problem of testing
the null $R\beta =r$ versus the alternative $R\beta \neq r$, where $R$ is a $%
q\times k$ matrix of rank $q$, $q\geq 1$, and $r\in \mathbb{R}^{q}$. To be
more precise and to emphasize that the testing problem is in fact a compound
one, the testing problem needs to be written as%
\begin{equation}
H_{0}:R\beta =r,0<\sigma ^{2}<\infty ,\Sigma \in \mathfrak{C}~~\text{ vs. }%
~~H_{1}:R\beta \neq r,0<\sigma ^{2}<\infty ,\Sigma \in \mathfrak{C}.
\label{testing problem 0}
\end{equation}%
This is important to stress, because size and power properties of tests
critically depend on nuisance parameters and, in particular, on the
complexity of $\mathfrak{C}$. Define the affine space 
\begin{equation*}
\mathfrak{M}_{0}=\left\{ \mu \in \mathfrak{M}:\mu =X\beta \text{ and }R\beta
=r\right\}
\end{equation*}%
and let 
\begin{equation*}
\mathfrak{M}_{1}=\mathfrak{M}\backslash \mathfrak{M}_{0}=\left\{ \mu \in 
\mathfrak{M}:\mu =X\beta \text{ and }R\beta \neq r\right\} .
\end{equation*}%
Adopting these definitions, the above testing problem can also be written as 
\begin{equation}
H_{0}:\mu \in \mathfrak{M}_{0},0<\sigma ^{2}<\infty ,\Sigma \in \mathfrak{C}%
~~\text{ vs. }~~H_{1}:\mu \in \mathfrak{M}_{1},0<\sigma ^{2}<\infty ,\Sigma
\in \mathfrak{C}.  \label{testing problem}
\end{equation}

Two remarks are in order: First, the Gaussiantiy assumption is not really a
restriction for the negative results in the paper, since they hold a
fortiori in any enlarged model that allows not only for Gaussian but also
for non-Gaussian disturbances. Furthermore, a large portion of the results
in the paper (positive or negative) continues to hold for certain classes of
non-Gaussian distributions such as, e.g., elliptical distributions, see
Subsection \ref{other_distr}. Second, if $X$ were allowed to be stochastic
but independent of $\mathbf{U}$, the results of the paper apply to size and
power conditional on $X$. Because $X$ is observable, one could then argue in
the spirit of conditional inference (see, e.g., \cite{RO1979}) that
conditional size and power and not their unconditional counterparts are the
more relevant characteristics of a test.

Recall that a (randomized) test is a Borel-measurable function $\varphi $
from the sample space $\mathbb{R}^{n}$ to $[0,1]$. If $\varphi =\boldsymbol{1%
}_{W}$, the set $W$ is called the rejection region of the test. As usual,
the size of a test $\varphi $ is the supremum over all rejection
probabilities under the null hypothesis $H_{0}$ and thus is given by $%
\sup_{\mu \in \mathfrak{M}_{0}}\sup_{0<\sigma ^{2}<\infty }\sup_{\Sigma \in 
\mathfrak{C}}E_{\mu ,\sigma ^{2}\Sigma }\left( \varphi \right) $ where $%
E_{\mu ,\sigma ^{2}\Sigma }$ refers to expectation under the probability
measure $P_{\mu ,\sigma ^{2}\Sigma }$.

Throughout the paper we shall always reserve the symbol $\hat{\beta}(y)$ for 
$\left( X^{\prime }X\right) ^{-1}X^{\prime }y$, where $X$ is the design
matrix appearing in (\ref{lm}) and $y\in \mathbb{R}^{n}$. Furthermore,
random vectors and random variables are always written in bold capital and
bold lower case letters, respectively. Lebesgue measure on $\mathbb{R}^{n}$
will be denoted by $\lambda _{\mathbb{R}^{n}}$, whereas Lebesgue measure on
an affine subspace $\mathcal{A}$ of $\mathbb{R}^{n}$ (but viewed as a
measure on the Borel-sets of $\mathbb{R}^{n}$) will be denoted by $\lambda _{%
\mathcal{A}}$, with zero-dimensional Lebesgue measure being interpreted as
point mass. We shall write $\limfunc{int}(A)$, $\limfunc{cl}(A)$, and $%
\limfunc{bd}(A)$ for the interior, closure, and boundary of a set $%
A\subseteq \mathbb{R}^{n}$, respectively, taken with respect to the
Euclidean topology. The Euclidean norm is denoted by $\left\Vert \cdot
\right\Vert $, while $d(x,A)$ denotes the Euclidean distance of the point $%
x\in \mathbb{R}^{n}$ to the set $A\subseteq \mathbb{R}^{n}$. Let $B^{\prime
} $ denote the transpose of a matrix $B$ and let $\limfunc{span}\left(
B\right) $ denote the space spanned by the columns of $B$. For a linear
subspace $\mathcal{L}$ of $\mathbb{R}^{n}$ we let $\mathcal{L}^{\bot }$
denote its orthogonal complement and we let $\Pi _{\mathcal{L}}$ denote the
orthogonal projection onto $\mathcal{L}$. For a vector $x$ in Euclidean
space we define the symbol $\left\langle x\right\rangle $ to denote $\pm x$
for $x\neq 0$, the sign being chosen in such a way that the first nonzero
component of $\left\langle x\right\rangle $ is positive, and we set $%
\left\langle 0\right\rangle =0$. The $j$-th standard basis vector in $%
\mathbb{R}^{n}$ is denoted by $e_{j}(n)$. The set of real matrices of
dimension $m\times n$ is denoted by $\mathbb{R}^{m\times n}$. We also
introduce the following terminology.

\begin{definition}
\label{CD}Let $\mathfrak{C}$ be a set of symmetric and positive definite $%
n\times n$ matrices. An $l$-dimensional linear subspace $\mathcal{Z}$ of $%
\mathbb{R}^{n}$ with $0\leq l<n$ is called a \emph{concentration space} of $%
\mathfrak{C}$, if there exists a sequence $(\Sigma _{m})_{m\in \mathbb{N}}$
in $\mathfrak{C}$, such that $\Sigma _{m}\rightarrow \bar{\Sigma}$ and $%
\limfunc{span}(\bar{\Sigma})=\mathcal{Z}$.
\end{definition}

While we shall in the sequel often refer to $\mathfrak{C}$ as the covariance
model, one should keep in mind that the set of all feasible covariance
matrices corresponding to (\ref{lm2})\ is given by $\left\{ \sigma
^{2}\Sigma :0<\sigma ^{2}<\infty ,\Sigma \in \mathfrak{C}\right\} $. In this
context we note that two covariance models $\mathfrak{C}$ and $\mathfrak{C}%
^{\ast }$ can be equivalent in the sense of giving rise to the same set of
feasible covariance matrices, but need not have the same concentration
spaces.\footnote{%
In applying the general results in Section \ref{sec_neg} or Corollary \ref%
{CW} to a particular problem some skill in choosing between equivalent $%
\mathfrak{C}$ and $\mathfrak{C}^{\ast }$ may thus be required as one choice
for $\mathfrak{C}$ may lead to more interesting results than does another
choice.}

\section{Size and Power of Tests of Linear Restrictions in Regression Models
with Autocorrelated Disturbances\label{tsr}}

In this section we investigate size and power properties of autocorrelation
robust tests that have been designed for use in case of stationary
disturbances. Studies of the properties of such tests in the literature (%
\cite{NW87,NW94}, \cite{A91}, \cite{A92}, \cite{KVB2000}, \cite%
{KiefVogl2002,KV2002,KV2005}, \cite{Jan2002, J04}, \cite{SPJ08, SPJ11})
maintain assumptions that allow for nonparametric models for the spectral
distribution of the disturbances. For example, a typical nonparametric model
results from assuming that the disturbance vector consists of $n$
consecutive elements of a weakly stationary process with spectral density
equal to%
\begin{equation*}
f(\omega )=(2\pi )^{-1}\left\vert \sum_{j=0}^{\infty }c_{j}\exp (-\iota
j\omega )\right\vert ^{2},
\end{equation*}%
where the coefficients $c_{j}$ are not all equal to zero and, for $\xi \geq
0 $ a given number, satisfy the summability condition $\sum_{j=0}^{\infty
}j^{\xi }\left\vert c_{j}\right\vert <\infty $. Here $\iota $ denotes the
imaginary unit. Let $\mathfrak{F}_{\xi }$ denote the collection of all such
spectral densities $f$. The corresponding covariance model $\mathfrak{C}%
_{\xi }$ is then given by $\left\{ \Sigma \left( f\right) :f\in \mathfrak{F}%
_{\xi }\right\} $ where $\Sigma \left( f\right) $ is the $n\times n$
correlation matrix%
\begin{equation*}
\Sigma \left( f\right) =\left( \int_{-\pi }^{\pi }\exp \left( -\iota \omega
\left( i-j\right) \right) f(\omega )d\omega \left/ \int_{-\pi }^{\pi
}f(\omega )d\omega \right. \right) _{i,j=1}^{n}.
\end{equation*}%
Certainly, $\mathfrak{F}_{\xi }$ contains all spectral densities of
stationary autoregressive moving average models of arbitrary large order.
Hence, the following assumption on the covariance model $\mathfrak{C}$ that
we shall impose for most results in this section is very mild and is
satisfied by the typical nonparametric model allowed for in the above
mentioned literature. It certainly covers the case where $\mathfrak{C}=%
\mathfrak{C}_{\xi }$ or where $\mathfrak{C}$ corresponds to an
autoregressive model of order $p\geq 1$.

\begin{assumption}
\label{AAR(1)} $\mathfrak{C}_{AR(1)}\subseteq \mathfrak{C}$.
\end{assumption}

Here $\mathfrak{C}_{AR(1)}$ denotes the set of correlation matrices
corresponding to $n$ successive elements of a stationary autoregressive
processes of order $1$, i.e., $\mathfrak{C}_{AR(1)}=\left\{ \Lambda (\rho
):\rho \in (-1,1)\right\} $ where the $\left( i,j\right) $-th entry in the $%
n\times n$ matrix $\Lambda (\rho )$ is given by $\rho ^{|i-j|}$. As hinted
at in the introduction, parameter values $\left( \mu ,\sigma ^{2},\Sigma
\right) $ with $\Sigma =\Lambda (\rho )$ where $\rho $ gets close to $\pm 1$
and $\sigma ^{2}$ is constant will play an important r\^{o}le as they will
be instrumental for establishing the bad size and power properties of the
tests presented below.\footnote{%
If we parameterized in terms of $\rho $ and the innovation variance $\sigma
_{\varepsilon }^{2}=\sigma ^{2}\left( 1-\rho ^{2}\right) $, this would
correspond to $\sigma _{\varepsilon }^{2}\rightarrow 0$ at the appropriate
rate.} We want to stress here that, as $\rho \rightarrow \pm 1$, the
corresponding stationary process does \emph{not} converge to an integrated
process but rather to a harmonic process.\footnote{%
To see this note that the covariance function of the disturbances converges
to that of a (very simple) harmonic process as $\rho \rightarrow \pm 1$. In
view of Gaussianity, this implies convergence of finite-dimensional
distributions and hence weak convergence of the entire process, cf. \cite%
{Billing}, p.19.} But see also Remark B(i) in Subsection \ref{Disc} for a
discussion that holding $\sigma ^{2}$ constant is actually not a restriction.

For later use we note that under Assumption \ref{AAR(1)} the matrices $%
e_{+}e_{+}^{\prime }$ and $e_{-}e_{-}^{\prime }$ are limit points of the
covariance model $\mathfrak{C}$ where $e_{+}=(1,\ldots ,1)^{\prime }$ and $%
e_{-}=(-1,1,\ldots ,\left( -1\right) ^{n})^{\prime }$ are $n\times 1$
vectors (since $\Lambda (\rho _{m})$ converges to $e_{+}e_{+}^{\prime }$ ($%
e_{-}e_{-}^{\prime }$, respectively) if $\rho _{m}\rightarrow 1$ ($\rho
_{m}\rightarrow -1$, respectively)). Other singular limit points of $%
\mathfrak{C}$ are possible, but $e_{+}e_{+}^{\prime }$ and $%
e_{-}e_{-}^{\prime }$ are the only singular limit points of $\mathfrak{C}%
_{AR(1)}$.

\subsection{Some preliminary results for the location model\label{prelim}}

Before we present the results for common nonparametrically based
autocorrelation robust tests in the next subsection and for parametrically
based tests in Subsection \ref{GLS}, it is perhaps helpful to gain some
understanding for these results from a very special case, namely from the
location model. We should, however, warn the reader that only some, but not
all, phenomena that we shall later observe in the case of a general
regression model will occur in the case of the location model, because it
represents an oversimplification of the general case. Hence, while gaining
intuition in the location model is certainly helpful, this intuition does
not paint a complete and faithful picture of the situation in a general
regression model.

Consider now the location model, i.e., model (\ref{lm}) with $k=1$ and $%
X=e_{+}$. Let Assumption \ref{AAR(1)} hold and assume that we want to test $%
\beta =\beta _{0}$ against the alternative $\beta \neq \beta _{0}$. Consider
the commonly used autocorrelation robust test statistic%
\begin{equation*}
\tau _{loc}(y)=(\hat{\beta}(y)-\beta _{0})^{2}/\hat{\omega}^{2}\left(
y\right)
\end{equation*}%
where $\hat{\beta}(y)$ is the arithmetic mean $n^{-1}e_{+}^{\prime }y$ and
where $\hat{\omega}^{2}\left( y\right) $ is one of the usual autocorrelation
robust estimators for the variance of the least squares estimator. As usual,
the null hypothesis is rejected if $\tau _{loc}(y)\geq C$ for some
user-specified critical value $C$ satisfying $0<C<\infty $. For definiteness
of the discussion assume that one has chosen the Bartlett estimator,
although any estimator based on weights satisfying Assumption \ref{AW} given
below could be used instead. It is then not difficult to see (cf. Lemma \ref%
{LRVPD} given below) that $\hat{\omega}^{2}\left( y\right) $ is positive,
and hence $\tau _{loc}(y)$ is well-defined, except when $y$ is proportional
to $e_{+}$; in this case we set $\tau _{loc}(y)$ equal to $0$, which, of
course, is a completely arbitrary choice, but has no effect on the rejection
probability of the resulting test as the event that $y$ is proportional to $%
e_{+}$ has probability zero under all the distributions in the model.

Consider now the points $\left( \beta _{0},1,\Lambda (\rho )\right) $ in the
null hypothesis, where we have set $\sigma ^{2}=1$ for simplicity and where
we let $\rho \in (-1,1)$ converge to $1$. Writing $P_{\rho }$ for $%
P_{e_{+}\beta _{0},\Lambda (\rho )}$, i.e., for the distribution of the
data, observe that under $P_{\rho }$ the distribution of $\hat{\beta}%
(y)-\beta _{0}=n^{-1}e_{+}^{\prime }y-\beta _{0}$ is $N\left(
0,n^{-2}e_{+}^{\prime }\Lambda (\rho )e_{+}\right) $. Noting that $\Lambda
(\rho )\rightarrow e_{+}e_{+}^{\prime }$ for $\rho \rightarrow 1$, we see
that under $P_{\rho }$ the distribution of the numerator of the test
statistic converges weakly for $\rho \rightarrow 1$ to a chi-square
distribution with one degree of freedom. Concerning the denominator, observe
that $\hat{\omega}^{2}\left( y\right) $ is a quadratic form in the residual
vector $y-e_{+}\hat{\beta}(y)=\left( I_{n}-n^{-1}e_{+}e_{+}^{\prime }\right)
y$, this vector being distributed under $P_{\rho }$ as $N\left( 0,A\left(
\rho \right) \right) $ with $A\left( \rho \right) =\left(
I_{n}-n^{-1}e_{+}e_{+}^{\prime }\right) \Lambda (\rho )\left(
I_{n}-n^{-1}e_{+}e_{+}^{\prime }\right) $. Now for $\rho \rightarrow 1$ we
see that $A\left( \rho \right) $ converges to the zero matrix, and therefore
the distribution of the residual vector under $P_{\rho }$ converges to
pointmass at zero. Consequently, the distribution of the quadratic form $%
\hat{\omega}^{2}\left( y\right) $ under $P_{\rho }$ collapses to pointmass
at zero. But this shows that all of the mass of the distribution of the test
statistic $\tau _{loc}$ under $P_{\rho }$ escapes to infinity for $\rho
\rightarrow 1$, entailing convergence of the rejection probabilities $%
P_{\rho }\left( \tau _{loc}(y)\geq C\right) $ to $1$, although the
distributions $P_{\rho }$ correspond to points $\left( \beta _{0},1,\Lambda
(\rho )\right) $ in the null hypothesis. This of course then implies that
the size of the test equals $1$.

In a similar vein, consider the points $\left( \beta _{0},1,\Lambda (\rho
)\right) $ in the null hypothesis where now $\rho $ converges to $-1$. Note
that $P_{\rho }$ then converges weakly to $N\left( e_{+}\beta
_{0},e_{-}e_{-}^{\prime }\right) $ which is the distribution of $e_{+}\beta
_{0}+e_{-}\mathbf{g}$ where $\mathbf{g}$ is a standard normal random
variable. Similar computations as before show that under $P_{\rho }$ the
distribution of the numerator of the test statistic now converges weakly to
the distribution of $n^{-2}\left( e_{+}^{\prime }e_{-}\right) ^{2}\mathbf{g}%
^{2}$ and that the distribution of the residual vector converges weakly to
the distribution of $\left( I_{n}-n^{-1}e_{+}e_{+}^{\prime }\right) e_{-}%
\mathbf{g}$, the weak convergence occurring jointly. Because of $\hat{\omega}%
^{2}\left( y\right) =\hat{\omega}^{2}\left( \left(
I_{n}-n^{-1}e_{+}e_{+}^{\prime }\right) y\right) $, it follows from the
continuous mapping theorem that the distribution of the denominator of the
test statistic under $P_{\rho }$ converges weakly to the distribution of $%
\hat{\omega}^{2}\left( \left( I_{n}-n^{-1}e_{+}e_{+}^{\prime }\right) e_{-}%
\mathbf{g}\right) $ (and convergence is joint with the numerator). Note that 
$\hat{\omega}^{2}\left( \left( I_{n}-n^{-1}e_{+}e_{+}^{\prime }\right) e_{-}%
\mathbf{g}\right) $ equals $\hat{\omega}^{2}\left(
e_{-}-n^{-1}e_{+}e_{+}^{\prime }e_{-}\right) \mathbf{g}^{2}$ by homogeneity
of $\hat{\omega}^{2}$. Now, if sample size $n$ is even, we see that $%
e_{+}^{\prime }e_{-}=0$, entailing that the distribution of the test
statistic under $P_{\rho }$ converges to pointmass at zero for $\rho
\rightarrow -1$ (since $\hat{\omega}^{2}\left( e_{-}\right) \mathbf{g}^{2}$
is almost surely positive). As a consequence, if sample size $n$ is even the
rejection probabilities $P_{\rho }\left( \tau _{loc}(y)\geq C\right) $
converge to zero as $\rho \rightarrow -1$ since $C>0$. Next consider the
case where $n$ is odd. Then $e_{+}^{\prime }e_{-}=-1$ and the limiting
distribution of the test statistic is pointmass at $n^{-2}\hat{\omega}%
^{-2}\left( e_{-}+n^{-1}e_{+}\right) $ which is \emph{positive }(and is
well-defined since $\hat{\omega}^{2}\left( e_{-}+n^{-1}e_{+}\right) >0$ as $%
e_{-}+n^{-1}e_{+}$ is not proportional to $e_{+}$). Hence, if $n$ is odd, we
learn that the rejection probabilities $P_{\rho }\left( \tau _{loc}(y)\geq
C\right) $ converge to zero or one as $\rho \rightarrow -1$ depending on
whether $C$ satisfies $C>n^{-2}\hat{\omega}^{-2}\left(
e_{-}+n^{-1}e_{+}\right) $ or $C<n^{-2}\hat{\omega}^{-2}\left(
e_{-}+n^{-1}e_{+}\right) $.

In summary we have learned that the size of the autocorrelation robust test
in the location model is always equal to one, an "offending" sequence
leading to this result being, e.g., $\left( \beta _{0},1,\Lambda (\rho
)\right) $ with $\rho \rightarrow 1$. We have also learned that if $n$ is
even, or if $n$ is odd and the critical value $C$ is larger than $n^{-2}\hat{%
\omega}^{-2}\left( e_{-}+n^{-1}e_{+}\right) $, the test is severely biased
as the rejection probabilities get arbitrarily close to zero in certain
parts of the null hypothesis; of course, this implies dismal power
properties of the test in certain parts of the alternative hypothesis. The
"offending" sequence in this case being again $\left( \beta _{0},1,\Lambda
(\rho )\right) $, but now with $\rho \rightarrow -1$. It is worth noting
that in the case where $n$ is odd and $C<n^{-2}\hat{\omega}^{-2}\left(
e_{-}+n^{-1}e_{+}\right) $ holds, this "offending" sequence does not inform
us about biasedness of the test, but rather provides a second sequence along
which the null rejection probabilities converge to $1$. We note here also
that due to certain invariance properties of the test statistic in fact any
sequence $\left( \beta _{0},\sigma ^{2},\Lambda (\rho )\right) $ with $\rho
\rightarrow \pm 1$ and \emph{arbitrary} behavior of $\sigma ^{2}$, $0<\sigma
^{2}<\infty $, is an "offending" sequence in the same way as $\left( \beta
_{0},1,\Lambda (\rho )\right) $ is. The results obtained above heavily
exploit the fact that $\rho $ can be chosen arbitrarily close to $\pm 1$
(entailing that $\Lambda (\rho )$ becomes singular in the limit). To what
extent an assumption \emph{restricting} the parameter space $\mathfrak{C}$
in such a way, that the matrices $\Sigma \in \mathfrak{C}$ do not have limit
points that are singular, can provide an escape route avoiding the size and
power problems observed above is discussed in Subsection \ref{Disc}.

We would like to stress once more that not all cases that can arise in a
general regression model (see Theorems \ref{thmlrv} and \ref{TU_2}) appear
already in the location model discussed above. For example, for other design
matrices and/or linear hypothesis to be tested, the roles of the "offending"
sequences $\rho \rightarrow 1$ and $\rho \rightarrow -1$ may be reversed, or
both sequences may lead to rejection probabilities converging to $1$, etc.
Furthermore, there exist cases where the above mentioned sequences are not
"offending" at all, see Theorem \ref{TU_2}.

We close this subsection with some comments on a heuristic argument that
tries to explain the above results. The argument is as follows: Suppose one
enlarges the model by adjoining the limit points $\left( \beta ,\sigma
^{2},\Lambda (\rho )\right) $ with $\rho =1$. Then the test problem now also
contains the problem of testing $\beta =\beta _{0}$ against $\beta \neq
\beta _{0}$ in the family $\mathcal{P}_{1}=\left\{ P_{e_{+}\beta ,\sigma
^{2}\Lambda (1)}:\beta \in \mathbb{R},0<\sigma ^{2}<\infty \right\} $ as a
subproblem.\footnote{%
We stress that the parameters $\beta $ and $\sigma ^{2}$ are identifiable in
the model $\mathcal{P}_{1}$.} Because of $\Lambda (1)=e_{+}e_{+}^{\prime }$,
this subproblem is equivalent to testing $\beta =\beta _{0}$ against $\beta
\neq \beta _{0}$ in the family $\left\{ N\left( \beta ,\sigma ^{2}\right)
:\beta \in \mathbb{R},0<\sigma ^{2}<\infty \right\} $. Obviously, there is
no "reasonable" test for the latter testing problem, and thus for the test
problem in the family $\mathcal{P}_{1}$. The intuitively appealing argument
now is that the absence of a "reasonable" test in the family $\mathcal{P}%
_{1} $ should necessarily imply trouble for tests, and in particular for
autocorrelation robust tests, in the original test problem in the family $%
\mathcal{P}_{orig}=\left\{ P_{e_{+}\beta ,\sigma ^{2}\Lambda (\rho )}:\beta
\in \mathbb{R},0<\sigma ^{2}<\infty ,\left\vert \rho \right\vert <1\right\} $
whenever $\rho $ is close to one. While this argument has some appeal, it
seems to rest on some sort of tacit continuity assumption regarding the
rejection probabilities at the point $\rho =1$, which is unjustified as we
now show: If $\varphi $ is any test, i.e., is a measurable function on $%
\mathbb{R}^{n}$ with values in $\left[ 0,1\right] $, then any test $\varphi
^{\ast }$ that coincides with $\varphi $ on $\mathbb{R}^{n}\backslash 
\limfunc{span}\left( e_{+}\right) $ has the same rejection probabilities in
the model $\mathcal{P}_{orig}$ as has $\varphi $; and any test $\varphi
^{\ast \ast }$ that coincides with $\varphi $ on $\limfunc{span}\left(
e_{+}\right) $ has the same rejection probabilities in the model $\mathcal{P}%
_{1}$ as has $\varphi $. This is so since the distributions in $\mathcal{P}%
_{1}$ are concentrated on $\limfunc{span}\left( e_{+}\right) $, whereas this
set is a null set for the distributions in $\mathcal{P}_{orig}$. As a
consequence, the sequence of rejection probabilities of a test $\varphi $
under $P_{\rho }$ with $\rho <1$ but $\rho \rightarrow 1$\ is unaffected by
modifying the test on $\limfunc{span}\left( e_{+}\right) $, whereas such a
modification will substantially affect the rejection probability under $%
\mathcal{P}_{1}$ (e.g., we can make it equal to $0$ or to $1$ by suitable
modifications of $\varphi $ on $\limfunc{span}\left( e_{+}\right) $). This,
of course, then shows that rejection probabilities of a test $\varphi $ will
in general not be continuous at the point $\rho =1$. Put differently, in the
case of the test statistic $\tau _{loc}$ the rejection probabilities under $%
\mathcal{P}_{1}$ depend only on the (completely arbitrary) way $\tau _{loc}$
is defined on $\limfunc{span}\left( e_{+}\right) $, while the rejection
probabilities under $\mathcal{P}_{orig}$ are completely unaffected by the
way $\tau _{loc}$ is defined on $\limfunc{span}\left( e_{+}\right) $. Hence,
any attempt to obtain information on the behavior of $P_{\rho }\left( \tau
_{loc}(y)\geq C\right) $ for $\rho \rightarrow 1$ from the rejection
probabilities of the test statistic under the limiting family $\mathcal{P}%
_{1}$ \emph{alone} is necessarily futile. [At the heart of the matter lies
here the fact, that while the distributions in $\mathcal{P}_{1}$ can be
approximated by distributions in $\mathcal{P}_{orig}$ in the sense of weak
convergence, this has little consequences for closeness of rejection
probabilities in general, especially since the distributions in $\mathcal{P}%
_{1}$ and $\mathcal{P}_{orig}$ are orthogonal and the tests one is
interested in are not continuous everywhere.] In a similar way one could try
to predict the behavior of the rejection probabilities for $\rho \rightarrow
-1$ from the limiting experiment corresponding to the family $\mathcal{P}%
_{-1}=\left\{ P_{e_{+}\beta ,\sigma ^{2}\Lambda (-1)}:\beta \in \mathbb{R}%
,0<\sigma ^{2}<\infty \right\} $, the argument now being as follows: Since $%
n>1$ is always assumed, the parameter $\beta $ can be estimated without
error in the model $\mathcal{P}_{-1}$. Thus, we can test the hypothesis $%
\beta =\beta _{0}$ without committing any error, seemingly suggesting that $%
P_{\rho }\left( \tau _{loc}(y)\geq C\right) $ should converge to zero for $%
\rho \rightarrow -1$. However, as we have shown above, $P_{\rho }\left( \tau
_{loc}(y)\geq C\right) $ does \emph{not} always converge to zero for $\rho
\rightarrow -1$, namely it converges to one if $n$ is odd and $C<n^{-2}\hat{%
\omega}^{-2}\left( e_{-}+n^{-1}e_{+}\right) $ holds.\footnote{%
Note that the arbitrariness in the definition of the test statistic $\tau
_{loc}(y)$ on $\limfunc{span}\left( e_{+}\right) $ has no effect on the
rejection probabilities under the experiment $\mathcal{P}_{-1}$. Hence, one
could hope to derive the behavior of \ $P_{\rho }\left( \tau
_{loc}(y)>C\right) $ for $\rho \rightarrow -1$ by first computing the
rejection probability in the limiting experiment $\mathcal{P}_{-1}$ and then
by arguing that the map $\rho \mapsto P_{\rho }\left( \tau
_{loc}(y)>C\right) $ \ is continuous at $\rho =-1$. However, this would just
amount to reproducing our direct argument given earlier.} Summarizing we see
that, while the heuristic arguments are interesting, they do not really
capture the underlying mechanism; cf. the discussion following Theorem \ref%
{thmlrv}. Furthermore, the heuristic arguments just discussed are specific
to the location model (i.e., to the case $X=e_{+}$), whereas severe size
distortions can also arise in more general regression models as will be
shown in the next subsection.

\subsection{Nonparametrically based autocorrelation robust tests \label{HAR}}

Commonly used autocorrelation robust tests for the null hypothesis $H_{0}$
given by (\ref{testing problem 0}) are based on test statistics of the form $%
(R\hat{\beta}(y)-r)^{\prime }\hat{\Omega}^{-1}\left( y\right) (R\hat{\beta}%
(y)-r)$, with the statistic typically being undefined if $\hat{\Omega}\left(
y\right) $ is singular. Here%
\begin{equation}
\hat{\Omega}\left( y\right) =nR(X^{\prime }X)^{-1}\hat{\Psi}(y)(X^{\prime
}X)^{-1}R^{\prime }  \label{omega}
\end{equation}%
and $\hat{\Psi}$ is a nonparametric estimator for $n^{-1}\mathbb{E}%
(X^{\prime }\mathbf{UU}^{\prime }X)$. The type of estimator $\hat{\Psi}$ we
consider in this subsection is obtained as a weighted sum of sample
autocovariances of $\hat{v}_{t}(y)=\hat{u}_{t}(y)x_{t\mathbf{\cdot }%
}^{\prime }$, where $\hat{u}_{t}(y)$ is the $t$-th coordinate of the least
squares residual vector $\hat{u}(y)=y-X\hat{\beta}(y)$ and $x_{t\cdot }$
denotes the $t$-th row vector of $X$. That is%
\begin{equation}
\hat{\Psi}(y)=\hat{\Psi}_{w}(y)=\sum\limits_{j=-(n-1)}^{n-1}w(j,n)\hat{\Gamma%
}_{j}(y)  \label{lrve}
\end{equation}%
for every $y\in \mathbb{R}^{n}$ with $\hat{\Gamma}_{j}(y)=n^{-1}%
\sum_{t=j+1}^{n}\hat{v}_{t}(y)\hat{v}_{t-j}(y)^{\prime }$ if $j\geq 0$ and $%
\hat{\Gamma}_{j}\left( y\right) =\hat{\Gamma}_{-j}(y)^{\prime }$ else. The
associated estimator $\hat{\Omega}$ will be denoted by $\hat{\Omega}_{w}$.
We make the following assumption on the weights.

\begin{assumption}
\label{AW} The weights $w(j,n)$ for $j=-(n-1),\ldots ,n-1$ are
data-independent and satisfy $w(0,n)=1$ as well as $w\left( -j,n\right)
=w\left( j,n\right) $. Furthermore, the symmetric $n\times n$ Toeplitz
matrix $\mathcal{W}_{n}$ with elements $w\left( i-j,n\right) $ is positive
definite.\footnote{%
For the case where $\mathcal{W}_{n}$ is only nonnegative definite see
Subsection \ref{alternative}.}
\end{assumption}

The positive definiteness assumption on $\mathcal{W}_{n}$ is weaker than the
frequently employed assumption that the Fourier transform $w\dag \left(
\omega \right) $ of the weights is nonnegative for all $\omega \in \lbrack
-\pi ,\pi ]$.\footnote{%
Note that the quadratic form $\alpha ^{\prime }\mathcal{W}_{n}\alpha $ can
be represented as $\int_{-\pi }^{\pi }\left\vert \sum_{j=1}^{n}\alpha
_{j}\exp \left( \iota j\omega \right) \right\vert ^{2}w\dag \left( \omega
\right) d\omega $. If $w\dag \left( \omega \right) \geq 0$ for all $\omega
\in \lbrack -\pi ,\pi ]$ is assumed, the integrand is nonnegative; and if $%
\alpha \neq 0$ it is positive almost everywhere (since it is then a product
of two nontrivial trigonometric polynomials).} It certainly implies that $%
\hat{\Psi}_{w}(y)$, and hence $\hat{\Omega}_{w}\left( y\right) $, is always
nonnegative definite, but it will allow us to show more, see Lemma \ref%
{LRVPD} below. In many applications the weights take the form $%
w(j,n)=w_{0}\left( |j|/M_{n}\right) $, where the lag-window $w_{0}$ is an
even function with $w_{0}(0)=1$ and where $M_{n}>0$ is a truncation lag
(bandwidth) parameter. In this case the first part of the above assumption
means that we are considering deterministic bandwidths only (as is the case,
e.g., in \cite{NW87}, Sections 3-5 of \cite{A91}, \cite{H92}, \cite%
{KV2002,KV2005}, and \cite{Jan2002,J04}). Extensions of the results in this
subsection to data-dependent bandwidth choices and prewhitening will be
discussed in \cite{Prein2013b}. Assumption \ref{AW} is known to be
satisfied, e.g., for the (modified) Bartlett, Parzen, or the Quadratic
Spectral lag-window, but is not satisfied, e.g., for the rectangular
lag-window (with $M_{n}>1$).\footnote{%
The estimator in \cite{KKW91} coincides with ($n$ times) the estimator given
by (\ref{omega}) if the rectangular lag-window is used and $R=I_{k}$.} See 
\cite{Anderson71} or \cite{Hannan70} for more discussion. It is also
satisfied for many exponentiated lag-windows as used in \cite{SPJ06,SPJ07}
and \cite{SPJ11}.

In the typical asymptotic analysis of this sort of tests in the literature
the event where the estimator $\hat{\Omega}_{w}$ is singular is
asymptotically negligible (as $\hat{\Omega}_{w}$ converges to a positive
definite or almost surely positive definite matrix), and hence there is no
need to be specific about the definition of the test statistic on this
event. However, if one is concerned with finite-sample properties, one has
to think about the definition of the test statistic also in the case where $%
\hat{\Omega}_{w}\left( y\right) $ is singular. We thus define the test
statistic as follows:\footnote{%
Some authors (e.g., \cite{KV2002,KV2005}) choose to normalize also by $q$,
the number of restrictions to be tested. This is of course immaterial as
long as one accordingly adjusts the critical vlaue.} 
\begin{equation}
T(y)=\left\{ 
\begin{array}{cc}
(R\hat{\beta}(y)-r)^{\prime }\hat{\Omega}_{w}^{-1}\left( y\right) (R\hat{%
\beta}(y)-r) & \text{if }\det \hat{\Omega}_{w}\left( y\right) \neq 0, \\ 
0 & \text{if }\det \hat{\Omega}_{w}\left( y\right) =0.%
\end{array}%
\right.  \label{tslrv}
\end{equation}%
Of course, assigning the test statistic $T$ the value zero on the set where $%
\hat{\Omega}_{w}\left( y\right) $ is singular is arbitrary. However, it will
be irrelevant for size and power properties of the test \emph{provided} we
can ensure that the set of $y\in \mathbb{R}^{n}$ for which $\det \hat{\Omega}%
_{w}\left( y\right) =0$ holds is a $\lambda _{\mathbb{R}^{n}}$-null set
(since all relevant distributions $P_{\mu ,\sigma ^{2}\Sigma }$ are
absolutely continuous w.r.t. $\lambda _{\mathbb{R}^{n}}$ due to the fact
that every element of $\Sigma \in \mathfrak{C}$ is positive definite by
assumption). We thus need to study under which circumstances this is
ensured. This will be done in the subsequent lemma. It will prove useful to
introduce the following matrix for every $y\in \mathbb{R}^{n}$%
\begin{eqnarray}
B(y) &=&R(X^{\prime }X)^{-1}X^{\prime }\limfunc{diag}\left( \hat{u}%
_{1}(y),\ldots ,\hat{u}_{n}(y)\right)  \notag \\
&=&R(X^{\prime }X)^{-1}X^{\prime }\limfunc{diag}\left( e_{1}^{\prime }(n)\Pi
_{\limfunc{span}(X)^{\bot }}y,\ldots ,e_{n}^{\prime }(n)\Pi _{\limfunc{span}%
(X)^{\bot }}y\right) ,  \label{B_matrix}
\end{eqnarray}%
as well as the following assumption on the design matrix $X$ (and on the
restriction matrix $R$):

\begin{assumption}
\label{R_and_X}Let $1\leq i_{1}<\ldots <i_{s}\leq n$ denote all the indices
for which $e_{i_{j}}(n)\in \limfunc{span}(X)$ holds where $e_{j}(n)$ denotes
the $j$-th standard basis vector in $\mathbb{R}^{n}$. If no such index
exists, set $s=0$. Let $X^{\prime }\left( \lnot (i_{1},\ldots i_{s})\right) $
denote the matrix which is obtained from $X^{\prime }$ by deleting all
columns with indices $i_{j}$, $1\leq i_{1}<\ldots <i_{s}\leq n$ (if $s=0$ no
column is deleted). Then $\limfunc{rank}\left( R(X^{\prime }X)^{-1}X^{\prime
}\left( \lnot (i_{1},\ldots i_{s})\right) \right) =q$ holds.
\end{assumption}

The lemma is now as follows. Note that the matrix $B\left( y\right) $ does 
\emph{not} depend on the weights $w\left( j,n\right) $.

\begin{lemma}
\label{LRVPD} Suppose Assumption \ref{AW} is satisfied. Then the following
holds:

\begin{enumerate}
\item $\hat{\Omega}_{w}\left( y\right) $ is nonnegative definite for every $%
y\in \mathbb{R}^{n}$.

\item $\hat{\Omega}_{w}\left( y\right) $ is singular if and only if $%
\limfunc{rank}\left( B(y)\right) <q$.

\item $\hat{\Omega}_{w}\left( y\right) =0$ if and only if $B(y)=0$.

\item The set of all $y\in \mathbb{R}^{n}$ for which $\hat{\Omega}_{w}\left(
y\right) $ is singular (or, equivalently, for which $\limfunc{rank}\left(
B(y)\right) <q$) is either a $\lambda _{\mathbb{R}^{n}}$-null set or the
entire sample space $\mathbb{R}^{n}$. The latter occurs if and only if
Assumption \ref{R_and_X} is violated.
\end{enumerate}
\end{lemma}

\begin{remark}
(i) Setting $R=X^{\prime }X$ and $q=k$ shows that a necessary and sufficient
condition for $\hat{\Psi}_{w}$ to be $\lambda _{\mathbb{R}^{n}}$-almost
everywhere nonsingular is that $e_{i}(n)\notin \limfunc{span}(X)$ for all $%
i=1,\ldots ,n$. [If this condition is not satisfied $\hat{\Psi}_{w}(y)$ is
singular for every $y\in \mathbb{R}^{n}$.] In particular, it follows that
under this simple condition $\hat{\Omega}_{w}\left( y\right) $ is
nonsingular $\lambda _{\mathbb{R}^{n}}$-almost everywhere for \emph{every}
choice of the restriction matrix $R$.

(ii) In the case $q=1$ Assumption \ref{R_and_X} is easily seen to be
violated if and only if%
\begin{equation*}
R(X^{\prime }X)^{-1}X^{\prime }e_{i}(n)=0\text{ \ or \ }e_{i}(n)\in \limfunc{%
span}(X)\text{ holds for every }i=1,\ldots ,n.
\end{equation*}
\end{remark}

We learn from the preceding lemma that, provided Assumption \ref{R_and_X} is
satisfied (which only depends on $X$ and $R$ and hence can be verified by
the user), our choice of defining the test statistic $T$ to be zero on the
set where $\hat{\Omega}_{w}$ is singular is immaterial and has no effect on
the size and power properties of the test. We also learn from that lemma
that, in case Assumption \ref{R_and_X} is violated, the commonly used
autocorrelation robust tests break down completely in a trivial way as $\hat{%
\Omega}_{w}(y)$ is then singular for \emph{every} data point $y$. We are
therefore forced to impose Assumption \ref{R_and_X} on the design matrix $X$
if we want commonly used autocorrelation robust tests to make any sense at
all. We shall thus impose Assumption \ref{R_and_X} in the following
development. We also note that, given a restriction matrix $R$, the set of
design matrices that lead to a violation of Assumption \ref{R_and_X} is a
"thin" subset in the set of all $n\times k$ matrices of full rank.

As usual, the test based on $T$ rejects $H_{0}$ if $T(y)\geq C$ where $C>0$
is an appropriate critical value. In applications the critical value is
usually taken from the asymptotic distribution of $T$ (obtained either under
assumptions that guarantee consistency of $\hat{\Omega}_{w}$ or under the
assumption of a "fixed bandwidth", i.e., $M_{n}/n>0$ independent of $n$). In
the subsequent theorem, which discusses size and power properties of
autocorrelation robust tests based on $T$, we allow for \emph{arbitrary}
(nonrandom) critical values $C>0$.\footnote{%
Because the theorem is a finite-sample result, we are free to imagine that $%
C $ depends on sample size $n$. In fact, there is nothing in the theory that
prohibits us from imagining that $C$ depends even on the design matrix $X$,
on the restriction given by $(R,r)$, or on the weights $w(j,n)$.} Because of
this, and since the theorem is a finite-sample result, it applies equally
well to standard autocorrelation robust tests (for which one fancies that $%
M_{n}\rightarrow \infty $ and $M_{n}/n\rightarrow 0$ if $n$ would increase
to infinity) and to so-called "fixed-bandwidth" tests (which assume $%
M_{n}/n>0$ independent of $n$).

\begin{theorem}
\label{thmlrv} Suppose Assumptions \ref{AAR(1)}, \ref{AW}, and \ref{R_and_X}
are satisfied. Let $T$ be the test statistic defined in (\ref{tslrv}) with $%
\hat{\Psi}_{w}$ as in (\ref{lrve}). Let $W(C)=\left\{ y\in \mathbb{R}%
^{n}:T(y)\geq C\right\} $ be the rejection region where $C$ is a real number
satisfying $0<C<\infty $. Then the following holds:

\begin{enumerate}
\item Suppose $\limfunc{rank}\left( B(e_{+})\right) =q$ and $T(e_{+}+\mu
_{0}^{\ast })>C$ hold for some (and hence all) $\mu _{0}^{\ast }\in 
\mathfrak{M}_{0}$, or $\limfunc{rank}\left( B(e_{-})\right) =q$ and $%
T(e_{-}+\mu _{0}^{\ast })>C$ hold for some (and hence all) $\mu _{0}^{\ast
}\in \mathfrak{M}_{0}$. Then%
\begin{equation}
\sup\limits_{\Sigma \in \mathfrak{C}}P_{\mu _{0},\sigma ^{2}\Sigma }\left(
W\left( C\right) \right) =1  \label{rp_1}
\end{equation}%
holds for every $\mu _{0}\in \mathfrak{M}_{0}$ and every $0<\sigma
^{2}<\infty $. In particular, the size of the test is equal to one.

\item Suppose $\limfunc{rank}\left( B(e_{+})\right) =q$ and $T(e_{+}+\mu
_{0}^{\ast })<C$ hold for some (and hence all) $\mu _{0}^{\ast }\in 
\mathfrak{M}_{0}$, or $\limfunc{rank}\left( B(e_{-})\right) =q$ and $%
T(e_{-}+\mu _{0}^{\ast })<C$ hold for some (and hence all) $\mu _{0}^{\ast
}\in \mathfrak{M}_{0}$. Then 
\begin{equation}
\inf_{\Sigma \in \mathfrak{C}}P_{\mu _{0},\sigma ^{2}\Sigma }\left( W\left(
C\right) \right) =0  \label{rp_2}
\end{equation}%
holds for every $\mu _{0}\in \mathfrak{M}_{0}$ and every $0<\sigma
^{2}<\infty $, and hence%
\begin{equation*}
\inf_{\mu _{1}\in \mathfrak{M}_{1}}\inf_{\Sigma \in \mathfrak{C}}P_{\mu
_{1},\sigma ^{2}\Sigma }\left( W\left( C\right) \right) =0
\end{equation*}%
holds for every $0<\sigma ^{2}<\infty $. In particular, the test is biased.
Furthermore, the nuisance-infimal rejection probability at every point $\mu
_{1}\in \mathfrak{M}_{1}$ is zero, i.e.,%
\begin{equation*}
\inf\limits_{0<\sigma ^{2}<\infty }\inf\limits_{\Sigma \in \mathfrak{C}%
}P_{\mu _{1},\sigma ^{2}\Sigma }(W\left( C\right) )=0.
\end{equation*}%
In particular, the infimal power of the test is equal to zero.

\item Suppose $B(e_{+})=0$ and $R\hat{\beta}(e_{+})\neq 0$ hold, or $%
B(e_{-})=0$ and $R\hat{\beta}(e_{-})\neq 0$ hold. Then 
\begin{equation}
\sup\limits_{\Sigma \in \mathfrak{C}}P_{\mu _{0},\sigma ^{2}\Sigma }\left(
W\left( C\right) \right) =1  \label{rp_3}
\end{equation}%
holds for every $\mu _{0}\in \mathfrak{M}_{0}$ and every $0<\sigma
^{2}<\infty $. In particular, the size of the test is equal to one.
\end{enumerate}
\end{theorem}

\begin{remark}
\label{rem_thmlrv}(i)\ As a point of interest we note that the rejection
probabilities $P_{\mu ,\sigma ^{2}\Sigma }(W(C))$ can be shown to depend on $%
\left( \mu ,\sigma ^{2},\Sigma \right) $ only through $\left( \left( R\beta
-r\right) /\sigma ,\Sigma \right) $ (in fact, only through $\left(
\left\langle \left( R\beta -r\right) /\sigma \right\rangle ,\Sigma \right) $%
), see Lemma \ref{aux_100} in Appendix \ref{App_A}.

(ii) Because of (i), the rejection probabilities $P_{\mu _{0},\sigma
^{2}\Sigma }\left( W\left( C\right) \right) $ are constant w.r.t. $\left(
\mu _{0},\sigma ^{2}\right) \in \mathfrak{M}_{0}\times \left( 0,\infty
\right) $ for every $\Sigma \in \mathfrak{C}$. Consequently, we could have
equivalently written (\ref{rp_1}) and (\ref{rp_3}) by inserting an infimum
over $\left( \mu _{0},\sigma ^{2}\right) \in \mathfrak{M}_{0}\times \left(
0,\infty \right) $ in between the supremum and $P_{\mu _{0},\sigma
^{2}\Sigma }\left( W\left( C\right) \right) $. Similarly, we could have
inserted a supremum over $\left( \mu _{0},\sigma ^{2}\right) \in \mathfrak{M}%
_{0}\times \left( 0,\infty \right) $ in between the infimum and $P_{\mu
_{0},\sigma ^{2}\Sigma }\left( W\left( C\right) \right) $ in (\ref{rp_2}). A
similar remark also applies to other results in the paper such as, e.g.,
Theorems \ref{thmlrv2}, \ref{thmGLS}, \ref{thmhet}, and Corollary \ref{CW}.

(iii) Although trivial, it is useful to note that the conclusions of the
preceding theorem also apply to any rejection region $W^{\ast }\in \mathcal{B%
}(\mathbb{R}^{n})$ which differs from $W\left( C\right) $ by a $\lambda _{%
\mathbb{R}^{n}}$-null set.

(iv) By the way $T$ is defined in (\ref{tslrv}), the condition $T(e_{+}+\mu
_{0}^{\ast })>C$ ($T(e_{-}+\mu _{0}^{\ast })>C$, respectively) in Part 1 of
the preceding theorem already implies $\limfunc{rank}\left( B(e_{+})\right)
=q$ ($\limfunc{rank}\left( B(e_{-})\right) =q$, respectively). For reasons
of comparability with Part 2 we have nevertheless included this rank
condition into the formulation of Part 1.
\end{remark}

\begin{remark}
\label{rem_thmlrv_2}(i) Inspection of the proof of Theorem \ref{thmlrv}
shows that Assumption \ref{AAR(1)} can obviously be weakened to the
assumption that $\mathfrak{C}$ contains AR(1) correlation matrices $\Lambda
(\rho _{m}^{(1)})$ and $\Lambda (\rho _{m}^{(2)})$ for two sequences $\rho
_{m}^{(i)}\in \left( -1,1\right) $ with $\rho _{m}^{(1)}\rightarrow 1$ and $%
\rho _{m}^{(2)}\rightarrow -1$. In fact, this can be further weakened to the
assumption that there exist $\Sigma _{m}^{(i)}\in \mathfrak{C}$ with $\Sigma
_{m}^{(1)}\rightarrow e_{+}e_{+}^{\prime }$ and $\Sigma
_{m}^{(2)}\rightarrow e_{-}e_{-}^{\prime }$ for $m\rightarrow \infty $.

(ii) For a discussion on how Theorem \ref{thmlrv} has to be modified in case
only $e_{+}e_{+}^{\prime }$ (or $e_{-}e_{-}^{\prime }$) arises as a singular
accumulation point of $\mathfrak{C}$ see Subsection \ref{Disc}.
\end{remark}

\bigskip

The conditions in Parts 1-3 of the theorem only depend on the design matrix $%
X$, the restriction $\left( R,r\right) $, the vector $e_{+}$ ($e_{-}$,
respectively), the critical value $C$, and the weights $w\left( j,n\right) $
(via $T(e_{+}+\mu _{0}^{\ast })$ or $T(e_{-}+\mu _{0}^{\ast })$,
respectively). Hence, in any particular application it can be decided
whether (and which of) these conditions are satisfied. Furthermore, as will
become transparent from the examples to follow and from Proposition \ref%
{generic} below, in the majority of applications at least one of these
conditions will be satisfied, implying that common autocorrelation robust
tests have size $1$ and/or have power arbitrarily close to $0$ in certain
parts of the alternative hypothesis. Before we turn to these examples, we
want to provide some intuition for Theorem \ref{thmlrv}: Consider a sequence 
$\rho _{m}\in \left( -1,1\right) $ with $\rho _{m}\rightarrow 1$ ($\rho
_{m}\rightarrow -1$, respectively) as $m\rightarrow \infty $. Then $\Sigma
_{m}=\Lambda \left( \rho _{m}\right) \in \mathfrak{C}$ by Assumption \ref%
{AAR(1)} and $\Lambda \left( \rho _{m}\right) \rightarrow e_{+}e_{+}^{\prime
}$ ($e_{-}e_{-}^{\prime }$) holds. Consequently, $P_{\mu _{0},\sigma
^{2}\Sigma _{m}}$ concentrates more and more around the one-dimensional
subspace $\limfunc{span}\left( e_{+}\right) $ ($\limfunc{span}\left(
e_{-}\right) $, respectively) in the sense that it converges weakly to the
singular Gaussian distribution $P_{\mu _{0},\sigma ^{2}e_{+}e_{+}^{\prime }}$
($P_{\mu _{0},\sigma ^{2}e_{-}e_{-}^{\prime }}$, respectively). The
conditions in Part 1 (or Part 3) of the preceding theorem then essentially
allow one to show that (i) the measure $P_{\mu _{0},\sigma
^{2}e_{+}e_{+}^{\prime }}$ ($P_{\mu _{0},\sigma ^{2}e_{-}e_{-}^{\prime }}$,
respectively) is supported by $W\left( C\right) $ (more precisely, after $%
W\left( C\right) $ has been modified by a suitable $\lambda _{\mathbb{R}%
^{n}} $-null set), and (ii) that $P_{\mu _{0},\sigma ^{2}e_{+}e_{+}^{\prime
}}$ ($P_{\mu _{0},\sigma ^{2}e_{-}e_{-}^{\prime }}$, respectively) puts no
mass on the boundary of the (modified) set $W\left( C\right) $. By the
Portmanteau theorem we can then conclude that the sequence of measures $%
P_{\mu _{0},\sigma ^{2}\Sigma _{m}}$ puts more and more mass on $W\left(
C\right) $ in the sense that $P_{\mu _{0},\sigma ^{2}\Sigma _{m}}\left(
W\left( C\right) \right) \rightarrow 1$ as $m\rightarrow \infty $, which
establishes the conclusion of Part 1 of the theorem. The proof of the first
claim in Part 2 works along similar lines but where concentration is now on
the complement of the rejection region $W\left( C\right) $. For more
discussion see Subsection \ref{sec_neg}. The remaining results in Part 2 are
obtained from the first claim in Part 2 exploiting invariance and continuity
properties of the rejection probabilities. While concentration of the
probability measures $P_{\mu _{0},\sigma ^{2}\Sigma _{m}}$constitutes an
important ingredient in the proof of Theorem \ref{thmlrv}, it should,
however, be stressed that there are also other cases (cf. Theorems \ref{TU_2}
and \ref{TU_3}), where despite concentration of $P_{\mu _{0},\sigma
^{2}\Sigma _{m}}$ as above, the conditions for an application of the
Portmanteau theorem are \emph{not} satisfied; in fact, in some of these
cases size $<1$ and infimal power $>0$ can be shown.

We now consider a few examples that illustrate the implications of the
preceding theorem. As in most applications the regression model contains an
intercept, we concentrate on this case in the examples.

\begin{example}
\label{EX1}\textit{(Testing a restriction involving the intercept)} Suppose
that Assumptions \ref{AAR(1)}, \ref{AW}, and \ref{R_and_X} hold. For
definiteness assume that the first column of $X$ corresponds to the
intercept (i.e., the first column of $X$ is $e_{+}$). Assume also that the
restriction involves the intercept, i.e., the first column of $R$ is
nonzero. Then it is easy to see that $B\left( e_{+}\right) =0$ and $R\hat{%
\beta}(e_{+})\neq 0$ holds (the latter since $\hat{\beta}(e_{+})=e_{1}\left(
k\right) $). Consequently, Part 3 of Theorem \ref{thmlrv} applies and shows
that the size of the test $T$ is always $1$. Additionally, the power
deficiency results in Part 2 of the theorem will apply whenever $\limfunc{%
rank}\left( B(e_{-})\right) =q$ and $T(e_{-}+\mu _{0}^{\ast })<C$ hold.
[Whether or not this is the case will depend on $C$, $X$, $R$, and the
weights.]
\end{example}

\begin{example}
\label{EX2}\textit{(Location model)} Suppose that Assumptions \ref{AAR(1)}
and \ref{AW} hold. Suppose $X=e_{+}$ and the hypothesis is $\beta =\beta
_{0} $ (hence $k=q=1$). As just noted in Example \ref{EX1}, the size of the
test $T$ is then always $1$ (as Assumption \ref{R_and_X} is certainly
satisfied). In this simple model the conditions for the power deficiencies
to arise can be made more explicit: Note that $B(e_{-})\neq 0$ clearly
always holds, and hence $\limfunc{rank}B(e_{-})=1=q$. If $n$ is even, it is
also easy to see that $T(e_{-}+\beta _{0}e_{+})=0<C$ always holds.
Consequently, Part 2 of Theorem \ref{thmlrv} applies and shows that the
power of the test gets arbitrarily close to zero in certain parts of the
parameter space as described in the theorem. If $n$ is odd, then $%
T(e_{-}+\beta _{0}e_{+})=n^{-1}\hat{\Psi}_{w}^{-1}(e_{-})$ and the same
conclusion applies provided this quantity is less than $C$.\footnote{%
The discussion in this example so far just reproduces results obtained in
Subsection \ref{prelim}.} For example, for the (modified) Bartlett
lag-window numerical computations show that $n^{-1}\hat{\Psi}%
_{w}^{-1}(e_{-}) $ is less than $1.563$ for every odd $n$ in the range $%
1<n<1000$ and every choice of $M_{n}/n\in (0,1]$; hence, if $C$ has been
chosen to be larger than or equal to $1.563$, which is typically the case at
conventional nominal significance levels, the power deficiencies are also
guaranteed to arise. We note here that this simple location model is often
used in Monte Carlo studies that try to assess finite-sample properties of
autocorrelation robust tests. Furthermore, autocorrelation robust testing of
the location parameter plays an important r\^{o}le in the field of
simulation, see, e.g., \cite{Heidel1981}, \cite{FJ2010}.
\end{example}

\begin{example}
\label{EX3}\textit{(Testing a zero restriction on a slope parameter)}
Consider the same regression model as in Example \ref{EX1} with the same
assumptions, but now suppose that the hypothesis is $\beta _{i}=0$ for some $%
i>1$, i.e., we are interested in testing a slope parameter. Since in this
case $B(e_{+})=0$ and $R\hat{\beta}(e_{+})=0$ obviously hold, where $%
R=e_{i}^{\prime }\left( k\right) $, we need to investigate the behavior of $%
B(e_{-})$ in order to be able to apply Theorem \ref{thmlrv}. If $\limfunc{%
rank}B(e_{-})=1$ holds (which will generically be the case) then size equals 
$1$ in case $T(e_{-})>C$ and the power deficiencies arise in case $%
T(e_{-})<C $.
\end{example}

\begin{example}
\label{EX4}\textit{(Testing for a change in mean)} A special case of the
preceding example is the case where $k=2$, the first column of\textit{\ }$X$
is $e_{+}$ and the second column has entries $x_{t2}=0$ for $1\leq t\leq
t_{\ast }$ and $x_{t2}=1$ else. We assume $t_{\ast }$ to be known and to
satisfy $1<t_{\ast }<n$. The hypothesis to be tested is $\beta _{2}=0$. It
is then easy to see that Assumption \ref{R_and_X} is satisfied. Furthermore,
some simple computations show that $\limfunc{rank}B(e_{-})=q=1$ always
holds. Hence, the test $T$ has size $1$ if $T(e_{-})>C$ and the power
deficiencies arise if $T(e_{-})<C$. In case $n$ as well as $n-t_{0}$ are
even, the latter case always arises since $T(e_{-})=0$ holds. [If $n$ or $%
n-t_{0}$ is odd, $T(e_{-})$ can of course be computed and depends only on $n$%
, $t_{0}$, and $\hat{\Psi}_{w}^{-1}(e_{-})$. We omit the details.]
\end{example}

The cases in Theorem \ref{thmlrv} leading to size $1$ or to power
deficiencies of the test based on $T$, while not being exhaustive, are often
satisfied in applications. We make this formal in the subsequent proposition
in that we prove that, for given restriction $(R,r)$ and critical value $C$,
the conditions in Theorem \ref{thmlrv} involving $X$ are generically
satisfied. The first part of the proposition shows that these conditions are
generically satisfied in the universe of all possible $n\times k$ design
matrices of rank $k$. Parts 2 and 3 show that the same is true if we impose
that the regression model has to contain an intercept. In the subsequent
proposition the dependence of $B\left( y\right) $, of $T(y)$, as well as of $%
\hat{\Omega}_{w}\left( y\right) $ on $X$ will be important and thus we shall
write $B_{X}\left( y\right) $, $T_{X}\left( y\right) $, and $\hat{\Omega}%
_{w,X}\left( y\right) $ for these quantities in the result to follow.

\begin{proposition}
\label{generic}Suppose Assumption \ref{AAR(1)} holds. Fix $\left( R,r\right) 
$ with $\limfunc{rank}\left( R\right) =q$, fix $0<C<\infty $, and fix the
weights $w(j,n)$ which are assumed to satisfy Assumption \ref{AW}. Let $T$
be the test statistic defined in (\ref{tslrv}) with $\hat{\Psi}_{w}$ as in (%
\ref{lrve}) and let $\mu _{0}^{\ast }\in \mathfrak{M}_{0}$ be arbitrary.

\begin{enumerate}
\item Define 
\begin{equation*}
\mathfrak{X}_{0}=\left\{ X\in \mathbb{R}^{n\times k}:\limfunc{rank}\left(
X\right) =k\right\} ,\quad \mathfrak{X}_{1}\left( e_{+}\right) =\left\{ X\in 
\mathfrak{X}_{0}:\limfunc{rank}\left( B_{X}(e_{+})\right) <q\right\} ,
\end{equation*}%
\begin{equation*}
\mathfrak{X}_{2}\left( e_{+}\right) =\left\{ X\in \mathfrak{X}_{0}\backslash 
\mathfrak{X}_{1}\left( e_{+}\right) :T_{X}(e_{+}+\mu _{0}^{\ast })=C\right\}
,
\end{equation*}%
and similarly define $\mathfrak{X}_{1}\left( e_{-}\right) $, $\mathfrak{X}%
_{2}\left( e_{-}\right) $. [Note that $\mathfrak{X}_{2}\left( e_{+}\right) $
and $\mathfrak{X}_{2}\left( e_{-}\right) $ do not depend on the choice of $%
\mu _{0}^{\ast }$.] Then $\mathfrak{X}_{1}\left( e_{+}\right) $, $\mathfrak{X%
}_{2}\left( e_{+}\right) $, $\mathfrak{X}_{1}\left( e_{-}\right) $, and $%
\mathfrak{X}_{2}\left( e_{-}\right) $ are $\lambda _{\mathbb{R}^{n\times k}}$%
-null sets. The set of all design matrices $X\in \mathfrak{X}_{0}$ for which
Theorem \ref{thmlrv} does not apply is a subset of $\left( \mathfrak{X}%
_{1}\left( e_{+}\right) \cup \mathfrak{X}_{2}\left( e_{+}\right) \right)
\cap \left( \mathfrak{X}_{1}\left( e_{-}\right) \cup \mathfrak{X}_{2}\left(
e_{-}\right) \right) $ and hence is a $\lambda _{\mathbb{R}^{n\times k}}$%
-null set. It thus is a "negligible" subset of $\mathfrak{X}_{0}$ in view of
the fact that $\mathfrak{X}_{0}$ differs from$\ \mathbb{R}^{n\times k}$ only
by a $\lambda _{\mathbb{R}^{n\times k}}$-null set.

\item Suppose $k\geq 2$, $X$ has $e_{+}$ as its first column, i.e., $%
X=\left( e_{+},\tilde{X}\right) $, and suppose the first column of $R$
consists of zeros only. Define%
\begin{eqnarray*}
\mathfrak{\tilde{X}}_{0} &=&\left\{ \tilde{X}\in \mathbb{R}^{n\times \left(
k-1\right) }:\limfunc{rank}\left( \left( e_{+},\tilde{X}\right) \right)
=k\right\} , \\
\mathfrak{\tilde{X}}_{1}\left( e_{-}\right) &=&\left\{ \tilde{X}\in 
\mathfrak{\tilde{X}}_{0}:\limfunc{rank}\left( B_{\left( e_{+},\tilde{X}%
\right) }(e_{-})\right) <q\right\} , \\
\mathfrak{\tilde{X}}_{2}\left( e_{-}\right) &=&\left\{ \tilde{X}\in 
\mathfrak{\tilde{X}}_{0}\backslash \mathfrak{\tilde{X}}_{1}\left(
e_{-}\right) :T_{\left( e_{+},\tilde{X}\right) }(e_{-}+\mu _{0}^{\ast
})=C\right\} ,
\end{eqnarray*}%
and note that $\mathfrak{\tilde{X}}_{2}\left( e_{-}\right) $ does not depend
on the choice of $\mu _{0}^{\ast }$. Then $\mathfrak{\tilde{X}}_{1}\left(
e_{-}\right) $ and $\mathfrak{\tilde{X}}_{2}\left( e_{-}\right) $ are $%
\lambda _{\mathbb{R}^{n\times \left( k-1\right) }}$-null sets (with the
analogously defined sets $\mathfrak{\tilde{X}}_{1}\left( e_{+}\right) $ and $%
\mathfrak{\tilde{X}}_{2}\left( e_{+}\right) $ satisfying $\mathfrak{\tilde{X}%
}_{1}\left( e_{+}\right) =\mathfrak{\tilde{X}}_{0}$ and $\mathfrak{\tilde{X}}%
_{2}\left( e_{+}\right) =\emptyset $.). The set of all matrices $\tilde{X}%
\in \mathfrak{\tilde{X}}_{0}$ such that Theorem \ref{thmlrv} does not apply
to the design matrix $X=\left( e_{+},\tilde{X}\right) $ is a subset of $%
\mathfrak{\tilde{X}}_{1}\left( e_{-}\right) \cup \mathfrak{\tilde{X}}%
_{2}\left( e_{-}\right) $ and hence is a $\lambda _{\mathbb{R}^{n\times
\left( k-1\right) }}$-null set. It thus is a "negligible" subset of $%
\mathfrak{\tilde{X}}_{0}$ in view of the fact that $\mathfrak{\tilde{X}}_{0}$
differs from$\ \mathbb{R}^{n\times \left( k-1\right) }$ only by a $\lambda _{%
\mathbb{R}^{n\times \left( k-1\right) }}$-null set.

\item Suppose $k\geq 2$, $X=\left( e_{+},\tilde{X}\right) $, and suppose the
first column of $R$ is nonzero. Then Theorem \ref{thmlrv} applies to the
design matrix $X=\left( e_{+},\tilde{X}\right) $ for every $\tilde{X}\in 
\mathfrak{\tilde{X}}_{0}$ (provided $X$ satisfies Assumption \ref{R_and_X}).%
\footnote{%
If $X$ does not satisfy Assumption \ref{R_and_X}, then the test breaks down
in a trivial way as already discussed.}
\end{enumerate}
\end{proposition}

The proof of the proposition actually shows more, namely that the set of
design matrices for which Theorem \ref{thmlrv} does not apply is contained
in an algebraic set. We also remark that if the regressor matrix $X$ is
viewed as randomly drawn from a distribution that is absolutely continuous
w.r.t. $\lambda _{\mathbb{R}^{n\times k}}$, Proposition \ref{generic}
implies that then the conditions of Theorem \ref{thmlrv} are almost surely
satisfied; if $X$ is also independent of $\mathbf{U}$, Theorem \ref{thmlrv}
then establishes negative results for the conditional rejection
probabilities for almost all realizations of $X$.

We next discuss an exceptional case to which Theorem \ref{thmlrv} does not
apply and which is interesting in that a positive result can be established,
at least if the covariance model $\mathfrak{C}$ is assumed to be $\mathfrak{C%
}_{AR(1)}$ or is approximated by $\mathfrak{C}_{AR(1)}$ near the singular
points in the sense of Remark \ref{rem_appr}(i) below. This positive result
will then guide us to an improved version of the test statistic $T$.

\begin{theorem}
\label{TU_2}Suppose $\mathfrak{C}=\mathfrak{C}_{AR(1)}$ and suppose
Assumptions \ref{AW} and \ref{R_and_X} are satisfied. Let $T$ be the test
statistic defined in (\ref{tslrv}) with $\hat{\Psi}_{w}$ as in (\ref{lrve}).
Let $W(C)=\left\{ y\in \mathbb{R}^{n}:T(y)\geq C\right\} $ be the rejection
region where $C$ is a real number satisfying $0<C<\infty $. If $%
e_{+},e_{-}\in \mathfrak{M}$ and $R\hat{\beta}(e_{+})=R\hat{\beta}(e_{-})=0$
is satisfied, then the following holds:

\begin{enumerate}
\item The size of the rejection region $W(C)$ is strictly less than $1$,
i.e.,%
\begin{equation*}
\sup\limits_{\mu _{0}\in \mathfrak{M}_{0}}\sup\limits_{0<\sigma ^{2}<\infty
}\sup\limits_{-1<\rho <1}P_{\mu _{0},\sigma ^{2}\Lambda (\rho )}\left(
W(C)\right) <1.
\end{equation*}%
Furthermore,%
\begin{equation*}
\inf_{\mu _{0}\in \mathfrak{M}_{0}}\inf_{0<\sigma ^{2}<\infty }\inf_{-1<\rho
<1}P_{\mu _{0},\sigma ^{2}\Lambda (\rho )}\left( W(C)\right) >0.
\end{equation*}

\item The infimal power is bounded away from zero, i.e., 
\begin{equation*}
\inf_{\mu _{1}\in \mathfrak{M}_{1}}\inf\limits_{0<\sigma ^{2}<\infty
}\inf\limits_{-1<\rho <1}P_{\mu _{1},\sigma ^{2}\Lambda (\rho )}(W(C))>0.
\end{equation*}

\item For every $0<c<\infty $%
\begin{equation*}
\inf_{\substack{ \mu _{1}\in \mathfrak{M}_{1},0<\sigma ^{2}<\infty  \\ %
d\left( \mu _{1},\mathfrak{M}_{0}\right) /\sigma \geq c}}P_{\mu _{1},\sigma
^{2}\Lambda (\rho _{m})}(W(C))\rightarrow 1
\end{equation*}%
holds for $m\rightarrow \infty $ and for any sequence $\rho _{m}\in (-1,1)$
satisfying $\left\vert \rho _{m}\right\vert \rightarrow 1$. Furthermore, for
every sequence $0<c_{m}<\infty $ and every $0<\varepsilon <1$%
\begin{equation*}
\inf_{\substack{ \mu _{1}\in \mathfrak{M}_{1},  \\ d\left( \mu _{1},%
\mathfrak{M}_{0}\right) \geq c_{m}}}\inf_{-1+\varepsilon \leq \rho \leq
1-\varepsilon }P_{\mu _{1},\sigma _{m}^{2}\Lambda (\rho )}(W(C))\rightarrow 1
\end{equation*}%
holds for $m\rightarrow \infty $ whenever $0<\sigma _{m}^{2}<\infty $ and $%
c_{m}/\sigma _{m}\rightarrow \infty $. [The very last statement holds even
without the conditions $e_{+},e_{-}\in \mathfrak{M}$ and $R\hat{\beta}%
(e_{+})=R\hat{\beta}(e_{-})=0$.]

\item For every $\delta $, $0<\delta <1$, there exists a $C(\delta )$, $%
0<C(\delta )<\infty $, such that%
\begin{equation*}
\sup\limits_{\mu _{0}\in \mathfrak{M}_{0}}\sup\limits_{0<\sigma ^{2}<\infty
}\sup\limits_{-1<\rho <1}P_{\mu _{0},\sigma ^{2}\Lambda (\rho )}(W(C(\delta
)))\leq \delta .
\end{equation*}
\end{enumerate}
\end{theorem}

The first statement of the theorem says that, in contrast to the cases
considered in Theorem \ref{thmlrv}, the size of the test $T$ is now bounded
away from $1$ for any choice of the critical value $C$. Moreover, the last
part of the theorem shows that the size can be controlled to be less than or
equal to any prespecified significance level $\delta $ by a suitable choice
of the critical value $C(\delta )$. Because $P_{\mu _{0},\sigma ^{2}\Lambda
(\rho )}(W(C))$ does not depend on $\mu _{0}$ and $\sigma ^{2}$ but only on $%
\rho $ (see Proposition \ref{inv_rej_prob}) and because this probability can
be computed via simulation, the supremum of this probability over $\mu _{0}$%
, $\sigma ^{2}$, and $\rho $ can be easily found by a grid search;
exploiting monotonicity of the probability with respect to $C$, the value of 
$C(\delta )$ can then be found by a simple search algorithm. The theorem
furthermore shows that, again in contrast to the scenario considered in
Theorem \ref{thmlrv}, the infimal power of the test is at least bounded away
from zero. The power even approaches $1$ if either $\left\Vert \left( R\beta
^{(1)}-r\right) /\sigma \right\Vert $ is bounded away from zero and $%
\left\vert \rho \right\vert \rightarrow 1$, or if $\left\Vert \left( R\beta
^{(1)}-r\right) /\sigma \right\Vert \rightarrow \infty $ and $\left\vert
\rho \right\vert $ is bounded away from $1$. [Here $\beta ^{(1)}$ is the
parameter vector corresponding to $\mu _{1}$. Note that $d\left( \mu _{1},%
\mathfrak{M}_{0}\right) $ is bounded from above as well as from below by
multiples of $\left\Vert R\beta ^{(1)}-r\right\Vert $, where the constants
involved are positive and depend only on $X$, $R$, and $r$.]

The preceding theorem required $e_{+},e_{-}\in \mathfrak{M}$ and $R\hat{\beta%
}(e_{+})=R\hat{\beta}(e_{-})=0$. To illustrate, these conditions are, e.g.,
satisfied if $e_{+}$ and $e_{-}$ constitute the first two columns of the
matrix $X$ and the hypothesis tested only involves coefficients $\beta _{i}$
with $i\geq 3$ (i.e., the first two columns of $R$ are zero). While an
intercept will typically be present in a regression model and thus $e_{+}$
appears as one of the regressors (and hence satisfies $e_{+}\in \mathfrak{M}$%
), $e_{-}$ will not necessarily be an element of $\mathfrak{M}$, and hence
the preceding theorem will not apply. However, the following theorem shows
how we can nevertheless extend the same positive results to this case if we
apply a simple adjustment to the test statistic $T$.

\begin{theorem}
\label{TU_3}Suppose $\mathfrak{C}=\mathfrak{C}_{AR(1)}$ and suppose
Assumption \ref{AW} is satisfied. Suppose one of the following scenarios
applies:

\begin{enumerate}
\item $e_{+}\in \mathfrak{M}$ with $R\hat{\beta}(e_{+})=0$ and $e_{-}\notin 
\mathfrak{M}$. Furthermore, $k+1<n$ holds and the $n\times \left( k+1\right) 
$ matrix $\bar{X}=\left( X,e_{-}\right) $ (which necessarily has rank $k+1$)
satisfies Assumption \ref{R_and_X} relative to the $q\times \left(
k+1\right) $ restriction matrix $\bar{R}=\left( R,0\right) $. Define $\bar{%
\beta}\left( y\right) =\left( I_{k},0\right) \left( \bar{X}^{\prime }\bar{X}%
\right) ^{-1}\bar{X}^{\prime }y$.

\item $e_{+}\notin \mathfrak{M}$ and $e_{-}\in \mathfrak{M}$ with $R\hat{%
\beta}(e_{-})=0$. Furthermore, $k+1<n$ holds and the $n\times \left(
k+1\right) $ matrix $\bar{X}=\left( X,e_{+}\right) $ (which necessarily has
rank $k+1$) satisfies Assumption \ref{R_and_X} relative to the $q\times
\left( k+1\right) $ restriction matrix $\bar{R}=\left( R,0\right) $. Define $%
\bar{\beta}\left( y\right) =\left( I_{k},0\right) \left( \bar{X}^{\prime }%
\bar{X}\right) ^{-1}\bar{X}^{\prime }y$.

\item $e_{+}\notin \mathfrak{M}$ and $e_{-}\notin \mathfrak{M}$ with $%
\limfunc{rank}\left( X,e_{+},e_{-}\right) =k+2$. Furthermore, $k+2<n$ holds
and the $n\times \left( k+2\right) $ matrix $\bar{X}=\left(
X,e_{+},e_{-}\right) $ (which necessarily has rank $k+2$) satisfies
Assumption \ref{R_and_X} relative to the $q\times \left( k+2\right) $
restriction matrix $\bar{R}=\left( R,0,0\right) $. Define $\bar{\beta}\left(
y\right) =\left( I_{k},0,0\right) \left( \bar{X}^{\prime }\bar{X}\right)
^{-1}\bar{X}^{\prime }y$.

\item $e_{+}\notin \mathfrak{M}$ and $e_{-}\notin \mathfrak{M}$ with $%
\limfunc{rank}\left( X,e_{+},e_{-}\right) =k+1$. Furthermore, $k+1<n$ holds
and the $n\times \left( k+1\right) $ matrix $\bar{X}=\left( X,e_{+}\right) $
(which necessarily has rank $k+1$) satisfies Assumption \ref{R_and_X}
relative to the $q\times \left( k+1\right) $ restriction matrix $\bar{R}%
=\left( R,0\right) $. Suppose further that $\bar{R}\left( \bar{X}^{\prime }%
\bar{X}\right) ^{-1}\bar{X}^{\prime }e_{-}=0$ holds. Define $\bar{\beta}%
\left( y\right) =\left( I_{k},0\right) \left( \bar{X}^{\prime }\bar{X}%
\right) ^{-1}\bar{X}^{\prime }y$.

\item $e_{+}\notin \mathfrak{M}$ and $e_{-}\notin \mathfrak{M}$ with $%
\limfunc{rank}\left( X,e_{+},e_{-}\right) =k+1$. Furthermore, $k+1<n$ holds
and the $n\times \left( k+1\right) $ matrix $\bar{X}=\left( X,e_{-}\right) $
(which necessarily has rank $k+1$) satisfies Assumption \ref{R_and_X}
relative to the $q\times \left( k+1\right) $ restriction matrix $\bar{R}%
=\left( R,0\right) $. Suppose further that $\bar{R}\left( \bar{X}^{\prime }%
\bar{X}\right) ^{-1}\bar{X}^{\prime }e_{+}=0$ holds. Define $\bar{\beta}%
\left( y\right) =\left( I_{k},0\right) \left( \bar{X}^{\prime }\bar{X}%
\right) ^{-1}\bar{X}^{\prime }y$.
\end{enumerate}

In all five scenarios define%
\begin{equation*}
\bar{T}(y)=\left\{ 
\begin{array}{cc}
(R\bar{\beta}(y)-r)^{\prime }\bar{\Omega}_{w}^{-1}(y)(R\bar{\beta}(y)-r) & 
\text{if }\det \bar{\Omega}_{w}\left( y\right) \neq 0, \\ 
0 & \text{if }\det \bar{\Omega}_{w}\left( y\right) =0,%
\end{array}%
\right.
\end{equation*}%
where $\bar{\Omega}_{w}\left( y\right) =n\bar{R}(\bar{X}^{\prime }\bar{X}%
)^{-1}\bar{\Psi}_{w}(y)(\bar{X}^{\prime }\bar{X})^{-1}\bar{R}^{\prime }$,
and $\bar{\Psi}_{w}(y)$ is computed from (\ref{lrve}) based on $\bar{v}%
_{t}(y)=\bar{u}_{t}(y)\bar{x}_{t\mathbf{\cdot }}^{\prime }$ instead of $\hat{%
v}_{t}(y)$. Here $\bar{u}_{t}(y)$ are the residuals from the regression of $%
y $ on $\bar{X}$, and $\bar{x}_{t\mathbf{\cdot }}$ are the rows of $\bar{X}$%
. Let $\bar{W}(C)=\left\{ y\in \mathbb{R}^{n}:\bar{T}(y)\geq C\right\} $ be
the rejection region where $C$ is a real number satisfying $0<C<\infty $.
Then for each of the five scenarios the conclusions of Theorem \ref{TU_2}
hold with $W(C)$ replaced by $\bar{W}(C)$.
\end{theorem}

Theorem \ref{thmlrv} together with Proposition \ref{generic} has shown that
generically the commonly used test based on the statistic $T$ has severe
size or power deficiencies even for $\mathfrak{C}=\mathfrak{C}_{AR(1)}$,
while Theorem \ref{TU_2} has isolated a special case where this is not so.
Theorem \ref{TU_3} now shows that in many of the cases falling under the
wrath of Theorem \ref{thmlrv} the ensuing problems can be circumvented (if $%
\mathfrak{C}=\mathfrak{C}_{AR(1)}$) by making use of the adjusted version $%
\bar{T}$ of the test statistic. The adjustment mechanism is simple and
amounts to basing the test statistic on estimators $\bar{\beta}$ and $\bar{%
\Omega}_{w}$ that are obtained from a "working model" that always adds the
regressors $e_{+}$ and/or $e_{-}$ to the design matrix. Note that these
regressors effect a purging of the residuals from harmonic components of
angular frequency $0$ and $\pi $. This purging effect together with the fact
that the restrictions to be tested do not involve the coefficients of the
"purging" regressors $e_{+}$ and $e_{-}$ lies at the heart of the positive
results expressed in Theorems \ref{TU_2} and \ref{TU_3}. Numerical results
that will be presented elsewhere support the theoretical result and show
that the adjusted test based on $\bar{T}$ considerably improves over the
unadjusted one based on $T$.

We next illustrate Theorems \ref{TU_2} and \ref{TU_3} in the context of
Examples \ref{EX1}-\ref{EX4}: In Examples \ref{EX1} and \ref{EX2} we have $%
e_{+}\in \mathfrak{M}$ but $R\hat{\beta}(e_{+})\neq 0$, hence neither
Theorem \ref{TU_2} nor Theorem \ref{TU_3} is applicable. In contrast, in
Example \ref{EX3} we have $e_{+}\in \mathfrak{M}$ and $R\hat{\beta}(e_{+})=0$
since $R=e_{i}^{\prime }\left( k\right) $ with $i>1$. In case $e_{-}\notin 
\mathfrak{M}$, which is the typical case and which is, in particular,
satisfied in Example \ref{EX4}, we can then use the adjusted test statistic $%
\bar{T}$ which is obtained from the auxiliary model using the enlarged
design matrix $\bar{X}=\left( X,e_{-}\right) $. Part 1 of Theorem \ref{TU_3}
then informs us that the so-adjusted test does not suffer from the severe
size/power distortions discussed in \ref{EX3} for the unadjusted
autocorrelation robust test (provided the conditions on $\bar{X}$ in the
theorem are satisfied, which generically will be the case). In case $%
e_{-}\in \mathfrak{M}$, Theorem \ref{TU_2} applies to the problem considered
in Example \ref{EX3} whenever $R\hat{\beta}(e_{-})=0$ holds, showing that in
this case already the unadjusted test does not suffer from the severe
size/power distortions. Note that here the condition $R\hat{\beta}(e_{-})=0$
will hold, for example, if $e_{-}$ is one of the columns of $X$ and the
slope parameter that is subjected to test is not the coefficient of $e_{-}$.

\begin{remark}
(i) Suppose the scenario in Part 1 of the above theorem applies except that $%
k+1=n$ holds or $\bar{X}=\left( X,e_{-}\right) $ does not satisfy Assumption %
\ref{R_and_X}. Then the test statistic $\bar{T}$ is identically zero and the
adjustment procedure does not work. A similar remark applies to Parts 2-5.

(ii) Suppose the scenario of Part 4 of the above theorem applies except that 
$\bar{R}\left( \bar{X}^{\prime }\bar{X}\right) ^{-1}\bar{X}^{\prime
}e_{-}\neq 0$ holds. Applying Part 3 of Theorem \ref{thmlrv} to $\bar{T}$
shows that this test has size $1$ and hence the adjustment procedure fails.
A similar comment applies to the scenario of Part 5.
\end{remark}

\begin{remark}
\label{rem_appr}(i) The results in Theorems \ref{TU_2} and \ref{TU_3} have
assumed $\mathfrak{C}=\mathfrak{C}_{AR(1)}$. The results immediately extend
to other covariance models $\mathfrak{C}$ as long as $\mathfrak{C}$ is
norm-bounded, the only singular accumulation points of $\mathfrak{C}$ are $%
e_{+}e_{+}^{\prime }$ and $e_{-}e_{-}^{\prime }$, and for every $\Sigma
_{m}\in \mathfrak{C}$ converging to one of these limit points there exists a
sequence $(\rho _{m})_{m\in \mathbb{N}}$ in $(-1,1)$ such that $\Lambda
^{-1/2}(\rho _{m})\Sigma _{m}\Lambda ^{-1/2}(\rho _{m})\rightarrow I_{n}$
for $m\rightarrow \infty $ (that is, near the "singular boundary" the
covariance model $\mathfrak{C}$ behaves similar to $\mathfrak{C}_{AR(1)}$).
This can be seen from an inspection of the proof. An extension of Theorems %
\ref{TU_2} and \ref{TU_3} to even more general covariance models will be
discussed elsewhere.

(ii) For a discussion of a version of Theorem \ref{TU_2} for the case where $%
\mathfrak{C}=\mathfrak{C}_{AR(1)}^{+}=\left\{ \Lambda \left( \rho \right)
:0\leq \rho <1\right\} $ or $\mathfrak{C}=\left\{ \Lambda \left( \rho
\right) :-1+\varepsilon <\rho <1\right\} $, $\varepsilon >0$, see Subsection %
\ref{Disc}.
\end{remark}

\subsubsection{Alternative nonparametric estimators for the variance
covariance matrix\label{alternative}}

We next discuss test statistics of the form (\ref{tslrv}) that use
estimators other than $\hat{\Psi}_{w}$.

\bigskip

\textbf{A.} \emph{(General quadratic estimators based on }$\hat{v}_{t}$\emph{%
) }The estimator $\hat{\Psi}_{w}$ given by (\ref{lrve}) is a special case of
general quadratic estimators $\hat{\Psi}_{GQ}\left( y\right) $ of the form 
\begin{equation*}
\hat{\Psi}_{GQ}\left( y\right) =\sum_{t,s=1}^{n}w\left( t,s;n\right) \hat{v}%
_{t}(y)\hat{v}_{s}(y)^{\prime }
\end{equation*}%
for every $y\in \mathbb{R}^{n}$, where the $n\times n$ weighting matrix $%
\mathcal{W}_{n}^{\ast }=\left( w\left( t,s;n\right) \right) _{t,s}$ is
symmetric and data-independent. While estimators of this more general form
have been studied in the early literature on spectral estimation, much of
the literature has focused on the special case of weighted autocovariance
estimators of the form $\hat{\Psi}_{w}$ (partly as a consequence of a result
in \cite{GR57} that the restriction to the smaller class of estimators does
not lead to inferior estimators in a certain asymptotic sense). However, if
the data are preprocessed by tapering before an estimator like $\hat{\Psi}%
_{w}$ is computed from the tapered data, the final estimator belongs to the
class of general quadratic estimators. Also, many modern spectral estimators
studied in the engineering literature fall into this class (see \cite{Thom82}%
), but not into the more narrow class of weighted autocovariance estimators.
Another example are the estimators proposed in \cite{Phill05}, \cite{Sun2012}%
, and \cite{ZhangShao2013b}. We now distinguish two cases:

\textit{Case 1:} The weighting matrix $\mathcal{W}_{n}^{\ast }=\left(
w\left( t,s;n\right) \right) _{t,s}$ is positive definite. Inspection of the
proofs then shows that all results given above for the tests $T$ based on $%
\hat{\Psi}_{w}$ remain valid as they stand if $\hat{\Psi}_{w}$ is replaced
by $\hat{\Psi}_{GQ}$ in the definition of the test statistic.

\textit{Case 2: }The weighting matrix $\mathcal{W}_{n}^{\ast }=\left(
w\left( t,s;n\right) \right) _{t,s}$ is only assumed to be nonnegative
definite (as is, e.g., the case for the estimators considered in \cite%
{Phill05} and \cite{Sun2012}). Arguing similar as in the proof of Lemma \ref%
{LRVPD} one can show the following:

\begin{lemma}
\label{LRVPD_2}Suppose $\mathcal{W}_{n}^{\ast }=\left( w\left( t,s;n\right)
\right) _{t,s}$ is nonnegative definite and define 
\begin{equation*}
\hat{\Omega}_{GQ}\left( y\right) =nR(X^{\prime }X)^{-1}\hat{\Psi}%
_{GQ}(y)(X^{\prime }X)^{-1}R^{\prime }.
\end{equation*}%
Then the following hold:

\begin{enumerate}
\item $\hat{\Omega}_{GQ}\left( y\right) $ is nonnegative definite for every $%
y\in \mathbb{R}^{n}$.

\item $\hat{\Omega}_{GQ}\left( y\right) $ is singular if and only if $%
\limfunc{rank}\left( B(y)\mathcal{W}_{n}^{\ast }\right) <q$ (or,
equivalently, if $\limfunc{rank}\left( B(y)\mathcal{W}_{n}^{\ast 1/2}\right)
<q$).

\item $\hat{\Omega}_{GQ}\left( y\right) =0$ if and only if $B(y)\mathcal{W}%
_{n}^{\ast }=0$ (or, equivalently, if $B(y)\mathcal{W}_{n}^{\ast 1/2}=0$).

\item The set of all $y\in \mathbb{R}^{n}$ for which $\hat{\Omega}%
_{GQ}\left( y\right) $ is singular (or, equivalently, for which $\limfunc{%
rank}\left( B(y)\mathcal{W}_{n}^{\ast }\right) <q$) is either a $\lambda _{%
\mathbb{R}^{n}}$-null set or the entire sample space $\mathbb{R}^{n}$.
\end{enumerate}
\end{lemma}

As a consequence we see that two cases can arise: In the first case $\hat{%
\Omega}_{GQ}\left( y\right) $ is singular for all $y\in \mathbb{R}^{n}$, in
which case the test statistic $T$ breaks down in a trivial way. Note that
this arises precisely if and only if $\limfunc{rank}\left( B(y)\mathcal{W}%
_{n}^{\ast }\right) <q$ for all $y\in \mathbb{R}^{n}$, which is a condition
solely on the design matrix $X$, the restriction matrix $R$, and the
weighting matrix $\mathcal{W}_{n}^{\ast }$, and thus can be verified in any
particular application. Now suppose that the second case arises, i.e., $%
\limfunc{rank}\left( B(y)\mathcal{W}_{n}^{\ast }\right) =q$ for $\lambda _{%
\mathbb{R}^{n}}$-almost all $y$. Then inspection of the proofs shows that
Theorems \ref{thmlrv} and \ref{TU_2} continue to hold for the test statistic 
$T$ based on $\hat{\Psi}_{GQ}$ provided Assumption \ref{R_and_X} is replaced
by the just mentioned condition $\limfunc{rank}\left( B(y)\mathcal{W}%
_{n}^{\ast }\right) =q$ for $\lambda _{\mathbb{R}^{n}}$-almost all $y$, and
the matrix $B(y)$ in those theorems is replaced by $B(y)\mathcal{W}%
_{n}^{\ast }$. Also Theorem \ref{TU_3} generalizes with the obvious changes.

\bigskip

\textbf{B.} \emph{(An estimator based on }$\hat{u}$\emph{) }Because $n^{-1}%
\mathbb{E}(X^{\prime }\mathbf{UU}^{\prime }X)=n^{-1}X^{\prime }\mathbb{E}(%
\mathbf{UU}^{\prime })X$, a natural estimator is%
\begin{equation*}
\hat{\Psi}_{E}\left( y\right) =n^{-1}X^{\prime }\hat{K}\left( y\right) X
\end{equation*}%
for every $y\in \mathbb{R}^{n}$, where $\hat{K}\left( y\right) $ is the
symmetric $n\times n$ Toeplitz matrix with block elements $%
n^{-1}\sum_{l=j+1}^{n}\hat{u}_{l}(y)\hat{u}_{l-j}(y)$ in the $j$-th diagonal
above the main diagonal. This estimator has already been discussed in \cite%
{E67}, but does not seem to have been used much in the econometrics
literature. It is not difficult to see that $\hat{\Psi}_{E}\left( y\right) $
is always nonnegative definite. It is positive definite if and only if $%
y\notin \limfunc{span}\left( X\right) $; and it is equal to zero for $y\in 
\limfunc{span}\left( X\right) $. Define the statistic $T_{E}$ via (\ref%
{tslrv}) with $\hat{\Omega}_{w}\left( y\right) $ replaced by $\hat{\Omega}%
_{E}\left( y\right) $ where the latter is obtained from (\ref{omega}) by
replacing $\hat{\Psi}_{w}\left( y\right) $ by $\hat{\Psi}_{E}\left( y\right) 
$. It is then easy to see that Theorems \ref{thmlrv}, \ref{TU_2}, and \ref%
{TU_3} carry over to the test based on $T_{E}$ provided Assumption \ref%
{R_and_X} is deleted from the formulation, the condition $\limfunc{rank}%
\left( B(e_{+})\right) =q$ ($\limfunc{rank}\left( B(e_{-})\right) =q$,
respectively)\ is replaced by $e_{+}\notin \limfunc{span}\left( X\right) $ ($%
e_{-}\notin \limfunc{span}\left( X\right) $), and the condition $B(e_{+})=0$
($B(e_{-})=0$, respectively) is replaced by $e_{+}\in \limfunc{span}\left(
X\right) $ ($e_{-}\in \limfunc{span}\left( X\right) $). [While \cite{E67}
provided conditions on the regressors under which consistency of $\hat{\Psi}%
_{E}\left( y\right) $ results, it may not be consistent for some common
forms of regressors (as noted in \cite{E67}). Therefore one may want to
replace $\hat{K}\left( y\right) $ by a variant where the empirical second
moments are downweighted (or more generally are obtained from an estimate of
the spectral density of the errors $\mathbf{u}_{t}$). Similar results can
then be obtained for this variant of the test. We omit the details.]

\bigskip

\textbf{C.} \emph{(Data-driven bandwidth, prewhitening, flat-top kernels,
autoregressive estimates, random critical values) }Tests based on weighted
autocovariance estimators $\hat{\Psi}$, but where the weights are allowed to
depend on the data (e.g., lag-window estimators with data-driven bandwidth
choice), or where prewhitening is used, are discussed in detail in \cite%
{Prein2013b}. Like the results given above, they are obtained by applying
the very general results provided in Subsection \ref{Impclass}. The results
in Subsection \ref{Impclass} essentially rely only on a certain equivariance
property of the estimator $\hat{\Omega}$. These results also accommodate
situations where the estimator $\hat{\Omega}\left( y\right) $ is only
well-defined for $\lambda _{\mathbb{R}^{n}}$-almost all $y$, a case that
arises often when data-driven bandwidth or prewithening are employed, or is
not always nonnegative definite (as is the case with so-called flat-top
kernels, see \cite{Polit2011}). Furthermore, certain cases where the
critical value is allowed to be random are covered by these results, see
Remark \ref{omega_rem}(ii) in Section \ref{Impclass} (and this is also true
for Subsection \ref{GLS} and Section \ref{Het}). Finally, tests based on
estimators $\hat{\Psi}$ obtained from parametric models like vector
autoregressions (see, e.g., \cite{HL97} or \cite{SunKapl2012} and references
therein) also fall into the domain of the results in Subsection \ref%
{Impclass}, but we abstain from a detailed analysis. See, however,
Subsection \ref{GLS} for a related analysis.

\subsubsection{Some discussion\label{Disc}}

\textbf{A. }We next discuss to what extent restricting the space $\mathfrak{C%
}$ of admissible covariance structures in such a way that it has fewer
singular limit points is helpful in ameliorating the properties of
autocorrelation robust tests. In the course of this we also discuss versions
of the results in Theorems \ref{thmlrv}, \ref{TU_2}, and \ref{TU_3} adapted
to such restricted spaces $\mathfrak{C}$. In the subsequent discussion we
concentrate for definiteness on the test statistic $T$ that is based on the
estimator $\hat{\Omega}_{w}$. However, the discussion carries over mutatis
mutandis to the case where the alternative estimators $\hat{\Omega}$
discussed in Subsection \ref{alternative} are used. Similar remarks also
apply to the test statistics considered in Subsection \ref{GLS}.

(i) The negative results in Theorem \ref{thmlrv} (i.e., size equal to $1$
and/or nuisance-infimal rejection probability equal to $0$) are driven by
the fact that, due to Assumption \ref{AAR(1)}, the covariance model $%
\mathfrak{C}$ has $e_{+}e_{+}^{\prime }$ and $e_{-}e_{-}^{\prime }$ as limit
points; cf. also Remark \ref{rem_thmlrv_2}. Suppose now that one would be
willing to assume that $\mathfrak{C}$ does not have any singular limit point
(and is norm-bounded which is not really a restriction here). Then the
negative results in Theorem \ref{thmlrv} do not apply. In fact, an
application of Theorem \ref{TU_1} shows that size is now strictly less than $%
1$ and the infimal power is larger than $0$. Does such an assumption on $%
\mathfrak{C}$ now solve the problem? We do not think so for at least two
reasons: First, making the assumption that the covariance model $\mathfrak{C}
$ does not have singular limit points (like $e_{+}e_{+}^{\prime }$ and $%
e_{-}e_{-}^{\prime }$) is highly questionable, especially in view of the
fact that the main motivation for the development of autocorrelation robust
tests has been the desire to avoid strong assumptions on $\mathfrak{C}$
which could lead to misspecification issues. In particular, in the not
unreasonable case where $\mathfrak{C}$ contains AR(1) correlation matrices $%
\Lambda (\rho )$, such an assumption would require to restrict $\rho $ to an
interval $\left( -1+\varepsilon ,1-\varepsilon \right) $ for some positive $%
\varepsilon $. Given the emphasis on unit root and near unit root processes
in econometrics, such an assumption seems untenable. Second, even if one is
willing to make such a heroic assumption, size or power problems can be
present. To see this assume for definiteness of the discussion that $%
\mathfrak{C}=\mathfrak{C}_{AR(1)}\left( \varepsilon ,\varepsilon \right)
=\left\{ \Lambda (\rho ):\rho \in \left( -1+\varepsilon ,1-\varepsilon
\right) \right\} $ for some small $\varepsilon >0$. As mentioned above, the
size of the test based on $T$ will be less than $1$ and the infimal power
will be larger than $0$. However, an upshot of Theorem \ref{thmlrv} still is
that the size will be close to $1$ and/or the infimal power will be close to 
$0$ for generic design\ matrices $X$, provided $\varepsilon $ is small (more
precisely, for given sample size $n$ this will happen for sufficiently small 
$\varepsilon $).\footnote{%
Of course, size could be reduced to any prescribed value in this situation
by increasing the critical value, but this would then come at the price of \
even further reduced power.} Hence, even under such an assumption,
size/power problems will disappear (or will be moderate) only if one is
willing to assume a relatively large $\varepsilon $ (in relation to sample
size $n$), making the assumption look even more heroic.

(ii) If $\mathfrak{C}$ has $e_{+}e_{+}^{\prime }$ ($e_{-}e_{-}^{\prime }$,
respectively) as its only singular limit point, inspection of the proof of
Theorem \ref{thmlrv} shows that a version of that theorem, in which now
every reference to $e_{-}$ ($e_{+}$, respectively) is deleted, continues to
hold. For example, if $\mathfrak{C}=\mathfrak{C}_{AR(1)}\left( \varepsilon
,0\right) =\left\{ \Lambda (\rho ):\rho \in \left( -1+\varepsilon ,1\right)
\right\} $ with $\varepsilon >0$, such a version of Theorem \ref{thmlrv}
applies. As an illustration, assume that $\mathfrak{C}=\mathfrak{C}%
_{AR(1)}\left( \varepsilon ,0\right) $, that the regression model contains
an intercept, and the hypothesis involves the intercept (in the sense that $R%
\hat{\beta}\left( e_{+}\right) \neq 0$). Then we can conclude from this
version of Theorem \ref{thmlrv} that the size of the test is equal to $1$.
Note that this result covers the case of testing in a location model.

(iii) Suppose $\mathfrak{C}$ has $e_{+}e_{+}^{\prime }$ as its only singular
limit point. Then in the important special case where an intercept is
present in the regression and the hypothesis tested does \emph{not} involve
the intercept (in the sense that $R\hat{\beta}\left( e_{+}\right) =0$), a 
\emph{positive} result (similar to Theorem \ref{TU_2}) is immediately
obtained from Theorem \ref{TU_1}, namely that the test based on $T$ now has
size $<$ $1$ and infimal power $>0$; moreover, the size can be controlled at
any given level $\delta $ by an appropriate choice of the critical value $%
C\left( \delta \right) $. [To be precise, Assumptions \ref{AW} and \ref%
{R_and_X} have to be satisfied, $\mathfrak{C}$ has to be norm-bounded, and
matrices in $\mathfrak{C}$ that approach $e_{+}e_{+}^{\prime }$ have to do
so in the particular manner required in Theorem \ref{TU_1}.] An important
example, where $\mathfrak{C}$ has $e_{+}e_{+}^{\prime }$ as its only
singular limit point (and is norm-bounded and satisfies the just mentioned
assumption required for Theorem \ref{TU_1}, cf. Lemma \ref{AR_1} in Appendix %
\ref{AR_100}), is $\mathfrak{C}=\mathfrak{C}_{AR(1)}\left( \varepsilon
,0\right) $ defined above. While an assumption like $\mathfrak{C}=\mathfrak{C%
}_{AR(1)}\left( \varepsilon ,0\right) $ is perhaps a bit more palatable than
the assumption $\mathfrak{C}=\mathfrak{C}_{AR(1)}\left( \varepsilon
,\varepsilon \right) $, it still imposes an adhoc restriction on the
covariance model $\mathfrak{C}_{AR(1)}$ that is debatable, especially if $%
\varepsilon $ is not small (as is, e.g., the case when $\rho $ is restricted
to be positive). Furthermore, note that, while the extreme size and power
problems (i.e., size equal one and infimal power equal zero) are absent in
the case we discuss here, less extreme, but nevertheless substantial, size
or power problems will generically still be present if $\varepsilon $ is
small as explained in (i) above. In case there is no intercept in the
regression, an appropriate version of Theorem \ref{TU_3} can be used to
generate an adjusted test by adding the intercept as a regressor, thus
bringing one back to the situation just discussed. [With the appropriate
modifications, similar remarks apply to the case where $e_{-}e_{-}^{\prime }$
is the only singular limit point of $\mathfrak{C}$.]

(iv) Regarding the preceding discussion in (iii) one should recall that in
case $\mathfrak{C}=\mathfrak{C}_{AR(1)}$ Theorems \ref{TU_2} and \ref{TU_3}
show how tests, which have size less than one and infimal power larger than
zero, can easily be obtained without \emph{any} need of bounding $\rho $
away from $1$ or $-1$, and thus without introducing any such adhoc
restrictions on $\mathfrak{C}$. Therefore, it would be desirable to free
Theorems \ref{TU_2} and \ref{TU_3} from the assumption $\mathfrak{C}=%
\mathfrak{C}_{AR(1)}$. To what extent this can be achieved without
introducing implausible assumptions like the ones discussed in the preceding
paragraphs will be discussed elsewhere.

\bigskip

\textbf{B.} (i) The results concerning the extreme size distortion and
biasedness of the tests under consideration in Theorems \ref{thmlrv} and \ref%
{thmGLS} are obtained by considering "offending" sequences of the form $%
\left( \mu _{0},\sigma ^{2},\Sigma _{m}\right) $ belonging to the null
hypothesis where $\mu _{0}\in \mathfrak{M}_{0}$ and where $\Sigma _{m}$
converges to $e_{+}e_{+}^{\prime }$ or $e_{-}e_{-}^{\prime }$. For example,
if $\Sigma _{m}=\Lambda (\rho _{m})$ with $\rho _{m}\rightarrow \pm 1$, then
the disturbance processes with covariance matrix $\sigma ^{2}\Sigma _{m}$
converge weakly to a harmonic process as discussed subsequent to Assumption %
\ref{AAR(1)}. However, it follows from Remark \ref{rem_thmlrv}(i) that also
the sequences $\left( \mu _{0},\sigma _{m}^{2},\Sigma _{m}\right) $, where $%
\mu _{0}$ and $\Sigma _{m}$ are as before and $\sigma _{m}^{2}$, $0<\sigma
_{m}^{2}<\infty $, is an \emph{arbitrary }sequence, are "offending"
sequences in the same way. Note that in case $\Sigma _{m}=\Lambda (\rho
_{m}) $ with $\rho _{m}\rightarrow \pm 1$ the corresponding disturbance
processes then need \emph{not }converge weakly to a harmonic process: As an
example, consider the case where one chooses $\sigma _{m}^{2}=\sigma
_{\varepsilon }^{2}\left( 1-\rho _{m}^{2}\right) ^{-1}$ with a constant
innovation variance $\sigma _{\varepsilon }^{2}>0$.

(ii) The covariance model $\mathfrak{C}$ maintained in this section (i.e.,
Section 3) supposes that the disturbances in the regression model are weakly
stationary and that all stationary AR(1) processes are allowed for. For
definiteness of the subsequent discussion assume that $\mathfrak{C}=%
\mathfrak{C}_{AR(1)}$. Now an \emph{alternative} model assumption could be
that the disturbances $\mathbf{u}_{t}$ satisfy 
\begin{equation*}
\mathbf{u}_{t}=\rho \mathbf{u}_{t-1}+\boldsymbol{\varepsilon }_{t},\quad
\quad 1\leq t\leq n
\end{equation*}%
where $\left\vert \rho \right\vert <1$, where the innovations $\boldsymbol{%
\varepsilon }_{t}$ are i.i.d. $N(0,\sigma _{\varepsilon }^{2})$, say, and
where $\mathbf{u}_{0}$ is a (possibly random) starting value with mean zero.
If $\mathbf{u}_{0}$ is treated as a \emph{fixed} random variable (i.e.,
being the same for all choices of the parameters in the model), then the
resulting model is not covered by the results in our paper. [Of course, this
does by no means guarantee that usual autocorrelation robust tests have good
size and power properties; cf. Footnote \ref{FN_1}.] We note, however, that
the assumption that $\mathbf{u}_{0}$ is fixed in the above sense assigns a
special meaning to the time point $t=0$, and hence may be debatable.
Therefore one may rather want to treat $\mathbf{u}_{0}$, more precisely its
distribution, as a further "parameter" of the problem. For example, one
could assume that $\mathbf{u}_{0}$ is $N\left( 0,\sigma _{\ast }^{2}\right) $%
-distributed independently of the innovations $\boldsymbol{\varepsilon }_{t}$
for $t\geq 1$ and with $0<\sigma _{\ast }^{2}<\infty $, where $\sigma _{\ast
}^{2}$ can vary independently of $\rho $ and $\sigma _{\varepsilon }^{2}$.
But then the resulting covariance model $\mathfrak{C}_{\ast }$ contains $%
\mathfrak{C}=\mathfrak{C}_{AR(1)}$ as a subset. Hence, all the results in
the paper concerning size equal to $1$ or infimal power equal to $0$, apply
a fortiori to this larger model $\mathfrak{C}_{\ast }$.

\bigskip

\textbf{C. }In a recent paper \cite{Perron2011} argue that the impossibility
results in \cite{Potscher2002} for estimating the value of the spectral
density at frequency zero are irrelevant in the context of autocorrelation
robust testing: In the framework of a Gaussian location model they compare
the behavior of common autocorrelation robust tests $t_{Robust}$, which are
standardized with the help of a spectral density estimate $\hat{f}_{n}(0)$,
with a benchmark given by the infeasible test statistic $t_{f(0)}$ that uses
the value of the unknown spectral density at frequency zero for
standardization. They find that common autocorrelation robust tests beat the
infeasible test statistic along a sequence of DGPs similar to the ones that
have been used in \cite{Potscher2002} to establish ill-posedness of the
spectral density estimation problem. This is certainly true and in fact easy
to understand: Consider as another benchmark the infeasible test statistic $%
t_{ideal}$, say, which uses the (unknown) finite-sample variance $s_{n}$ of
the arithmetic mean for standardization rather than the asymptotic variance $%
2\pi f(0)$, and observe that this statistic is exactly $N(0,1)$ distributed
(under the null) and has well-behaved size and power properties. Because $%
s_{n}$ does in general not converge uniformly to the asymptotic variance $%
2\pi f(0)$ (for the very same reasons that underlie the impossibility result
in \cite{Potscher2002}) $t_{f(0)}$ is \emph{not} uniformly close to the
ideal test $t_{ideal}$. The fact that $\hat{f}_{n}(0)$ is also not uniformly
close to $f(0)$ (due to the ill-posedness results in \cite{Potscher2002}) is
now "helpful" in the sense that it in principle allows for the possibility
that $2\pi \hat{f}_{n}(0)$ might be closer to the ideal standardization
factor $s_{n}$ than is $2\pi f(0)$, thus allowing for the possibility that $%
t_{Robust}$ might be closer to the ideal test $t_{ideal}$ than to $t_{f(0)}$%
. [Observe that $2\pi \hat{f}_{n}(0)$ as well as $s_{n}$ each not being
uniformly close to $2\pi f(0)$ does in principle not preclude (uniform)
closeness between $2\pi \hat{f}_{n}(0)$ and $s_{n}$.] In other words,
"aiming" at $f(0)$ in standardizing the test statistic is simply the wrong
thing to do. In that sense, the ill-posedness of estimating $f(0)$ is then
indeed irrelevant for autocorrelation robust testing (simply because the
benchmark $t_{f(0)}$ is irrelevant). As a matter of fact, there is no
statement to the contrary in \cite{Potscher2002}: Note that \cite%
{Potscher2002} only discusses ill-posedness of the problem of estimating $%
f(0)$ (considered to be the parameter of interest), and does not make any
statements regarding consequences of this ill-posedness for autocorrelation
robust tests that use $2\pi \hat{f}_{n}(0)$ as an estimate of the variance
nuisance parameter. The claim opening the last but one paragraph on p.1 in 
\cite{Perron2011} is thus simply false. Finally, the preceding discussion
begs the question whether or not uniform closeness of $2\pi \hat{f}_{n}(0)$
and $s_{n}$ can indeed be established under sufficiently general assumptions
on the underlying correlation structure. If possible, this would then
immediately transfer the good size and power properties of $t_{ideal}$\ to $%
t_{Robust}$. However, unfortunately this is not possible: Recall from
Example \ref{EX2} that in the location model considered in \cite{Perron2011}
the size of common autocorrelation robust tests like $t_{Robust}$ is always
equal to $1$.

\subsubsection{Further obstructions to favorable size and power properties 
\label{obstruct}}

The negative results given in Theorem \ref{thmlrv} rest on Assumption \ref%
{AAR(1)}, i.e., $\mathfrak{C}\supseteq \mathfrak{C}_{AR(1)}$, and the fact
that there exist sequences $\Sigma _{m}\in \mathfrak{C}_{AR(1)}$ that
converge to the singular matrices $e_{+}e_{+}^{\prime }$ or $e_{-}^{\prime
}e_{-}$ leading to a concentration phenomenon as discussed in the wake of
Theorem \ref{thmlrv}. The commonly used nonparametric covariance models like 
$\mathfrak{C}_{\xi }$ discussed at the beginning of Section \ref{tsr} of
course also satisfy $\mathfrak{C}_{\xi }\supseteq \mathfrak{C}_{AR(p)}$ for
every $p$, where $\mathfrak{C}_{AR(p)}$ is the set of all $n\times n$
correlation matrices arising from stationary autoregressive process of order
not larger than $p$. In this case additional singular limit matrices arise
which lead to additional conditions under which size equals $1$ or infimal
power equals $0$. We illustrate this shortly for the case where $\mathfrak{C}%
\supseteq \mathfrak{C}_{AR(2)}$. To this end define for $\nu \in (0,\pi )$
the matrix $E(\nu )$ as the $n\times 2$ matrix with $t$-th row equal to $%
(\cos (t\nu ),\sin (t\nu ))$. Furthermore set $E(0)=e_{+}$ and $E(\pi
)=e_{-} $. In Lemma \ref{cpar2}\textbf{\ }in Appendix \ref{AR_100} we show
that the matrices $E(\nu )E(\nu )^{\prime }$ for $\nu \in \lbrack 0,\pi ]$
arise as limits of sequences of matrices in $\mathfrak{C}_{AR(2)}$.
Obviously, $E(\nu )E(\nu )^{\prime }$ is singular whenever $n\geq 3$.
Restricting $\nu $ to the set $\left\{ 0,\pi \right\} $ in the subsequent
theorem reproduces the conditions appearing in Theorem \ref{thmlrv} (albeit
under the stronger assumptions that $\mathfrak{C}\supseteq \mathfrak{C}%
_{AR(2)}$ and $n\geq 3$).

\begin{theorem}
\label{thmlrv2} Suppose $\mathfrak{C}\supseteq \mathfrak{C}_{AR(2)}$,
Assumptions \ref{AW} and \ref{R_and_X} are satisfied, and $n\geq 3$ holds.
Let $T$ be the test statistic defined in (\ref{tslrv}) with $\hat{\Psi}_{w}$
as in (\ref{lrve}). Let $W(C)=\left\{ y\in \mathbb{R}^{n}:T(y)\geq C\right\} 
$ be the rejection region where $C$ is a real number satisfying $0<C<\infty $%
. Then the following holds:

\begin{enumerate}
\item Suppose there exists a $\nu \in \lbrack 0,\pi ]$ such that $\limfunc{%
rank}\left( B(z)\right) =q$ and $T(z+\mu _{0}^{\ast })>C$ hold for some (and
hence all) $\mu _{0}^{\ast }\in \mathfrak{M}_{0}$ and for $\lambda _{%
\limfunc{span}\left( E(\nu )\right) }$-almost all $z\in \limfunc{span}\left(
E(\nu )\right) $. Then%
\begin{equation*}
\sup\limits_{\Sigma \in \mathfrak{C}}P_{\mu _{0},\sigma ^{2}\Sigma }\left(
W\left( C\right) \right) =1
\end{equation*}%
holds for every $\mu _{0}\in \mathfrak{M}_{0}$ and every $0<\sigma
^{2}<\infty $. In particular, the size of the test is equal to one.

\item Suppose there exists a $\nu \in \lbrack 0,\pi ]$ such that $\limfunc{%
rank}\left( B(z)\right) =q$ and $T(z+\mu _{0}^{\ast })<C$ hold for some (and
hence all) $\mu _{0}^{\ast }\in \mathfrak{M}_{0}$ and for $\lambda _{%
\limfunc{span}\left( E(\nu )\right) }$-almost all $z\in \limfunc{span}\left(
E(\nu )\right) $. Then 
\begin{equation*}
\inf_{\Sigma \in \mathfrak{C}}P_{\mu _{0},\sigma ^{2}\Sigma }\left( W\left(
C\right) \right) =0
\end{equation*}%
holds for every $\mu _{0}\in \mathfrak{M}_{0}$ and every $0<\sigma
^{2}<\infty $, and hence%
\begin{equation*}
\inf_{\mu _{1}\in \mathfrak{M}_{1}}\inf_{\Sigma \in \mathfrak{C}}P_{\mu
_{1},\sigma ^{2}\Sigma }\left( W\left( C\right) \right) =0
\end{equation*}%
holds for every $0<\sigma ^{2}<\infty $. In particular, the test is biased.
Furthermore, the nuisance-infimal rejection probability at every point $\mu
_{1}\in \mathfrak{M}_{1}$ is zero, i.e.,%
\begin{equation*}
\inf\limits_{0<\sigma ^{2}<\infty }\inf\limits_{\Sigma \in \mathfrak{C}%
}P_{\mu _{1},\sigma ^{2}\Sigma }(W\left( C\right) )=0.
\end{equation*}%
In particular, the infimal power of the test is equal to zero.

\item Suppose there exists a $\nu \in \lbrack 0,\pi ]$ such that $B(z)=0$
and $R\hat{\beta}(z)\neq 0$ hold for $\lambda _{\limfunc{span}\left( E(\nu
)\right) }$-almost all $z\in \limfunc{span}\left( E(\nu )\right) $. Then 
\begin{equation*}
\sup\limits_{\Sigma \in \mathfrak{C}}P_{\mu _{0},\sigma ^{2}\Sigma }\left(
W\left( C\right) \right) =1
\end{equation*}%
holds for every $\mu _{0}\in \mathfrak{M}_{0}$ and every $0<\sigma
^{2}<\infty $. In particular, the size of the test is equal to one.
\end{enumerate}
\end{theorem}

To illustrate the value added of the preceding theorem when compared to
Theorem \ref{thmlrv} consider the following example: Assume that $e_{+}$ and 
$e_{-}$ are both elements of $\mathfrak{M}$ and $R\hat{\beta}(e_{+})=R\hat{%
\beta}(e_{-})=0$. Then none of the conditions in Theorem \ref{thmlrv} are
satisfied and thus this theorem is not applicable. Suppose now that the
design matrix $X$ contains $E(\nu )$ for some $\nu \in (0,\pi )$ as a
submatrix, i.e., seasonal regressors are included. Without loss of
generality assume that $X=(E(\nu ),X^{(2)})$. If we want to test for absence
of seasonality at angular frequency $\nu $, this corresponds to $R=(I_{2},0)$
and $r=0$. In case Assumption \ref{R_and_X} holds, the conditions in Case 3
of the preceding theorem are then obviously satisfied and we conclude that
the size of the test for absence of seasonality is equal to one. [In case
Assumption \ref{R_and_X} is violated, the test breaks down in a trivial way
as noted earlier.]

We finally ask what happens if we allow for covariance structures deriving
from even higher-order autoregressive models, i.e., $\mathfrak{C}\supseteq 
\mathfrak{C}_{AR(p)}$ with $p>2$. While additional concentration spaces
arise and theorems like the one above can be easily obtained from Corollary %
\ref{CW}, these theorems will often not generate new obstructions to good
size and power properties. The reason for this is that any of the newly
arising concentration spaces already contains one of the concentration
spaces $\limfunc{span}\left( E(\nu )\right) $ for $\nu \in \lbrack 0,\pi ]$
as a subset.

\subsection{Parametrically based autocorrelation robust tests\label{GLS}}

The results in Subsection \ref{HAR} were given for autocorrelation robust
tests that make use of a nonparametric estimator $\hat{\Omega}$. In this
subsection we show that the phenomena encountered in Subsection \ref{HAR}
(size distortions and power deficiencies) are not a consequence of the
nonparametric nature of the estimator, but can equally arise if a parametric
estimator is being used (and even if the parametric model employed correctly
describes the covariance structure of the errors). We illustrate this for
the case where the test statistic is obtained from a feasible generalized
least squares (GLS) estimator predicated on an AR(1) covariance structure,
as well as for the case where the test statistic is obtained from the
ordinary least squares (OLS)\ estimator combined with an estimator for the
variance covariance matrix again predicated on the same covariance
structure. The theoretical results derived below are in line with Monte
Carlo results provided in \cite{ParkMitch1980} and \cite{M89}.

We start with the estimator $\hat{\rho}$ that will be used in the feasible
GLS procedure as well as in the estimator for the variance covariance matrix
of the OLS estimator.

\begin{assumption}
\label{ARER} For $a_{1}\in \left\{ 1,2\right\} $ and $a_{2}\in \left\{
n-1,n\right\} $ with $a_{1}\leq a_{2}$ the estimator $\hat{\rho}$ is of the
form 
\begin{equation*}
\hat{\rho}(y)=\sum\limits_{t=2}^{n}\hat{u}_{t}(y)\hat{u}_{t-1}(y)\left/
\sum\limits_{t=a_{1}}^{a_{2}}\hat{u}_{t}^{2}(y)\right.
\end{equation*}%
for all $y\in \mathbb{R}^{n}\backslash N_{0}(a_{1},a_{2})$ and it is
undefined for $y\in N_{0}(a_{1},a_{2})=\left\{ y\in \mathbb{R}%
^{n}:\sum\limits_{t=a_{1}}^{a_{2}}\hat{u}_{t}^{2}(y)=0\right\} $.
\end{assumption}

The Yule-Walker estimator, which we shall abbreviate by $\hat{\rho}_{YW}$,
corresponds to $a_{1}=1$, $a_{2}=n$, while the least squares estimator $\hat{%
\rho}_{LS}$ corresponds to $a_{1}=1$, $a_{2}=n-1$. The estimators which use $%
a_{1}=2$, $a_{2}=n-1$ or $a_{1}=2$, $a_{2}=n$ have also been considered in
the literature (see, e.g., \cite{ParkMitch1980}, \cite{M89}).

\begin{remark}
\label{rho_hat}\emph{(Some properties of }$\hat{\rho}$\emph{) }(i)\ For the
Yule-Walker estimator $\hat{\rho}_{YW}$ we have $N_{0}(1,n)=\mathfrak{M}$,
i.e., $\hat{\rho}_{YW}$ is well-defined for every $y\in \mathbb{R}%
^{n}\backslash \mathfrak{M}$. Furthermore, $\hat{\rho}_{YW}$ is bounded away
from $1$ in modulus uniformly over its domain of definition, i.e., $%
\sup_{y\in \mathbb{R}^{n}\backslash \mathfrak{M}}|\hat{\rho}_{YW}(y)|<1$
holds. This follows easily from the well-known fact that $|\hat{\rho}%
_{YW}(y)|<1$, that the supremum in question does not change its value if the
range for $y$ is replaced by the compact set $\left\{ y\in \mathfrak{M}%
^{\perp }:\left\Vert y\right\Vert =1\right\} $, and the fact that $\hat{\rho}%
_{YW}$ is continuous on this set. [It can also be derived from the
discussion in Section 3.5 in \cite{GR57}.]

(ii) The least squares estimator $\hat{\rho}_{LS}$ exhibits a somewhat
different behavior: First, $\hat{\rho}_{LS}$ is well defined only on $%
\mathbb{R}^{n}\backslash N_{0}(1,n-1)$, with $N_{0}(1,n-1)$ given by $%
\left\{ y\in \mathbb{R}^{n}:\hat{u}(y)\in \limfunc{span}(e_{n}\left(
n\right) )\right\} $. Note that $\mathbb{R}^{n}\backslash N_{0}(1,n-1)$ is
contained in $\mathbb{R}^{n}\backslash \mathfrak{M}$, but is strictly
smaller in case $e_{n}\left( n\right) $ is orthogonal to each column of $X$.
Second, $\hat{\rho}_{LS}$ is not bounded away from one in modulus, in fact $%
\left\vert \hat{\rho}_{LS}\right\vert \geq 1$ can occur.\footnote{%
There are even cases where $\hat{\rho}_{LS}$ is unbounded.}

(iii) The behavior of the remaining two estimators $\hat{\rho}$ is similar
to the behavior of $\hat{\rho}_{LS}$.

(iv) The set $N_{0}(a_{1},a_{2})$ is always a closed subset of $\mathbb{R}%
^{n}$. It is guaranteed to be a $\lambda _{\mathbb{R}^{n}}$-null set
provided $k\leq a_{2}-a_{1}$ holds, cf. Lemma \ref{LDEF} below. This
condition on $k$ is no restriction in the case of the Yule-Walker estimator
(since we have assumed $k<n$ from the beginning), and is a very mild
condition in the other cases (requiring $k\leq n-2$ or $k\leq n-3$ at most).
\end{remark}

The definition of the test statistics further below will require inversion
of $\Lambda (\hat{\rho})$. While $\Lambda (\hat{\rho})$ is nonsingular if $%
\left\vert \hat{\rho}\right\vert \neq 1$, $\Lambda (\hat{\rho})$ is singular
if $\left\vert \hat{\rho}\right\vert =1$, and hence we need to study the set
of $y$ where $\left\vert \hat{\rho}\left( y\right) \right\vert =1$ (or $\hat{%
\rho}\left( y\right) $ is undefined).

\begin{lemma}
\label{LDEF} Let $\hat{\rho}$ satisfy Assumption \ref{ARER}. Then $\mathfrak{%
M}\subseteq N_{0}(a_{1},a_{2})\subseteq N_{1}(a_{1},a_{2})$ where%
\begin{equation*}
N_{1}(a_{1},a_{2})=\left\{ y\in \mathbb{R}^{n}:\left\vert
\sum\limits_{t=2}^{n}\hat{u}_{t}(y)\hat{u}_{t-1}(y)\right\vert
=\sum\limits_{t=a_{1}}^{a_{2}}\hat{u}_{t}^{2}(y)\right\} .
\end{equation*}%
The set $N_{1}(a_{1},a_{2})$ is a closed subset of $\mathbb{R}^{n}$ and is
precisely the set where the estimator $\hat{\rho}$ is either not
well-defined or is equal to $1$ in modulus. The estimator $\hat{\rho}$ is
continuous on $\mathbb{R}^{n}\backslash N_{0}(a_{1},a_{2})\supseteq \mathbb{R%
}^{n}\backslash N_{1}(a_{1},a_{2})$. If $k\leq a_{2}-a_{1}$ holds, the set $%
N_{1}(a_{1},a_{2})$ is a $\lambda _{\mathbb{R}^{n}}$-null set.
\end{lemma}

While for the Yule-Walker estimator $N_{1}(1,n)=N_{0}(1,n)$ holds as a
consequence of Remark \ref{rho_hat}(i), for the other estimators $\hat{\rho}$
the corresponding set $N_{1}(a_{1},a_{2})$ can be a proper superset of $%
N_{0}(a_{1},a_{2})$.

Given an estimator $\hat{\rho}$ satisfying Assumption \ref{ARER} we now
introduce the test statistic%
\begin{equation*}
T_{FGLS}\left( y\right) =%
\begin{cases}
(R\tilde{\beta}(y)-r)^{\prime }\tilde{\Omega}^{-1}(y)(R\tilde{\beta}(y)-r) & 
\text{ if }y\in \mathbb{R}^{n}\backslash N_{2}^{\ast }(a_{1},a_{2}), \\ 
0 & \text{ else.}%
\end{cases}%
\end{equation*}%
where 
\begin{equation*}
\tilde{\beta}(y)=(X^{\prime }\Lambda ^{-1}(\hat{\rho}(y))X)^{-1}X^{\prime
}\Lambda ^{-1}(\hat{\rho}(y))y,
\end{equation*}%
\begin{equation*}
\tilde{\sigma}^{2}(y)=(n-k)^{-1}(y-X\tilde{\beta}(y))^{\prime }\Lambda ^{-1}(%
\hat{\rho}(y))(y-X\tilde{\beta}(y)),
\end{equation*}%
\begin{equation*}
\tilde{\Omega}\left( y\right) =\tilde{\sigma}^{2}(y)R(X^{\prime }\Lambda
^{-1}(\hat{\rho}(y))X)^{-1}R^{\prime }.
\end{equation*}%
Here $N_{2}^{\ast }(a_{1},a_{2})$ is defined via%
\begin{equation*}
\mathbb{R}^{n}\backslash N_{2}^{\ast }(a_{1},a_{2})=\left\{ y\in \mathbb{R}%
^{n}\backslash N_{2}(a_{1},a_{2}):\tilde{\sigma}^{2}(y)\neq 0,\det \left(
R(X^{\prime }\Lambda ^{-1}(\hat{\rho}(y))X)^{-1}R^{\prime }\right) \neq
0\right\} ,
\end{equation*}%
where $N_{2}(a_{1},a_{2})$ is given by%
\begin{equation*}
\mathbb{R}^{n}\backslash N_{2}(a_{1},a_{2})=\left\{ y\in \mathbb{R}%
^{n}\backslash N_{1}(a_{1},a_{2}):\det \left( X^{\prime }\Lambda ^{-1}(\hat{%
\rho}(y))X\right) \neq 0\right\} .
\end{equation*}%
Note that $\tilde{\beta}$, $\tilde{\sigma}^{2}$, and $\tilde{\Omega}$ are
well-defined on $\mathbb{R}^{n}\backslash N_{2}(a_{1},a_{2})$, with $\tilde{%
\Omega}\left( y\right) $ being nonsingular if and only if $y\in \mathbb{R}%
^{n}\backslash N_{2}^{\ast }(a_{1},a_{2})$, see Lemma \ref{lem_GLS} in
Appendix \ref{App_B}. Furthermore, define%
\begin{equation*}
T_{OLS}\left( y\right) =%
\begin{cases}
(R\hat{\beta}(y)-r)^{\prime }\hat{\Omega}^{-1}(y)(R\hat{\beta}(y)-r) & \text{
if }y\in \mathbb{R}^{n}\backslash N_{0}^{\ast }(a_{1},a_{2}), \\ 
0 & \text{ else,}%
\end{cases}%
\end{equation*}%
where $\hat{\beta}(y)$ is the OLS-estimator, $\hat{\sigma}^{2}(y)=(n-k)^{-1}%
\hat{u}^{\prime }(y)\hat{u}(y)$, and%
\begin{equation*}
\hat{\Omega}\left( y\right) =\hat{\sigma}^{2}(y)R(X^{\prime
}X)^{-1}X^{\prime }\Lambda (\hat{\rho}(y))X(X^{\prime }X)^{-1}R^{\prime }.
\end{equation*}%
Here $N_{0}^{\ast }(a_{1},a_{2})$ is defined via%
\begin{equation*}
\mathbb{R}^{n}\backslash N_{0}^{\ast }(a_{1},a_{2})=\left\{ y\in \mathbb{R}%
^{n}\backslash N_{0}(a_{1},a_{2}):\det \left( R(X^{\prime }X)^{-1}X^{\prime
}\Lambda (\hat{\rho}(y))X(X^{\prime }X)^{-1}R^{\prime }\right) \neq
0\right\} .
\end{equation*}%
Of course, $\hat{\beta}$ and $\hat{\sigma}^{2}$ are well-defined on all of $%
\mathbb{R}^{n}$, while $\hat{\Omega}$ is well-defined on $\mathbb{R}%
^{n}\backslash N_{0}(a_{1},a_{2})\supseteq \mathbb{R}^{n}\backslash
N_{0}^{\ast }(a_{1},a_{2})$. Furthermore, $\hat{\Omega}\left( y\right) $ is
nonsingular for $y\in \mathbb{R}^{n}\backslash N_{0}^{\ast }(a_{1},a_{2})$,
see Lemma \ref{lem_GLS} in Appendix \ref{App_B}. We note that the
exceptional sets $N_{0}^{\ast }(a_{1},a_{2})$ and $N_{2}^{\ast
}(a_{1},a_{2}) $, respectively, appearing in the definition of the test
statistics are $\lambda _{\mathbb{R}^{n}}$-null sets provided $k\leq
a_{2}-a_{1}$ holds, see Lemma \ref{lem_GLS}. [For the case of the
Yule-Walker estimator actually $N_{2}^{\ast
}(1,n)=N_{2}(1,n)=N_{1}(1,n)=N_{0}^{\ast }(1,n)=N_{0}(1,n)=\mathfrak{M}$
holds, because $\Lambda (\hat{\rho}_{YW}\left( y\right) )$ is positive
definite for every $y\notin N_{0}(1,n)=\mathfrak{M}$ in view of $\left\vert 
\hat{\rho}_{YW}\left( y\right) \right\vert <1$, cf. Remark \ref{rho_hat}(i).]

As already noted in Remark \ref{rho_hat}, except for the Yule-Walker
estimator we can not rule out that $\hat{\rho}\left( y\right) $ is larger
than one in absolute value. For such values of $y$ the matrix $\Lambda (\hat{%
\rho}\left( y\right) )$, although being nonsingular, is indefinite. [To see
this, note that $\det \Lambda (\hat{\rho}\left( y\right) )=(1-\hat{\rho}%
^{2}\left( y\right) )^{n-1}$, which is negative for $\left\vert \hat{\rho}%
\left( y\right) \right\vert >1$ if $n$ is even. Hence there must exist a
negative and a positive eigenvalue. For odd $n$ $>1$ the claim then follows
from Cauchy's interlacing theorem.] In fact, if $\left\vert \hat{\rho}\left(
y\right) \right\vert >1$ occurs for some $y$, then it occurs on a set of
positive $\lambda _{\mathbb{R}^{n}}$-measure in view of continuity of $\hat{%
\rho}$. As a consequence, $\tilde{\Omega}\left( y\right) $ and $\hat{\Omega}%
\left( y\right) $ are not guaranteed to be $\lambda _{\mathbb{R}^{n}}$%
-almost everywhere nonnegative definite (except if the Yule-Walker estimator
is being used), although they are $\lambda _{\mathbb{R}^{n}}$-almost
everywhere nonsingular in case $k\leq a_{2}-a_{1}$. Of course, the
probability of the event $\left\vert \hat{\rho}\left( y\right) \right\vert
>1 $ will go to zero as sample size goes to infinity, but this is not
relevant for the present finite-sample analysis and the complications
ensuing from $\left\vert \hat{\rho}\left( y\right) \right\vert >1$ have to
be dealt with. Fortunately, the theory in Subsection \ref{Impclass} does not
require the estimated variance covariance matrices to be nonnegative
definite almost everywhere but only requires some weaker properties to be
satisfied which are formalized in Assumptions \ref{omega1} and \ref{omega2}
in Subsection \ref{Impclass}. Lemma \ref{lem_GLS_2} in Appendix \ref{App_B}
shows that $\tilde{\Omega}$ and $\hat{\Omega}$ satisfy these assumptions.

The subsequent theorem provides a negative result that is similar in spirit
to Theorem \ref{thmlrv}.

\begin{theorem}
\label{thmGLS} Suppose Assumptions \ref{AAR(1)} and \ref{ARER} are satisfied
and $k\leq a_{2}-a_{1}$ holds. Let $W_{FGLS}(C)=\left\{ y\in \mathbb{R}%
^{n}:T_{FGLS}(y)\geq C\right\} $ and $W_{OLS}(C)=\left\{ y\in \mathbb{R}%
^{n}:T_{OLS}(y)\geq C\right\} $ be the rejection regions corresponding to
the test statistics $T_{FGLS}$ and $T_{OLS}$, respectively, where $C$ is a
real number satisfying $0<C<\infty $. Then the following holds:

\begin{enumerate}
\item Suppose $e_{+}\notin N_{2}^{\ast }\left( a_{1},a_{2}\right) $ and $%
T_{FGLS}(e_{+}+\mu _{0}^{\ast })>C$ hold for some (and hence all) $\mu
_{0}^{\ast }\in \mathfrak{M}_{0}$, or $e_{-}\notin N_{2}^{\ast }\left(
a_{1},a_{2}\right) $ and $T_{FGLS}(e_{-}+\mu _{0}^{\ast })>C$ hold for some
(and hence all) $\mu _{0}^{\ast }\in \mathfrak{M}_{0}$. Then%
\begin{equation*}
\sup\limits_{\Sigma \in \mathfrak{C}}P_{\mu _{0},\sigma ^{2}\Sigma }\left(
W_{FGLS}\left( C\right) \right) =1
\end{equation*}%
holds for every $\mu _{0}\in \mathfrak{M}_{0}$ and every $0<\sigma
^{2}<\infty $. In particular, the size of the test is equal to one.

\item Suppose $e_{+}\notin N_{2}^{\ast }\left( a_{1},a_{2}\right) $ and $%
T_{FGLS}(e_{+}+\mu _{0}^{\ast })<C$ hold for some (and hence all) $\mu
_{0}^{\ast }\in \mathfrak{M}_{0}$, or $e_{-}\notin N_{2}^{\ast }\left(
a_{1},a_{2}\right) $ and $T_{FGLS}(e_{-}+\mu _{0}^{\ast })<C$ hold for some
(and hence all) $\mu _{0}^{\ast }\in \mathfrak{M}_{0}$. Then 
\begin{equation*}
\inf_{\Sigma \in \mathfrak{C}}P_{\mu _{0},\sigma ^{2}\Sigma }\left(
W_{FGLS}\left( C\right) \right) =0
\end{equation*}%
holds for every $\mu _{0}\in \mathfrak{M}_{0}$ and every $0<\sigma
^{2}<\infty $, and hence%
\begin{equation*}
\inf_{\mu _{1}\in \mathfrak{M}_{1}}\inf_{\Sigma \in \mathfrak{C}}P_{\mu
_{1},\sigma ^{2}\Sigma }\left( W_{FGLS}\left( C\right) \right) =0
\end{equation*}%
holds for every $0<\sigma ^{2}<\infty $. In particular, the test is biased.
Furthermore, the nuisance-infimal rejection probability at every point $\mu
_{1}\in \mathfrak{M}_{1}$ is zero, i.e.,%
\begin{equation*}
\inf\limits_{0<\sigma ^{2}<\infty }\inf\limits_{\Sigma \in \mathfrak{C}%
}P_{\mu _{1},\sigma ^{2}\Sigma }(W_{FGLS}\left( C\right) )=0.
\end{equation*}%
In particular, the infimal power of the test is equal to zero.

\item Suppose that $e_{+}\in \mathfrak{M}$ and $R\hat{\beta}(e_{+})\neq 0$
hold. Then there exists a constant $K_{FGLS}\left( e_{+}\right) $, which
depends only on $e_{+}$, $R$, and $X$, such that for every $\mu _{0}\in 
\mathfrak{M}_{0}$, every $\sigma $ with $0<\sigma <\infty $, and every $%
M\geq 0$ we have%
\begin{equation*}
\inf_{\gamma \in \mathbb{R},\left\vert \gamma \right\vert \geq
M}\inf_{\Sigma \in \mathfrak{C}}P_{\mu _{0}+\gamma e_{+},\sigma ^{2}\Sigma
}\left( W_{FGLS}(C)\right) \leq K_{FGLS}\left( e_{+}\right) \leq
\sup_{\Sigma \in \mathfrak{C}}P_{\mu _{0},\sigma ^{2}\Sigma }\left(
W_{FGLS}(C)\right) ;
\end{equation*}%
Note that $\mu _{0}+\gamma e_{+}\in \mathfrak{M}_{1}$ for $\gamma \neq 0$.
Furthermore, if $\hat{\rho}\equiv \hat{\rho}_{YW}$, then $K_{FGLS}\left(
e_{+}\right) =1$ and hence 
\begin{equation}
\sup\limits_{\Sigma \in \mathfrak{C}}P_{\mu _{0},\sigma ^{2}\Sigma }\left(
W_{FGLS}\left( C\right) \right) =1  \label{same_conc}
\end{equation}%
holds for every $\mu _{0}\in \mathfrak{M}_{0}$ and every $0<\sigma
^{2}<\infty $. If $e_{-}\in \mathfrak{M}$ and $R\hat{\beta}(e_{-})\neq 0$
hold then the analogous statements hold with $e_{+}$ replaced by $e_{-}$
where the constant $K_{FGLS}\left( e_{-}\right) $ now depends only on $e_{-}$%
, $R$, and $X$.

\item Statements analogous to 1.-3. hold true if $T_{FGLS}$ is replaced by $%
T_{OLS}$, $W_{FGLS}\left( C\right) $ is replaced by $W_{OLS}(C)$, the set $%
N_{2}^{\ast }\left( a_{1},a_{2}\right) $ is replaced by $N_{0}^{\ast }\left(
a_{1},a_{2}\right) $, and the constants $K_{FGLS}\left( \cdot \right) $ are
replaced by constants $K_{OLS}\left( \cdot \right) $.
\end{enumerate}
\end{theorem}

The meaning of Parts 1 and 2 of the preceding theorem is similar to the
meaning of the corresponding parts of Theorem \ref{thmlrv}. We note that in
the case where the Yule-Walker estimator $\hat{\rho}_{YW}$ is used the
exceptional null sets appearing in Parts 1 and 2 (and in the corresponding
portion of Part 4) satisfy $N_{2}^{\ast }\left( 1,n\right) =N_{0}^{\ast
}\left( 1,n\right) =\mathfrak{M}$. Part 3 differs somewhat from the
corresponding part of the earlier theorem, and tells us that, given the
conditions in Part 3 are met, there exist points in the alternative,
arbitrarily far away from the null hypothesis, at which power is not larger
than the size of the test. The reason for the difference between Part 3 of
Theorem \ref{thmlrv} and Part 3 of the preceding theorem lies in the fact
that the variance covariance matrix estimator $\tilde{\Omega}$ used in the
present subsection can be indefinite and that the concentration direction $%
e_{+}$ ($e_{-}$, respectively) belongs to the null set on which $\tilde{%
\Omega}$ is not defined. This requires one in the proof of the preceding
theorem to resort to Theorem \ref{prop_101} rather than to using Part 3 of
Corollary \ref{CW} (even when the Yule-Walker estimator $\hat{\rho}_{YW}$ is
used). A similar remark applies also to the corresponding portion of Part 4
of the preceding theorem. In view of the general results in Subsection \ref%
{Impclass} there is little doubt that similar negative results can also be
obtained for FGLS or OLS\ based tests that are constructed on the basis of
higher order autoregressive AR models or of other more profligate parametric
models (as long as $\mathfrak{C}\supseteq \mathfrak{C}_{AR(1)}$ is assumed).
Hence it is to be expected that autocorrelation robust tests based on
autoregressive estimates (cf. \cite{Berk1974}, \cite{HL97}, \cite%
{SunKapl2012}) will also suffer from severe size and power problems.

The results given in the preceding theorem reveal serious size and power
problems of the tests based on $T_{FGLS}$ and $T_{OLS}$. Note that these
problems arise even if $\mathfrak{C}=\mathfrak{C}_{AR(1)}$, i.e., even if
the construction of the test statistics makes use of the correct covariance
model. If $\mathfrak{C}=\mathfrak{C}_{AR(1)}$ holds, it is interesting to
contrast the above results with the size and power properties of the
corresponding infeasible tests based on $T_{GLS}^{\ast }$ and $T_{OLS}^{\ast
}$ which are defined in a similar way as $T_{FGLS}$ and $T_{OLS}$ are, but
with $\hat{\rho}$ replaced by the true value of $\rho $: These tests are
standard $F$-tests (except for not being standardized by $q$), have
well-known and reasonable size and power properties, and do not suffer from
the size and power problems exhibited by their feasible counterparts.

Similar to the situation in Subsection \ref{HAR}, the conditions in Parts
1-3 of the preceding theorem only depend on $a_{1}$ and $a_{2}$ (i.e., on
the choice of estimator $\hat{\rho}$), the design matrix $X$, the
restriction $\left( R,r\right) $, the vector $e_{+}$ ($e_{-}$,
respectively), and the critical value $C$. Hence, in any particular
application it can be decided whether or not (and which of) these conditions
are satisfied. We furthermore note that remarks analogous to Remarks \ref%
{rem_thmlrv} and \ref{rem_thmlrv_2} also apply mutatis mutandis to the
preceding theorem. We also note that a result analogous to Theorem \ref%
{thmlrv2} could be given here, but we do not spell out the details.

We next show that the conditions of Theorem \ref{thmGLS} involving the
design matrix $X$ are generically satisfied. The first part of the
subsequent proposition shows that these conditions are generically satisfied
in the class of all possible design matrices of rank $k$. Parts 2 and 3 show
a corresponding result if we impose that the regression model has to contain
an intercept. In the proposition the dependence of several quantities like $%
T_{FGLS}$, $T_{OLS}$, $N_{2}^{\ast }\left( a_{1},a_{2}\right) $, etc on the
design matrix $X$ will be important and thus we shall write $T_{FGLS,X}$, $%
T_{OLS,X}$, $N_{2,X}^{\ast }\left( a_{1},a_{2}\right) $, etc for these
quantities in the result to follow.

\begin{proposition}
\label{generic_2}Suppose Assumption \ref{AAR(1)} holds. Fix $\left(
R,r\right) $ with $\limfunc{rank}\left( R\right) =q$, fix $0<C<\infty $, and
fix $a_{1}\in \left\{ 1,2\right\} $ and $a_{2}\in \left\{ n-1,n\right\} $ in
Assumption \ref{ARER}. Suppose\ $k\leq a_{2}-a_{1}$ holds. Let $T_{FGLS,X}$
and $T_{OLS,X}$ be the test statistics defined above and let $\mu _{0}^{\ast
}\in \mathfrak{M}_{0}$ be arbitrary.

\begin{enumerate}
\item With $\mathfrak{X}_{0}$ defined in Proposition \ref{generic} define
now 
\begin{eqnarray*}
\mathfrak{X}_{1,FGLS}\left( e_{+}\right) &=&\left\{ X\in \mathfrak{X}%
_{0}:e_{+}\in N_{2,X}^{\ast }\left( a_{1},a_{2}\right) \right\} , \\
\mathfrak{X}_{2,FGLS}\left( e_{+}\right) &=&\left\{ X\in \mathfrak{X}%
_{0}\backslash \mathfrak{X}_{1,FGLS}\left( e_{+}\right)
:T_{FGLS,X}(e_{+}+\mu _{0}^{\ast })=C\right\} ,
\end{eqnarray*}%
and similarly define $\mathfrak{X}_{1,FGLS}\left( e_{-}\right) $, $\mathfrak{%
X}_{2,FGLS}\left( e_{-}\right) $. [Note that $\mathfrak{X}_{2,FGLS}\left(
e_{+}\right) $ and $\mathfrak{X}_{2,FGLS}\left( e_{-}\right) $ do not depend
on the choice of $\mu _{0}^{\ast }$.] Then $\mathfrak{X}_{1,FGLS}\left(
e_{+}\right) $ and $\mathfrak{X}_{1,FGLS}\left( e_{-}\right) $ are $\lambda
_{\mathbb{R}^{n\times k}}$-null sets. The same is true for $\mathfrak{X}%
_{2,FGLS}\left( e_{+}\right) $ ($\mathfrak{X}_{2,FGLS}\left( e_{-}\right) $,
respectively) under the provision that it is a proper subset of $\mathfrak{X}%
_{0}\backslash \mathfrak{X}_{1,FGLS}\left( e_{+}\right) $ ($\mathfrak{X}%
_{0}\backslash \mathfrak{X}_{1,FGLS}\left( e_{-}\right) $, respectively).
The set of all design matrices $X\in \mathfrak{X}_{0}$ for which Theorem \ref%
{thmGLS} does not apply is a subset of 
\begin{equation*}
\left( \mathfrak{X}_{1,FGLS}\left( e_{+}\right) \cup \mathfrak{X}%
_{2,FGLS}\left( e_{+}\right) \right) \cap \left( \mathfrak{X}_{1,FGLS}\left(
e_{-}\right) \cup \mathfrak{X}_{2,FGLS}\left( e_{-}\right) \right) .
\end{equation*}%
Hence it is a $\lambda _{\mathbb{R}^{n\times k}}$-null set provided the
preceding provision holds for at least one of $\mathfrak{X}_{2,FGLS}\left(
e_{+}\right) $ or $\mathfrak{X}_{2,FGLS}\left( e_{-}\right) $; it thus is a
"negligible" subset of $\mathfrak{X}_{0}$ in view of the fact that $%
\mathfrak{X}_{0}$ differs from$\ \mathbb{R}^{n\times k}$ only by a $\lambda
_{\mathbb{R}^{n\times k}}$-null set.

\item Suppose $k\geq 2$ and $n\geq 4$ hold and suppose $X$ has $e_{+}$ as
its first column, i.e., $X=\left( e_{+},\tilde{X}\right) $. With $\mathfrak{%
\tilde{X}}_{0}$ defined in Proposition \ref{generic} define%
\begin{eqnarray*}
\mathfrak{\tilde{X}}_{1,FGLS}\left( e_{-}\right) &=&\left\{ \tilde{X}\in 
\mathfrak{\tilde{X}}_{0}:e_{-}\in N_{2,\left( e_{+},\tilde{X}\right) }^{\ast
}\left( a_{1},a_{2}\right) \right\} , \\
\mathfrak{\tilde{X}}_{2,FGLS}\left( e_{-}\right) &=&\left\{ \tilde{X}\in 
\mathfrak{\tilde{X}}_{0}\backslash \mathfrak{\tilde{X}}_{1,FGLS}\left(
e_{-}\right) :T_{FGLS,\left( e_{+},\tilde{X}\right) }(e_{-}+\mu _{0}^{\ast
})=C\right\} ,
\end{eqnarray*}%
and note that $\mathfrak{\tilde{X}}_{2,FGLS}\left( e_{-}\right) $ does not
depend on the choice of $\mu _{0}^{\ast }$. Then $\mathfrak{\tilde{X}}%
_{1,FGLS}\left( e_{-}\right) $ is a $\lambda _{\mathbb{R}^{n\times \left(
k-1\right) }}$-null set. The set $\mathfrak{\tilde{X}}_{2,FGLS}\left(
e_{-}\right) $ is a $\lambda _{\mathbb{R}^{n\times \left( k-1\right) }}$%
-null set under the provision that it is a proper subset of $\mathfrak{%
\tilde{X}}_{0}\backslash \mathfrak{\tilde{X}}_{1,FGLS}\left( e_{-}\right) $.
[The analogously defined sets $\mathfrak{\tilde{X}}_{1,FGLS}\left(
e_{+}\right) $ and $\mathfrak{\tilde{X}}_{2,FGLS}\left( e_{+}\right) $
satisfy $\mathfrak{\tilde{X}}_{1,FGLS}\left( e_{+}\right) =\mathfrak{\tilde{X%
}}_{0}$ and $\mathfrak{\tilde{X}}_{2,FGLS}\left( e_{+}\right) =\emptyset $.]
The set of all matrices $\tilde{X}\in \mathfrak{\tilde{X}}_{0}$ such that
Theorem \ref{thmGLS} does not apply to the design matrix $X=\left( e_{+},%
\tilde{X}\right) $ is a subset of $\mathfrak{\tilde{X}}_{1,FGLS}\left(
e_{-}\right) \cup \mathfrak{\tilde{X}}_{2,FGLS}\left( e_{-}\right) $ and
hence is a $\lambda _{\mathbb{R}^{n\times \left( k-1\right) }}$-null set
under the preceding provision; it thus is a "negligible" subset of $%
\mathfrak{\tilde{X}}_{0}$ in view of the fact that $\mathfrak{\tilde{X}}_{0}$
differs from$\ \mathbb{R}^{n\times \left( k-1\right) }$ only by a $\lambda _{%
\mathbb{R}^{n\times \left( k-1\right) }}$-null set.

\item Define $\mathfrak{X}_{1,OLS}\left( \cdot \right) $ and $\mathfrak{X}%
_{2,OLS}\left( \cdot \right) $ analogously, but with $N_{0,X}^{\ast }\left(
a_{1},a_{2}\right) $ replacing $N_{2,X}^{\ast }\left( a_{1},a_{2}\right) $
and $T_{OLS,X}$ replacing $T_{FGLS,X}$. Similarly define $\mathfrak{\tilde{X}%
}_{1,OLS}\left( \cdot \right) $ and $\mathfrak{\tilde{X}}_{2,OLS}\left(
\cdot \right) $. Then Part 1 (Part 2, respectively) holds analogously for $%
\mathfrak{X}_{1,OLS}\left( \cdot \right) $ and $\mathfrak{X}_{2,OLS}\left(
\cdot \right) $ ($\mathfrak{\tilde{X}}_{1,OLS}\left( \cdot \right) $ and $%
\mathfrak{\tilde{X}}_{2,OLS}\left( \cdot \right) $, respectively) with
obvious changes.

\item Suppose $X=\left( e_{+},\tilde{X}\right) $, and suppose the first
column of $R$ is nonzero. Then Part 3 of Theorem \ref{thmGLS} applies to the
design matrix $X=\left( e_{+},\tilde{X}\right) $ for every $\tilde{X}\in 
\mathfrak{\tilde{X}}_{0}$ (for the FGLS- as well as for the OLS-based test).
\end{enumerate}
\end{proposition}

The preceding genericity result maintains in Part 1 the provision that $%
\mathfrak{X}_{2,FGLS}\left( e_{+}\right) $ is a proper subset of $\mathfrak{X%
}_{0}\backslash \mathfrak{X}_{1,FGLS}\left( e_{+}\right) $ or that $%
\mathfrak{X}_{2,FGLS}\left( e_{-}\right) $ is a proper subset of $\mathfrak{X%
}_{0}\backslash \mathfrak{X}_{1,FGLS}\left( e_{-}\right) $. Note that the
provision depends on the critical value $C$. If the provision is satisfied
for the given $C$, we can conclude from Part 1 that the set of all design
matrices $X\in \mathfrak{X}_{0}$ for which Theorem \ref{thmGLS} is not
applicable to the test statistic $T_{FGLS}$ is "negligible". If the
provision is not satisfied, i.e., if $\mathfrak{X}_{2,FGLS}\left(
e_{+}\right) =\mathfrak{X}_{0}\backslash \mathfrak{X}_{1,FGLS}\left(
e_{+}\right) $ \emph{and }$\mathfrak{X}_{2,FGLS}\left( e_{-}\right) =%
\mathfrak{X}_{0}\backslash \mathfrak{X}_{1,FGLS}\left( e_{-}\right) $ holds,
and thus we cannot draw the desired conclusion for the given value of $C$,
we immediately see that the provision must then be satisfied for any \emph{%
other} choice $C^{\prime }$ of the critical value; hence, negligibility of
the set of design matrices for which Theorem \ref{thmGLS} is not applicable
to the test statistic $T_{FGLS}$ can then be concluded for any $C^{\prime
}\neq C$. Summarizing we see that the provision is always satisfied except
possible for one particular choice of the critical value. A similar comment
applies to Parts 2 and 3 of the proposition.\footnote{%
For example, if $T_{OLS}$ is used, $a_{1}=1$, $a_{2}=n$ (Yule-Waker
estimator), and $X$ is not restricted to be of the form $\left( e_{+},\tilde{%
X}\right) $, it is not difficult to show that the provision is in fact
satisfied for $\emph{every}$ choice of $C$. This can also be shown for other
choices of $a_{1}$ and $a_{2}$ and/or for the case where $X=\left( e_{+},%
\tilde{X}\right) $ under additional assumptions on $R$. It may actually be
true in general, but we do not want to pursue this.}

Similar as in Subsection \ref{HAR}, we next discuss an exceptional case to
which Theorem \ref{thmGLS} does not apply and which allows for a positive
result, at least if the covariance model $\mathfrak{C}$ is assumed to be $%
\mathfrak{C}_{AR(1)}$ or is approximated by $\mathfrak{C}_{AR(1)}$ near the
singular points (in the sense of Remark \ref{rem_appr}(i)).

\begin{theorem}
\label{excGLS2} Suppose $\mathfrak{C}=\mathfrak{C}_{AR(1)}$, Assumption \ref%
{ARER} is satisfied, and $k\leq a_{2}-a_{1}$ holds. Let $W_{FGLS}(C)=\left\{
y\in \mathbb{R}^{n}:T_{FGLS}(y)\geq C\right\} $ and $W_{OLS}(C)=\left\{ y\in 
\mathbb{R}^{n}:T_{OLS}(y)\geq C\right\} $ be the rejection regions
corresponding to the test statistics $T_{FGLS}$ and $T_{OLS}$, respectively,
where $C$ is a real number satisfying $0<C<\infty $. If $e_{+},e_{-}\in 
\mathfrak{M}$ and $R\hat{\beta}(e_{+})=R\hat{\beta}(e_{-})=0$ is satisfied,
then the following holds for $W(C)=W_{FGLS}(C)$ as well as $W(C)=W_{OLS}(C)$:

\begin{enumerate}
\item The size of the rejection region $W(C)$ is strictly less than $1$,
i.e.,%
\begin{equation*}
\sup\limits_{\mu _{0}\in \mathfrak{M}_{0}}\sup\limits_{0<\sigma ^{2}<\infty
}\sup\limits_{-1<\rho <1}P_{\mu _{0},\sigma ^{2}\Lambda (\rho )}\left(
W(C)\right) <1.
\end{equation*}%
Furthermore,%
\begin{equation*}
\inf_{\mu _{0}\in \mathfrak{M}_{0}}\inf_{0<\sigma ^{2}<\infty }\inf_{-1<\rho
<1}P_{\mu _{0},\sigma ^{2}\Lambda (\rho )}\left( W(C)\right) >0.
\end{equation*}

\item The infimal power is bounded away from zero, i.e., 
\begin{equation*}
\inf_{\mu _{1}\in \mathfrak{M}_{1}}\inf\limits_{0<\sigma ^{2}<\infty
}\inf\limits_{-1<\rho <1}P_{\mu _{1},\sigma ^{2}\Lambda (\rho )}(W(C))>0.
\end{equation*}

\item Suppose that $a_{1}=1$ and $a_{2}=n$. Then for every $0<c<\infty $%
\begin{equation*}
\inf_{\substack{ \mu _{1}\in \mathfrak{M}_{1},0<\sigma ^{2}<\infty  \\ %
d\left( \mu _{1},\mathfrak{M}_{0}\right) /\sigma \geq c}}P_{\mu _{1},\sigma
^{2}\Lambda (\rho _{m})}(W(C))\rightarrow 1
\end{equation*}%
holds for $m\rightarrow \infty $ and for any sequence $\rho _{m}\in (-1,1)$
satisfying $\left\vert \rho _{m}\right\vert \rightarrow 1$. Furthermore, for
every sequence $0<c_{m}<\infty $ and every $0<\varepsilon <1$%
\begin{equation*}
\inf_{\substack{ \mu _{1}\in \mathfrak{M}_{1},  \\ d\left( \mu _{1},%
\mathfrak{M}_{0}\right) \geq c_{m}}}\inf_{-1+\varepsilon \leq \rho \leq
1-\varepsilon }P_{\mu _{1},\sigma _{m}^{2}\Lambda (\rho )}(W(C))\rightarrow 1
\end{equation*}%
holds for $m\rightarrow \infty $ whenever $0<\sigma _{m}^{2}<\infty $ and $%
c_{m}/\sigma _{m}\rightarrow \infty $. [The very last statement holds even
without the conditions $e_{+},e_{-}\in \mathfrak{M}$ and $R\hat{\beta}%
(e_{+})=R\hat{\beta}(e_{-})=0$.]

\item For every $\delta $, $0<\delta <1$, there exists a $C(\delta )$, $%
0<C(\delta )<\infty $, such that%
\begin{equation*}
\sup\limits_{\mu _{0}\in \mathfrak{M}_{0}}\sup\limits_{0<\sigma ^{2}<\infty
}\sup\limits_{-1<\rho <1}P_{\mu _{0},\sigma ^{2}\Lambda (\rho )}(W(C(\delta
)))\leq \delta .
\end{equation*}
\end{enumerate}
\end{theorem}

A discussion similar to the one following Theorem \ref{TU_2} applies also
here. Furthermore, a result paralleling Theorem \ref{TU_3} can again be
obtained by a combined application of Theorem \ref{TU_1} and Proposition \ref%
{enforce_inv}. The so-obtained result shows how adjusted test statistics $%
\bar{T}_{FGLS}$ and $\bar{T}_{OLS}$ can be constructed that have size/power
properties as given in the preceding theorem also in many cases which fall
under the wrath of Theorem \ref{thmGLS} (and for which the tests based on $%
T_{FGLS}$ and $T_{OLS}$ suffer from extreme size or power deficiencies). The
adjustment mechanism again amounts to using a "working model" that always
adds the regressors $e_{+}$ and/or $e_{-}$ to the design matrix. We abstain
from providing details.

\subsection{Some remarks on the $F$-test without correction for
autocorrelation \label{Krae}}

As mentioned in the introduction, a considerable body of literature is
concerned with the properties of the \emph{standard} $F$-test (i.e., the $F$%
-test without correction for autocorrelation) in the presence of
autocorrelation. Much of this literature concentrates on the case where the
errors follow a stationary autoregressive process of order $1$, i.e., $%
\mathfrak{C}=\mathfrak{C}_{AR(1)}$. As the correlation in the errors is not
accounted for in the standard $F$-test, bad performance of the standard $F$%
-test for large values of the correlation $\rho $ can be expected. This has
been demonstrated formally in \cite{Kr89}, \cite{KKB90}, and subsequently in 
\cite{Ban00}: These papers determine the limit as $\rho \rightarrow 1$ of
the error of the first kind of the standard $F$-test and show that (i) this
limit is $1$ if the regression contains an intercept and the restrictions to
be tested involve the intercept (i.e., the $n\times 1$ vector $e_{+}=\left(
1,\ldots ,1\right) ^{\prime }$ belongs to the span of the design matrix and $%
R\hat{\beta}(e_{+})\neq 0$ holds) or if the regression does not contain an
intercept (i.e., $e_{+}$ does not belong to the span of the design matrix)
and a certain observable quantity, $A$ say, is positive, (ii) it is $0$ if
the regression does not contain an intercept and the observable quantity $A$
is negative, and (iii) it is a value between $0$ and $1$ if the regression
contains an intercept but the restrictions to be tested do not involve the
intercept (i.e., $e_{+}$ belongs to the span of the design matrix and $R\hat{%
\beta}(e_{+})=0$ holds).\footnote{\cite{Ban00} claim in their Theorem 5 that
the expression $\Pr \left( F(0)>\delta \right) $ converges to zero if $%
Mi\neq 0$ and $\bar{F}(0)\leq \delta $. In case $\bar{F}(0)=\delta $ the
argument given there is, however, incorrect, because $F(0)\rightarrow \bar{F}%
(0)=\delta $ in probability does not imply $\Pr \left( F(0)>\delta \right)
\rightarrow 0$ in general.} It perhaps comes as a surprise that
autocorrelation robust tests, which have built into them a correction for
autocorrelation, exhibit a similar behavior as shown in Section \ref{tsr} of
the present paper. We mention that, due to the relatively simple structure
of the standard $F$-test statistic as a ratio of quadratic forms, the method
of proof in \cite{Kr89}, \cite{KKB90}, and \cite{Ban00} is by direct
computation of the limit (as $\rho \rightarrow 1$)\ of the test statistic.
In contrast, the results for the much more complicated test statistics
considered in the present paper rely on quite different methods which make
use of invariance considerations and are of a more geometric flavor.
Needless to say, the just mentioned results in \cite{Kr89}, \cite{KKB90},
and \cite{Ban00} can be rederived through a straightforward application of
the general results in Subsection \ref{Impclass} to the standard $F$-test.

In light of the fact that the standard $F$-test makes no correction for
autocorrelation at all, a perhaps surprising observation is that
nevertheless an analogue to Theorems \ref{TU_2} and \ref{excGLS2} can be
established for the standard $F$-test by a simple application of Theorem \ref%
{TU_1}. Even more, the adjustment procedure described in Proposition \ref%
{enforce_inv} can be applied to the standard $F$-test leading to a result
analogous to Theorem \ref{TU_3}. While these results show that the size and
power of the so-adjusted standard $F$-test do not "break down" completely
for extreme correlations, they do not tell us much about the performance of
the adjusted test for moderate correlations.

\section{Size and Power of Tests of Linear Restrictions in Regression Models
with Heteroskedastic Disturbances\label{Het}}

We next turn to size and power properties of commonly used
heteroskedasticity robust tests. To this end we allow for heteroskedasticity
of unknown form as is common in the literature and thus allow that the
errors in the regression model have a variance covariance matrix $\sigma
^{2}\Sigma $ where $\Sigma $ is an element of the covariance model given by

\begin{equation*}
\mathfrak{C}_{Het}=\left\{ \limfunc{diag}(\tau _{1}^{2},\ldots ,\tau
_{n}^{2}):\tau _{i}^{2}>0,i=1,\ldots ,n,\sum_{i=1}^{n}\tau
_{i}^{2}=1\right\} .
\end{equation*}%
The normalization for $\Sigma $ chosen is of course arbitrary and could
equally well be replaced, e.g., by the normalization $\tau _{1}^{2}=1$. The
heteroskedasticity robust test statistic considered is given by%
\begin{equation}
T_{Het}\left( y\right) =\left\{ 
\begin{array}{cc}
(R\hat{\beta}\left( y\right) -r)^{\prime }\hat{\Omega}_{Het}^{-1}\left(
y\right) (R\hat{\beta}\left( y\right) -r) & \text{if }\det \hat{\Omega}%
_{Het}\left( y\right) \neq 0, \\ 
0 & \text{if }\det \hat{\Omega}_{Het}\left( y\right) =0,%
\end{array}%
\right.  \label{T_het}
\end{equation}%
where $\hat{\Omega}_{Het}=R\hat{\Psi}_{Het}R^{\prime }$ and $\hat{\Psi}%
_{Het} $ is a heteroskedasticity robust estimator. Such estimators were
introduced in \cite{E63,E67} and have later found their way into the
econometrics literature (e.g., \cite{W80}). They are of the form%
\begin{equation*}
\hat{\Psi}_{Het}\left( y\right) =(X^{\prime }X)^{-1}X^{\prime }\limfunc{diag}%
\left( d_{1}\hat{u}_{1}^{2}\left( y\right) ,\ldots ,d_{n}\hat{u}%
_{n}^{2}\left( y\right) \right) X(X^{\prime }X)^{-1}
\end{equation*}%
where the constants $d_{i}>0$ may depend on the design matrix. Typical
choices for $d_{i}$ are $d_{i}=1$, $d_{i}=n/(n-k)$, $d_{i}=\left(
1-h_{ii}\right) ^{-1}$, or $d_{i}=\left( 1-h_{ii}\right) ^{-2}$ where $%
h_{ii} $ denotes the $i$-th diagonal element of the projection matrix $%
X(X^{\prime }X)^{-1}X^{\prime }$, see \cite{LE2000} for an overview. Another
suggestion is $d_{i}=\left( 1-h_{ii}\right) ^{-\delta _{i}}$ for suitable
choice of $\delta _{i}$, see \cite{Crib2004}. For the last three choices of $%
d_{i}$ we use the convention that we set $d_{i}=1$ in case $h_{ii}=1$. Note
that $h_{ii}=1$ implies $\hat{u}_{i}\left( y\right) =0$ for every $y$, and
hence it is irrelevant which real value is assigned to $d_{i}$ in case $%
h_{ii}=1$.

Similar as in Subsection \ref{HAR} we need to ensure that $\hat{\Omega}%
_{Het}\left( y\right) $ is nonsingular $\lambda _{\mathbb{R}^{n}}$-almost
everywhere. As shown in the subsequent lemma this is the case provided
Assumption \ref{R_and_X} introduced in Subsection \ref{HAR} is satisfied.
The lemma also shows that in case this assumption is violated the matrix $%
\hat{\Omega}_{Het}\left( y\right) $ is singular \emph{everywhere}, leading
to a complete and trivial breakdown of the test. Recall the definition of
the matrix $B\left( y\right) $ given in (\ref{B_matrix}) and note that it is
independent of the constants $d_{i}$.

\begin{lemma}
\label{Definiteness_het}

\begin{enumerate}
\item $\hat{\Omega}_{Het}\left( y\right) $ is nonnegative definite for every 
$y\in \mathbb{R}^{n}$.

\item $\hat{\Omega}_{Het}\left( y\right) $ is singular if and only if $%
\limfunc{rank}\left( B(y)\right) <q$.

\item $\hat{\Omega}_{Het}\left( y\right) =0$ if and only if $B(y)=0$.

\item The set of all $y\in \mathbb{R}^{n}$ for which $\hat{\Omega}%
_{Het}\left( y\right) $ is singular (or, equivalently, for which $\limfunc{%
rank}\left( B(y)\right) <q$) is either a $\lambda _{\mathbb{R}^{n}}$-null
set or the entire sample space $\mathbb{R}^{n}$. The latter occurs if and
only if Assumption \ref{R_and_X} is violated.
\end{enumerate}
\end{lemma}

The proof of the preceding lemma is completely analogous to the proof of
Lemma \ref{LRVPD} and hence is omitted. We are now in the position to state
the result on size and power of tests based on the statistic $T_{Het}$ given
in (\ref{T_het}).

\begin{theorem}
\label{thmhet} Suppose $\mathfrak{C}\supseteq \mathfrak{C}_{Het}$ holds and
Assumption \ref{R_and_X} is satisfied. Let $T_{Het}$ be the test statistic
defined in (\ref{T_het}) and let $W_{Het}(C)=\left\{ y\in \mathbb{R}%
^{n}:T(y)\geq C\right\} $ be the rejection region where $C$ is a real number
satisfying $0<C<\infty $. Then the following holds:

\begin{enumerate}
\item Suppose for some $i$, $1\leq i\leq n$, we have $\limfunc{rank}\left(
B(e_{i}\left( n\right) )\right) =q$ and $T_{Het}(e_{i}\left( n\right) +\mu
_{0}^{\ast })>C$ for some (and hence all) $\mu _{0}^{\ast }\in \mathfrak{M}%
_{0}$. Then%
\begin{equation*}
\sup\limits_{\Sigma \in \mathfrak{C}}P_{\mu _{0},\sigma ^{2}\Sigma }\left(
W_{Het}\left( C\right) \right) =1
\end{equation*}%
holds for every $\mu _{0}\in \mathfrak{M}_{0}$ and every $0<\sigma
^{2}<\infty $. In particular, the size of the test is equal to one.

\item Suppose for some $i$, $1\leq i\leq n$, we have $\limfunc{rank}\left(
B(e_{i}\left( n\right) )\right) =q$ and $T_{Het}(e_{i}\left( n\right) +\mu
_{0}^{\ast })<C$ for some (and hence all) $\mu _{0}^{\ast }\in \mathfrak{M}%
_{0}$. Then 
\begin{equation*}
\inf_{\Sigma \in \mathfrak{C}}P_{\mu _{0},\sigma ^{2}\Sigma }\left(
W_{Het}\left( C\right) \right) =0
\end{equation*}%
holds for every $\mu _{0}\in \mathfrak{M}_{0}$ and every $0<\sigma
^{2}<\infty $, and hence%
\begin{equation*}
\inf_{\mu _{1}\in \mathfrak{M}_{1}}\inf_{\Sigma \in \mathfrak{C}}P_{\mu
_{1},\sigma ^{2}\Sigma }\left( W_{Het}\left( C\right) \right) =0
\end{equation*}%
holds for every $0<\sigma ^{2}<\infty $. In particular, the test is biased.
Furthermore, the nuisance-infimal rejection probability at every point $\mu
_{1}\in \mathfrak{M}_{1}$ is zero, i.e.,%
\begin{equation*}
\inf\limits_{0<\sigma ^{2}<\infty }\inf\limits_{\Sigma \in \mathfrak{C}%
}P_{\mu _{1},\sigma ^{2}\Sigma }(W_{Het}\left( C\right) )=0.
\end{equation*}%
In particular, the infimal power of the test is equal to zero.

\item Suppose for some $i$, $1\leq i\leq n$, we have $B(e_{i}\left( n\right)
)=0$ and $R\hat{\beta}(e_{i}\left( n\right) )\neq 0$. Then 
\begin{equation*}
\sup\limits_{\Sigma \in \mathfrak{C}}P_{\mu _{0},\sigma ^{2}\Sigma }\left(
W_{Het}\left( C\right) \right) =1
\end{equation*}%
holds for every $\mu _{0}\in \mathfrak{M}_{0}$ and every $0<\sigma
^{2}<\infty $. In particular, the size of the test is equal to one.
\end{enumerate}
\end{theorem}

We note that Remark \ref{rem_thmlrv} as well as most of the discussion
following Theorem \ref{thmlrv} apply mutatis mutandis also here. Similar as
in Subsection \ref{HAR} it is also not difficult to show (for typical
choices of $d_{i}$) that the set of design matrices $X$ for which the
conditions in Theorem \ref{thmhet} are not satisfied is a negligible set. We
omit a formal statement. In contrast to the case considered in Subsection %
\ref{HAR}, however, no (nontrivial) analogues to the positive results given
in Theorems \ref{TU_2} and \ref{TU_3} are possible due to the fact that in
the present setting there are now too many concentration spaces (which
together in fact span all of $\mathbb{R}^{n}$). Furthermore, the above
theorem and its proof exploits only the one-dimensional concentration spaces 
$\mathcal{Z}_{i}=$span$(e_{i}\left( n\right) )$. While every linear space of
the form $\limfunc{span}(e_{i_{1}}\left( n\right) ,\ldots ,e_{i_{p}}\left(
n\right) )$ for $0<p<n$ and $1\leq i_{1}<\ldots <i_{p}\leq n$ is a
concentration space of the model $\mathfrak{C}$, using all these
concentration spaces in conjunction with Corollary \ref{CW} will often not
deliver additional obstructions to good size or power properties, the reason
being that each of these spaces already contains a concentration space $%
\mathcal{Z}_{i}$ as a subset. As a further point of interest we note that
the assumptions imposed in \cite{E63,E67} require all variances $\sigma
^{2}\tau _{i}^{2}$ to be bounded away from zero in order to achieve
uniformity in the convergence to the limiting distribution. Hence, Eicker's
assumptions rule out the concentration effect that drives the above result.%
\footnote{%
Imposing the assumption that all elements $\Sigma $ of $\mathfrak{C}%
\subseteq \mathfrak{C}_{Het}$ have all their diagonal elements bounded from
below by a given positive constant $\varepsilon $ is only a partial cure.
While it saves the heteroskedasticity robust test from the \emph{extreme}
size and power distortions as described in Theorem \ref{thmhet}, substantial
size/power distortions will nevertheless be present if $\varepsilon $ is
small (relative to sample size). Cf. the discussion in Subsection \ref{Disc}.%
} It appears that this insight in \cite{E63,E67} has not been fully
appreciated in the ensuing econometrics literature.

In connection with the preceding theorem, which points out size distortions
and/or power deficiencies of heteroskedasticity robust tests even under a
normality assumption,\ a result in Section 4.2 of \cite{Dufour2003} needs to
be mentioned which shows that the size of heteroskedasticity robust tests is
always $1$ if one allows for a sufficiently \emph{large nonparametric} class
of distributions for the errors $\mathbf{U}$.

We briefly discuss the standard $F$-test statistic without any correction
for heteroskedasticity. Let%
\begin{equation*}
T_{uncorr}\left( y\right) =\left\{ 
\begin{array}{cc}
\left( \left( n-k\right) /q\right) (R\hat{\beta}\left( y\right) -r)^{\prime
}\left( R\left( X^{\prime }X\right) ^{-1}R^{\prime }\right) ^{-1}(R\hat{\beta%
}\left( y\right) -r)/\left( \hat{u}^{\prime }\left( y\right) \hat{u}\left(
y\right) \right) & \text{if }y\notin \mathfrak{M} \\ 
0 & \text{if }y\in \mathfrak{M}%
\end{array}%
\right.
\end{equation*}%
and define $W_{uncorr}(C)$ in the obvious way. It is then easy to see that a
variant of Theorem \ref{thmhet} also holds with $T_{uncorr}$ and $%
W_{uncorr}(C)$ replacing $T_{Het}$ and $W_{Het}(C)$, respectively, if in
this variant of the theorem Assumption \ref{R_and_X} is dropped, the
condition $\limfunc{rank}\left( B(e_{i}\left( n\right) )\right) =q$ is
replaced by the condition $e_{i}\left( n\right) \notin \mathfrak{M}$, and
the condition $B(e_{i}\left( n\right) )=0$ is replaced by the condition $%
e_{i}\left( n\right) \in \mathfrak{M}$. In a recent paper \cite%
{IbragMuell2010} consider the standard $t$-test for testing $\mu =0$ versus $%
\mu \neq 0$ in a Gaussian location model and discuss a result by \cite%
{BakiSzek2005} to the effect that the size of this test under
heteroskedasticity of unknown form equals the nominal significance level $%
\delta $ as long as $n\geq 2$ and $\delta \leq 0.08326$. It is not difficult
to see that in this location problem $T_{uncorr}\left( e_{i}\left( n\right)
\right) =1$ holds for every $i$ (note that $\mu _{0}^{\ast }=0$) and thus
the inequality $T_{uncorr}\left( e_{i}\left( n\right) \right) <C$ always
holds whenever $C>1$. Hence Case 1 of the variant of Theorem \ref{thmhet}
just discussed does not arise whenever $C>1$ which is in line with the
results in \cite{BakiSzek2005}. However, note that Case 2 of that theorem
then always applies (since obviously $e_{i}\left( n\right) \notin \mathfrak{M%
}=\limfunc{span}\left( e_{+}\right) $), showing that the standard $t$-test
suffers from severe power deficiencies under heteroskedasticity of unknown
form in case $n\geq 2$ and $\delta \leq 0.08326$ (noting that the squared
standard $t$-statistic is the standard $F$-statistic).

\section{General Principles Underlying Size and Power Results for Tests of
Linear Restrictions in Regression Models with Nonspherical Disturbances\label%
{General}}

The results on size and power properties given in the previous sections are
obtained as special cases of a more general theory that applies to a large
class of tests and to general covariance models $\mathfrak{C}$ (which thus
are not restricted to covariance structures resulting from stationary
disturbances or from heteroskedasticity). This theory is provided in the
present section. We use the notation and assumptions of Section \ref%
{HTFramework}. Since invariance properties of tests will play an important r%
\^{o}le in some of the results to follow, the next subsection collects some
relevant results related to invariance. In Subsection \ref{sec_neg} we
provide conditions under which the tests considered have highly unpleasant
size or power properties. This result is based on a "concentration" effect.
In contrast, Subsection \ref{sec_pos} provides conditions under which tests
do not suffer from the size and power problems just mentioned. Subsection %
\ref{Impclass} then specializes the results of the preceding subsections to
a class of tests which can be described as nonsphericity-corrected $F$-type
tests. This class of tests contains virtually all so-called
heteroskedasticity and autocorrelation robust tests available in the
literature as special cases. Furthermore, Subsection \ref{Impclass} also
contains another negative result, the derivation of which exploits the
particular structure of these tests.

\subsection{Some preliminaries on groups and invariance\label{sec_inv}}

Let $G$ be a group of bijective Borel-measurable transformations of $\mathbb{%
R}^{n}$ into itself, the group operation being the composition of
transformations. A function $S$ defined on $\mathbb{R}^{n}$ is said to be
invariant under the group $G$ if $S(g(y))=S(y)$ for all $y\in \mathbb{R}^{n}$
and all $g\in G$. A subset $A$ of $\mathbb{R}^{n}$ is said to be invariant
under $G$ if $g(A)\subseteq A$ holds for every $g\in G$. Since with $g$ also 
$g^{-1}$ belongs to $G$, this is equivalent to $g(A)=A$ for every $g\in G$,
and thus to invariance of the indicator function of $A$ as defined before.%
\footnote{%
If $G$ is only a collection of bijective transformations on $\mathbb{R}^{n}$
but is not a group, then invariance of $A$ does not imply $g(A)=A$ in
general, and in particular does not coincide with the notion of invariance
of the indicator function of $A$.} Clearly, invariance of $S:\mathbb{R}%
^{n}\rightarrow \overline{\mathbb{R}}$, the extended real line, under the
group $G$ implies invariance of the super-level sets $W=\left\{ y:S(y)\geq
C\right\} $. Furthermore, a function $S$ defined on $\mathbb{R}^{n}$ is said
to be almost invariant under the group $G$ if $S(g(y))=S(y)$ holds for all $%
g\in G$ and all $y\in \mathbb{R}^{n}\backslash N(g)$ with Borel-sets $N(g)$
satisfying $\lambda _{\mathbb{R}^{n}}\left( N(g)\right) =0$ and also $%
\lambda _{\mathbb{R}^{n}}\left( g^{\prime -1}(N(g))\right) =0$ for all $%
g^{\prime }\in G$.\footnote{%
The additional requirement $\lambda _{\mathbb{R}^{n}}\left( g^{\prime
-1}(N(g))\right) =0$ for all $g^{\prime }\in G$ of course implies $\lambda _{%
\mathbb{R}^{n}}\left( N(g)\right) =0$ and may appear artificial at first
sight. However, it arises naturally in the context of testing problems that
are invariant under the group $G$ and for which the relevant family of
probability measures is equivalent to $\lambda _{\mathbb{R}^{n}}$, cf. \cite%
{LR05}, Section 6.5. Regardless of this, the additional requirement already
follows from $\lambda _{\mathbb{R}^{n}}\left( N(g)\right) =0$ in case the
group $G$ is a group of affine transformations on $\mathbb{R}^{n}$, which
will be the groups we are interested in.} A subset $A$ of $\mathbb{R}^{n}$
is said to be almost invariant if $g(A)\subseteq A\cup N(g)$ holds for every 
$g\in G$ with the Borel-sets $N(g)$ satisfying $\lambda _{\mathbb{R}%
^{n}}\left( N(g)\right) =0$ and $\lambda _{\mathbb{R}^{n}}\left( g^{\prime
-1}(N(g))\right) =0$ for all $g^{\prime }\in G$. It is easy to see that this
is equivalent to $g(A)\bigtriangleup A\subseteq N^{\ast }(g)$ for every $%
g\in G$, with Borel-sets $N^{\ast }(g)$ satisfying $\lambda _{\mathbb{R}%
^{n}}\left( N^{\ast }(g)\right) =0$ and $\lambda _{\mathbb{R}^{n}}\left(
g^{\prime -1}(N^{\ast }(g))\right) =0$ for all $g^{\prime }\in G$; thus it
is equivalent to almost invariance of the indicator function of $A$.
Clearly, almost invariance of $S:\mathbb{R}^{n}\rightarrow \overline{\mathbb{%
R}}$ under the group $G$ implies almost invariance of the super-level sets $%
W=\left\{ y:S(y)\geq C\right\} $.

We are interested in some particular groups of affine transformations. For
an affine subspace $\mathfrak{N}$ of $\mathbb{R}^{n}$ let 
\begin{equation*}
G(\mathfrak{N})=\left\{ g_{\alpha ,\nu ,\nu ^{\prime }}:\alpha \neq 0\text{, 
}\nu ^{\prime }\in \mathfrak{N}\right\}
\end{equation*}%
for some fixed but arbitrary $\nu \in \mathfrak{N}$, where the affine map $%
g_{\alpha ,\nu ,\nu ^{\prime }}$ is given by $g_{\alpha ,\nu ,\nu ^{\prime
}}(y)=\alpha (y-\nu )+\nu ^{\prime }$ with $\alpha \in \mathbb{R}$. Observe
that $G(\mathfrak{N})$ does not depend on the choice of $\nu $ (in
particular, if $\mathfrak{N}$ is a linear subspace, one may choose $\nu =0$%
). Hence, $G(\mathfrak{N})$ can also be written in a redundant way as 
\begin{equation*}
G(\mathfrak{N})=\left\{ g_{\alpha ,\nu ,\nu ^{\prime }}:\alpha \neq 0\text{, 
}\nu \in \mathfrak{N}\text{, }\nu ^{\prime }\in \mathfrak{N}\right\} \text{.}
\end{equation*}%
It is easy to see that $G(\mathfrak{N})$ is a group w.r.t. composition which
is non-abelian except if $\mathfrak{N}$ is a singleton. For later use we
also note that $\mathfrak{N}$ as well as $\mathbb{R}^{n}\backslash \mathfrak{%
N}$ are invariant under $G(\mathfrak{N})$, and that $G(\mathfrak{N})$ acts
transitively on $\mathfrak{N}$ (but not on $\mathbb{R}^{n}\backslash 
\mathfrak{N}$ in general). Furthermore, note that the elements of $G(%
\mathfrak{N})$ can also be written as $g_{\alpha ,\nu ,\nu ^{\prime
}}(y)=\alpha y+(1-\alpha )\nu +(\nu ^{\prime }-\nu )$.

\begin{remark}
We make an observation on the structure of $G(\mathfrak{N})$. Let $G_{1}(%
\mathfrak{N})$ denote the collection of transformations $g_{\alpha ,\nu ,\nu
}(y)$ for every $\alpha \neq 0$ and every $\nu \in \mathfrak{N}$, and let $%
G_{2}(\mathfrak{N})$ denote the collection of transformations $g_{1,\nu ,\nu
^{\prime }}(y)$ for every pair $\nu ,\nu ^{\prime }\in \mathfrak{N}$.
Obviously, $G_{1}(\mathfrak{N})$ as well as $G_{2}(\mathfrak{N})$ are
subsets of $G(\mathfrak{N})$, and every element of $G(\mathfrak{N})$ is the
composition of an element in $G_{2}(\mathfrak{N})$ with an element of $G_{1}(%
\mathfrak{N})$. While $G_{2}(\mathfrak{N})$ is a subgroup, $G_{1}(\mathfrak{N%
})$ is not (as it is not closed under composition) except in the trivial
case where $\mathfrak{N}$ is a singleton. However, the group generated by $%
G_{1}(\mathfrak{N})$ is precisely $G(\mathfrak{N})$. As a consequence, any
function $S$ which is invariant under the elements of $G_{1}(\mathfrak{N})$
(meaning that $S(g(y))=S(y)$ for all $y\in \mathbb{R}^{n}$ and all $g\in
G_{1}(\mathfrak{N})$) is already invariant under the entire group $G(%
\mathfrak{N})$, and a similar statement holds for almost invariance.
\end{remark}

\begin{proposition}
\label{max_inv}A maximal invariant for $G(\mathfrak{N})$ is given by 
\begin{equation*}
h(y)=\left\langle \Pi _{\left( \mathfrak{N}-\nu _{\ast }\right) ^{\bot
}}(y-\nu _{\ast })/\left\Vert \Pi _{\left( \mathfrak{N}-\nu _{\ast }\right)
^{\bot }}(y-\nu _{\ast })\right\Vert \right\rangle ,
\end{equation*}%
where $\nu _{\ast }$ is an arbitrary element of $\mathfrak{N}$. The maximal
invariant $h$ in fact does not depend on the choice of $\nu _{\ast }\in 
\mathfrak{N}$. [Here we use the convention $x/\left\Vert x\right\Vert =0$ if 
$x=0$.]
\end{proposition}

\begin{remark}
Specializing to the case $\mathfrak{N}=\mathfrak{M}_{0}$ it is obvious that $%
\Pi _{\left( \mathfrak{M}_{0}-\mu _{0}\right) ^{\bot }}(y-\mu _{0})$ can be
computed as $y-X\hat{\beta}_{rest}(y)$, where $\hat{\beta}_{rest}$ denotes
the restricted ordinary least squares estimator. It follows that any test
that is invariant under $G(\mathfrak{M}_{0})$ depends only on the normalized
restricted least squares residuals, in fact only on $\left\langle y-X\hat{%
\beta}_{rest}(y)/\left\Vert y-X\hat{\beta}_{rest}(y)\right\Vert
\right\rangle $. [For the tests considered in Subsection \ref{Impclass} one
can obtain this result also directly from the definition of the tests.]
\end{remark}

\bigskip

Consider now the problem of testing $H_{0}$ versus $H_{1}$ as defined in (%
\ref{testing problem}). First observe that the sets $\mathfrak{M}_{0}$ and $%
\mathfrak{M}_{1}$ are invariant under the transformations in $G(\mathfrak{M}%
_{0})$. This implies that the parameter spaces $\mathfrak{M}_{i}\times
(0,\infty )\times \mathfrak{C}$ corresponding to $H_{i}$ (for $i=0,1$) are
each invariant under the associated group $\overline{G(\mathfrak{M}_{0})}$,
i.e., the group consisting of all transformations $\bar{g}_{\alpha ,\mu
_{0},\mu _{0}^{\prime }}$ defined on $\mathfrak{M}\times (0,\infty )\times 
\mathfrak{C}$ given by 
\begin{equation*}
\bar{g}_{\alpha ,\mu _{0},\mu _{0}^{\prime }}(\mu ,\sigma ^{2},\Sigma
)=(\alpha (\mu -\mu _{0})+\mu _{0}^{\prime },\alpha ^{2}\sigma ^{2},\Sigma )
\end{equation*}%
where $\alpha \neq 0$, $\mu _{0}\in \mathfrak{M}_{0}$, $\mu _{0}^{\prime
}\in \mathfrak{M}_{0}$. [Note that the associated group strictly speaking
also depends on $\mathfrak{C}$, but we suppress this in the notation.]
Second, the probability measures associated with $H_{0}$ and $H_{1}$ clearly
satisfy%
\begin{equation}
P_{\mu ,\sigma ^{2}\Sigma }\left( A\right) =P_{\alpha (\mu -\mu _{0})+\mu
_{0}^{\prime },\alpha ^{2}\sigma ^{2}\Sigma }\left( \alpha (A-\mu _{0})+\mu
_{0}^{\prime }\right)  \label{transf_formula}
\end{equation}%
for every $(\mu ,\sigma ^{2},\Sigma )\in \mathfrak{M}\times (0,\infty
)\times \mathfrak{C}$ and every Borel set $A\subseteq \mathbb{R}^{n}$. This
shows that the testing problem considered in (\ref{testing problem}) is
invariant under the group $G(\mathfrak{M}_{0})$ in the sense of \cite{LR05},
Chapters 6 and 8. While trivial, it will be useful to note that (\ref%
{transf_formula}) continues to hold if $\Sigma \in \mathfrak{C}$ is replaced
by an arbitrary nonnegative definite symmetric $n\times n$ matrix $\Phi $.
The next proposition discusses invariance properties of the rejection
probabilities of an almost invariant test $\varphi $ that will be needed in
subsequent subsections. As will be seen later, it is useful to consider in
that proposition the rejection probabilities $E_{\mu ,\sigma ^{2}\Phi
}(\varphi )$ also for $\Phi $ a positive (or sometimes only nonnegative)
definite symmetric $n\times n$ matrix not necessarily belonging to the
assumed covariance model $\mathfrak{C}$.

\begin{proposition}
\label{inv_rej_prob}Let $\varphi :\mathbb{R}^{n}\rightarrow \lbrack 0,1]$ be
a Borel-measurable function that is almost invariant under $G(\mathfrak{M}%
_{0})$.

\begin{enumerate}
\item For every $(\mu ,\sigma ^{2})\in \mathfrak{M}\times (0,\infty )$ and
for every positive definite symmetric $n\times n$ matrix $\Phi $ the
rejection probabilities satisfy%
\begin{equation}
E_{\mu ,\sigma ^{2}\Phi }(\varphi )=E_{\alpha (\mu -\mu _{0})+\mu
_{0}^{\prime },\alpha ^{2}\sigma ^{2}\Phi }(\varphi )  \label{power_inv}
\end{equation}%
for all $\alpha \neq 0$, $\mu _{0}\in \mathfrak{M}_{0}$, $\mu _{0}^{\prime
}\in \mathfrak{M}_{0}$.

\item For every $(\mu ,\sigma ^{2})\in \mathfrak{M}\times (0,\infty )$ and
every positive definite symmetric $n\times n$ matrix $\Phi $ we have the
representation%
\begin{equation}
E_{\mu ,\sigma ^{2}\Phi }(\varphi )=E_{\Pi _{\left( \mathfrak{M}_{0}-\mu
_{0}\right) ^{\bot }}(\mu -\mu _{0})/\sigma +\mu _{0},\Phi }(\varphi
)=E_{\left\langle \Pi _{\left( \mathfrak{M}_{0}-\mu _{0}\right) ^{\bot
}}(\mu -\mu _{0})/\sigma \right\rangle +\mu _{0},\Phi }(\varphi )
\label{power_inv_2}
\end{equation}%
where $\mu _{0}$ is an arbitrary element of $\mathfrak{M}_{0}$. [Note that $%
\Pi _{\left( \mathfrak{M}_{0}-\mu _{0}\right) ^{\bot }}(\mu -\mu
_{0})/\sigma $ actually does not depend on the choice of $\mu _{0}$, and $%
\Pi _{\left( \mathfrak{M}_{0}-\mu _{0}\right) ^{\bot }}(\mu -\mu _{0})$ can
be computed as $\mu -X\hat{\beta}_{rest}(\mu )$.]

\item The rejection probability $E_{\mu ,\sigma ^{2}\Phi }(\varphi )$
depends on $\left( \mu ,\sigma ^{2}\right) \in \mathfrak{M}\times (0,\infty
) $ and $\Phi $ ($\Phi $ symmetric and positive definite) only through $%
\left( \left\langle \Pi _{\left( \mathfrak{M}_{0}-\mu _{0}\right) ^{\bot
}}(\mu -\mu _{0})/\sigma \right\rangle ,\Phi \right) $. Furthermore, $\Pi
_{\left( \mathfrak{M}_{0}-\mu _{0}\right) ^{\bot }}(\mu -\mu _{0})/\sigma $
is in a bijective correspondence with $\left( R\beta -r\right) /\sigma $
where $\beta $ denotes the coordinates of $\mu $ in the basis given by the
columns of $X$. Thus the rejection probability $E_{\mu ,\sigma ^{2}\Phi
}(\varphi )$ depends on $\left( \mu ,\sigma ^{2}\right) \in \mathfrak{M}%
\times (0,\infty ) $ and $\Phi $ only through $\left( \left\langle \left(
R\beta -r\right) /\sigma \right\rangle ,\Phi \right) $.

\item If $\varphi $ is invariant under $G(\mathfrak{M}_{0})$, then (\ref%
{power_inv}) and (\ref{power_inv_2}) hold even if $\Phi $ is only
nonnegative definite and symmetric (and consequently in this case also the
claim in Part 3 continues to hold for such $\Phi $).
\end{enumerate}
\end{proposition}

\begin{remark}
(i) For $\Phi =\Sigma \in \mathfrak{C}$ relation (\ref{power_inv}) expresses
the fact that the rejection probability of the almost invariant test $%
\varphi $ is invariant under the associated group $\overline{G(\mathfrak{M}%
_{0})}$.

(ii) Setting $\alpha =1$ in (\ref{power_inv}) and holding $\sigma ^{2}$ and $%
\Phi $ fixed, we see that the rejection probability is, in particular,
constant along that translation of $\mathfrak{M}_{0}$ which passes through $%
\mu $.

(iii) If $\mu \in \mathfrak{M}_{0}$, choosing $\mu _{0}=\mu $, $\alpha
=\sigma ^{-1}$ in (\ref{power_inv}), and fixing $\mu _{0}^{\prime }\in 
\mathfrak{M}_{0}$, shows that $E_{\mu ,\sigma ^{2}\Phi }(\varphi )=E_{\mu
_{0}^{\prime },\Phi }(\varphi )$. Hence, for $\mu \in \mathfrak{M}_{0}$, the
rejection probability is constant in $\left( \mu ,\sigma ^{2}\right) $ and
only depends on $\Phi $.

(iv) Occasionally we consider tests $\varphi $ that are only required to be
almost invariant under the subgroup of transformations $y\mapsto \alpha
y+\left( 1-\alpha \right) \mu _{0}$ for a fixed $\mu _{0}\in \mathfrak{M}%
_{0} $, i.e., under the group $G\left( \left\{ \mu _{0}\right\} \right) $.
The results in the above propositions can be easily adapted to this case and
we refrain from spelling out the details. We only note that the analogue to (%
\ref{power_inv}) in this case is given by%
\begin{equation}
E_{\mu ,\sigma ^{2}\Phi }(\varphi )=E_{\alpha (\mu -\mu _{0})+\mu
_{0},\alpha ^{2}\sigma ^{2}\Phi }(\varphi )  \label{power_inv_3}
\end{equation}%
for all $\alpha \neq 0$.
\end{remark}

\bigskip

Part 2 of the above proposition has shown that the rejection probability
depends on the parameters only through $\left( \left\langle \Pi _{\left( 
\mathfrak{M}_{0}-\mu _{0}\right) ^{\bot }}(\mu -\mu _{0})/\sigma
\right\rangle ,\Sigma \right) $. This quantity is recognized as a maximal
invariant in the next result.

\begin{proposition}
\label{prop_3}Let $\mu _{0}\in \mathfrak{M}_{0}$ be arbitrary. Then $\left(
\left\langle \Pi _{\left( \mathfrak{M}_{0}-\mu _{0}\right) ^{\bot }}(\mu
-\mu _{0})/\sigma \right\rangle ,\Sigma \right) $ is a maximal invariant for
the associated group $\overline{G(\mathfrak{M}_{0})}$.
\end{proposition}

\subsection{Negative results\label{sec_neg}}

We next establish a negative result providing conditions under which (i) the
size of a test is $1$, and/or (ii) the power function of a test gets
arbitrarily close to zero. The theorem is based on a "concentration effect"
that we explain now: Suppose one can find a sequence $\Sigma _{m}\in 
\mathfrak{C}$ converging to a singular matrix $\bar{\Sigma}$ and let $%
\mathcal{Z}$ denote the span of the columns of $\bar{\Sigma}$. Let $\mu
_{0}\in \mathfrak{M}_{0}$. Since the probability measures $P_{\mu
_{0},\sigma ^{2}\Sigma _{m}}$ converge weakly to $P_{\mu _{0},\sigma ^{2}%
\bar{\Sigma}}$, which has support $\mu _{0}+\mathcal{Z}$, they concentrate
their mass more and more around $\mu _{0}+\mathcal{Z}$. Suppose first that
one can show that $\mu _{0}+\mathcal{Z}$ is essentially contained in the
interior of the rejection region $W$ in the sense that the set of points in $%
\mu _{0}+\mathcal{Z}$ which are not interior points of $W$ has $\lambda
_{\mu _{0}+\mathcal{Z}}$-measure zero. It then follows that $P_{\mu
_{0},\sigma ^{2}\Sigma _{m}}\left( W\right) $ converges to $P_{\mu
_{0},\sigma ^{2}\bar{\Sigma}}\left( W\right) \geq P_{\mu _{0},\sigma ^{2}%
\bar{\Sigma}}\left( \mu _{0}+\mathcal{Z}\right) =1$, establishing that the
size of the test is $1$. Now, in some cases of interest it turns out that $%
\mu _{0}+\mathcal{Z}$ fails to satisfy the just mentioned "interiority"
condition with respect to the rejection region $W$, but it also turns out
that it does satisfy the "interiority" condition with respect to an
"equivalent" rejection region $W^{\prime }$, which is obtained by adjoining
a $\lambda _{\mathbb{R}^{n}}$-null set to $W$ (for example, for $W^{\prime
}=W\cup \left( \mu _{0}+\mathcal{Z}\right) $). Since the rejection
probabilities corresponding to $W$ and $W^{\prime }$ are identical (as any $%
\Sigma \in \mathfrak{C}$ is positive definite) and thus the two tests have
the same size, the above reasoning can then be applied to $W^{\prime }$,
again showing that the size of the test based on $W$ is $1$ for these cases.
Part 1 of Theorem \ref{inv} below formalizes this reasoning. The same
"concentration effect" reasoning applied to $\mathbb{R}^{n}\backslash W$
instead of $W$ then gives (\ref{inf}). [The remaining claims in Part 2 as
well as Part 3 are then consequences of (\ref{inf}) combined with continuity
or invariance properties of the power function.] It should, however, be
stressed that weak convergence of $P_{\mu _{0},\sigma ^{2}\Sigma _{m}}$ to $%
P_{\mu _{0},\sigma ^{2}\bar{\Sigma}}$ together with the inclusion $\mu _{0}+%
\mathcal{Z}\subseteq W$ (except possibly for a $\lambda _{\mu _{0}+\mathcal{Z%
}}$-null set) alone is \emph{not} sufficient to allow one to draw the
conclusion -- as tempting as it may be -- that $P_{\mu _{0},\sigma
^{2}\Sigma _{m}}\left( W\right) \rightarrow 1$ although "in the limit" $%
P_{\mu _{0},\sigma ^{2}\bar{\Sigma}}\left( W\right) =1$ holds.
Counterexamples where $P_{\mu _{0},\sigma ^{2}\Sigma _{m}}$ converges weakly
to $P_{\mu _{0},\sigma ^{2}\bar{\Sigma}}$ and $\mu _{0}+\mathcal{Z}\subseteq
W$ (and thus $P_{\mu _{0},\sigma ^{2}\bar{\Sigma}}\left( W\right) =1$)
holds, but where $P_{\mu _{0},\sigma ^{2}\Sigma _{m}}\left( W\right) $
converges to a positive number less than $1$ are easily found with the help
of Theorem \ref{TU}. We furthermore note that in a different testing context 
\cite{Mart10} provides a result which also makes use of a "concentration
effect", but his result is not correct as given. For a discussion of these
issues and corrected results see \cite{Prein2013}.

The "concentration effect" reasoning underlying Theorem \ref{inv} of course
hinges crucially on the "interiority" condition (either w.r.t. $W$ or w.r.t. 
$\mathbb{R}^{n}\backslash W$), raising the question why we should expect
this to be satisfied in the applications we have in mind, rather than expect
that $\mu _{0}+\mathcal{Z}$ intersects with both $W$ and $\mathbb{R}%
^{n}\backslash W$ in such a way that the "interiority" condition is neither
satisfied w.r.t. $W$ nor w.r.t. $\mathbb{R}^{n}\backslash W$. Consider the
case where $\mathcal{Z}$ is one-dimensional, a case of paramount importance
in the applications, and suppose also that $W$ is invariant under the group $%
G\left( \mathfrak{M}_{0}\right) $. Then we have the dichotomy that $\left(
\mu _{0}+\mathcal{Z}\right) \backslash \left\{ \mu _{0}\right\} $ either
lies entirely in $W$ or in $\mathbb{R}^{n}\backslash W$, showing that --
except possibly for the point $\mu _{0}$ -- the set $\mu _{0}+\mathcal{Z}$
never intersects both $W$ and $\mathbb{R}^{n}\backslash W$. Moreover, if an
element of $\left( \mu _{0}+\mathcal{Z}\right) \backslash \left\{ \mu
_{0}\right\} $ belongs to the interior of $W$ (of $\mathbb{R}^{n}\backslash
W $, respectively), then $\left( \mu _{0}+\mathcal{Z}\right) \backslash
\left\{ \mu _{0}\right\} $ in its entirety is a subset of the interior of $W$
(of $\mathbb{R}^{n}\backslash W$, respectively). Hence, under the mentioned
invariance and for one-dimensional $\mathcal{Z}$, one can expect the
"interiority" conditions in the subsequent theorem to be satisfied not
infrequently.

\begin{theorem}
\label{inv} Let $W$ be a Borel set in $\mathbb{R}^{n}$, the rejection region
of a test. Furthermore, assume that $\mathcal{Z}$ is a concentration space
of the covariance model $\mathfrak{C}$. Then the following holds:

\begin{enumerate}
\item If $\mu _{0}\in \mathfrak{M}_{0}$ satisfies%
\begin{equation}
\lambda _{\mu _{0}+\mathcal{Z}}\left( \limfunc{bd}\left( W\cup \left( \mu
_{0}+\mathcal{Z}\right) \right) \right) =0,  \label{int_cond}
\end{equation}%
then for every $0<\sigma ^{2}<\infty $%
\begin{equation*}
\sup\limits_{\Sigma \in \mathfrak{C}}P_{\mu _{0},\sigma ^{2}\Sigma }(W)=1
\end{equation*}%
holds; in particular, the size of the test equals $1$. [In case $W$ is of
the form $\left\{ y\in \mathbb{R}^{n}:T(y)\geq C\right\} $ for some
Borel-measurable function $T:\mathbb{R}^{n}\mapsto \overline{\mathbb{R}}$
and $0<C<\infty $, a sufficient condition for (\ref{int_cond}) is that for $%
\lambda _{\mathcal{Z}}$-almost every $z\in \mathcal{Z}$ the test statistic $%
T $ satisfies $T(\mu _{0}+z)>C$ and is lower semicontinuous at $\mu _{0}+z$.]

\item If $\mu _{0}\in \mathfrak{M}_{0}$ satisfies%
\begin{equation}
\lambda _{\mu _{0}+\mathcal{Z}}\left( \limfunc{bd}\left( \left( \mathbb{R}%
^{n}\backslash W\right) \cup \left( \mu _{0}+\mathcal{Z}\right) \right)
\right) =0,  \label{int_cond_2}
\end{equation}%
then for every $0<\sigma ^{2}<\infty $%
\begin{equation}
\inf\limits_{\Sigma \in \mathfrak{C}}P_{\mu _{0},\sigma ^{2}\Sigma }(W)=0,
\label{inf}
\end{equation}%
and hence%
\begin{equation*}
\inf_{\mu _{1}\in \mathfrak{M}_{1}}\inf\limits_{\Sigma \in \mathfrak{C}%
}P_{\mu _{1},\sigma ^{2}\Sigma }(W)=0,
\end{equation*}%
holds for every $0<\sigma ^{2}<\infty $. In particular, the test is biased
(except in the trivial case where its size is zero). [In case $W$ is of the
form $\left\{ y\in \mathbb{R}^{n}:T(y)\geq C\right\} $ for some
Borel-measurable function $T:\mathbb{R}^{n}\mapsto \overline{\mathbb{R}}$
and $0<C<\infty $, a sufficient condition for (\ref{int_cond_2}) is that for 
$\lambda _{\mathcal{Z}}$-almost every $z\in \mathcal{Z}$ the test statistic $%
T$ satisfies $T(\mu _{0}+z)<C$ and is upper semicontinuous at $\mu _{0}+z$.]

\item Suppose that condition (\ref{inf}) is satisfied for some $\mu _{0}\in 
\mathfrak{M}_{0}$ and some $0<\sigma ^{2}<\infty $. Furthermore, assume that 
$W$ is almost invariant under the group $G\left( \left\{ \mu _{0}\right\}
\right) $. Then for every $\mu _{1}\in \mathfrak{M}_{1}$ we have 
\begin{equation*}
\inf\limits_{0<\sigma ^{2}<\infty }\inf\limits_{\Sigma \in \mathfrak{C}%
}P_{\mu _{1},\sigma ^{2}\Sigma }(W)=0.
\end{equation*}%
[In case $W$ is of the form $\left\{ y\in \mathbb{R}^{n}:T(y)\geq C\right\} $
for some Borel-measurable function $T:\mathbb{R}^{n}\mapsto \overline{%
\mathbb{R}}$ and $0<C<\infty $, almost invariance of $W$ under the group $%
G\left( \left\{ \mu _{0}\right\} \right) $ follows from almost invariance of 
$T$ under $G\left( \left\{ \mu _{0}\right\} \right) $.]
\end{enumerate}
\end{theorem}

\begin{remark}
\label{trivial}(i) The conclusions of the above theorem immediately also
apply to every test statistic $T^{\prime }$ that is $\lambda _{\mathbb{R}%
^{n}}$-almost everywhere equal to a test statistic $T$ satisfying the
assumptions of the theorem.

(ii) Let $\varphi :\mathbb{R}^{n}\mapsto \lbrack 0,1]$ be Borel-measurable,
i.e., a test. If the set $\left\{ y:\varphi (y)=1\right\} $ satisfies the
assumptions on $W$ in Part 1 of the above theorem, then for every $0<\sigma
^{2}<\infty $%
\begin{equation*}
\sup\limits_{\Sigma \in \mathfrak{C}}E_{\mu _{0},\sigma ^{2}\Sigma }\left(
\varphi \right) =1
\end{equation*}%
holds. If the set $\left\{ y:\varphi (y)=0\right\} $ satisfies the
assumptions on $\mathbb{R}^{n}\backslash W$ in Part 2 of the above theorem
then for every $0<\sigma ^{2}<\infty $%
\begin{equation*}
\inf\limits_{\Sigma \in \mathfrak{C}}E_{\mu _{0},\sigma ^{2}\Sigma }\left(
\varphi \right) =0
\end{equation*}%
holds. A similar remark applies to Part 3 of the theorem, provided $\varphi $
is almost invariant under $G\left( \left\{ \mu _{0}\right\} \right) $.
\end{remark}

\begin{remark}
\label{special_AR_1}If the covariance model $\mathfrak{C}$ contains AR(1)
correlation matrices $\Lambda (\rho _{m})$ for some sequence $\rho _{m}\in
\left( -1,1\right) $ with $\rho _{m}\rightarrow 1$ ($\rho _{m}\rightarrow -1$%
, respectively), then $\limfunc{span}\left( e_{+}\right) $ ($\limfunc{span}%
\left( e_{-}\right) $, respectively) is a concentration space of $\mathfrak{C%
}$ (cf. Lemma \ref{AR_1} in Appendix \ref{AR_100}). Hence Theorem \ref{inv}
applies with $\mathcal{Z}=\limfunc{span}\left( e_{+}\right) $ ($\mathcal{Z}=%
\limfunc{span}\left( e_{-}\right) $, respectively). In particular, if $%
\mathfrak{C}$ contains $\mathfrak{C}_{AR(1)}$, then Theorem \ref{inv}
applies with $\mathcal{Z}=\limfunc{span}\left( e_{+}\right) $ as well as
with $\mathcal{Z}=\limfunc{span}\left( e_{-}\right) $.
\end{remark}

\subsection{Positive results\label{sec_pos}}

The next theorem isolates conditions under which a test does not suffer from
the extreme size and power problems encountered in the preceding subsection.
In particular, we provide conditions which guarantee that the size is
bounded away from one and that the power function is bounded away from zero.
The theorem assumes that the test $\varphi $ -- apart from being (almost)
invariant under the group $G(\mathfrak{M}_{0})$ -- is also invariant under
addition of elements of $J(\mathfrak{C})$ defined below. This additional
invariance assumption will be automatically satisfied in the important
special case where $\varphi $ is invariant under the group $G(\mathfrak{M}%
_{0})$ and where $J(\mathfrak{C})\subseteq \mathfrak{M}_{0}-\mu _{0}$ for
some $\mu _{0}\in \mathfrak{M}_{0}$ (and hence for all $\mu _{0}\in 
\mathfrak{M}_{0}$) as then the maps $x\mapsto x+z$ for $z\in J(\mathfrak{C})$
are elements of $G(\mathfrak{M}_{0})$; see also Proposition \ref{enforce_inv}
and the attending discussion in Subsection \ref{Impclass}. A second
assumption of the subsequent theorem is that the covariance model $\mathfrak{%
C}$ is bounded which is typically a harmless assumption in applications as
it is, e.g., always satisfied if the elements of $\mathfrak{C}$ are
normalized such that the largest diagonal element is $1$, or such that the
trace is $1$. The theorem also maintains a further assumption on the
covariance model $\mathfrak{C}$ related to the way sequences of elements in $%
\mathfrak{C}$ approach singular matrices. This condition has to be verified
for the covariance model $\mathfrak{C}$ in any particular application. A
verification for $\mathfrak{C}_{AR(1)}$ is given in Appendix \ref{AR_100},
cf. also Remarks \ref{rem_AR_1} and \ref{special_AR_1_1}.

For a covariance model $\mathfrak{C}$ define now 
\begin{equation*}
J(\mathfrak{C})=\bigcup \left\{ \limfunc{span}(\bar{\Sigma}):\det \bar{\Sigma%
}=0\text{, }\bar{\Sigma}=\lim_{m\rightarrow \infty }\Sigma _{m}\text{ for a
sequence }\Sigma _{m}\in \mathfrak{C}\right\} ,
\end{equation*}%
i.e., $J(\mathfrak{C})$\ is the union of all concentration spaces of the
covariance model $\mathfrak{C}$. [Note that the subsequent results remain
valid in the case where $J(\mathfrak{C})$ is empty.]

\begin{theorem}
\label{TU}Let $\varphi :\mathbb{R}^{n}\rightarrow \lbrack 0,1]$ be a
Borel-measurable function that is almost invariant under $G(\mathfrak{M}%
_{0}) $. Suppose that $\varphi $ is neither $\lambda _{\mathbb{R}^{n}}$%
-almost everywhere equal to $1$ nor $\lambda _{\mathbb{R}^{n}}$-almost
everywhere equal to $0$. Suppose further that 
\begin{equation}
\varphi (x+z)=\varphi (x)\qquad \text{for every }x\in \mathbb{R}^{n}\text{
and every }z\in J(\mathfrak{C}).  \label{phi_inv_wrt_z}
\end{equation}%
Assume that $\mathfrak{C}$ is bounded (as a subset of $\mathbb{R}^{n\times
n} $). Assume also that for every sequence $\Sigma _{m}\in \mathfrak{C}$
converging to a singular $\bar{\Sigma}$ there exists a subsequence $%
(m_{i})_{i\in \mathbb{N}}$ and a sequence of positive real numbers $%
s_{m_{i}} $ such that the sequence of matrices $D_{m_{i}}=\Pi _{\limfunc{span%
}(\bar{\Sigma})^{\bot }}\Sigma _{m_{i}}\Pi _{\limfunc{span}(\bar{\Sigma}%
)^{\bot }}/s_{m_{i}}$ converges to a matrix $D$ which is regular on the
orthogonal complement of $\limfunc{span}(\bar{\Sigma})$ (meaning that the
linear map corresponding to $D$ is injective when restricted to the
orthogonal complement of $\limfunc{span}(\bar{\Sigma})$)\footnote{%
Of course, $D$ maps every element of $\limfunc{span}(\bar{\Sigma})$ into
zero by construction.}. Then the following holds:

\begin{enumerate}
\item The size of the test $\varphi $ is strictly less than $1$, i.e.,%
\begin{equation*}
\sup\limits_{\mu _{0}\in \mathfrak{M}_{0}}\sup\limits_{0<\sigma ^{2}<\infty
}\sup\limits_{\Sigma \in \mathfrak{C}}E_{\mu _{0},\sigma ^{2}\Sigma
}(\varphi )<1.
\end{equation*}%
Furthermore,%
\begin{equation*}
\inf_{\mu _{0}\in \mathfrak{M}_{0}}\inf_{0<\sigma ^{2}<\infty }\inf_{\Sigma
\in \mathfrak{C}}E_{\mu _{0},\sigma ^{2}\Sigma }(\varphi )>0.
\end{equation*}

\item Suppose additionally that for every sequence $\nu _{m}\in \Pi _{\left( 
\mathfrak{M}_{0}-\mu _{0}\right) ^{\bot }}(\mathfrak{M}_{1}-\mu _{0})$ with $%
\left\Vert \nu _{m}\right\Vert \rightarrow \infty $ and for every sequence $%
\Phi _{m}$ of positive definite symmetric $n\times n$ matrices with $\Phi
_{m}\rightarrow \Phi $, $\Phi $ positive definite, we have 
\begin{equation}
\liminf_{m\rightarrow \infty }E_{\nu _{m}+\mu _{0},\Phi _{m}}(\varphi )>0,
\label{power_ass}
\end{equation}%
where $\mu _{0}$ is an element of $\mathfrak{M}_{0}$. [This condition
clearly does not depend on the particular choice of $\mu _{0}\in \mathfrak{M}%
_{0}$.]. Then the infimal power is bounded away from zero, i.e., 
\begin{equation*}
\inf_{\mu _{1}\in \mathfrak{M}_{1}}\inf\limits_{0<\sigma ^{2}<\infty
}\inf\limits_{\Sigma \in \mathfrak{C}}E_{\mu _{1},\sigma ^{2}\Sigma
}(\varphi )>0.
\end{equation*}

\item Suppose that the limit inferior in (\ref{power_ass}) is $1$ for every
sequence $\nu _{m}$ and $\Phi _{m}$ as specified above. Then for every $%
0<c<\infty $%
\begin{equation}
\inf_{\substack{ \mu _{1}\in \mathfrak{M}_{1},0<\sigma ^{2}<\infty  \\ %
d\left( \mu _{1},\mathfrak{M}_{0}\right) /\sigma \geq c}}E_{\mu _{1},\sigma
^{2}\Sigma _{m}}(\varphi )\rightarrow 1  \label{unit_power}
\end{equation}%
holds for $m\rightarrow \infty $ and for any sequence $\Sigma _{m}\in 
\mathfrak{C}$ satisfying $\Sigma _{m}\rightarrow \bar{\Sigma}$ with $\bar{%
\Sigma}$ a singular matrix. Furthermore, for every sequence $0<c_{m}<\infty $%
\begin{equation}
\inf_{\substack{ \mu _{1}\in \mathfrak{M}_{1},  \\ d\left( \mu _{1},%
\mathfrak{M}_{0}\right) \geq c_{m}}}E_{\mu _{1},\sigma _{m}^{2}\Sigma
_{m}}(\varphi )\rightarrow 1  \label{unit_power_2}
\end{equation}%
holds for $m\rightarrow \infty $ whenever $0<\sigma _{m}^{2}<\infty $, $%
c_{m}/\sigma _{m}\rightarrow \infty $, and the sequence $\Sigma _{m}\in 
\mathfrak{C}$ satisfies $\Sigma _{m}\rightarrow \bar{\Sigma}$ with $\bar{%
\Sigma}$ a positive definite matrix. [The very last statement even holds
without recourse to condition (\ref{phi_inv_wrt_z}) and the condition on $%
\mathfrak{C}$ following (\ref{phi_inv_wrt_z}).]
\end{enumerate}
\end{theorem}

The first two parts of the preceding theorem provide conditions under which
the size is strictly less than $1$ and the infimal power is strictly
positive, while the third part provides conditions under which the power
approaches $1$ in certain parts of the parameter space, the parts being
characterized by the property that either $\left\Vert \left( R\beta
^{(1)}-r\right) /\sigma \right\Vert $ is bounded away from zero and $\Sigma
_{m}$ approaches a singular matrix, or that $\left\Vert \left( R\beta
^{(1)}-r\right) /\sigma \right\Vert \rightarrow \infty $ and $\Sigma _{m}$
approaches a positive definite matrix. Here $\beta ^{(1)}$ is the parameter
vector corresponding to $\mu _{1}$. Note that $d\left( \mu _{1},\mathfrak{M}%
_{0}\right) $ is bounded from above as well as from below by multiples of $%
\left\Vert R\beta ^{(1)}-r\right\Vert $, where the constants involved are
positive and depend only on $X$, $R$, and $r$.

\begin{remark}
\label{rem_positive}(i) Because $J(\mathfrak{C})$ as a union of linear
spaces is homogenous, condition (\ref{phi_inv_wrt_z}) is equivalent to the
condition that $\varphi (x+z)=\varphi (x)$ holds for every $x\in \mathbb{R}%
^{n}$ and every $z\in \limfunc{span}\left( J(\mathfrak{C})\right) $.

(ii) If condition (\ref{power_ass}) in Theorem \ref{TU} is replaced by the
weaker condition%
\begin{equation}
\liminf_{m\rightarrow \infty }E_{d_{m}(\mu _{1}-\mu _{0})+\mu _{0},\Phi
_{m}}(\varphi )>0,  \label{power_ass_weaker}
\end{equation}%
for every $\mu _{1}\in \mathfrak{M}_{1}$, for every $d_{m}\rightarrow \infty 
$ and every sequence $\Phi _{m}$ of positive definite symmetric $n\times n$
matrices with $\Phi _{m}\rightarrow \Phi $, $\Phi $ a positive definite
matrix, then we can only establish for every $\mu _{1}\in \mathfrak{M}_{1}$
that 
\begin{equation*}
\inf\limits_{0<\sigma ^{2}<\infty }\inf\limits_{\Sigma \in \mathfrak{C}%
}E_{\mu _{1},\sigma ^{2}\Sigma }(\varphi )>0.
\end{equation*}%
If the limes inferior in (\ref{power_ass_weaker}) is $1$ for every $\mu _{1}$%
, $d_{m}$, and $\Phi _{m}$ as specified above, then for every $\mu _{1}\in 
\mathfrak{M}_{1}$ and every $0<\sigma _{\ast }^{2}<\infty $ we have 
\begin{equation*}
\inf_{0<\sigma ^{2}\leq \sigma _{\ast }^{2}}E_{\mu _{1},\sigma ^{2}\Sigma
_{m}}(\varphi )\rightarrow 1
\end{equation*}%
for any sequence $\Sigma _{m}\in \mathfrak{C}$ satisfying $\Sigma
_{m}\rightarrow \bar{\Sigma}$ with $\bar{\Sigma}$ a singular matrix; and
also $E_{\mu _{1},\sigma _{m}^{2}\Sigma _{m}}(\varphi )\rightarrow 1$ holds
whenever $\sigma _{m}^{2}\rightarrow 0$ and the sequence $\Sigma _{m}\in 
\mathfrak{C}$ satisfies $\Sigma _{m}\rightarrow \bar{\Sigma}$ with $\bar{%
\Sigma}$ a positive definite matrix. [The very last statement even holds
without recourse to condition (\ref{phi_inv_wrt_z}) and the condition on $%
\mathfrak{C}$ following (\ref{phi_inv_wrt_z}).]
\end{remark}

\bigskip

The subsequent theorem elaborates on Part 1 of Theorem \ref{TU} and shows
that under the additional assumptions one can not only guarantee that the
size of the test is smaller than $1$, but one can, for any prescribed
significance level $\delta $ ($0<\delta <1$), construct the test in such a
way that it has size not exceeding $\delta $. The result applies in
particular to the important case where the tests are of the form $\varphi
_{C}=\boldsymbol{1}\left( T\geq C\right) $ for some test statistic $T$. Note
that for any $C_{k}\uparrow \infty $ the sequence of tests $\varphi _{C_{k}}$
clearly satisfies condition (\ref{monotone}) in the subsequent theorem
provided $\left\{ y:T(y)=\infty \right\} $ is a $\lambda _{\mathbb{R}^{n}}$%
-null set. Thus in this case the theorem shows that for any given
significance level $\delta $, $0<\delta <1$, we can find a critical value $%
C(\delta )$ such that the test $\varphi _{C(\delta )}$ has a size not
exceeding $\delta $.

\begin{theorem}
\label{TUU}Let $\varphi _{k}:\mathbb{R}^{n}\rightarrow \lbrack 0,1]$ for $%
k\geq 1$ be a sequence of Borel-measurable functions each of which satisfies
the assumptions for Part 1 of Theorem \ref{TU}, and let $\mathfrak{C}$ also
satisfy the assumptions of that theorem. Furthermore assume that the
sequence $\varphi _{k}$ satisfies%
\begin{equation}
E_{\mu _{0}^{\ast },\Phi }(\varphi _{k})\downarrow 0  \label{monotone}
\end{equation}%
as $k\uparrow \infty $ for some $\mu _{0}^{\ast }\in \mathfrak{M}_{0}$ and
all positive definite symmetric $n\times n$ matrices $\Phi $. Then for every 
$\delta $, $0<\delta <1$, there exists a $k_{0}=k_{0}(\delta )$ such that%
\begin{equation*}
\sup\limits_{\mu _{0}\in \mathfrak{M}_{0}}\sup\limits_{0<\sigma ^{2}<\infty
}\sup\limits_{\Sigma \in \mathfrak{C}}E_{\mu _{0},\sigma ^{2}\Sigma
}(\varphi _{k_{0}})\leq \delta .
\end{equation*}
\end{theorem}

\begin{remark}
(i) The assumption in Theorem \ref{TU} that $\varphi _{k}$ is not $\lambda _{%
\mathbb{R}^{n}}$-almost everywhere equal to $0$ is of course irrelevant for
the result in Theorem \ref{TUU}.

(ii) Of course, the second part of Part 1 of Theorem \ref{TU} immediately
applies to $\varphi _{k_{0}}$; and Parts 2 and 3 of that theorem also apply
to $\varphi _{k_{0}}$ provided $\varphi _{k_{0}}$ satisfies the respective
additional conditions.
\end{remark}

\begin{remark}
\label{rem_AR_1}(i) In case the covariance model $\mathfrak{C}$ equals $%
\mathfrak{C}_{AR(1)}$, the boundedness condition in Theorems \ref{TU} and %
\ref{TUU} is clearly satisfied and $J(\mathfrak{C})$ reduces to $\limfunc{%
span}\left( e_{+}\right) \cup \limfunc{span}\left( e_{-}\right) $.
Furthermore, the condition on the covariance model $\mathfrak{C}$ in those
theorems expressed in terms of the matrices $D_{m}$ is then also satisfied
as shown in Lemma \ref{AR_1} in Appendix \ref{AR_100}. Also note that in
this case the sequences $\Sigma _{m}$ in Part 3 of Theorem \ref{TU}
converging to a singular matrix are of the form $\Lambda \left( \rho
_{m}\right) $ with $\rho _{m}\rightarrow 1$ or $\rho _{m}\rightarrow -1$.

(ii) More generally suppose that $\mathfrak{C}$ is norm-bounded, has $%
e_{+}e_{+}^{\prime }$ and $e_{-}e_{-}^{\prime }$ as the only singular
accumulation points, and has the property that for every sequence $\Sigma
_{m}\in \mathfrak{C}$ converging to one of these limit points there exists a
sequence $(\rho _{m})_{m\in \mathbb{N}}$ in $(-1,1)$ such that $\Lambda
^{-1/2}(\rho _{m})\Sigma _{m}\Lambda ^{-1/2}(\rho _{m})\rightarrow I_{n}$
for $m\rightarrow \infty $ (that is, near the "singular boundary" the
covariance model $\mathfrak{C}$ behaves similar to $\mathfrak{C}_{AR(1)}$).
Then $J(\mathfrak{C})$ is as in (i) and again the conditions on the
covariance model $\mathfrak{C}$ in Theorems \ref{TU} and \ref{TUU} are
satisfied.
\end{remark}

\subsection{Size and power properties of a common class of tests:
Nonsphericity-corrected $F$-type tests\label{Impclass}}

In this subsection we specialize the preceding results to a broad class of
tests of linear restrictions in linear regression models with nonspherical
errors and derive a further result specific to this class. The class
considered in this subsection contains the vast majority of tests proposed
in the literature for this testing problem. We start with a pair of
estimators $\check{\beta}$ and $\check{\Omega}$, where $\check{\Omega}$
typically has the interpretation of an estimator of the variance covariance
matrix of $R\check{\beta}-r$ under the null hypothesis. Similar as in
previous sections, the estimators are viewed as functions of $y\in \mathbb{R}%
^{n}$, but it proves useful to allow for cases where the estimators are not
defined for some exceptional values of $y$. We impose the following
assumption on the estimators.

\begin{assumption}
\label{AE} (i) The estimators$\ \check{\beta}:\mathbb{R}^{n}\backslash
N\rightarrow \mathbb{R}^{k}$ and $\check{\Omega}:\mathbb{R}^{n}\backslash
N\rightarrow \mathbb{R}^{q\times q}$ are well-defined and continuous on the
complement of a closed $\lambda _{\mathbb{R}^{n}}$-null set $N$ in the
sample space $\mathbb{R}^{n}$, with $\check{\Omega}$ also being symmetric on 
$\mathbb{R}^{n}\backslash N$.

(ii) The set $\mathbb{R}^{n}\backslash N$ is invariant under the group $G(%
\mathfrak{M})$, i.e., $y\in \mathbb{R}^{n}\backslash N$ implies $\alpha
y+X\gamma \in \mathbb{R}^{n}\backslash N$ for every $\alpha \neq 0$ and
every $\gamma \in \mathbb{R}^{k}$.

(iii) The estimators satisfy the equivariance properties $\check{\beta}%
(\alpha y+X\gamma )=\alpha \check{\beta}(y)+\gamma $ and $\check{\Omega}%
(\alpha y+X\gamma )=\alpha ^{2}\check{\Omega}(y)$ for every $y\in \mathbb{R}%
^{n}\backslash N$, for every $\alpha \neq 0$, and for every $\gamma \in 
\mathbb{R}^{k}$.

(iv) $\check{\Omega}$ is $\lambda _{\mathbb{R}^{n}}$-almost everywhere
nonsingular on $\mathbb{R}^{n}\backslash N$.
\end{assumption}

We make a few obvious observations: First, the invariance of $\mathbb{R}%
^{n}\backslash N$ under the group $G(\mathfrak{M})$ expressed in Assumption %
\ref{AE} is equivalent to the same invariance property of $N$ itself.
Second, since $N$ is closed by Assumption \ref{AE}, it follows that either $%
N $ is empty or otherwise must at least contain $\mathfrak{M}$ (to see this
note that $y\in N$ implies $\alpha y\in N$ for $\alpha $ arbitrarily close
to zero which in turn implies $0\in N$ by closedness of $N$). Third, given
Assumption \ref{AE} holds, the sets $\left\{ y\in \mathbb{R}^{n}\backslash
N:\det {\check{\Omega}(y)}=0\right\} $ and $\left\{ y\in \mathbb{R}%
^{n}\backslash N:\det {\check{\Omega}(y)}\neq 0\right\} $ are invariant
under the transformations in $G(\mathfrak{M})$, and the set 
\begin{equation}
N^{\ast }=N\cup \left\{ y\in \mathbb{R}^{n}\backslash N:\det {\check{\Omega}%
(y)}=0\right\}  \label{Nstar}
\end{equation}%
is a closed $\lambda _{\mathbb{R}^{n}}$-null set that is also invariant
under the transformations in $G(\mathfrak{M})$; cf. Lemma \ref{aux} in
Appendix \ref{App_E}. Hence, the set $\left\{ y\in \mathbb{R}^{n}\backslash
N:\det {\check{\Omega}(y)}=0\right\} $ could in principle have been absorbed
into $N$ in the above assumption; however, we shall not do so since keeping
the exceptional set $N$ as small as possible will lead to stronger results.
Furthermore, $\mathfrak{M}\subseteq N^{\ast }$ always holds. To see this
note that $\mathfrak{M}\subseteq N\subseteq N^{\ast }$ holds if $N$ is not
empty as noted above; in case $N$ is empty, $\check{\Omega}(y)$ is
well-defined for every $y$ and $\check{\Omega}(0)=\check{\Omega}(\alpha
0)=\alpha ^{2}\check{\Omega}(0)$ must hold, implying $\check{\Omega}(0)=0$
and thus also $\check{\Omega}(X\gamma )=\check{\Omega}(\alpha 0+X\gamma
)=\alpha ^{2}\check{\Omega}(0)=0$. In particular, this shows that either $%
\check{\Omega}$ is not defined on $\mathfrak{M}$ or is zero on $\mathfrak{M}$%
.

Given estimators $\check{\beta}$ and $\check{\Omega}$ satisfying Assumption %
\ref{AE} we define the test statistic 
\begin{equation}
T(y)=%
\begin{cases}
(R\check{\beta}(y)-r)^{\prime }\check{\Omega}^{-1}(y)(R\check{\beta}(y)-r),
& y\in \mathbb{R}^{n}\backslash N^{\ast }, \\ 
0, & y\in N^{\ast }\text{.}%
\end{cases}
\label{DT}
\end{equation}

We note that assigning the test statistic the value zero at points $y\in 
\mathbb{R}^{n}$ for which either $y\in N$ or $\det (\check{\Omega})(y)=0$
holds is arbitrary, but has no effect on the rejection probabilities of the
test, since $N^{\ast }$ is a $\lambda _{\mathbb{R}^{n}}$-null set as noted
above and since all relevant probability measures $P_{\mu ,\sigma ^{2}\Sigma
}$ are absolutely continuous w.r.t. Lebesgue measure on $\mathbb{R}^{n}$.

In line with the interpretation of $\check{\Omega}$ as an estimator for a
variance covariance matrix, the leading case is when $\check{\Omega}$ is
positive definite almost everywhere (which under Assumption \ref{AE} is
equivalent to nonnegative definiteness almost everywhere). However,
sometimes we encounter situations where this is not guaranteed for a given
fixed sample size (cf. Subsection \ref{GLS}), although typically the
probability of being positive definite will go to one for each fixed value
of the parameters as sample size increases. In order to be able to
accommodate also such cases, Assumption \ref{AE} does not contain a
requirement that $\check{\Omega}$ is positive definite almost everywhere.
Nevertheless, in light of what has just been said, we shall consider the
rejection region to be of the form $\left\{ y\in \mathbb{R}^{n}:T(y)\geq
C\right\} $ for $C$ a real number satisfying $0<C<\infty $.

For some of the results that follow we shall need further conditions on $%
\check{\Omega}$ which, however, are much weaker than the almost everywhere
positive definiteness requirement just mentioned.

\begin{assumption}
\label{omega1} There exists $v\in \mathbb{R}^{q}$, $v\neq 0$, and a $y\in 
\mathbb{R}^{n}\backslash N^{\ast }$ such that $v^{\prime }\check{\Omega}%
^{-1}(y)v>0$ holds.
\end{assumption}

Since under Assumption \ref{AE} the matrix $\check{\Omega}^{-1}(y)$ is
continuous on $\mathbb{R}^{n}\backslash N^{\ast }$, it follows that
Assumption \ref{omega1} in fact implies that $v^{\prime }\check{\Omega}%
^{-1}(y)v>0$ holds on an open set of $y$'s. The condition expressed in the
next assumption is also certainly satisfied if $\check{\Omega}$ is positive
definite almost everywhere. At first glance it may seem that this condition
rules out the case where $\check{\Omega}(y)$ is allowed to be indefinite on
a set of positive Lebesgue measure, but this is not so as $v$ is not allowed
to depend on $y$ in this condition.

\begin{assumption}
\label{omega2}For every $v\in \mathbb{R}^{q}$ with $v\neq 0$ we have $%
\lambda _{\mathbb{R}^{n}}\left( \left\{ y\in \mathbb{R}^{n}\backslash
N^{\ast }:v^{\prime }\check{\Omega}^{-1}(y)v=0\right\} \right) =0$.
\end{assumption}

The following lemma collects some properties of the test statistic that will
be useful in the sequel.

\begin{lemma}
\label{LWT} Suppose Assumption \ref{AE} is satisfied and let $T$ be the test
statistic defined in (\ref{DT}). Then the following holds:

\begin{enumerate}
\item The set $\mathbb{R}^{n}\backslash N^{\ast }$ is invariant under the
elements of $G(\mathfrak{M})$.

\item The test statistic $T$ is continuous on $\mathbb{R}^{n}\backslash
N^{\ast }$; in particular, $T$ is $\lambda _{\mathbb{R}^{n}}$-almost
everywhere continuous on $\mathbb{R}^{n}$.

\item The test statistic $T$ is invariant under the group $G(\mathfrak{M}%
_{0})$. Consequently, the rejection region $W(C)=\left\{ y\in \mathbb{R}%
^{n}:T(y)\geq C\right\} $ and its complement are invariant under $G(%
\mathfrak{M}_{0})$.

\item The set $\left\{ y\in \mathbb{R}^{n}:T(y)=C\right\} $ is a $\lambda _{%
\mathbb{R}^{n}}$-null set for every $0<C<\infty $.

\item Suppose $0<C<\infty $ holds. Then $\left\{ y\in \mathbb{R}%
^{n}\backslash N^{\ast }:T(y)>C\right\} (=\left\{ y\in \mathbb{R}%
^{n}:T(y)>C\right\} )$ is an open set in $\mathbb{R}^{n}$, which is
guaranteed to be non-empty under Assumption \ref{omega1}. Consequently,
under Assumption \ref{omega1} the rejection region $W(C)$ contains a
non-empty open set and thus satisfies $\lambda _{\mathbb{R}^{n}}(W(C))>0$.

\item Suppose $0<C<\infty $ holds. Then $\left\{ y\in \mathbb{R}%
^{n}\backslash N^{\ast }:T(y)<C\right\} $ is a non-empty open set in $%
\mathbb{R}^{n}$. Consequently, the complement of the rejection region $W(C)$
contains a non-empty open set and thus satisfies $\lambda _{\mathbb{R}^{n}}(%
\mathbb{R}^{n}\backslash W(C))>0$.

\item Suppose Assumption \ref{omega2} and $0<C<\infty $ hold. Then, for
every $\mu _{0}\in \mathfrak{M}_{0}$, every sequence $\nu _{m}\in \Pi
_{\left( \mathfrak{M}_{0}-\mu _{0}\right) ^{\bot }}(\mathfrak{M}_{1}-\mu
_{0})$ with $\left\Vert \nu _{m}\right\Vert \rightarrow \infty $, and for
every sequence $\Phi _{m}$ of positive definite symmetric $n\times n$
matrices with $\Phi _{m}\rightarrow \Phi $, $\Phi $ a positive definite
matrix, we have that%
\begin{eqnarray}
\liminf_{m\rightarrow \infty }P_{\nu _{m}+\mu _{0},\Phi _{m}}(W(C))
&=&\inf_{v\in A(\left( \nu _{m}\right) _{m\geq 1})}\Pr \left( {v}^{\prime }%
\check{\Omega}^{-1}(\Phi ^{1/2}\mathbf{G})v\geq 0\right)  \notag \\
&=&\inf_{v\in A(\left( \nu _{m}\right) _{m\geq 1})}\Pr \left( {v}^{\prime }%
\check{\Omega}^{-1}(\Phi ^{1/2}\mathbf{G})v>0\right)  \label{power_far_away}
\end{eqnarray}%
where $A(\left( \nu _{m}\right) _{m\geq 1})$ is the set of all accumulation
points of the sequence 
\begin{equation*}
R\left( X^{\prime }X\right) ^{-1}X^{\prime }\nu _{m}/\left\Vert R\left(
X^{\prime }X\right) ^{-1}X^{\prime }\nu _{m}\right\Vert ,
\end{equation*}%
and where $\mathbf{G}$ is a standard normal $n$-vector. A lower bound that
does not depend on the sequence $\nu _{m}$ is as follows: 
\begin{eqnarray}
\liminf_{m\rightarrow \infty }P_{\nu _{m}+\mu _{0},\Phi _{m}}(W(C)) &\geq
&\inf_{v\in \mathbb{R}^{q},\left\Vert v\right\Vert =1}\Pr \left( {v}^{\prime
}\check{\Omega}^{-1}(\Phi ^{1/2}\mathbf{G})v\geq 0\right)  \notag \\
&=&\inf_{v\in \mathbb{R}^{q},\left\Vert v\right\Vert =1}\Pr \left( {v}%
^{\prime }\check{\Omega}^{-1}(\Phi ^{1/2}\mathbf{G})v>0\right)  \notag \\
&\geq &\Pr \left( \check{\Omega}(\Phi ^{1/2}\mathbf{G})\text{ is nonnegative
definite}\right) .  \label{lower_bound_power_far_away}
\end{eqnarray}%
In particular, if $\check{\Omega}$ is nonnegative definite $\lambda _{%
\mathbb{R}^{n}}$-almost everywhere (implying that Assumption \ref{omega2} is
satisfied), this lower bound is $1$.
\end{enumerate}
\end{lemma}

\begin{remark}
\label{omega_rem}(i) Because $A(\left( \nu _{m}\right) _{m\geq 1})$ is a
closed subset of the unit ball in $\mathbb{R}^{q}$ and because the map $%
v\mapsto \Pr \left( {v}^{\prime }\check{\Omega}^{-1}(\Phi ^{1/2}\mathbf{G}%
)v\geq 0\right) $ is continuous on the unit ball under Assumption \ref%
{omega2}, we see that the expressions in (\ref{power_far_away}) are positive
if and only if 
\begin{equation}
\lambda _{\mathbb{R}^{n}}\left( \left\{ y\in \mathbb{R}^{n}\backslash
N^{\ast }:{v}^{\prime }\check{\Omega}^{-1}(y)v\geq 0\right\} \right) >0
\label{lam}
\end{equation}%
holds for every $v\in A(\left( \nu _{m}\right) _{m\geq 1})$. Under
Assumption \ref{omega2} we have $\lambda _{\mathbb{R}^{n}}\left( \left\{
y\in \mathbb{R}^{n}\backslash N^{\ast }:{v}^{\prime }\check{\Omega}%
^{-1}(y)v\geq 0\right\} \right) =\lambda _{\mathbb{R}^{n}}\left( \left\{
y\in \mathbb{R}^{n}\backslash N^{\ast }:{v}^{\prime }\check{\Omega}%
^{-1}(y)v>0\right\} \right) $ for every $v\neq 0$ and hence, by continuity
of $\check{\Omega}^{-1}(y)$ on $\mathbb{R}^{n}\backslash N^{\ast }$,
condition (\ref{lam}) for some $v\neq 0$ is in turn equivalent to ${v}%
^{\prime }\check{\Omega}^{-1}(y)v>0$ for some $y=y(v)\in \mathbb{R}%
^{n}\backslash N^{\ast }$.

(ii) Let $\check{\beta}$ and $\check{\Omega}$ satisfy Assumption \ref{AE},
let $T$ be the test statistic defined in (\ref{DT}), and suppose that we now
use a "random" critical value $\check{C}=\check{C}(y)>0$ for $y\in \mathbb{R}%
^{n}$. Suppose that $\check{C}$ is continuous on $\mathbb{R}^{n}\backslash N$
and satisfies the invariance condition $\check{C}(\alpha y+X\gamma )=\check{C%
}(y)$ for every $y\in \mathbb{R}^{n}\backslash N$, every $\alpha \neq 0$,
and for every $\gamma \in \mathbb{R}^{k}$. Rewriting the rejection region $%
\left\{ y\in \mathbb{R}^{n}:T(y)\geq \check{C}\right\} $ as $\left\{ y\in 
\mathbb{R}^{n}:T(y)/\check{C}\geq 1\right\} $ and observing that $\bar{\Omega%
}(y)=\check{C}(y)\check{\Omega}(y)$ satisfies Assumption \ref{AE} shows that
the results of this subsection also apply to the test with rejection region $%
\left\{ y\in \mathbb{R}^{n}:T(y)\geq \check{C}\right\} $.
\end{remark}

\bigskip

As a corollary to Theorem \ref{inv}, we now obtain negative size and power
results for tests of the form (\ref{DT}). The semicontinuity conditions in
Theorem \ref{inv} are implied by continuity properties of the estimators $%
\check{\Omega}$ and $\check{\beta}$ used in the construction of the test.
The sufficient conditions so obtained are easy to verify in practice and
become particularly simple in the practically relevant case where $\dim
\left( \mathcal{Z}\right) =1$, cf. the remark following the corollary.

\begin{corollary}
\label{CW} Let $\check{\beta}$ and $\check{\Omega}$ satisfy Assumption \ref%
{AE} and let $T$ be the test statistic defined in (\ref{DT}). Furthermore,
let $W(C)=\left\{ y\in \mathbb{R}^{n}:T(y)\geq C\right\} $ with $0<C<\infty $
be the rejection region. Suppose that $\mathcal{Z}$ is a concentration space
of the covariance model $\mathfrak{C}$. Recall that $N$ is the exceptional
set in Assumption \ref{AE} and that $N^{\ast }$ is given by (\ref{Nstar}).
Then the following holds:

\begin{enumerate}
\item Suppose we have for some $\mu _{0}^{\ast }\in \mathfrak{M}_{0}$ that $%
z\in \mathbb{R}^{n}\backslash N^{\ast }$ and $T(\mu _{0}^{\ast }+z)>C$ hold
simultaneously $\lambda _{\mathcal{Z}}$-almost everywhere. Then 
\begin{equation*}
\sup\limits_{\Sigma \in \mathfrak{C}}P_{\mu _{0},\sigma ^{2}\Sigma }(W(C))=1
\end{equation*}%
holds for every $\mu _{0}\in \mathfrak{M}_{0}$ and every $0<\sigma
^{2}<\infty $. In particular, the size of the test is equal to one.

\item Suppose we have for some $\mu _{0}^{\ast }\in \mathfrak{M}_{0}$ that $%
z\in \mathbb{R}^{n}\backslash N^{\ast }$ and $T(\mu _{0}^{\ast }+z)<C$ hold
simultaneously $\lambda _{\mathcal{Z}}$-almost everywhere. Then%
\begin{equation*}
\inf\limits_{\Sigma \in \mathfrak{C}}P_{\mu _{0},\sigma ^{2}\Sigma }(W(C))=0
\end{equation*}%
holds for every $\mu _{0}\in \mathfrak{M}_{0}$ and every $0<\sigma
^{2}<\infty $, and hence%
\begin{equation*}
\inf_{\mu _{1}\in \mathfrak{M}_{1}}\inf\limits_{\Sigma \in \mathfrak{C}%
}P_{\mu _{1},\sigma ^{2}\Sigma }(W(C))=0,
\end{equation*}%
holds for every $0<\sigma ^{2}<\infty $. In particular, the test is biased
(except in the trivial case where its size is zero). Furthermore, the
nuisance-infimal rejection probability at every point $\mu _{1}\in \mathfrak{%
M}_{1}$ is zero, i.e., 
\begin{equation*}
\inf\limits_{0<\sigma ^{2}<\infty }\inf\limits_{\Sigma \in \mathfrak{C}%
}P_{\mu _{1},\sigma ^{2}\Sigma }(W(C))=0.
\end{equation*}%
In particular, the infimal power of the test is equal to zero.

\item Suppose $\check{\Omega}$ is nonnegative definite on $\mathbb{R}%
^{n}\backslash N$. If $z\in \mathbb{R}^{n}\backslash N$, $\check{\Omega}%
(z)=0 $, and $R\check{\beta}(z)\neq 0$ hold simultaneously $\lambda _{%
\mathcal{Z}}$-almost everywhere, then 
\begin{equation*}
\sup\limits_{\Sigma \in \mathfrak{C}}P_{\mu _{0},\sigma ^{2}\Sigma }(W(C))=1
\end{equation*}%
holds for every $\mu _{0}\in \mathfrak{M}_{0}$ and every $0<\sigma
^{2}<\infty $. In particular, the size of the test is equal to one.
\end{enumerate}
\end{corollary}

\begin{remark}
\label{remnew}(i) Since $T$ in the above corollary is invariant under $G(%
\mathfrak{M}_{0})$, the condition in the corollary does not depend on the
particular choice of $\mu _{0}^{\ast }\in \mathfrak{M}_{0}$. Furthermore, if 
$\mathcal{Z}$ is one-dimensional, the invariance of $T$ shows that $T(\mu
_{0}^{\ast }+z)>C$ already holds for \emph{all} $z\in \mathcal{Z}$ with $%
z\neq 0$ provided it holds for one $z\in \mathcal{Z}$ with $z\neq 0$. In a
similar vein, Part 1 of Lemma \ref{LWT} implies for one-dimensional $%
\mathcal{Z}$ that $z\in \mathbb{R}^{n}\backslash N^{\ast }$ holds for \emph{%
all} $z\in \mathcal{Z}$ with $z\neq 0$ if and only if $z\in \mathbb{R}%
^{n}\backslash N^{\ast }$ holds for at least one $z\in \mathcal{Z}$ with $%
z\neq 0$. In view of Assumption \ref{AE} a similar statement also applies to
the relations $z\in \mathbb{R}^{n}\backslash N$, $\check{\Omega}(z)=0$, and $%
R\check{\beta}(z)\neq 0$.

(ii) We note that the rejection probabilities under the null hypothesis,
i.e., $P_{\mu _{0},\sigma ^{2}\Sigma }(W(C))$, do not depend on $\left( \mu
_{0},\sigma ^{2}\right) \in \mathfrak{M}_{0}\times \left( 0,\infty \right) $%
. Hence Remark \ref{rem_thmlrv}(ii) applies here.

(iii) In case the covariance model $\mathfrak{C}$ contains AR(1) correlation
matrices, a remark analogous to Remark \ref{special_AR_1} also applies here.
Furthermore, note that the concentration spaces derived from the AR(1)
correlation matrices are one-dimensional, and hence the discussion in (i)
above applies.
\end{remark}

\bigskip

The negative result in the preceding corollary does not apply if substantial
portions of $\mathcal{Z}$ belong to the exceptional set $N$ (which in
particular occurs if $\mathcal{Z}\subseteq \mathfrak{M}$ holds and $N$ is
not empty as then $\mathcal{Z}\subseteq \mathfrak{M}\subseteq N$). For this
case we provide a further negative result which is applicable provided (\ref%
{test_conc_100}) given below holds. For example, if $\mathcal{Z}=\limfunc{%
span}\left( e_{+}\right) $ and the design matrix contains an intercept, we
immediately obtain $\mathcal{Z}\subseteq \mathfrak{M}$, and (\ref%
{test_conc_100}) holds if and only if the column in $R$ corresponding to the
intercept is nonzero. The significance of the subsequent theorem is that it
provides an \emph{upper} bound $K_{1}$ for the power in certain directions
which is less than or equal to a \emph{lower} bound for the size. This will
typically imply biasedness of the test (except if equality holds in (\ref%
{Power_size_ineq})). Furthermore, note that the result implies that the test
has size $1$ in case $\check{\Omega}$ is positive definite $\lambda _{%
\mathbb{R}^{n}}$-almost everywhere since then $K_{1}=K_{2}=1$ follows. The
condition on the covariance model $\mathfrak{C}$ is often satisfied, see
Remark \ref{special_AR_1_1} following the theorem.

\begin{theorem}
\label{prop_101}Let $\check{\beta}$ and $\check{\Omega}$ satisfy Assumptions %
\ref{AE} and \ref{omega2}, let $T$ be the test statistic defined in (\ref{DT}%
), and let $W(C)=\left\{ y\in \mathbb{R}^{n}:T(y)\geq C\right\} $ with $%
0<C<\infty $ be the rejection region. Assume that there is a sequence $%
\Sigma _{m}\in \mathfrak{C}$ such that $\Sigma _{m}\rightarrow \bar{\Sigma}$
for $m\rightarrow \infty $ where $\bar{\Sigma}$ is singular with $l:=\dim 
\limfunc{span}(\bar{\Sigma})>0$. Suppose that for some sequence of positive
real numbers $s_{m}$ the matrix $D_{m}=\Pi _{\limfunc{span}(\bar{\Sigma}%
)^{\bot }}\Sigma _{m}\Pi _{\limfunc{span}(\bar{\Sigma})^{\bot }}/s_{m}$
converges to a matrix $D$, which is regular on $\limfunc{span}(\bar{\Sigma}%
)^{\bot }$, and that $\Pi _{\limfunc{span}(\bar{\Sigma})^{\bot }}\Sigma
_{m}\Pi _{\limfunc{span}(\bar{\Sigma})}/s_{m}^{1/2}\rightarrow 0$. Suppose
further that $\limfunc{span}(\bar{\Sigma})\subseteq \mathfrak{M}$, and let $%
Z $ be a matrix, the columns of which form a basis for $\limfunc{span}(\bar{%
\Sigma})$. Assume also that%
\begin{equation}
R\hat{\beta}(z)\neq 0\qquad \lambda _{\limfunc{span}(\bar{\Sigma})}\text{-}%
a.e.  \label{test_conc_100}
\end{equation}%
is satisfied. Then for every $\mu _{0}\in \mathfrak{M}_{0}$, every $\sigma $
with $0<\sigma <\infty $, and every $M\geq 0$ we have%
\begin{equation}
\inf_{\gamma \in \mathbb{R}^{l},\left\Vert \gamma \right\Vert \geq
M}\inf_{\Sigma \in \mathfrak{C}}P_{\mu _{0}+Z\gamma ,\sigma ^{2}\Sigma
}\left( W(C)\right) \leq K_{1}\leq K_{2}\leq \sup_{\Sigma \in \mathfrak{C}%
}P_{\mu _{0},\sigma ^{2}\Sigma }\left( W(C)\right) .  \label{Power_size_ineq}
\end{equation}%
The constants $K_{1}$ and $K_{2}$ are given by%
\begin{equation*}
K_{1}=\inf_{\gamma \in \mathbb{R}^{l}}\Pr \left( \bar{\xi}\left( \gamma
\right) \geq 0\right) =\inf_{\left\Vert \gamma \right\Vert =1}\Pr \left( 
\bar{\xi}\left( \gamma \right) \geq 0\right)
\end{equation*}%
and%
\begin{equation*}
K_{2}=\int \Pr \left( \bar{\xi}\left( \gamma \right) \geq 0\right)
dP_{0,A}(\gamma )
\end{equation*}%
with the random variable $\bar{\xi}\left( \gamma \right) $ given by%
\begin{equation*}
\bar{\xi}\left( \gamma \right) =\left( R\hat{\beta}\left( Z\gamma \right)
\right) ^{\prime }\check{\Omega}^{-1}\left( \left( \bar{\Sigma}%
^{1/2}+D^{1/2}\right) \boldsymbol{G}\right) R\hat{\beta}\left( Z\gamma
\right)
\end{equation*}%
on the event $\left\{ \left( \bar{\Sigma}^{1/2}+D^{1/2}\right) \boldsymbol{G}%
\in \mathbb{R}^{n}\backslash N^{\ast }\right\} $ and by $\bar{\xi}\left(
\gamma \right) =0$ otherwise, where $\boldsymbol{G}$ is a standard normal $n$%
-vector. The matrix $A$ denotes $\left( Z^{\prime }Z\right) ^{-1}Z^{\prime }%
\bar{\Sigma}Z\left( Z^{\prime }Z\right) ^{-1}$, which is nonsingular, and $%
P_{0,A}$ denotes the Gaussian distribution on $\mathbb{R}^{l}$ with mean
zero and variance covariance matrix $A$.
\end{theorem}

\begin{remark}
\label{special_AR_1_1}Suppose the covariance model $\mathfrak{C}$ contains $%
\mathfrak{C}_{AR(1)}$, or, more generally, $\mathfrak{C}$ contains AR(1)
correlation matrices $\Lambda (\rho _{m})$ for some sequence $\rho _{m}\in
\left( -1,1\right) $ with $\rho _{m}\rightarrow 1$ ($\rho _{m}\rightarrow -1$%
, respectively). Then all the conditions on the covariance model in the
preceding theorem are satisfied with $\bar{\Sigma}=e_{+}e_{+}^{\prime }$, $%
\limfunc{span}(\bar{\Sigma})=\limfunc{span}(e_{+})$, and $Z=e_{+}$ ($\bar{%
\Sigma}=e_{-}e_{-}^{\prime }$, $\limfunc{span}(\bar{\Sigma})=\limfunc{span}%
(e_{-})$, and $Z=e_{-}$, respectively); cf. Lemma \ref{AR_1} in Appendix \ref%
{AR_100}. Furthermore, condition (\ref{test_conc_100}) simplifies to $R\hat{%
\beta}(e_{+})\neq 0$ ($R\hat{\beta}(e_{-})\neq 0$, respectively).
\end{remark}

\bigskip

The subsequent theorem specializes the positive result given in Theorems \ref%
{TU} and \ref{TUU} to the class of tests considered in the present
subsection.

\begin{theorem}
\label{TU_1} Let $\check{\beta}$ and $\check{\Omega}$ satisfy Assumptions %
\ref{AE}, \ref{omega1}, and \ref{omega2}. Let $T$ be the test statistic
defined in (\ref{DT}). Furthermore, let $W(C)=\left\{ y\in \mathbb{R}%
^{n}:T(y)\geq C\right\} $ with $0<C<\infty $ be the rejection region.
Suppose further that 
\begin{equation}
T(y+z)=T(y)\qquad \text{for every }y\in \mathbb{R}^{n}\text{ and every }z\in
J(\mathfrak{C}).  \label{T_inv_wrt_z}
\end{equation}%
Assume that $\mathfrak{C}$ is bounded (as a subset of $\mathbb{R}^{n\times
n} $). Assume also that for every sequence $\Sigma _{m}\in \mathfrak{C}$
converging to a singular $\bar{\Sigma}$ there exists a subsequence $%
(m_{i})_{i\in \mathbb{N}}$ and a sequence of positive real numbers $%
s_{m_{i}} $ such that the sequence of matrices $D_{m_{i}}=\Pi _{\limfunc{span%
}(\bar{\Sigma})^{\bot }}\Sigma _{m_{i}}\Pi _{\limfunc{span}(\bar{\Sigma}%
)^{\bot }}/s_{m_{i}}$ converges to a matrix $D$ which is regular on the
orthogonal complement of $\limfunc{span}(\bar{\Sigma})$. Then the following
holds:

\begin{enumerate}
\item The size of the rejection region $W(C)$ is strictly less than $1$,
i.e.,%
\begin{equation*}
\sup\limits_{\mu _{0}\in \mathfrak{M}_{0}}\sup\limits_{0<\sigma ^{2}<\infty
}\sup\limits_{\Sigma \in \mathfrak{C}}P_{\mu _{0},\sigma ^{2}\Sigma }\left(
W(C)\right) <1.
\end{equation*}%
Furthermore,%
\begin{equation*}
\inf_{\mu _{0}\in \mathfrak{M}_{0}}\inf_{0<\sigma ^{2}<\infty }\inf_{\Sigma
\in \mathfrak{C}}P_{\mu _{0},\sigma ^{2}\Sigma }\left( W(C)\right) >0.
\end{equation*}

\item Suppose that $\lambda _{\mathbb{R}^{n}}\left( \left\{ y\in \mathbb{R}%
^{n}\backslash N^{\ast }:{v}^{\prime }\check{\Omega}^{-1}(y)v\geq 0\right\}
\right) >0$ for every $v\in \mathbb{R}^{q}$ with $\left\Vert v\right\Vert =1$%
. Then the infimal power is bounded away from zero, i.e., 
\begin{equation*}
\inf_{\mu _{1}\in \mathfrak{M}_{1}}\inf\limits_{0<\sigma ^{2}<\infty
}\inf\limits_{\Sigma \in \mathfrak{C}}P_{\mu _{1},\sigma ^{2}\Sigma
}(W(C))>0.
\end{equation*}

\item Suppose that $\check{\Omega}$ is nonnegative definite $\lambda _{%
\mathbb{R}^{n}}$-almost everywhere. Then for every $0<c<\infty $%
\begin{equation*}
\inf_{\substack{ \mu _{1}\in \mathfrak{M}_{1},0<\sigma ^{2}<\infty  \\ %
d\left( \mu _{1},\mathfrak{M}_{0}\right) /\sigma \geq c}}P_{\mu _{1},\sigma
^{2}\Sigma _{m}}(W(C))\rightarrow 1
\end{equation*}%
holds for $m\rightarrow \infty $ and for any sequence $\Sigma _{m}\in 
\mathfrak{C}$ satisfying $\Sigma _{m}\rightarrow \bar{\Sigma}$ with $\bar{%
\Sigma}$ a singular matrix. Furthermore, for every sequence $0<c_{m}<\infty $
\begin{equation*}
\inf_{\substack{ \mu _{1}\in \mathfrak{M}_{1},  \\ d\left( \mu _{1},%
\mathfrak{M}_{0}\right) \geq c_{m}}}P_{\mu _{1},\sigma _{m}^{2}\Sigma
_{m}}(W(C))\rightarrow 1
\end{equation*}%
holds for $m\rightarrow \infty $ whenever $0<\sigma _{m}^{2}<\infty $, $%
c_{m}/\sigma _{m}\rightarrow \infty $, and the sequence $\Sigma _{m}\in 
\mathfrak{C}$ satisfies $\Sigma _{m}\rightarrow \bar{\Sigma}$ with $\bar{%
\Sigma}$ a positive definite matrix. [The very last statement even holds
without recourse to condition (\ref{T_inv_wrt_z}) and the condition on $%
\mathfrak{C}$ following (\ref{T_inv_wrt_z}).]

\item For every $\delta $, $0<\delta <1$, there exists a $C(\delta )$, $%
0<C(\delta )<\infty $, such that%
\begin{equation*}
\sup\limits_{\mu _{0}\in \mathfrak{M}_{0}}\sup\limits_{0<\sigma ^{2}<\infty
}\sup\limits_{\Sigma \in \mathfrak{C}}P_{\mu _{0},\sigma ^{2}\Sigma
}(W(C(\delta )))\leq \delta .
\end{equation*}
\end{enumerate}
\end{theorem}

\begin{remark}
(i) In case the covariance model $\mathfrak{C}$ equals $\mathfrak{C}_{AR(1)}$%
, a remark analogous to Remark \ref{rem_AR_1} also applies here.

(ii) Under the assumptions of the preceding theorem, the additional
condition in Part 2 of the theorem is equivalent to ${v}^{\prime }\check{%
\Omega}^{-1}(y)v>0$ for every $v\in \mathbb{R}^{q}$ with $\left\Vert
v\right\Vert =1$ and a suitable $y=y(v)\in \mathbb{R}^{n}\backslash N^{\ast
} $. Cf. Remark \ref{omega_rem}(i).
\end{remark}

\bigskip

We now discuss when the preceding theorem can be expected to apply and how
the crucial condition (\ref{T_inv_wrt_z}) can be enforced. As already noted
prior to Theorem \ref{TU}, a sufficient condition for (\ref{T_inv_wrt_z}) to
be satisfied for \emph{any} test statistic $T$ of the form (\ref{DT}), based
on estimators $\check{\beta}$ and $\check{\Omega}$ satisfying Assumption \ref%
{AE}, is that $J(\mathfrak{C})\subseteq \mathfrak{M}_{0}-\mu _{0}$ for some
(and hence all) $\mu _{0}\in \mathfrak{M}_{0}$ holds. This sufficient
condition is equivalent to $J(\mathfrak{C})\subseteq \mathfrak{M}$ and $R%
\hat{\beta}(z)=0$ for every $z\in J(\mathfrak{C})$, because $\mathfrak{M}%
_{0}-\mu _{0}$ coincides with the set $\left\{ \mu \in \mathfrak{M}:R\hat{%
\beta}(\mu )=0\right\} $. [Note that replacing $J(\mathfrak{C})$ by $%
\limfunc{span}\left( J(\mathfrak{C})\right) $ in the preceding two sentences
leads to equivalent statements because $\mathfrak{M}_{0}-\mu _{0}$ as well
as $\mathfrak{M}$ are linear spaces.] Now consider the general case where $J(%
\mathfrak{C})$, or equivalently $\limfunc{span}\left( J(\mathfrak{C})\right) 
$, may not be a subset of $\mathfrak{M}_{0}-\mu _{0}$: If there exists a $%
z\in \limfunc{span}\left( J(\mathfrak{C})\right) \cap \mathfrak{M}$ with $%
z\notin \mathfrak{M}_{0}-\mu _{0}$ (i.e., with $R\hat{\beta}(z)\neq 0$),
then \emph{any} test statistic $T$ of the form (\ref{DT}), based on
estimators $\check{\beta}$ and $\check{\Omega}$ satisfying Assumptions \ref%
{AE} and \ref{omega2}, does \emph{not} satisfy the invariance condition (\ref%
{T_inv_wrt_z}), see Lemma \ref{nec} in Appendix \ref{App_E}. Hence, $%
\limfunc{span}\left( J(\mathfrak{C})\right) \cap \mathfrak{M}\subseteq 
\mathfrak{M}_{0}-\mu _{0}$, or in other words $R\hat{\beta}(z)=0$ for every $%
z\in \limfunc{span}\left( J(\mathfrak{C})\right) \cap \mathfrak{M}$, is a
necessary condition for (\ref{T_inv_wrt_z}) to be satisfied for \emph{some }$%
T$ as above. We next show how a test statistic of the form (\ref{DT})\
satisfying the crucial invariance condition (\ref{T_inv_wrt_z}) can in fact
be constructed if we impose this necessary condition.

\begin{proposition}
\label{enforce_inv}Let $\mathfrak{C}$ be a covariance model and suppose that 
$\limfunc{span}\left( J(\mathfrak{C})\right) \cap \mathfrak{M}\subseteq 
\mathfrak{M}_{0}-\mu _{0}$ holds.

\begin{enumerate}
\item Let $\mathfrak{\bar{M}}$ be the linear space spanned by $J(\mathfrak{C}%
)\cup \mathfrak{M}$. Define $\bar{X}=\left( X,\bar{x}_{1},\ldots ,\bar{x}%
_{p}\right) $ where $\bar{x}_{i}\in \limfunc{span}\left( J(\mathfrak{C})\cup
\left( \mathfrak{M}_{0}-\mu _{0}\right) \right) $ are chosen in such a way
that the columns of $\bar{X}$ form a basis of $\mathfrak{\bar{M}}$. Assume
that $k<k+p<n$ holds. Suppose $\bar{\theta}$ and $\bar{\Omega}$ are
estimators satisfying the analogue of Assumption \ref{AE} obtained by
replacing $k$ by $k+p$, $X$ by $\bar{X}$, and $\mathfrak{M}$ by $\mathfrak{%
\bar{M}}$. Let $\bar{N}$ denote the null set appearing in that analogue of
Assumption \ref{AE} and $\bar{N}^{\ast }=\bar{N}\cup \left\{ y\in \mathbb{R}%
^{n}\backslash \bar{N}:\det {\bar{\Omega}(y)}=0\right\} $. Define $\bar{\beta%
}=\left( I_{k},0\right) \bar{\theta}$. Then $\bar{\beta}$ and ${\bar{\Omega}}
$ satisfy the original Assumption \ref{AE} (with $N$ given by $\bar{N}$),
and the test statistic $\bar{T}$ given by%
\begin{equation*}
\bar{T}(y)=%
\begin{cases}
(R\bar{\beta}(y)-r)^{\prime }\bar{\Omega}^{-1}(y)(R\bar{\beta}(y)-r), & y\in 
\mathbb{R}^{n}\backslash \bar{N}^{\ast }, \\ 
0, & y\in \bar{N}^{\ast }\text{.}%
\end{cases}%
\end{equation*}%
satisfies the invariance condition (\ref{T_inv_wrt_z}).

\item Let $\mathfrak{\bar{M}}$ and $\bar{X}$ be as above and $k<k+p<n$.
Suppose $\bar{\theta}\left( y\right) =\left( \bar{X}^{\prime }\bar{X}\right)
^{-1}\bar{X}^{\prime }y$ is the least squares estimator based on $\bar{X}$.
Then the requirements on $\bar{\theta}$ postulated in the above mentioned
analogue of Assumption \ref{AE} are satisfied, and $R\bar{\beta}\left(
z\right) =0$ holds for every $z\in \limfunc{span}\left( J(\mathfrak{C}%
)\right) $. Furthermore, if $X^{\ast }=\left( X,x_{1}^{\ast },\ldots
,x_{p}^{\ast }\right) $ is obtained in the same way as is $\bar{X}$ but for
another choice of elements $x_{i}^{\ast }\in \limfunc{span}\left( J(%
\mathfrak{C})\cup \left( \mathfrak{M}_{0}-\mu _{0}\right) \right) $ and if $%
\theta ^{\ast }$ denotes the least squares estimator w.r.t. the design
matrix $X^{\ast }$, then $R\bar{\beta}(y)=R\beta ^{\ast }(y)$ holds for
every $y\in \mathbb{R}^{n}$ with $\beta ^{\ast }$ denoting $\left(
I_{k},0\right) \theta ^{\ast }$.
\end{enumerate}
\end{proposition}

We next discuss ways of choosing $\bar{x}_{1},\ldots ,\bar{x}_{p}$ such that
they satisfy the requirements in the preceding proposition: One natural way
is to first find $z_{1}\ldots ,z_{r}$ in $J(\mathfrak{C})$ that form a basis
of $\limfunc{span}J(\mathfrak{C})$. From these vectors then select $\bar{x}%
_{1}=z_{i_{1}}\ldots ,\bar{x}_{p}=z_{i_{p}}$ to complement the columns of $X$
to a basis of $\mathfrak{\bar{M}}$. An alternative way is based on the
observation that adding elements of $\mathfrak{M}_{0}-\mu _{0}$ to each of
the previously found $z_{i_{j}}$ obviously gives rise to another feasible
choice of $\bar{x}_{i}$. It hence follows that an alternative feasible
choice for the $\bar{x}_{i}$ is to use the projections of the $z_{i_{j}}$
onto the orthogonal complement of $\mathfrak{M}_{0}-\mu _{0}$. Of course, if
the estimator $\bar{\theta}$ is chosen to be the least squares estimator,
then Part 2 of the preceding proposition informs us that the particular
choice of the $\bar{x}_{i}$ has no effect on $R\bar{\beta}(y)$ since it is
invariant under the choice of the $\bar{x}_{i}$.

Part 2 of Proposition \ref{enforce_inv} provides a particular estimator $%
\bar{\theta}$ that satisfies the assumptions on $\bar{\theta}$ maintained in
Part 1 of this proposition. Because no particular covariance model $%
\mathfrak{C}$ has been specified in Proposition \ref{enforce_inv}, we can
not provide a similar concrete construction of $\bar{\Omega}$ in that
proposition. The construction of an appropriate $\bar{\Omega}$ has to be
done on a case by case basis, depending on the covariance model employed in
the particular application. For an example of such a construction in the
context of autocorrelation robust testing see Theorem \ref{TU_3}. We
furthermore note that similar to the results in Part 2 of Proposition \ref%
{enforce_inv} such estimators $\bar{\Omega}$ will typically be unchanged
whether they are constructed on the basis of the design matrices $\bar{X}$
or $X^{\ast }$. In particular, this is the case for the estimator
constructed in Theorem \ref{TU_3}.

To summarize, the significance of Proposition \ref{enforce_inv} is that it
tells us (in conjunction with Theorem \ref{TU_1}) when and how we can
construct an adjusted test based on an auxiliary model that does not suffer
from the severe size and power distortions (i.e., size $1$ and/or infimal
power $0$), the adjustment consisting of adding appropriate auxiliary
regressors to the model. For a concrete implementation see Theorem \ref{TU_3}%
.

\begin{remark}
(i) Suppose that the assumptions of Proposition \ref{enforce_inv} hold,
except that now $p=0$ holds. Then $J(\mathfrak{C})\subseteq \mathfrak{M}$
and hence $R\hat{\beta}(z)=0$ holds for every $z\in J(\mathfrak{C})$,
implying that actually the sufficient condition mentioned prior to the
proposition is satisfied. Consequently, as discussed above, the invariance
condition (\ref{T_inv_wrt_z}) is already satisfied for \emph{every} $T$ of
the form (\ref{DT}) based on estimators $\check{\beta}$ and $\check{\Omega}$
satisfying Assumption \ref{AE}.

(ii) Suppose that the assumptions of Proposition \ref{enforce_inv} hold,
except that now $k+p=n$ holds (note that $k+p\leq n$ always holds). Suppose
further that $T$ is a test statistic of the form (\ref{DT}) based on
estimators $\check{\beta}$ and $\check{\Omega}$ satisfying Assumptions \ref%
{AE} and \ref{omega1}. Then $T$ can never satisfy (\ref{T_inv_wrt_z}) and
hence Theorem \ref{TU_1} does not apply in this situation. This can be seen
as follows: Because of $k+p=n$ it follows that every $y\in \mathbb{R}^{n}$
can be written as a linear combination of finitely many $z_{i}\in J(%
\mathfrak{C})$ plus an element $\mu $ in $\mathfrak{M}$. Because invariance
w.r.t. addition of elements $z\in J(\mathfrak{C})$ is equivalent to
invariance w.r.t. addition of elements $z\in \limfunc{span}\left( J(%
\mathfrak{C})\right) $ (cf. Remark \ref{rem_positive}(i)) we see that $%
T(y)=T(\mu )$ would have to hold under (\ref{T_inv_wrt_z}). As noted after
the introduction of Assumption \ref{AE}, either $\mathfrak{M}\subseteq
N\subseteq N^{\ast }$ holds or $N$ is empty. In the second case we have that 
$\check{\Omega}(\mu )=0$ as a consequence of equivariance. Hence in both
cases we arrive at $\mu \in N^{\ast }$ and thus at $T(\mu )=0$. But this
shows that $T$ is constant equal to zero, contradicting Part 5 of Lemma \ref%
{LWT}.

(iii) Proposition \ref{enforce_inv} uses the auxiliary matrix $\bar{X}$ and
the associated estimators $\bar{\theta}$ to construct an estimator $\bar{%
\beta}$ for the parameter $\beta $ in the originally given regression model (%
\ref{lm}) and this estimator $\bar{\beta}$ is then used to construct a test
statistic $\bar{T}$ for the testing problem (\ref{testing problem}) to which
Theorem \ref{TU_1} can be applied. In an alternative view we can consider
the auxiliary model $\mathbf{Y}=\bar{X}\theta +\mathbf{U}$ with $\theta
=\left( \beta ^{\prime },\zeta ^{\prime }\right) ^{\prime }$ as a model in
its own right. [Of course, if we maintain model (\ref{lm}) then $\zeta =0$
must hold in the auxiliary model.] Define the $q\times \left( k+p\right) $
matrix $\bar{R}=R\left( I_{k},0\right) $, define $\mathfrak{\bar{M}}%
_{0}=\left\{ \mu \in \mathfrak{\bar{M}}:\mu =\bar{X}\theta ,\bar{R}\theta
=r\right\} $ and set $\mathfrak{\bar{M}}_{1}=\mathfrak{\bar{M}}\backslash 
\mathfrak{\bar{M}}_{0}$, and define a null hypothesis $\bar{H}_{0}$ and an
alternative hypothesis $\bar{H}_{1}$ analogously as in (\ref{testing problem}%
). Proposition \ref{enforce_inv} can now be viewed as stating that condition
(\ref{T_inv_wrt_z}) is satisfied for the test statistic which is obtained by
using (\ref{DT}) based on the restriction matrix $\bar{R}$ and on the
estimators $\bar{\theta}$ and $\bar{\Omega}$ figuring in Proposition \ref%
{enforce_inv}. Consequently, Theorem \ref{TU_1} can be directly applied to
this test statistic (provided $\bar{\Omega}$ satisfies Assumptions \ref%
{omega1} and \ref{omega2}). It should be noted that the so-obtained result
now applies to the problem of testing $\bar{H}_{0}$ versus $\bar{H}_{1}$.
However, since $\mathfrak{M}_{0}\subseteq \mathfrak{\bar{M}}_{0}$ and $%
\mathfrak{M}_{1}\subseteq \mathfrak{\bar{M}}_{1}$ hold and since $T$ is
invariant under translation by elements in $\limfunc{span}\left( J(\mathfrak{%
C})\right) $, we essentially recover the same result as before.
\end{remark}

\subsection{Non-Gaussian distributions\label{other_distr}}

As already noted in Section \ref{HTFramework}, the negative results given in
this paper immediately extend in a trivial way without imposing the
Gaussianity assumption on the error vector $\mathbf{U}$ in (\ref{lm}) as
long as the assumptions on the feasible error distributions is weak enough
to ensure that the implied set of distributions for $\mathbf{Y}$ contains
the set $\left\{ P_{\mu ,\sigma ^{2}\Sigma }:\mu \in \mathfrak{M},0<\sigma
^{2}<\infty ,\Sigma \in \mathfrak{C}\right\} $, but possibly contains also
other distributions.

Another, less trivial, extension is as follows: Suppose that $\mathbf{U}$ is
elliptically distributed in the sense that it has the same distribution as $%
\mathbf{\varrho }\sigma \Sigma ^{1/2}\mathbf{E}$ where $0<\sigma <\infty $, $%
\Sigma \in \mathfrak{C}$, $\mathbf{E}$ is a random vector uniformly
distributed on the unit sphere $S^{n-1}$, and $\mathbf{\varrho }$ is a
random variable distributed independently of $\mathbf{E}$ satisfying $\Pr (%
\mathbf{\varrho }>0)=1$. [If $\mathbf{\varrho }$ is distributed as the
square root of a chi-square with $n$ degrees of freedom we recover the
Gaussian situation described in Section \ref{HTFramework}.] If $\varphi $ is
a test that is invariant under the group $G(\mathfrak{M}_{0})$ then it is
easy to see that for $\mu _{0}\in \mathfrak{M}_{0}$%
\begin{equation*}
\mathbb{E}(\varphi (\mu _{0}+\mathbf{\varrho }\sigma \Sigma ^{1/2}\mathbf{E}%
))=\mathbb{E}(\varphi (\mu _{0}+\Sigma ^{1/2}\mathbf{E}))
\end{equation*}%
holds.\footnote{%
Under an additional absolute continuity assumption this is also true for
almost invariant tests $\varphi $.} Since this does not depend on the
distribution of $\mathbf{\varrho }$ at all, we learn that the rejection
probability under the null hypothesis is therefore the same as in the
Gaussian case. As a consequence, all results concerning only the null
behavior of $\varphi $ obtained under Gaussianity in the paper extend
immediately to regression models in which the disturbance vector $\mathbf{U}$
is elliptically distributed in the above sense. Furthermore, all results
concerning rejection probabilities under the alternative which are obtained
from the behavior of the null rejection probabilities by an approximation
argument (e.g., Parts 2 and 3 of Theorem \ref{inv} as well as of Corollary %
\ref{CW}, and the corresponding applications of these results in Sections %
\ref{tsr} and \ref{Het}) also go through in view of Scheff\'{e}'s lemma
provided the density of $\mathbf{\varrho E}$ exists and is continuous almost
everywhere.

\appendix

\section{Appendix: Proofs for Subsection \protect\ref{LRVPD} \label{App_A}}

\textbf{Proof of Lemma \ref{LRVPD}: }Observe that $\hat{\Omega}_{w}\left(
y\right) =B\left( y\right) \mathcal{W}_{n}B^{\prime }\left( y\right) $.
Given that $\mathcal{W}_{n}$ is positive definite due to Assumption \ref{AW}%
, this immediately establishes Parts 1-3 of the Lemma. It remains to prove
Part 4. Let $s$ be as in Assumption \ref{R_and_X} and consider first the
case where this assumption is satisfied, i.e., where $\limfunc{rank}\left(
R(X^{\prime }X)^{-1}X^{\prime }\left( \lnot (i_{1},\ldots i_{s})\right)
\right) =q$ holds. If now $y$ is such that $\hat{\Omega}_{w}\left( y\right) $
is singular it follows, in view of the equivalent condition $\limfunc{rank}%
\left( B(y)\right) <q$, that $\hat{u}_{l}(y)=0$ must hold at least for some $%
l\notin \left\{ i_{1},\ldots i_{s}\right\} $ where $l$ may depend on $y$.
But this means that $y$ satisfies $e_{l}^{\prime }(n)\left( I_{n}-X\left(
X^{\prime }X\right) ^{-1}X^{\prime }\right) y=0$. Since $e_{l}^{\prime
}(n)\left( I_{n}-X\left( X^{\prime }X\right) ^{-1}X^{\prime }\right) \neq 0$
by construction of $l$, it follows that the set of $y$ for which $\hat{\Omega%
}_{w}\left( y\right) $ is singular is contained in a finite union of proper
linear subspaces, and hence is a $\lambda _{\mathbb{R}^{n}}$-null set. Next
consider the case where Assumption \ref{R_and_X} is not satisfied. Observe
that then $s>0$ must hold. Note that $\hat{u}_{i}(y)=0$ holds for all $y\in 
\mathbb{R}^{n}$ and all $i\in \left\{ i_{1},\ldots i_{s}\right\} $ by
construction of $\left\{ i_{1},\ldots i_{s}\right\} $. But then for every $%
y\in \mathbb{R}^{n}$%
\begin{eqnarray*}
\limfunc{rank}\left( B\left( y\right) \right) &=&\limfunc{rank}\left(
R(X^{\prime }X)^{-1}X^{\prime }\left( \lnot (i_{1},\ldots i_{s})\right)
A(y)\right) \\
&\leq &\limfunc{rank}\left( R(X^{\prime }X)^{-1}X^{\prime }\left( \lnot
(i_{1},\ldots i_{s})\right) \right) <q
\end{eqnarray*}%
is satisfied where $A(y)$ is obtained from $\limfunc{diag}\left( \hat{u}%
_{1}(y),\ldots ,\hat{u}_{n}(y)\right) $ by deleting rows and columns $i$
with $i\in \left\{ i_{1},\ldots i_{s}\right\} $. This completes the proof. $%
\blacksquare $

\begin{lemma}
\label{aux_100}Suppose Assumptions \ref{AW} and \ref{R_and_X} are satisfied.
Then $\hat{\beta}$ and $\hat{\Omega}_{w}$ satisfy Assumption \ref{AE}, \ref%
{omega1}, and \ref{omega2} with $N=\emptyset $. In fact, $\hat{\Omega}%
_{w}\left( y\right) $ is nonnegative definite for every $y\in \mathbb{R}^{n}$%
, and is positive definite $\lambda _{\mathbb{R}^{n}}$-almost everywhere.
The test statistic $T$ defined in (\ref{tslrv}), with $\hat{\Psi}_{w}$ as in
(\ref{lrve}), is invariant under the group $G\left( \mathfrak{M}_{0}\right) $
and the rejection probabilities $P_{\mu ,\sigma ^{2}\Sigma }(T\geq C)$
depend on $\left( \mu ,\sigma ^{2},\Sigma \right) \in \mathfrak{M}\times
(0,\infty )\times \mathfrak{C}$ only through $\left( \left( R\beta -r\right)
/\sigma ,\Sigma \right) $ (in fact, only through $\left( \left\langle \left(
R\beta -r\right) /\sigma \right\rangle ,\Sigma \right) $), where $\beta $
corresponds to $\mu $ via $\mu =X\beta $.
\end{lemma}

\begin{proof}
Clearly, $\hat{\beta}$ and $\hat{\Omega}_{w}$ are well-defined and
continuous on $\mathbb{R}^{n}$, hence we may set $N=\emptyset $ in
Assumption \ref{AE}. Symmetry of $\hat{\Omega}_{w}$ as well as the required
equivariance properties of $\hat{\beta}$ and $\hat{\Omega}_{w}$ are
obviously satisfied. By Assumption \ref{AW} $\hat{\Omega}_{w}\left( y\right) 
$ is nonnegative definite for every $y\in \mathbb{R}^{n}$. By Assumptions %
\ref{AW} and \ref{R_and_X} and Lemma \ref{LRVPD} the matrix $\hat{\Omega}%
_{w} $ is nonsingular (and hence positive definite) $\lambda _{\mathbb{R}%
^{n}}$-almost everywhere. Hence Assumptions \ref{AE}, \ref{omega1}, and \ref%
{omega2} are satisfied which proves the first claim. The remaining claims
follow immediately from Lemma \ref{LWT} and Proposition \ref{inv_rej_prob}.
\end{proof}

\textbf{Proof of Theorem \ref{thmlrv}:} By Lemma \ref{aux_100} we know that $%
\hat{\beta}$ and $\hat{\Omega}_{w}$ satisfy Assumption \ref{AE} and that $%
\hat{\Omega}_{w}\left( y\right) $ is nonnegative definite for every $y\in 
\mathbb{R}^{n}$. Furthermore, in view of this lemma and because $N=\emptyset 
$, the set $N^{\ast }$ in Corollary \ref{CW} is precisely the set of $y$ for
which $\limfunc{rank}\left( B(y)\right) <q$, cf. Lemma \ref{LRVPD}. By
Assumption \ref{AAR(1)} the spaces $\mathcal{Z}_{+}=\text{span}(e_{+})$ and $%
\mathcal{Z}_{-}=\text{span}(e_{-})$ are concentration spaces of $\mathfrak{C}
$. The theorem now follows by applying Corollary \ref{CW} and Remark \ref%
{remnew}(i) to $\mathcal{Z}_{+}$ as well as to $\mathcal{Z}_{-}$ and by
noting that $e_{+}\in \mathbb{R}^{n}\backslash N^{\ast }$ translates into $%
\limfunc{rank}\left( B(e_{+})\right) =q$ with a similar translation if $%
e_{+} $ is replaced by $e_{-}$. Also note that the size of the test can not
be zero in view of Part 5 of Lemma \ref{LWT} and Lemma \ref{aux_100}. $%
\blacksquare $

\textbf{Proof of Proposition \ref{generic}: (}1) Define the matrix $%
B_{X}^{\ast }\left( y\right) =\left( \det (X^{\prime }X)\right)
^{2}B_{X}\left( y\right) $ and observe that (for given $y$) every element of
this matrix is a multivariate polynomial in the elements $x_{ti}$ of $X$
because $(X^{\prime }X)^{-1}$ can be written as $\left( \det (X^{\prime
}X)\right) ^{-1}\limfunc{adj}(X^{\prime }X)$ (with the convention that $%
\limfunc{adj}(X^{\prime }X)=1$ if $k=1$). Because $\det (X^{\prime }X)\neq 0$
for $X\in \mathfrak{X}_{0}$ holds, we have%
\begin{equation*}
\mathfrak{X}_{1}\left( e_{+}\right) =\mathfrak{X}_{0}\cap \left\{ X\in 
\mathbb{R}^{n\times k}:\det \left( B_{X}^{\ast }(e_{+})B_{X}^{\ast \prime
}(e_{+})\right) =0\right\} .
\end{equation*}%
The set to the right of the intersection operation in the above display is
obviously the zero-set of a multivariate polynomial in the variables $x_{ti}$%
. Thus it is an algebraic set, and hence is either a $\lambda _{\mathbb{R}%
^{n\times k}}$-null set or is all of $\mathbb{R}^{n\times k}$. However, the
latter case can not arise because we can choose an $n\times k$ matrix $%
X^{\#}\in \mathfrak{X}_{0}$, say, such that all its columns are orthogonal
to $e_{+}$ (this being possible since $k<n$ by assumption) and this matrix
then satisfies $\limfunc{rank}\left( B_{X^{\#}}^{\ast }(e_{+})\right) =q$.
This shows that $\mathfrak{X}_{1}\left( e_{+}\right) $ is a $\lambda _{%
\mathbb{R}^{n\times k}}$-null set. Next consider $\mathfrak{X}_{2}\left(
e_{+}\right) $: Observe that for $X\in \mathfrak{X}_{0}\backslash \mathfrak{X%
}_{1}\left( e_{+}\right) $ we have $\det (\hat{\Omega}_{w,X}\left(
e_{+}\right) )\neq 0$ and hence for $X\in \mathfrak{X}_{0}\backslash 
\mathfrak{X}_{1}\left( e_{+}\right) $ the relation $T_{X}(e_{+}+\mu
_{0}^{\ast })=C$ can equivalently be written as%
\begin{equation*}
(R\limfunc{adj}(X^{\prime }X)X^{\prime }e_{+})^{\prime }\limfunc{adj}\hat{%
\Omega}_{w,X}\left( e_{+}\right) (R\limfunc{adj}(X^{\prime }X)X^{\prime
}e_{+})-\left( \det (X^{\prime }X)\right) ^{2}\det (\hat{\Omega}_{w,X}\left(
e_{+}\right) )C=0.
\end{equation*}%
Furthermore, for $X\in \mathfrak{X}_{0}$ we can write $\hat{\Omega}%
_{w,X}\left( e_{+}\right) $ as $\left( \det (X^{\prime }X)\right)
^{-4}B_{X}^{\ast }\left( e_{+}\right) \mathcal{W}_{n}B_{X}^{\ast \prime
}\left( e_{+}\right) $. Note that $B_{X}^{\ast }\left( e_{+}\right) \mathcal{%
W}_{n}B_{X}^{\ast \prime }\left( e_{+}\right) $ is a multivariate polynomial
in the variables $x_{ti}$. Consequently, for $X\in \mathfrak{X}%
_{0}\backslash \mathfrak{X}_{1}\left( e_{+}\right) $ the relation $%
T_{X}(e_{+}+\mu _{0}^{\ast })=C$ can, after multiplication by $\left( \det
(X^{\prime }X)\right) ^{4q-2}$, which is nonzero for $X\in \mathfrak{X}_{0}$%
, equivalently be written as%
\begin{eqnarray*}
\left( \det (X^{\prime }X)\right) ^{2}(R\limfunc{adj}(X^{\prime }X)X^{\prime
}e_{+})^{\prime }\limfunc{adj}\left( B_{X}^{\ast }\left( e_{+}\right) 
\mathcal{W}_{n}B_{X}^{\ast \prime }\left( e_{+}\right) \right) (R\limfunc{adj%
}(X^{\prime }X)X^{\prime }e_{+}) && \\
-\det (B_{X}^{\ast }\left( e_{+}\right) \mathcal{W}_{n}B_{X}^{\ast \prime
}\left( e_{+}\right) )C &=&0.
\end{eqnarray*}

The left-hand side of the above display is now a multivariate polynomial in
the elements $x_{ti}$. The polynomial does not vanish on all of $\mathbb{R}%
^{n\times k}$ since the matrix $X^{\#}$ constructed before provides an
element in $\mathfrak{X}_{0}\backslash \mathfrak{X}_{1}\left( e_{+}\right) $
for which $T_{X^{\#}}(e_{+}+\mu _{0}^{\ast })=0<C$ holds. The proofs for $%
\mathfrak{X}_{1}\left( e_{-}\right) $ and $\mathfrak{X}_{2}\left(
e_{-}\right) $ are completely analogous, as is the proof for the fact that $%
\mathbb{R}^{n\times k}\backslash \mathfrak{X}_{0}$ is a $\lambda _{\mathbb{R}%
^{n\times k}}$-null set. Finally, that the set of all design matrices $X\in 
\mathfrak{X}_{0}$ for which Theorem \ref{thmlrv} does not apply is a subset
of $\left( \mathfrak{X}_{1}\left( e_{+}\right) \cup \mathfrak{X}_{2}\left(
e_{+}\right) \right) \cap \left( \mathfrak{X}_{1}\left( e_{-}\right) \cup 
\mathfrak{X}_{2}\left( e_{-}\right) \right) $ is obvious upon observing that
the set of all $X\in \mathfrak{X}_{0}$ which do not satisfy Assumption \ref%
{R_and_X} is contained in $\mathfrak{X}_{1}\left( e_{+}\right) $ as well as
in $\mathfrak{X}_{1}\left( e_{-}\right) $.

(2) Similar arguments as in the proof of Part 1 show that $\mathfrak{\tilde{X%
}}_{1}\left( e_{-}\right) $ and $\mathfrak{\tilde{X}}_{2}\left( e_{-}\right) 
$ are each contained in an algebraic set. Define the matrix $X^{\sharp
}=\left( e_{+},\tilde{X}^{\sharp }\right) $ where the columns of $\tilde{X}%
^{\sharp }$ are $k-1$ linearly independent unit vectors that are orthogonal
to $e_{+}$ as well as $e_{-}$. It is then easy to see that $\tilde{X}%
^{\sharp }\in \mathfrak{\tilde{X}}_{0}\backslash \mathfrak{\tilde{X}}%
_{1}\left( e_{-}\right) $, implying that $\mathfrak{\tilde{X}}_{1}\left(
e_{-}\right) $ does not coincide with all of $\mathfrak{\tilde{X}}_{0}$.
Furthermore, simple computation shows that $T_{X^{\sharp }}(e_{-}+\mu
_{0}^{\ast })=0<C$ by the assumption on $R$, which implies that $\mathfrak{%
\tilde{X}}_{2}\left( e_{-}\right) $ is a proper subset of $\mathfrak{\tilde{X%
}}_{0}\backslash \mathfrak{\tilde{X}}_{1}\left( e_{-}\right) $. It follows
now as above that $\mathfrak{\tilde{X}}_{1}\left( e_{-}\right) $ and $%
\mathfrak{\tilde{X}}_{2}\left( e_{-}\right) $ are $\lambda _{\mathbb{R}%
^{n\times \left( k-1\right) }}$-null sets. The rest of the proof now
proceeds as before.

(3) See Example \ref{EX1}. $\blacksquare $

\textbf{Proof of Theorem \ref{TU_2}: }We verify the assumptions of Theorem %
\ref{TU_1}. By Lemma \ref{aux_100} Assumptions \ref{AE}, \ref{omega1}, and %
\ref{omega2} are satisfied. Because of $\mathfrak{C}=\mathfrak{C}_{AR(1)}$
we have that $J\left( \mathfrak{C}\right) =\limfunc{span}(e_{+})\cup 
\limfunc{span}(e_{-})$, see Lemma \ref{AR_1}, and because $e_{+},e_{-}\in 
\mathfrak{M}$ is assumed we conclude that $J\left( \mathfrak{C}\right)
\subseteq \mathfrak{M}$. The assumption $R\hat{\beta}(e_{+})=R\hat{\beta}%
(e_{-})=0$ then implies that even $J\left( \mathfrak{C}\right) \subseteq 
\mathfrak{M}_{0}-\mu _{0}$ holds. The invariance condition (\ref{T_inv_wrt_z}%
) in Theorem \ref{TU_1} is thus satisfied, because $T$ is $G\left( \mathfrak{%
M}_{0}\right) $-invariant by Lemma \ref{LWT}. The assumptions on $\mathfrak{C%
}$ in Theorem \ref{TU_1} are satisfied in view of Lemma \ref{AR_1}. Finally
the assumptions on $\hat{\Omega}_{w}$ in Parts 2 and 3 of Theorem \ref{TU_1}
are satisfied because $\hat{\Omega}_{w}$ is positive definite $\lambda _{%
\mathbb{R}^{n}}$-almost everywhere as shown in Lemma \ref{aux_100}. The
theorem now follows from Theorem \ref{TU_1} using a standard subsequence
argument for Part 3. The claim in parenthesis in Part 3 follows from the
corresponding claim in parenthesis in Theorem \ref{TU_1} and the observation
that the conditions on $e_{+}$ and $e_{-}$ in the theorem were only used to
verify condition (\ref{T_inv_wrt_z}). $\blacksquare $

\textbf{Proof of Theorem \ref{TU_3}: }Similar as in the preceding proof
verify the assumptions of Theorem \ref{TU_1} but now for $\bar{\beta}$ and $%
\bar{\Omega}_{w}$ by additionally making use of Proposition \ref{enforce_inv}%
. Note that the condition $\limfunc{span}\left( J(\mathfrak{C})\right) \cap 
\mathfrak{M}\subseteq \mathfrak{M}_{0}-\mu _{0}$ is satisfied in all five
parts of the theorem. This is obvious for Parts 1-3. For Part 4 this follows
from the following argument: Observe that $e_{-}=\delta e_{+}+X\gamma $ must
hold by the assumptions of Part 4. Now suppose $m\in \limfunc{span}\left( J(%
\mathfrak{C})\right) \cap \mathfrak{M}$. Then $\alpha _{+}e_{+}+\alpha
_{-}e_{-}=m=X\gamma ^{\ast }$ must hold. These relations together imply $%
(\alpha _{+}+\alpha _{-}\delta )e_{+}=X(\gamma ^{\ast }-\alpha _{-}\gamma )$%
. Because $e_{+}\notin \mathfrak{M}$, it follows that $\gamma ^{\ast
}-\alpha _{-}\gamma =0$. Thus 
\begin{equation*}
R\left( X^{\prime }X\right) ^{-1}X^{\prime }m=R\gamma ^{\ast }=\alpha
_{-}R\gamma =\alpha _{-}\bar{R}(\gamma ^{\prime }:\delta )^{\prime }=\alpha
_{-}\bar{R}\left( \bar{X}^{\prime }\bar{X}\right) ^{-1}\bar{X}^{\prime
}e_{-}=0,
\end{equation*}%
which establishes that $m\in \mathfrak{M}_{0}-\mu _{0}$. The verification
for Part 5 is completely analogous. $\blacksquare $

\textbf{Proof of Lemma \ref{LRVPD_2}: }Since $\hat{\Omega}_{w}\left(
y\right) =nB\left( y\right) \mathcal{W}_{n}^{\ast }B^{\prime }\left(
y\right) $, Parts 1-3 of the Lemma follow immediately from nonnegative
definiteness of $\mathcal{W}_{n}^{\ast }$. To prove Part 4 observe that $%
\hat{\Omega}_{w}\left( y\right) $ is singular if and only if $\det \left(
B\left( y\right) \mathcal{W}_{n}^{\ast }B^{\prime }\left( y\right) \right)
=0 $. Now observe that the l.h.s. of this equation is a multivariate
polynomial in $y$, hence the solution set is an algebraic set and thus is
either a $\lambda _{\mathbb{R}^{n}}$-null set or all of $\mathbb{R}^{n}$. $%
\blacksquare $

\textbf{Proof of Theorem \ref{thmlrv2}:} The proof is completely analogous
to the proof of Theorem \ref{thmlrv} using Lemma \ref{cpar2}\textbf{\ }in
case $\nu \in \left( 0,\pi \right) $. $\blacksquare $

\section{Appendix: Proofs for Subsection \protect\ref{GLS} \label{App_B}}

\textbf{Proof of Lemma \ref{LDEF}:} The inclusion $\mathfrak{M}\subseteq
N_{0}(a_{1},a_{2})$ is trivial since $\hat{u}\left( y\right) =0$ for $y\in 
\mathfrak{M}$. Because $a_{1}\in \left\{ 1,2\right\} $, $a_{2}\in \left\{
n-1,n\right\} $ with $a_{1}\leq a_{2}$ holds, $N_{0}(a_{1},a_{2})$ is
contained in $N_{1}(a_{1},a_{2})$, establishing the first claim. Closedness
of $N_{1}(a_{1},a_{2})$ is obvious. Given the just established inclusion $%
N_{0}(a_{1},a_{2})\subseteq N_{1}(a_{1},a_{2})$ the alternative description
of $N_{1}(a_{1},a_{2})$ given in the second claim is also immediately seen
to be true. Continuity of $\hat{\rho}$ on $\mathbb{R}^{n}\backslash
N_{0}(a_{1},a_{2})$ is obvious. Assume now that $k\leq a_{2}-a_{1}$ holds.
If $a_{1}=1$, $a_{2}=n$, i.e., $\hat{\rho}=\hat{\rho}_{YW}$, we have $%
N_{1}(a_{1},a_{2})=N_{0}(a_{1},a_{2})=\mathfrak{M}$ because $\hat{\rho}_{YW}$
is well-defined and bounded away from one in modulus on $\mathbb{R}%
^{n}\backslash \mathfrak{M}$ as shown in Remark \ref{rho_hat}(i). Hence, $%
N_{1}(a_{1},a_{2})$ is a $\lambda _{\mathbb{R}^{n}}$- null set in this case
as $k<n$ holds by assumption. To establish this result also for the other
choices of $a_{1}$ and $a_{2}$ note that $N_{1}(a_{1},a_{2})$ is the zero
set of a multivariate polynomial in $y$. It hence is a $\lambda _{\mathbb{R}%
^{n}}$- null set, provided we can show that the polynomial is not
identically zero. Observe that we now have $n-k\geq n-a_{2}+a_{1}\geq 2$ (as
we have already disposed off the case $a_{1}=1$, $a_{2}=n$). Let $%
y^{(1)},\ldots ,y^{(n-k)}$ be a basis for $\mathfrak{M}^{\bot }$. The
submatrix obtained from $\left( y^{(1)},\ldots ,y^{(n-k)}\right) $ by
selecting the rows with index $j$ satisfying $j<a_{1}$ as well as the rows
with $j>a_{2}$ has dimension $\left( n-a_{2}+a_{1}-1\right) \times \left(
n-k\right) $ and thus has rank at most $n-a_{2}+a_{1}-1<n-k$. Consequently,
we can find constants $c_{1},\ldots ,c_{n-k}$, not all equal to zero, such
that the $j$-th component of $y_{0}=\sum_{i=1}^{n-k}c_{i}y^{(i)}$ is zero
whenever $j<a_{1}$ or $j>a_{2}$. Because $y_{0}\in \mathfrak{M}^{\bot }$ and 
$y_{0}\neq 0$ by construction, we have $y_{0}\in \mathbb{R}^{n}\backslash 
\mathfrak{M}=\mathbb{R}^{n}\backslash N_{0}(1,n)$. Because $y_{0}\neq 0$ and
because the $j$-th component of $y_{0}=\hat{u}(y_{0})$ is zero whenever $%
j<a_{1}$ or $j>a_{2}$, we also have $y_{0}\in \mathbb{R}^{n}\backslash
N_{0}(a_{1},a_{2})$. Hence, $\hat{\rho}\left( y_{0}\right) $ as well as $%
\hat{\rho}_{YW}\left( y_{0}\right) $ are well-defined. Furthermore, they
coincide in view of the construction of $y_{0}=\hat{u}(y_{0})$. By what was
said above for the Yule-Walker estimator it follows that $\left\vert \hat{%
\rho}\left( y_{0}\right) \right\vert =\left\vert \hat{\rho}_{YW}\left(
y_{0}\right) \right\vert <1$. Hence $y_{0}\in \mathbb{R}^{n}\backslash
N_{1}(a_{1},a_{2})$, and the polynomial is not identically equal to zero. $%
\blacksquare $

\begin{lemma}
\label{lem_GLS} Suppose $\hat{\rho}$ satisfies Assumption \ref{ARER}.

\begin{enumerate}
\item The sets $\mathbb{R}^{n}\backslash N_{0}(a_{1},a_{2})$, $\mathbb{R}%
^{n}\backslash N_{1}(a_{1},a_{2})$, and $\mathbb{R}^{n}\backslash
N_{2}(a_{1},a_{2})$ are invariant under the group of transformations $%
y\mapsto \alpha y+X\gamma $ where $\alpha \neq 0$, $\gamma \in \mathbb{R}%
^{k} $.

\item The estimators $\tilde{\beta}$, $\tilde{\sigma}^{2}$, and $\tilde{%
\Omega}$ are well-defined and continuous on $\mathbb{R}^{n}\backslash
N_{2}(a_{1},a_{2})$. They satisfy the equivariance conditions $\tilde{\beta}%
(\alpha y+X\gamma )=\alpha \tilde{\beta}(y)+\gamma $, $\tilde{\sigma}%
^{2}(\alpha y+X\gamma )=\alpha ^{2}\tilde{\sigma}^{2}(y)$, and $\tilde{\Omega%
}(\alpha y+X\gamma )=\alpha ^{2}\tilde{\Omega}(y)$ for $\alpha \neq 0$, $%
\gamma \in \mathbb{R}^{k}$, and $y\in \mathbb{R}^{n}\backslash
N_{2}(a_{1},a_{2})$. The estimator $\tilde{\Omega}\left( y\right) $ is
(well-defined and) nonsingular if and only if $y\in \mathbb{R}^{n}\backslash
N_{2}^{\ast }(a_{1},a_{2})$. The sets $N_{2}(a_{1},a_{2})$ and $N_{2}^{\ast
}(a_{1},a_{2})$ are closed. If $k\leq a_{2}-a_{1}$ holds, $%
N_{2}(a_{1},a_{2}) $ and $N_{2}^{\ast }(a_{1},a_{2})$ are $\lambda _{\mathbb{%
R}^{n}}$-null sets.

\item The estimator $\hat{\Omega}$ is well-defined and continuous on $%
\mathbb{R}^{n}\backslash N_{0}(a_{1},a_{2})$, whereas $\hat{\beta}$ and $%
\hat{\sigma}^{2}$ are well-defined and continuous on all of $\mathbb{R}^{n}$%
. They satisfy the equivariance conditions $\hat{\beta}(\alpha y+X\gamma
)=\alpha \hat{\beta}(y)+\gamma $, $\hat{\sigma}^{2}(\alpha y+X\gamma
)=\alpha ^{2}\hat{\sigma}^{2}(y)$ for $\alpha \neq 0$, $\gamma \in \mathbb{R}%
^{k}$, and $y\in \mathbb{R}^{n}$, as well as $\hat{\Omega}(\alpha y+X\gamma
)=\alpha ^{2}\hat{\Omega}(y)$ for $\alpha \neq 0$, $\gamma \in \mathbb{R}%
^{k} $, and $y\in \mathbb{R}^{n}\backslash N_{0}(a_{1},a_{2})$. Furthermore, 
$\hat{\sigma}^{2}(y)>0$ holds for $y\in \mathbb{R}^{n}\backslash \mathfrak{M}%
\supseteq \mathbb{R}^{n}\backslash N_{0}^{\ast }(a_{1},a_{2})$, and hence $%
\hat{\Omega}\left( y\right) $ is (well-defined and) nonsingular if and only
if $y\in \mathbb{R}^{n}\backslash N_{0}^{\ast }(a_{1},a_{2})$. The set $%
N_{0}^{\ast }(a_{1},a_{2})$ is closed. If $k\leq a_{2}-a_{1}$ holds, $%
N_{0}^{\ast }(a_{1},a_{2})$ is a $\lambda _{\mathbb{R}^{n}}$-null set.
[Recall from Remark \ref{rho_hat}(iv) that $N_{0}(a_{1},a_{2})$ is always a
closed set, and is a $\lambda _{\mathbb{R}^{n}}$-null set in case $k\leq
a_{2}-a_{1}$.]
\end{enumerate}
\end{lemma}

\begin{proof}
(1) The invariance of the first two sets follows since $\hat{u}(\alpha
y+X\gamma )=\alpha \hat{u}(y)$ holds for every $y\in \mathbb{R}^{n}$, $%
\alpha \neq 0$, and $\gamma \in \mathbb{R}^{k}$. This property of the
residual vector implies $\hat{\rho}\left( \alpha y+X\gamma \right) =\hat{\rho%
}\left( y\right) $ for every $\alpha \neq 0$, $\gamma \in \mathbb{R}^{k}$
and $y\in \mathbb{R}^{n}\backslash N_{0}(a_{1},a_{2})\supseteq \mathbb{R}%
^{n}\backslash N_{1}(a_{1},a_{2})$. Together with the already established
invariance of $\mathbb{R}^{n}\backslash N_{1}(a_{1},a_{2})$ this implies
invariance of $\mathbb{R}^{n}\backslash N_{2}(a_{1},a_{2})$ upon observing
that $\Lambda ^{-1}(\hat{\rho}\left( y\right) )$ is well-defined for $y\in 
\mathbb{R}^{n}\backslash N_{1}(a_{1},a_{2})$. The latter holds because for $%
b\in \mathbb{R}$, $\left\vert b\right\vert \neq 1$ the matrix $\Lambda (b)$
is nonsingular. [This can, e.g., be seen from the fact that its inverse is
given by the symmetric tridiagonal matrix with diagonal equal to $\left(
1,1+b^{2},\ldots ,1+b^{2},1\right) /\left( 1-b^{2}\right) $ and with the
elements next to the diagonal given by $-b/\left( 1-b^{2}\right) $.]

(2) Using Lemma \ref{LDEF} and the just established fact that $\Lambda ^{-1}(%
\hat{\rho}\left( y\right) )$ is well-defined for $y\in \mathbb{R}%
^{n}\backslash N_{1}(a_{1},a_{2})$, we see that $\tilde{\beta}$, $\tilde{%
\sigma}^{2}$, and $\tilde{\Omega}$ are well-defined and continuous on $%
\mathbb{R}^{n}\backslash N_{2}(a_{1},a_{2})\subseteq \mathbb{R}%
^{n}\backslash N_{1}(a_{1},a_{2})$. Observing that $\hat{\rho}(\alpha
y+X\gamma )=\hat{\rho}(y)$ holds for $\alpha \neq 0$, $\gamma \in \mathbb{R}%
^{k}$, and $y\in \mathbb{R}^{n}\backslash N_{0}(a_{1},a_{2})\supseteq 
\mathbb{R}^{n}\backslash N_{1}(a_{1},a_{2})$, the claimed equivariance of $%
\tilde{\beta}$, $\tilde{\sigma}^{2}$, and $\tilde{\Omega}$ follows. The
third claim is obvious, and the fourth claim follows easily from Lemma \ref%
{LDEF}.\ We next prove the last claim for the Yule-Walker estimator, i.e.,
for $a_{1}=1$ and $a_{2}=n$: For this it suffices to show that $N_{2}^{\ast
}(1,n)\subseteq \mathfrak{M}$ since $\mathfrak{M}$ is a proper subspace of $%
\mathbb{R}^{n}$ in view of the assumption $k<n$. Now for arbitrary $y\notin 
\mathfrak{M}=N_{0}(1,n)$ we have that $\hat{\rho}_{YW}(y)$ is well-defined
and satisfies $\left\vert \hat{\rho}_{YW}(y)\right\vert <1$ (cf. Remark \ref%
{rho_hat}(i)) implying $y\in \mathbb{R}^{n}\backslash N_{1}(1,n)$ as well as
positive definiteness of $\Lambda (\hat{\rho}_{YW}(y))$. But this gives
positive definiteness, and hence nonsingularity, of $X^{\prime }\Lambda
^{-1}(\hat{\rho}_{YW}(y))X$, implying that $y\in \mathbb{R}^{n}\backslash
N_{2}(1,2)$. It also delivers positive definiteness of $R(X^{\prime }\Lambda
^{-1}(\hat{\rho}(y))X)^{-1}R^{\prime }$. Furthermore, $y\notin \mathfrak{M}$
implies $y-X\tilde{\beta}(y)\neq 0$ and thus $\tilde{\sigma}^{2}(y)>0$ in
view of the just established positive definiteness of $\Lambda (\hat{\rho}%
_{YW}(y))$. But this gives $y\in \mathbb{R}^{n}\backslash N_{2}^{\ast }(1,2)$%
, completing the proof for the case $a_{1}=1$ and $a_{2}=n$. To prove the
claim for the remaining values of $a_{1}$ and $a_{2}$ we first show that $%
N_{2}(a_{1},a_{2})$ is a $\lambda _{\mathbb{R}^{n}}$-null set: observe that $%
N_{2}(a_{1},a_{2})$ is the union of $N_{1}(a_{1},a_{2})$ and $\left\{ y\in 
\mathbb{R}^{n}\backslash N_{1}(a_{1},a_{2}):\det \left( X^{\prime }\Lambda
^{-1}(\hat{\rho}(y))X\right) =0\right\} $. In view of Lemma \ref{LDEF} it
hence suffices to show that the latter set is a $\lambda _{\mathbb{R}^{n}}$%
-null set.\ Using the relation $D^{-1}=\limfunc{adj}\left( D\right) /\det
\left( D\right) $ (with the convention that $\limfunc{adj}\left( D\right) =1$
if $D$ is $1\times 1$) and noting that $\det \left( \Lambda (\hat{\rho}%
(y))\right) \neq 0$ for $y\in \mathbb{R}^{n}\backslash N_{1}(a_{1},a_{2})$
the set in question can be rewritten as 
\begin{equation*}
A=\left\{ y\in \mathbb{R}^{n}\backslash N_{1}(a_{1},a_{2}):\det \left(
X^{\prime }\limfunc{adj}\left( \Lambda (\hat{\rho}(y))\right) X\right)
=0\right\} .
\end{equation*}%
Note that the equation in the set in the above display is polynomial in $%
\hat{\rho}(y)$. Upon multiplying the equation defining $A$ by $\left(
\sum_{t=a_{1}}^{a_{2}}\hat{u}_{t}^{2}(y)\right) ^{d}$, which is non-zero on $%
\mathbb{R}^{n}\backslash N_{1}(a_{1},a_{2})$, where $d=(n-1)^{2}k$, the set $%
A$ is seen to be the intersection of $\mathbb{R}^{n}\backslash
N_{1}(a_{1},a_{2})$ with the zero-set of a multivariate polynomial in $y$.
Hence, $A$ is a $\lambda _{\mathbb{R}^{n}}$-null set provided we can
establish that the polynomial is not identically zero. For this it suffices
to find an $y\in \mathbb{R}^{n}\backslash N_{1}(a_{1},a_{2})$ such that $%
\det \left( X^{\prime }\Lambda ^{-1}(\hat{\rho}(y))X\right) \neq 0$: Set $%
y=y_{0}$ where $y_{0}$ has been constructed in the proof of Lemma \ref{LDEF}%
. Observe that $\hat{\rho}(y_{0})=\hat{\rho}_{YW}(y_{0})$ for the estimator $%
\hat{\rho}$ specified by $a_{1}$ and $a_{2}$ and hence $y_{0}\in \mathbb{R}%
^{n}\backslash N_{1}(a_{1},a_{2})\subseteq \mathbb{R}^{n}\backslash 
\mathfrak{M}$ since $\left\vert \hat{\rho}_{YW}(y_{0})\right\vert <1$ holds.
But then $\det \left( X^{\prime }\Lambda ^{-1}(\hat{\rho}(y_{0}))X\right)
\neq 0$ holds because $\Lambda (\hat{\rho}_{YW}(y))$ is always positive
definite (whenever it is defined) as has been established before. This shows
that $N_{2}(a_{1},a_{2})$ is a $\lambda _{\mathbb{R}^{n}}$-null set. It
remains to show that $N_{2}^{\ast }(a_{1},a_{2})$ is a $\lambda _{\mathbb{R}%
^{n}}$-null set. For this it suffices to show that 
\begin{equation*}
B=\left\{ y\in \mathbb{R}^{n}\backslash N_{2}(a_{1},a_{2}):\tilde{\sigma}%
^{2}(y)=0\right\}
\end{equation*}%
as well as 
\begin{equation*}
C=\left\{ y\in \mathbb{R}^{n}\backslash N_{2}(a_{1},a_{2}):\det \left(
R(X^{\prime }\Lambda ^{-1}(\hat{\rho}(y))X)^{-1}R^{\prime }\right) =0\right\}
\end{equation*}%
are $\lambda _{\mathbb{R}^{n}}$-null sets. Noting that $\det \left( \Lambda (%
\hat{\rho}(y))\right) \neq 0$ as well as $\det \left( X^{\prime }\limfunc{adj%
}\left( \Lambda (\hat{\rho}(y))\right) X\right) \neq 0$ hold for $y\in 
\mathbb{R}^{n}\backslash N_{2}(a_{1},a_{2})$, the set $B$ can be rewritten as%
\begin{equation*}
B=\left\{ y\in \mathbb{R}^{n}\backslash N_{2}(a_{1},a_{2}):\det \left(
\Lambda (\hat{\rho}(y))\right) \left[ \det \left( X^{\prime }\limfunc{adj}%
\left( \Lambda (\hat{\rho}(y))\right) X\right) \right] ^{2}\tilde{\sigma}%
^{2}(y)=0\right\} .
\end{equation*}%
Again the equation in the set in the above display is polynomial in $y$ and $%
\hat{\rho}(y)$. Upon multiplying this by $\left( \sum_{t=a_{1}}^{a_{2}}\hat{u%
}_{t}^{2}(y)\right) ^{d}$, which is non-zero on $\mathbb{R}^{n}\backslash
N_{2}(a_{1},a_{2})$, where $d=(n-1)^{2}(2k+1)$, one sees that $B$ is the
intersection of $\mathbb{R}^{n}\backslash N_{2}(a_{1},a_{2})$ with the
zero-set of a multivariate polynomial in $y$. To establish that $B$ is a $%
\lambda _{\mathbb{R}^{n}}$-null set it thus suffices to find an $y\in 
\mathbb{R}^{n}\backslash N_{2}(a_{1},a_{2})$ with $\tilde{\sigma}^{2}(y)>0$.
Choose $y_{0}$ as above. Then we know that $y_{0}\in \mathbb{R}%
^{n}\backslash N_{1}(a_{1},a_{2})$ and $\det \left( X^{\prime }\Lambda ^{-1}(%
\hat{\rho}(y_{0}))X\right) \neq 0$ hold, i.e., $y_{0}\in \mathbb{R}%
^{n}\backslash N_{2}(a_{1},a_{2})$. Furthermore, as shown before $\Lambda (%
\hat{\rho}(y_{0}))$ is positive definite (since $\Lambda (\hat{\rho}%
(y_{0}))=\Lambda (\hat{\rho}_{YW}(y_{0}))$) and $y_{0}-X\tilde{\beta}%
(y_{0})\neq 0$ holds (since $y_{0}\notin \mathfrak{M}$). Consequently, $%
\tilde{\sigma}^{2}(y_{0})>0$ holds. The proof for $C$ is very similar.

(3) Well-definedness is trivial and continuity follows from continuity of $%
\hat{\rho}$ on the open set $\mathbb{R}^{n}\backslash N_{0}(a_{1},a_{2})$
(cf. Lemma \ref{LDEF}). Equivariance of $\hat{\beta}$ and $\hat{\sigma}^{2}$
is obvious, while the equivariance property of $\hat{\Omega}$ follows from
invariance of $\mathbb{R}^{n}\backslash N_{0}(a_{1},a_{2})$ and the
equivariance of $\hat{\rho}$ established in (1). The third claim is obvious.
Closedness of $N_{0}^{\ast }(a_{1},a_{2})$ follows from the continuity
property of $\hat{\rho}$ established in Lemma \ref{LDEF}. To prove the final
claim observe that $N_{0}^{\ast }(a_{1},a_{2})$ is the union of the $\lambda
_{\mathbb{R}^{n}}$-null set $N_{0}(a_{1},a_{2})$ with%
\begin{equation*}
\left\{ y\in \mathbb{R}^{n}\backslash N_{0}(a_{1},a_{2}):\det \left(
R(X^{\prime }X)^{-1}X^{\prime }\Lambda (\hat{\rho}(y))X(X^{\prime
}X)^{-1}R^{\prime }\right) =0\right\} .
\end{equation*}%
Multiplying the equation defining this set by $\left( \sum_{t=a_{1}}^{a_{2}}%
\hat{u}_{t}^{2}(y)\right) ^{q(n-1)}$, which is non-zero on $\mathbb{R}%
^{n}\backslash N_{0}(a_{1},a_{2})$, one sees that the above set is the
intersection of $\mathbb{R}^{n}\backslash N_{0}(a_{1},a_{2})$ with the
zero-set of a multivariate polynomial in $y$. Again perusing $y_{0}$
constructed before shows that the polynomial is not identically zero, which
then delivers the desired result.
\end{proof}

The following lemma is an immediate consequence of Lemma \ref{lem_GLS}.

\begin{lemma}
\label{lem_GLS_1} Suppose $\hat{\rho}$ satisfies Assumption \ref{ARER} and $%
k\leq a_{2}-a_{1}$ holds. Then $\tilde{\beta}$ and $\tilde{\Omega}$ satisfy
Assumption \ref{AE} with $N=N_{2}(a_{1},a_{2})$, and the set $N^{\ast }$
(cf. equation (\ref{Nstar})) is given by $N_{2}^{\ast }(a_{1},a_{2})$.
Similarly, $\hat{\beta}$ and $\hat{\Omega}$ satisfy Assumption \ref{AE} with 
$N=N_{0}(a_{1},a_{2})$, and the set $N^{\ast }$ is given by $N_{0}^{\ast
}(a_{1},a_{2})$. The sets $N_{0}^{\ast }(a_{1},a_{2})$ and $N_{2}^{\ast
}(a_{1},a_{2})$ are invariant under the group of transformations $y\mapsto
\alpha y+X\gamma $ where $\alpha \neq 0$, $\gamma \in \mathbb{R}^{k}$.
\end{lemma}

\begin{proof}
The lemma except for the last claim follows from Lemma \ref{lem_GLS}. The
last claim then follows from Lemma \ref{aux} in Appendix \ref{App_E}, cf.
also the discussion following Assumption \ref{AE}.
\end{proof}

\begin{lemma}
\label{lem_GLS_2}Suppose $\hat{\rho}$ satisfies Assumption \ref{ARER} and $%
k\leq a_{2}-a_{1}$ holds. Then $\tilde{\Omega}$ and $\hat{\Omega}$ satisfy
Assumptions \ref{omega1} and \ref{omega2} with $N^{\ast }=N_{2}^{\ast
}(a_{1},a_{2})$ in case of $\tilde{\Omega}$ and with $N^{\ast }=N_{0}^{\ast
}(a_{1},a_{2})$ in case of $\hat{\Omega}$.
\end{lemma}

\begin{proof}
Consider first the case of the Yule-Walker estimator, i.e., $a_{1}=1$ and $%
a_{2}=n$. Then $\Lambda (\hat{\rho}_{YW}(y))$ is positive definite for every 
$y\notin N_{0}(a_{1},a_{2})$. Hence $\tilde{\Omega}\left( y\right) $ is
positive definite for $y\notin N_{2}^{\ast }(a_{1},a_{2})$ and $\hat{\Omega}%
\left( y\right) $ is positive definite for $y\notin N_{0}^{\ast
}(a_{1},a_{2})$. Consequently, Assumptions \ref{omega1} and \ref{omega2} are
clearly satisfied. Next consider the case where $a_{1}\neq 1$ or $a_{2}\neq
n $. Then $y_{0}$ constructed in the proof of Lemma \ref{LDEF} satisfies $%
y_{0}\in \mathbb{R}^{n}\backslash N_{0}^{\ast }(a_{1},a_{2})$ as well as $%
y_{0}\in \mathbb{R}^{n}\backslash N_{2}^{\ast }(a_{1},a_{2})$\ as shown in
the proof of Lemma \ref{lem_GLS}. Because of $\hat{\rho}(y_{0})=\hat{\rho}%
_{YW}(y_{0})$, we also see that $\tilde{\Omega}(y_{0})$ as well as $\hat{%
\Omega}(y_{0})$ are positive definite (as the variance covariance estimators
based on $\hat{\rho}$ coincide with the ones based on the Yule-Walker
estimator). This shows that Assumption \ref{omega1} is satisfied for $\tilde{%
\Omega}$ and $\hat{\Omega}$. It remains to establish Assumption \ref{omega2}%
: Let $v\neq 0$, $v\in \mathbb{R}^{q}$ be arbitrary. The preceding argument
has shown $y_{0}\in \mathbb{R}^{n}\backslash N_{2}^{\ast }(a_{1},a_{2})$ and 
$y_{0}\in \mathbb{R}^{n}\backslash N_{0}^{\ast }(a_{1},a_{2})$ and also
shows that $v^{\prime }\tilde{\Omega}^{-1}\left( y_{0}\right) v>0$ and $%
v^{\prime }\hat{\Omega}^{-1}\left( y_{0}\right) v>0$ hold. To complete the
proof it suffices to show that the set $\left\{ y\in \mathbb{R}%
^{n}\backslash N_{2}^{\ast }(a_{1},a_{2}):v^{\prime }\tilde{\Omega}%
^{-1}\left( y\right) v=0\right\} $ is the intersection of $y\in \mathbb{R}%
^{n}\backslash N_{2}^{\ast }(a_{1},a_{2})$ with the zero-set of a
multivariate polynomial, and similarly for $\left\{ y\in \mathbb{R}%
^{n}\backslash N_{0}^{\ast }(a_{1},a_{2}):v^{\prime }\hat{\Omega}^{-1}\left(
y\right) v=0\right\} $. This is proved in a similar manner as in the proof
of Lemma \ref{lem_GLS} by rewriting all inverse matrices appearing in $%
v^{\prime }\tilde{\Omega}^{-1}\left( y\right) v$ ($v^{\prime }\hat{\Omega}%
^{-1}\left( y\right) v$, respectively) in terms of the adjoints and
determinants and observing that the determinants are all non-zero for $y\in 
\mathbb{R}^{n}\backslash N_{2}^{\ast }(a_{1},a_{2})$ ($y\in \mathbb{R}%
^{n}\backslash N_{0}^{\ast }(a_{1},a_{2})$, respectively). This shows that $%
v^{\prime }\tilde{\Omega}^{-1}\left( y\right) v=0$ ($v^{\prime }\hat{\Omega}%
^{-1}\left( y\right) v=0$, respectively) can be rewritten as a polynomial
equation in $\hat{\rho}(y)$. Multiplying this polynomial equation by a
suitable power of $\sum_{t=a_{1}}^{a_{2}}\hat{u}_{t}^{2}(y)$, which is
non-zero on $\mathbb{R}^{n}\backslash N_{2}^{\ast }(a_{1},a_{2})$ ($\mathbb{R%
}^{n}\backslash N_{0}^{\ast }(a_{1},a_{2})$, respectively) shows that these
equations can be rewritten as polynomial equations in $y$.
\end{proof}

\textbf{Proof of Theorem \ref{thmGLS}:} We first verify the assumptions of
Corollary \ref{CW}. Assumption \ref{AE} is satisfied for $\tilde{\beta}$ and 
$\tilde{\Omega}$ (with $N=N_{2}(a_{1},a_{2})$ and $N^{\ast }=N_{2}^{\ast
}(a_{1},a_{2})$) as well as for $\hat{\beta}$ and $\hat{\Omega}$ (with $%
N=N_{0}(a_{1},a_{2})$ and $N^{\ast }=N_{0}^{\ast }(a_{1},a_{2})$) in view of
Lemma \ref{lem_GLS_1}. In view of Assumption \ref{AAR(1)} we conclude from
Lemma \ref{AR_1} that $\mathcal{Z}_{+}=\limfunc{span}\left( e_{+}\right) $
as well as $\mathcal{Z}_{-}=\limfunc{span}\left( e_{-}\right) $ are
concentration spaces of $\mathfrak{C}$.\ Applying Parts 1 and 2 of Corollary %
\ref{CW} and Remark \ref{remnew}(i) to $\mathcal{Z}_{+}$ as well as to $%
\mathcal{Z}_{-}$ establishes (1) and (2) of the theorem as well as the
corresponding parts of (4), if we also note that the size of the test can
not be zero in view of Part 5 of Lemma \ref{LWT} and Lemma \ref{lem_GLS_2}.
In order to prove (3) of the theorem, we apply Theorem \ref{prop_101}. First
note that $\tilde{\Omega}$ satisfies Assumption \ref{omega2} because of
Lemma \ref{lem_GLS_2}. Furthermore, choose as the sequence $\Sigma _{m}$ in
that theorem $\Sigma _{m}=\Lambda \left( \rho _{m}\right) $ for some
sequence $\rho _{m}\rightarrow 1$, $\rho _{m}\in \left( -1,1\right) $. Then $%
\bar{\Sigma}=e_{+}e_{+}$ by Lemma \ref{AR_1}, which also provides the matrix 
$D$ and its required properties. Hence $l=1$ and $\limfunc{span}\left( \bar{%
\Sigma}\right) =\limfunc{span}\left( e_{+}\right) $ which is contained in $%
\mathfrak{M}$ since $e_{+}\in \mathfrak{M}$ has been assumed and $\mathfrak{M%
}$ is a linear space. Condition (\ref{test_conc_100}) in Theorem \ref%
{prop_101} is satisfied in view of the assumption $R\hat{\beta}(e_{+})\neq 0$
since $\limfunc{span}\left( \bar{\Sigma}\right) =\limfunc{span}\left(
e_{+}\right) $. Inspection of the constants $K_{1}$ and $K_{2}$ in Theorem %
\ref{prop_101} reveal that $K_{1}=K_{2}=:K_{FGLS}\left( e_{+}\right) $ since
in the present case $\gamma $ is one-dimensional. That $K_{FGLS}\left(
e_{+}\right) $ depends only on the quantities given in the theorem is
obvious from the formulas for $K_{1}$ and $K_{2}$. Furthermore, if $\hat{\rho%
}\equiv \hat{\rho}_{YW}$, then $\tilde{\Omega}$ is always positive definite
on $\mathbb{R}^{n}\backslash N_{2}(1,n)=\mathbb{R}^{n}\backslash \mathfrak{M}
$, because $\left\vert \hat{\rho}_{YW}\right\vert <1$ holds implying that $%
\Lambda \left( \hat{\rho}_{YW}\right) $ is positive definite on $\mathbb{R}%
^{n}\backslash N_{2}(1,n)$. Inspection of the constants $K_{1}$ and $K_{2}$
then reveals $K_{1}=K_{2}=1$ in that case. The claims in (3) with $e_{+}$
replaced by $e_{-}$ are proved analogously, and so are the remaining claims
in (4). $\blacksquare $

\textbf{Proof of Proposition \ref{generic_2}:} (1) First consider $\mathfrak{%
X}_{1,FGLS}\left( e_{+}\right) $. The condition $e_{+}\in N_{2,X}^{\ast
}\left( a_{1},a_{2}\right) $ is equivalent to $e_{+}\in N_{1,X}\left(
a_{1},a_{2}\right) $, or $e_{+}\in \mathbb{R}^{n}\backslash
N_{1,X}(a_{1},a_{2})$ but $\det \left( X^{\prime }\Lambda ^{-1}(\hat{\rho}%
_{X}(e_{+}))X\right) =0$, or to $e_{+}\in \mathbb{R}^{n}\backslash
N_{1,X}(a_{1},a_{2})$ and $\det \left( X^{\prime }\Lambda ^{-1}(\hat{\rho}%
_{X}(e_{+}))X\right) \neq 0$ but $\tilde{\sigma}_{X}^{2}(e_{+})\det \left(
R(X^{\prime }\Lambda ^{-1}(\hat{\rho}_{X}(e_{+}))X)^{-1}R^{\prime }\right)
=0 $. The first one of these three conditions can be written as%
\begin{equation}
\left( \sum\limits_{t=2}^{n}\hat{u}_{t}(e_{+})\hat{u}_{t-1}(e_{+})\right)
^{2}=\left( \sum\limits_{t=a_{1}}^{a_{2}}\hat{u}_{t}^{2}(e_{+})\right) ^{2}.
\label{eq_1}
\end{equation}%
Since $\det (X^{\prime }X)\neq 0$ holds for $X\in \mathfrak{X}_{0}$, the set
of $X\in \mathfrak{X}_{0}$ satisfying (\ref{eq_1}) is -- after
multiplication of both sides of (\ref{eq_1}) by the fourth power of $\det
(X^{\prime }X)$ -- seen to be included in the zero-set of a multivariate
polynomial in the variables $x_{ti}$. Observing that $\det \left( \Lambda (%
\hat{\rho}_{X}(e_{+}))\right) \neq 0$ and $\sum_{t=a_{1}}^{a_{2}}\hat{u}%
_{t,X}^{2}(e_{+})\neq 0$ for $e_{+}\in \mathbb{R}^{n}\backslash
N_{1,X}(a_{1},a_{2})$, the second one of the above conditions takes the
equivalent form%
\begin{equation}
\left( \sum_{t=a_{1}}^{a_{2}}\hat{u}_{t}^{2}(e_{+})\right) ^{k(n-1)^{2}}\det
\left( X^{\prime }\limfunc{adj}\left( \Lambda (\hat{\rho}_{X}(e_{+}))\right)
X\right) =0,\left( \sum\limits_{t=2}^{n}\hat{u}_{t}(e_{+})\hat{u}%
_{t-1}(e_{+})\right) ^{2}\neq \left( \sum\limits_{t=a_{1}}^{a_{2}}\hat{u}%
_{t}^{2}(e_{+})\right) ^{2}.  \label{eq_2}
\end{equation}%
For $X\in \mathfrak{X}_{0}$ satisfying the inequality in (\ref{eq_2}), the
left-hand side of the equation in the preceding display is easily seen to be
a polynomial in the variables $x_{ti}$ and $\hat{u}_{t}(e_{+})$. Since $\det
(X^{\prime }X)\hat{u}_{t}(e_{+})$ is polynomial in the variables $x_{ti}$
and $\det (X^{\prime }X)\neq 0$ for $X\in \mathfrak{X}_{0}$, we may rewrite
the equation in the preceding display by multiplying it by the $4k(n-1)^{2}$%
-th power of $\det (X^{\prime }X)$. The resulting equivalent equation is
obviously a polynomial in the variables $x_{ti}$. This shows that the set of 
$X\in \mathfrak{X}_{0}$ satisfying (\ref{eq_2}) is (a subset of) the
zero-set of a multivariate polynomial. Recalling that $\det \left( \Lambda (%
\hat{\rho}_{X}(e_{+}))\right) \neq 0$ and $\sum_{t=a_{1}}^{a_{2}}\hat{u}%
_{t}^{2}(e_{+})\neq 0$ for $e_{+}\in \mathbb{R}^{n}\backslash
N_{1,X}(a_{1},a_{2})$, and that $\det \left( X^{\prime }\Lambda ^{-1}(\hat{%
\rho}_{X}(e_{+}))X\right) \neq 0$ implies $\det \left( X^{\prime }\limfunc{%
adj}\left( \Lambda (\hat{\rho}_{X}(e_{+}))\right) X\right) \neq 0$, the
third one of the above conditions takes the equivalent form%
\begin{equation}
\left( \sum_{t=a_{1}}^{a_{2}}\hat{u}_{t}^{2}(e_{+})\right)
^{(n-1)^{2}(2k+1+q(k-1))}f(X)^{\prime }\limfunc{adj}\left( \Lambda (\hat{\rho%
}_{X}(e_{+}))\right) f(X)g(X)=0  \label{eq_2'}
\end{equation}%
subject to%
\begin{equation}
\left( \sum\limits_{t=2}^{n}\hat{u}_{t}(e_{+})\hat{u}_{t-1}(e_{+})\right)
^{2}\neq \left( \sum\limits_{t=a_{1}}^{a_{2}}\hat{u}_{t}^{2}(e_{+})\right)
^{2},\det \left( X^{\prime }\Lambda ^{-1}(\hat{\rho}_{X}(e_{+}))X\right)
\neq 0,  \label{eq_2''}
\end{equation}%
where 
\begin{eqnarray*}
f(X) &=&\left[ \det \left( X^{\prime }\limfunc{adj}\left( \Lambda (\hat{\rho}%
_{X}(e_{+}))\right) X)\right) I_{n}-X\limfunc{adj}\left( X^{\prime }\limfunc{%
adj}\left( \Lambda (\hat{\rho}_{X}(e_{+}))\right) X\right) X^{\prime }%
\limfunc{adj}\left( \Lambda (\hat{\rho}_{X}(e_{+}))\right) \right] e_{+} \\
g(X) &=&\det \left( R\limfunc{adj}(X^{\prime }\limfunc{adj}\left( \Lambda (%
\hat{\rho}_{X}(e_{+}))\right) X)R^{\prime }\right) .
\end{eqnarray*}%
The left-hand side of the equation in (\ref{eq_2'}) is a polynomial in the
variables $x_{ti}$ as well as $\hat{u}_{t,X}(e_{+})$ for all $X\in \mathfrak{%
X}_{0}$ satisfying the inequality in (\ref{eq_2''}). After multiplying the
left-hand side of the equation in (\ref{eq_2'}) by a suitable power of $\det
(X^{\prime }X)$, which is non-zero for $X\in \mathfrak{X}_{0}$, (\ref{eq_2'}%
) can be equivalently recast as an equation that is polynomial in $x_{ti}$,
showing that the set of $X\in \mathfrak{X}_{0}$ satisfying (\ref{eq_2'}) and
(\ref{eq_2''}) is a subset of the zero-set of a multivariate polynomial. It
follows that $\mathfrak{X}_{1,FGLS}\left( e_{+}\right) $ is a $\lambda _{%
\mathbb{R}^{n\times k}}$-null set provided we can show that each of the
three polynomials in the variables $x_{ti}$ mentioned before is not trivial.
For this it certainly suffices to construct a matrix $X\in \mathfrak{X}_{0}$
such that $e_{+}\notin N_{2,X}^{\ast }\left( a_{1},a_{2}\right) $ holds:
Consider first the case $n\geq 3$. Let the first column $x_{\cdot 1}^{\ast }$
of $X^{\ast }$ be equal to $(1,0,\ldots ,0,1)^{\prime }$, and choose the
remaining columns linearly independent in the orthogonal complement of the
space spanned by $x_{\cdot 1}^{\ast }$ and $e_{+}$. Then $X^{\ast }\in 
\mathfrak{X}_{0}$ holds and $\hat{u}_{X^{\ast }}(e_{+})=(0,1,1,\ldots
,1,1,0)^{\prime }$ and hence $\hat{\rho}_{X^{\ast }}\left( e_{+}\right) $ is
well-defined and equals $\hat{\rho}_{YW,X^{\ast }}\left( e_{+}\right) $,
which is always less than $1$ in absolute value. Consequently, $e_{+}\in 
\mathbb{R}^{n}\backslash N_{1,X^{\ast }}(a_{1},a_{2})$ holds. Furthermore, $%
\Lambda (\hat{\rho}_{X^{\ast }}(e_{+}))$ is then positive definite and hence 
$\det \left( X^{\ast \prime }\Lambda ^{-1}(\hat{\rho}_{X^{\ast
}}(e_{+}))X^{\ast }\right) \neq 0$ and $\det \left( R(X^{\ast \prime
}\Lambda ^{-1}(\hat{\rho}_{X^{\ast }}(e_{+}))X^{\ast })^{-1}R^{\prime
}\right) \neq 0$ hold; also $\tilde{\sigma}_{X^{\ast }}^{2}(e_{+})>0$
follows from positive definiteness of $\Lambda (\hat{\rho}_{X^{\ast
}}(e_{+}))$ and the fact that $e_{+}\notin \limfunc{span}\left( X^{\ast
}\right) $. But this establishes $e_{+}\in \mathbb{R}^{n}\backslash
N_{2,X^{\ast }}^{\ast }(a_{1},a_{2})$ in case $n\geq 3$. Next consider the
case $n=2$. Then $k=1$ must hold. The assumption $k\leq a_{2}-a_{1}$ entails 
$a_{2}=n=2$ and $a_{1}=1$, i.e., $\hat{\rho}$ must be the Yule-Walker
estimator implying that $N_{2,X^{\ast }}^{\ast }(a_{1},a_{2})=\limfunc{span}%
\left( X^{\ast }\right) $. Choose $X^{\ast }$ as an arbitrary vector
linearly independent of $e_{+}$ (which is possible since $n=2>1=k$). Then $%
X^{\ast }\in \mathfrak{X}_{0}$ and $e_{+}\in \mathbb{R}^{n}\backslash
N_{2,X^{\ast }}^{\ast }(a_{1},a_{2})$ are satisfied. The proof for $%
\mathfrak{X}_{1,FGLS}\left( e_{-}\right) $ is completely analogous where in
case $n\geq 3$ the matrix $X^{\ast }$ is now chosen in such a way that $%
x_{\cdot 1}^{\ast }$ is equal to $(-1,0,\ldots ,0,\left( -1\right)
^{n})^{\prime }$ and $e_{-}$ takes the r\^{o}le of $e_{+}$ in the
construction of the remaining columns. Next consider the set $\mathfrak{X}%
_{2,FGLS}\left( e_{+}\right) $. Observe that for $X\in \mathfrak{X}%
_{0}\backslash \mathfrak{X}_{1,FGLS}\left( e_{+}\right) $ the relation $%
T_{FGLS,X}(e_{+}+\mu _{0}^{\ast })=C$ can equivalently be written as%
\begin{equation}
(R\tilde{\beta}_{X}(e_{+}))^{\prime }\tilde{\Omega}_{X}^{-1}(e_{+})(R\tilde{%
\beta}_{X}(e_{+}))-C=0.  \label{eq_3}
\end{equation}%
Similar arguments as above show that for $X\in \mathfrak{X}_{0}\backslash 
\mathfrak{X}_{1,FGLS}\left( e_{+}\right) $ this equation can equivalently be
stated as $p(X)=0$ where $p(X)$ is a polynomial in the variables $x_{ti}$.
But this shows that the set of $X\in \mathfrak{X}_{0}\backslash \mathfrak{X}%
_{1,FGLS}\left( e_{+}\right) $ satisfying (\ref{eq_3}) is (a subset of) an
algebraic set. It follows that $\mathfrak{X}_{2,FGLS}\left( e_{+}\right) $
is a $\lambda _{\mathbb{R}^{n\times k}}$-null set provided the polynomial $p$
is not trivial, or in other words that there exists a matrix $X\in \mathfrak{%
X}_{0}\backslash \mathfrak{X}_{1,FGLS}\left( e_{+}\right) $ that violates (%
\ref{eq_3}). But this is guaranteed by the provision in the theorem. The
result for $\mathfrak{X}_{2,FGLS}\left( e_{-}\right) $ is proved in exactly
the same manner. The remaining claims of Part 1 are now obvious.

(2) Similar arguments as in the proof of Part 1 show that $\mathfrak{\tilde{X%
}}_{1,FGLS}\left( e_{-}\right) $ and $\mathfrak{\tilde{X}}_{2,FGLS}\left(
e_{-}\right) $ are each contained in an algebraic set. By the assumed
provision it follows immediately that $\mathfrak{\tilde{X}}_{2,FGLS}\left(
e_{-}\right) $ is a $\lambda _{\mathbb{R}^{n\times \left( k-1\right) }}$%
-null set. The same conclusion holds for $\mathfrak{\tilde{X}}%
_{1,FGLS}\left( e_{-}\right) $ if we can find a matrix $X^{\ast }=\left(
e_{+},\tilde{X}^{\ast }\right) $ such that $e_{-}\notin N_{2,X^{\ast
}}^{\ast }\left( a_{1},a_{2}\right) $. To this end let the $n\times 1$
vector $a=\left( -1,0,\ldots ,0,\left( -1\right) ^{n}\right) ^{\prime }$ be
the first column of $\tilde{X}^{\ast }$ and choose the remaining $k-2$
columns linearly independent in the orthogonal complement of the space
spanned by $e_{+}$, $e_{-}$, and $a$ (which is possible since $k<n$). Simple
computation now shows that $\hat{u}_{X^{\ast }}(e_{-})\neq 0$ (note that $%
n\geq 4$ has been assumed) and that the first and last entry of $\hat{u}%
_{X^{\ast }}(e_{-})$ is zero. Consequently, $\hat{\rho}_{X^{\ast }}\left(
e_{-}\right) $ is well-defined and equals $\hat{\rho}_{YW,X^{\ast }}\left(
e_{-}\right) $, which is always less than $1$ in absolute value, and the
same argument as in the proof of Part 1 shows that $e_{-}\notin N_{2,X^{\ast
}}^{\ast }\left( a_{1},a_{2}\right) $ is indeed satisfied. The remaining
claims of Part 2 are now obvious.

(3) First consider $\mathfrak{X}_{1,OLS}\left( e_{+}\right) $. The condition 
$e_{+}\in N_{0,X}^{\ast }\left( a_{1},a_{2}\right) $ is equivalent to $%
\sum_{t=a_{1}}^{a_{2}}\hat{u}_{t}^{2}(e_{+})=0$, or $\sum_{t=a_{1}}^{a_{2}}%
\hat{u}_{t}^{2}(e_{+})\neq 0$ but $\det \left( R(X^{\prime }X)^{-1}X^{\prime
}\Lambda (\hat{\rho}(y))X(X^{\prime }X)^{-1}R^{\prime }\right) =0$. Similar
arguments as in (1) then show that $\mathfrak{X}_{1,OLS}\left( e_{+}\right) $
is a subset of an algebraic set. The matrix $X^{\ast }$ constructed in (1)
is easily seen to satisfy $e_{+}\in \mathbb{R}^{n}\backslash N_{0,X^{\ast
}}^{\ast }\left( a_{1},a_{2}\right) $. Thus $\mathfrak{X}_{1,OLS}\left(
e_{+}\right) $ is a $\lambda _{\mathbb{R}^{n\times k}}$-null set. The proof
for $\mathfrak{X}_{1,OLS}\left( e_{-}\right) $ is exactly the same. Next
consider $\mathfrak{X}_{2,OLS}\left( e_{+}\right) $. Observe that for $X\in 
\mathfrak{X}_{0}\backslash \mathfrak{X}_{1,OLS}\left( e_{+}\right) $ the
relation $T_{OLS,X}(e_{+}+\mu _{0}^{\ast })=C$ can equivalently be written as%
\begin{equation}
(R\hat{\beta}_{X}(e_{+}))^{\prime }\hat{\Omega}_{X}^{-1}(e_{+})(R\hat{\beta}%
_{X}(e_{+}))-C=0.  \label{eq_4}
\end{equation}%
The same argument as in the proof of Part (1) shows that the set of $X\in 
\mathfrak{X}_{0}\backslash \mathfrak{X}_{1,OLS}\left( e_{+}\right) $
satisfying (\ref{eq_4}) is (a subset of) the zero-set of a multivariate
polynomial in the variables $x_{ti}$. It follows that $\mathfrak{X}%
_{2,OLS}\left( e_{+}\right) $ is a $\lambda _{\mathbb{R}^{n\times k}}$-null
set under the maintained provision that it is a proper subset of $\mathfrak{X%
}_{0}\backslash \mathfrak{X}_{1,OLS}\left( e_{+}\right) $. The proof for $%
\mathfrak{X}_{2,OLS}\left( e_{-}\right) $ is the same. The proof for $%
\mathfrak{\tilde{X}}_{1,OLS}\left( e_{-}\right) $ and $\mathfrak{\tilde{X}}%
_{2,OLS}\left( e_{-}\right) $ is similar to the proof for $\mathfrak{\tilde{X%
}}_{1,FGLS}\left( e_{-}\right) $ and $\mathfrak{\tilde{X}}_{2,FGLS}\left(
e_{-}\right) $.

(4) Note that the assumptions obviously imply $e_{+}\in \mathfrak{M}$ and $R%
\hat{\beta}(e_{+})\neq 0$. $\blacksquare $

\begin{remark}
In case $a_{1}=1$ and $a_{2}=n$ the argument in the above proof simplifies
due to the fact that $N_{2,X}^{\ast }\left( 1,n\right) =N_{0,X}^{\ast
}\left( 1,n\right) =\limfunc{span}\left( X\right) $.
\end{remark}

\textbf{Proof of Theorem \ref{excGLS2}:} We apply Theorem \ref{TU_1}. That $(%
\tilde{\beta},\tilde{\Omega})$ as well as $(\hat{\beta},\hat{\Omega})$
satisfy Assumptions \ref{AE}, \ref{omega1}, and \ref{omega2} has been shown
in Lemmata \ref{lem_GLS_1} and \ref{lem_GLS_2}. The covariance model $%
\mathfrak{C}_{AR(1)}$ satisfies the properties required in Theorem \ref{TU_1}
as shown in Lemma \ref{AR_1}. Furthermore, we have $J\left( \mathfrak{C}%
_{AR(1)}\right) =\limfunc{span}(e_{+})\cup \limfunc{span}(e_{-})$, see Lemma %
\ref{AR_1}, and because $e_{+},e_{-}\in \mathfrak{M}$ is assumed we conclude
that $J\left( \mathfrak{C}_{AR(1)}\right) \subseteq \mathfrak{M}$. The
assumption $R\hat{\beta}(e_{+})=R\hat{\beta}(e_{-})=0$ then implies that
even $J\left( \mathfrak{C}_{AR(1)}\right) \subseteq \mathfrak{M}_{0}-\mu
_{0} $ holds. The invariance condition (\ref{T_inv_wrt_z}) in Theorem \ref%
{TU_1} is thus satisfied, because $T$ is $G\left( \mathfrak{M}_{0}\right) $%
-invariant by Lemma \ref{LWT}. We next show that the additional condition in
Part 2 of Theorem \ref{TU_1} is satisfied. This is trivial in case the
Yule-Walker estimator is used (i.e., if $a_{1}=1$ and $a_{n}=n$) since then $%
\tilde{\Omega}\left( y\right) $ is positive definite for $y\notin
N_{2}^{\ast }(a_{1},a_{2})$ and $\hat{\Omega}\left( y\right) $ is positive
definite for $y\notin N_{0}^{\ast }(a_{1},a_{2})$ (see the proof of Lemma %
\ref{lem_GLS}) and since $N_{2}^{\ast }(a_{1},a_{2})$ and $N_{0}^{\ast
}(a_{1},a_{2})$ are $\lambda _{\mathbb{R}^{n}}$-null sets by Lemma \ref%
{lem_GLS}. If $a_{1}\neq 1$ or $a_{n}\neq n$, then $y_{0}$ constructed in
the proof of Lemma \ref{LDEF} satisfies $y_{0}\in \mathbb{R}^{n}\backslash
N_{2}^{\ast }(a_{1},a_{2})$ and $y_{0}\in \mathbb{R}^{n}\backslash
N_{0}^{\ast }(a_{1},a_{2})$ (cf. proof of Lemma \ref{lem_GLS}) as well as $%
\hat{\rho}(y_{0})=\hat{\rho}_{YW}(y_{0})$, implying that $\tilde{\Omega}%
(y_{0})$ as well as $\hat{\Omega}(y_{0})$ are positive definite. As shown in
Lemma \ref{lem_GLS}, the matrix $\tilde{\Omega}$ is, in particular,
continuous on the open set $\mathbb{R}^{n}\backslash N_{2}^{\ast
}(a_{1},a_{2})$ and the matrix $\hat{\Omega}$ is continuous on the open set $%
\mathbb{R}^{n}\backslash N_{0}^{\ast }(a_{1},a_{2})$. Consequently, $\tilde{%
\Omega}$ and $\hat{\Omega}$ are positive definite in a neighborhood of $%
y_{0} $ and thus the additional condition in Part 2 of Theorem \ref{TU_1} is
satisfied. Finally, the condition $a_{1}=1$ and $a_{2}=n$ implies that $%
\tilde{\Omega}$ and $\hat{\Omega}$ are $\lambda _{\mathbb{R}^{n}}$-almost
everywhere positive definite (since then $\hat{\rho}=\hat{\rho}_{YW}$),
verifying the extra condition in Part 3 of Theorem \ref{TU_1}. $\blacksquare 
$

\section{Appendix: Proofs for Section \protect\ref{Het} \label{App_B2}}

\textbf{Proof of Theorem \ref{thmhet}:} First observe that $\hat{\beta}$ and 
$\hat{\Omega}_{Het}$ satisfy Assumptions \ref{AE} and \ref{omega1} with $%
N=\emptyset $. In fact, $\hat{\Omega}_{Het}\left( y\right) $ is nonnegative
definite for every $y\in \mathbb{R}^{n}$, and is positive definite $\lambda
_{\mathbb{R}^{n}}$-almost everywhere under Assumption \ref{R_and_X} by Lemma %
\ref{Definiteness_het}. Furthermore, in view of this lemma and because $%
N=\emptyset $, the set $N^{\ast }$ in Corollary \ref{CW} is precisely the
set of $y$ for which $\limfunc{rank}\left( B(y)\right) <q$. It is trivial
that $\mathcal{Z}_{i}=$span$(e_{i}\left( n\right) )$ is a concentration
space of $\mathfrak{C}$ for every $i=1,\ldots ,n$. The theorem now follows
by applying Corollary \ref{CW} and Remark \ref{remnew}(i) to $\mathcal{Z}%
_{i} $ and by noting that $e_{i}\left( n\right) \in \mathbb{R}^{n}\backslash
N^{\ast }$ translates into $\limfunc{rank}\left( B(e_{i}\left( n\right)
)\right) =q$. Also note that the size of the test can not be zero in view of
Part 5 of Lemma \ref{LWT}. $\blacksquare $

\section{Appendix: Proofs for Subsection \protect\ref{sec_inv} \label{App_C}}

\textbf{Proof of Proposition \ref{max_inv}: }Since%
\begin{eqnarray*}
\Pi _{\left( \mathfrak{N}-\nu _{\ast }\right) ^{\bot }}\left( g_{\alpha ,\nu
,\nu ^{\prime }}(y)-\nu _{\ast }\right) &=&\alpha \Pi _{\left( \mathfrak{N}%
-\nu _{\ast }\right) ^{\bot }}\left( y-\nu \right) +\Pi _{\left( \mathfrak{N}%
-\nu _{\ast }\right) ^{\bot }}\left( \nu ^{\prime }-\nu _{\ast }\right) \\
&=&\alpha \Pi _{\left( \mathfrak{N}-\nu _{\ast }\right) ^{\bot }}\left(
y-\nu \right) =\alpha \Pi _{\left( \mathfrak{N}-\nu _{\ast }\right) ^{\bot
}}\left( y-\nu _{\ast }\right) ,
\end{eqnarray*}%
invariance of $h$ follows, and hence $h$ is constant on the orbits of $G(%
\mathfrak{N})$. Now suppose that $h(y)=h(y^{\prime })$. If $h(y)=h(y^{\prime
})=0$ holds, it follows that%
\begin{equation*}
\Pi _{\left( \mathfrak{N}-\nu _{\ast }\right) ^{\bot }}(y-\nu _{\ast })=\Pi
_{\left( \mathfrak{N}-\nu _{\ast }\right) ^{\bot }}(y^{\prime }-\nu _{\ast
})=0.
\end{equation*}%
Consequently, $y^{\prime }-y$ is of the form $\nu ^{\ast }-\nu _{\ast }$ for
some $\nu ^{\ast }\in \mathfrak{N}$. But this gives $y^{\prime }=(y-\nu
_{\ast })+\nu ^{\ast }=g_{1,\nu _{\ast },\nu ^{\ast }}(y)$, showing that $%
y^{\prime }$ is in the same orbit as $y$. Next consider the case where $%
h(y)=h(y^{\prime })\neq 0$. Then%
\begin{equation*}
\Pi _{\left( \mathfrak{N}-\nu _{\ast }\right) ^{\bot }}\left( \left\Vert \Pi
_{\left( \mathfrak{N}-\nu _{\ast }\right) ^{\bot }}(y^{\prime }-\nu _{\ast
})\right\Vert (y-\nu _{\ast })-c\left\Vert \Pi _{\left( \mathfrak{N}-\nu
_{\ast }\right) ^{\bot }}(y-\nu _{\ast })\right\Vert (y^{\prime }-\nu _{\ast
})\right) =0
\end{equation*}%
where $c=\pm 1$. It follows that the argument inside the projection operator
is of the form $\nu ^{\ast }-\nu _{\ast }$ for some $\nu ^{\ast }\in 
\mathfrak{N}$. Elementary calculations give%
\begin{equation*}
y^{\prime }=\frac{\left\Vert \Pi _{\left( \mathfrak{N}-\nu _{\ast }\right)
^{\bot }}(y^{\prime }-\nu _{\ast })\right\Vert }{c\left\Vert \Pi _{\left( 
\mathfrak{N}-\nu _{\ast }\right) ^{\bot }}(y-\nu _{\ast })\right\Vert }%
(y-\nu _{\ast })+\left( \nu _{\ast }+\frac{1}{c\left\Vert \Pi _{\left( 
\mathfrak{N}-\nu _{\ast }\right) ^{\bot }}(y-\nu _{\ast })\right\Vert }%
\left( \nu _{\ast }-\nu ^{\ast }\right) \right) \text{.}
\end{equation*}%
Since the last term in parenthesis on the right-hand side above is obviously
an element of $\mathfrak{N}$, we have obtained $y^{\prime }=g(y)$ for some $%
g\in G(\mathfrak{N})$, i.e., $y^{\prime }$ is in the same orbit as $y$. This
shows that $h$ is a maximal invariant. $\blacksquare $

\textbf{Proof of Proposition \ref{inv_rej_prob}: }(1) From (\ref%
{transf_formula}) and its extension discussed subsequently to (\ref%
{transf_formula}), as well as from the transformation theorem for integrals
we obtain%
\begin{equation*}
E_{\mu ,\sigma ^{2}\Phi }\left( \varphi \left( y\right) \right) =E_{\alpha
(\mu -\mu _{0})+\mu _{0}^{\prime },\alpha ^{2}\sigma ^{2}\Phi }\left(
\varphi \left( g_{\alpha ,\mu _{0},\mu _{0}^{\prime }}^{-1}(y)\right)
\right) .
\end{equation*}%
By almost invariance of $\varphi $ we have that $\varphi \left( y\right)
=\varphi \left( g_{\alpha ,\mu _{0},\mu _{0}^{\prime }}^{-1}(y)\right) $ for
all $y\in \mathbb{R}^{n}\backslash N$ with $\lambda _{\mathbb{R}^{n}}\left(
N\right) =0$ (where $N$ may depend on $g_{\alpha ,\mu _{0},\mu _{0}^{\prime
}}^{-1}$). Since $\Phi $ is positive definite, also $P_{\alpha (\mu -\mu
_{0})+\mu _{0}^{\prime },\alpha ^{2}\sigma ^{2}\Phi }(N)=0$ holds, and thus
the right-hand side of the above display equals $E_{\alpha (\mu -\mu
_{0})+\mu _{0}^{\prime },\alpha ^{2}\sigma ^{2}\Phi }\left( \varphi \left(
y\right) \right) $ which proves the first claim.

(2) Setting $\alpha =1$ in (\ref{power_inv}) shows that the rejection
probability is invariant under addition of elements that belong to $%
\mathfrak{M}_{0}-\mu _{0}$. Since $\mu =\Pi _{\left( \mathfrak{M}_{0}-\mu
_{0}\right) }\left( \mu -\mu _{0}\right) +\Pi _{\left( \mathfrak{M}_{0}-\mu
_{0}\right) ^{\bot }}\left( \mu -\mu _{0}\right) +\mu _{0}$ we thus conclude
that $E_{\mu ,\sigma ^{2}\Phi }(\varphi )=E_{\nu +\mu _{0},\sigma ^{2}\Phi
}(\varphi )$ where $\nu =\Pi _{\left( \mathfrak{M}_{0}-\mu _{0}\right)
^{\bot }}(\mu -\mu _{0})\in \mathfrak{M}$. Now applying (\ref{power_inv})
with $\alpha =\sigma ^{-1}$ and $\mu _{0}^{\prime }=\mu _{0}$ to $E_{\nu
+\mu _{0},\sigma ^{2}\Phi }(\varphi )$ establishes the first equality in (%
\ref{power_inv_2}). The second equality follows by the same argument by
setting $\alpha =\pm \sigma ^{-1}$, the sign equaling the sign of the first
non-zero component of $\nu $ if $\nu \neq 0$, and the choice of sign being
irrelevant if $\nu =0$.

(3) The first claim is an immediate consequence of (\ref{power_inv_2}). For
the second claim it suffices to show that $\mu -X\hat{\beta}_{rest}(\mu )$
(for $\mu \in \mathfrak{M}$) is an injective linear function of $R\beta -r$,
bijectivity of this mapping following from dimension considerations. To this
end note that%
\begin{eqnarray*}
\mu -X\hat{\beta}_{rest}(\mu ) &=&X\beta -X\left[ \hat{\beta}(\mu )-\left(
X^{\prime }X\right) ^{-1}R^{\prime }\left( R\left( X^{\prime }X\right)
^{-1}R^{\prime }\right) ^{-1}\left( R\hat{\beta}(\mu )-r\right) \right] \\
&=&X\beta -X\left[ \beta -\left( X^{\prime }X\right) ^{-1}R^{\prime }\left(
R\left( X^{\prime }X\right) ^{-1}R^{\prime }\right) ^{-1}\left( R\beta
-r\right) \right] \\
&=&\left( X^{\prime }X\right) ^{-1}R^{\prime }\left( R\left( X^{\prime
}X\right) ^{-1}R^{\prime }\right) ^{-1}\left( R\beta -r\right)
\end{eqnarray*}%
and that the matrix premultiplying $R\beta -r$ is of full column rank $q$.

(4) This follows similarly as in (1) observing that for invariant $\varphi $
the exceptional set $N$ is empty. $\blacksquare $

\textbf{Proof of Proposition \ref{prop_3}: }Set $\bar{h}(\mu ,\sigma
^{2})=\left\langle \Pi _{\left( \mathfrak{M}_{0}-\mu _{0}\right) ^{\bot
}}(\mu -\mu _{0})/\sigma \right\rangle $. The invariance of $\left( \bar{h}%
(\mu ,\sigma ^{2}),\Sigma \right) $ follows from a simple computation
similar to the one in the proof of Proposition \ref{max_inv}. Now assume
that $\left( \bar{h}(\mu ,\sigma ^{2}),\Sigma \right) =\left( \bar{h}(\mu
^{\prime },\sigma ^{\prime 2}),\Sigma ^{\prime }\right) $. We immediately
get $\bar{h}(\mu ,\sigma ^{2})=\bar{h}(\mu ^{\prime },\sigma ^{\prime 2})$
and $\Sigma =\Sigma ^{\prime }$. The former implies%
\begin{equation*}
\Pi _{\left( \mathfrak{M}_{0}-\mu _{0}\right) ^{\bot }}\left( \left( \mu
-\mu _{0}\right) -c\left( \sigma /\sigma ^{\prime }\right) \left( \mu
^{\prime }-\mu _{0}\right) \right) =0
\end{equation*}%
where $c=\pm 1$. Similar calculations as in the proof of Proposition \ref%
{max_inv} give%
\begin{equation*}
\mu ^{\prime }=c\left( \sigma ^{\prime }/\sigma \right) \left( \mu -\mu
^{\ast }\right) +\mu _{0}
\end{equation*}%
for some $\mu ^{\ast }\in \mathfrak{M}_{0}$. Together with $\Sigma =\Sigma
^{\prime }$ this shows that $(\mu ^{\prime },\sigma ^{\prime 2},\Sigma
^{\prime })$ is in the same orbit under the associated group as is $(\mu
,\sigma ^{2},\Sigma )$. $\blacksquare $

\section{Appendix: Proofs and Auxiliary Results for Subsections \protect\ref%
{sec_neg} and \protect\ref{sec_pos} \label{App_D}}

The next lemma is a simple consequence of a continuity property of the
characteristic function of a multivariate Gaussian probability measure and
of the Portmanteau theorem.

\begin{lemma}
\label{conc} Let $\Phi _{m}$ be a sequence of nonnegative definite symmetric 
$n\times n$ matrices converging to an $n\times n$ matrix $\Phi $ as $%
m\rightarrow \infty $, where $\Phi $ may be singular, and let $\mu _{m}\in 
\mathbb{R}^{n}$ be a sequence converging to $\mu \in \mathbb{R}^{n}$ as $%
m\rightarrow \infty $. Then $P_{\mu _{m},\Phi _{m}}$ converges weakly to $%
P_{\mu ,\Phi }$. If, in addition, $A\in \mathcal{B}(\mathbb{R}^{n})$
satisfies $\lambda _{\mu +\limfunc{span}(\Phi )}(\limfunc{bd}(A))=0$, then $%
P_{\mu _{m},\Phi _{m}}(A)\rightarrow P_{\mu ,\Phi }(A)$.
\end{lemma}

\textbf{Proof of Theorem \ref{inv}: }(1) Since $\mathcal{Z}$ is a
concentration space of $\mathfrak{C}$, there exists a sequence $(\Sigma
_{m})_{m\in \mathbb{N}}$ in $\mathfrak{C}$ converging to $\bar{\Sigma}$ such
that $\limfunc{span}(\bar{\Sigma})=\mathcal{Z}$. Note that $\mu _{0}+%
\mathcal{Z}$ is a $\lambda _{\mathbb{R}^{n}}$-null set because $\dim \left( 
\mathcal{Z}\right) <n$ in view of Definition \ref{CD}. Because $\Sigma _{m}$
is positive definite, we thus have%
\begin{equation*}
P_{\mu _{0},\sigma ^{2}\Sigma _{m}}(W)=P_{\mu _{0},\sigma ^{2}\Sigma
_{m}}(W\cup \left( \mu _{0}+\mathcal{Z}\right) )\text{.}
\end{equation*}%
By Lemma \ref{conc} we then have that $P_{\mu _{0},\sigma ^{2}\Sigma
_{m}}(W\cup \left( \mu _{0}+\mathcal{Z}\right) )$ converges to $P_{\mu
_{0},\sigma ^{2}\bar{\Sigma}}(W\cup \left( \mu _{0}+\mathcal{Z}\right) )$.
But the later probability is not less than $P_{\mu _{0},\sigma ^{2}\bar{%
\Sigma}}(\mu _{0}+\mathcal{Z})$ which equals $1$ since $P_{\mu _{0},\sigma
^{2}\bar{\Sigma}}$ is supported by $\mu _{0}+\mathcal{Z}$. To prove the
claim in parentheses observe that $T(\mu _{0}+z)>C$ and lower semicontinuity
of $T$ at $\mu _{0}+z$ implies that $T(w)>C$ holds for all $w$ in a
neighborhood of $\mu _{0}+z$; hence such points $\mu _{0}+z$ belong to $%
\limfunc{int}(W)\subseteq \limfunc{int}\left( W\cup \left( \mu _{0}+\mathcal{%
Z}\right) \right) $, and consequently do not belong to $\limfunc{bd}\left(
W\cup \left( \mu _{0}+\mathcal{Z}\right) \right) $. But this establishes (%
\ref{int_cond})$.$

(2) Apply the same argument as above to $\mathbb{R}^{n}\backslash W$. Also
note that $P_{\mu _{0},\sigma ^{2}\Sigma }(W)$ can be approximated
arbitrarily closely by $P_{\mu _{1},\sigma ^{2}\Sigma }(W)$ for suitable $%
\mu _{1}\in \mathfrak{M}_{1}$, since $\Vert {P_{\mu _{0},\sigma ^{2}\Sigma
}-P_{\mu _{1},\sigma ^{2}\Sigma }}\Vert _{TV}\rightarrow 0$ for $\mu
_{1}\rightarrow \mu _{0}$ holds by Scheff\'{e}'s Lemma as $\sigma ^{2}\Sigma 
$ is positive definite.

(3) Choose an arbitrary $\mu _{1}\in \mathfrak{M}_{1}$. By assumption we
have $\inf\limits_{\Sigma \in \mathfrak{C}}P_{\mu _{0},\sigma ^{2}\Sigma
}(W)=0$ for a suitable $\sigma ^{2}>0$. It hence suffices to show that for
every $\Sigma \in \mathfrak{C}$%
\begin{equation*}
P_{\mu _{0},\sigma ^{2}\Sigma }(W)-P_{\mu _{1},\sigma ^{2}\tau ^{2}\Sigma
}(W)\rightarrow 0
\end{equation*}%
holds for $\tau \rightarrow \infty $. By almost invariance of $W$ under $%
G\left( \left\{ \mu _{0}\right\} \right) $ we have that $W\bigtriangleup
\left( \tau W+(1-\tau )\mu _{0}\right) $ is a $\lambda _{\mathbb{R}^{n}}$%
-null set. Hence, by the reproductive property of the normal distribution 
\begin{equation*}
P_{\mu _{1},\sigma ^{2}\tau ^{2}\Sigma }(W)=P_{\mu _{1},\sigma ^{2}\tau
^{2}\Sigma }(\tau W+(1-\tau )\mu _{0})=P_{\mu _{0}+\tau ^{-1}(\mu _{1}-\mu
_{0}),\sigma ^{2}\Sigma }(W).
\end{equation*}%
But, since $\sigma ^{2}\Sigma $ is positive definite, we have by an
application of Scheff\'{e}'s Lemma 
\begin{equation*}
\Vert {P_{\mu _{0},\sigma ^{2}\Sigma }-P_{\mu _{0}+\tau ^{-1}(\mu _{1}-\mu
_{0}),\sigma ^{2}\Sigma }}\Vert _{TV}\rightarrow 0
\end{equation*}%
as $\tau \rightarrow \infty $, and hence ${P_{\mu _{0},\sigma ^{2}\Sigma
}(W)-P_{\mu _{0}+\tau ^{-1}(\mu _{1}-\mu _{0}),\sigma ^{2}\Sigma }}%
(W)\rightarrow 0$. The claim in parenthesis is obvious. $\blacksquare $

\begin{lemma}
\label{Lunif_1}Let $\varphi :\mathbb{R}^{n}\rightarrow \lbrack 0,1]$ be a
Borel-measurable function that is almost invariant under $G(\mathfrak{M}%
_{0}) $. Suppose $\Phi _{m}$ is a sequence of positive definite symmetric $%
n\times n$ matrices converging to a positive definite matrix $\Phi $,
suppose $\mu _{m}\in \mathfrak{M}$, and suppose the sequence $\sigma
_{m}^{2} $ satisfies $0<\sigma _{m}^{2}<\infty $. Then%
\begin{equation*}
\lim_{m\rightarrow \infty }E_{\mu _{m},\sigma _{m}^{2}\Phi _{m}}(\varphi
)=E_{\nu +\mu _{0},\Phi }(\varphi )
\end{equation*}%
provided $\nu _{m}^{\ast }=\Pi _{\left( \mathfrak{M}_{0}-\mu _{0}\right)
^{\bot }}(\mu _{m}-\mu _{0})/\sigma _{m}$ for some $\mu _{0}\in \mathfrak{M}%
_{0}$ converges to an element $\nu \in \mathbb{R}^{n}$ (which then
necessarily belongs to $\mathfrak{M}$). [Note that $\nu _{m}^{\ast }$, and
thus the result, does not depend on the choice of $\mu _{0}\in \mathfrak{M}%
_{0}$.]
\end{lemma}

\begin{proof}
By Proposition \ref{inv_rej_prob} we have that $E_{\mu _{m},\sigma
_{m}^{2}\Phi _{m}}(\varphi )=E_{\nu _{m}^{\ast }+\mu _{0},\Phi _{m}}(\varphi
)$. Since $\nu _{m}^{\ast }\rightarrow \nu $ and since $\Phi _{m}\rightarrow
\Phi $, with $\Phi $ positive definite, the result follows from total
variation distance convergence of $P_{\nu _{m}^{\ast }+\mu _{0},\Phi _{m}}$
to $P_{\nu +\mu _{0},\Phi }$.
\end{proof}

\begin{remark}
\label{rem_Lunif_1}(i) Consider the case where $\nu _{m}^{\ast }=\Pi
_{\left( \mathfrak{M}_{0}-\mu _{0}\right) ^{\bot }}(\mu _{m}-\mu
_{0})/\sigma _{m}$ does not converge. Then, as long as the sequence $\nu
_{m}^{\ast }$ is bounded, the above lemma can be applied by passing to
subsequences along which $\nu _{m}^{\ast }$ converges. In the case where the
sequence $\nu _{m}^{\ast }$ is unbounded, then, along subsequences such that
the norm of $\nu _{m}^{\ast }$ diverges, one would expect $E_{\mu
_{m},\sigma _{m}^{2}\Phi _{m}}(\varphi )=E_{\nu _{m}^{\ast }+\mu _{0},\Phi
_{m}}(\varphi )$ to converge to $1$ for any reasonable test since $\nu
_{m}^{\ast }+\mu _{0}$ moves farther and farther away from $\mathfrak{M}_{0}$
(and $\Phi _{m}$ stabilizes at a positive definite matrix). Indeed, such a
result can be shown for a large class of tests, see Lemma \ref{LWT}.

(ii) In the special case where $\mu _{m}\equiv \mu $ it is easy to see,
using Proposition \ref{inv_rej_prob}, that the limit in the above lemma is $%
E_{\mu ,\sigma ^{2}\Phi }(\varphi )$ if $\sigma _{m}^{2}\rightarrow \sigma
^{2}\in (0,\infty )$\ and $\mu \in \mathfrak{M}_{1}$, is $E_{\mu _{0},\Phi
}(\varphi )$ if $\sigma _{m}^{2}\rightarrow \infty $\ and $\mu \in \mathfrak{%
M}_{1}$, and is $E_{\mu ,\Phi }(\varphi )$ if $\mu \in \mathfrak{M}_{0}$.
\end{remark}

\begin{lemma}
\label{LUNIF}Let $\varphi :\mathbb{R}^{n}\rightarrow \lbrack 0,1]$ be a
Borel-measurable function that is almost invariant under $G(\mathfrak{M}%
_{0}) $. Suppose $\Phi _{m}$ is a sequence of positive definite symmetric $%
n\times n$ matrices converging to a singular matrix $\Phi $, suppose $\mu
_{m}\in \mathfrak{M}$, and $\sigma _{m}^{2}$ is a sequence satisfying $%
0<\sigma _{m}^{2}<\infty $. Assume further that $\varphi (x+z)=\varphi (x)$
holds for every $x\in \mathbb{R}^{n}$ and every $z\in \limfunc{span}(\Phi )$%
. Suppose that for some sequence of positive real numbers $s_{m}$ the matrix 
$D_{m}=\Pi _{\limfunc{span}(\Phi )^{\bot }}\Phi _{m}\Pi _{\limfunc{span}%
(\Phi )^{\bot }}/s_{m}$ converges to a matrix $D$, which is regular on the
orthogonal complement of $\limfunc{span}(\Phi )$. Then%
\begin{equation*}
\lim_{m\rightarrow \infty }E_{\mu _{m},\sigma _{m}^{2}\Phi _{m}}(\varphi
)=E_{\nu +\mu _{0},D+\Phi }(\varphi )=E_{\nu +\mu _{0},D}(\varphi )
\end{equation*}%
provided $\nu _{m}^{\ast \ast }=\Pi _{\left( \mathfrak{M}_{0}-\mu
_{0}\right) ^{\bot }}(\mu _{m}-\mu _{0})/\left( \sigma
_{m}s_{m}^{1/2}\right) $ for some $\mu _{0}\in \mathfrak{M}_{0}$ converges
to an element $\nu \in \mathbb{R}^{n}$ (which then necessarily belongs to $%
\mathfrak{M}$). [Note that $\nu _{m}^{\ast \ast }$, and thus the result,
does not depend on the choice of $\mu _{0}\in \mathfrak{M}_{0}$.]
Furthermore, the matrix $D+\Phi $ is positive definite.
\end{lemma}

\begin{proof}
Because $\Pi _{\limfunc{span}(\Phi )}\left( x-\mu _{m}\right) \in \limfunc{%
span}(\Phi )$, we obtain by the assumed invariance w.r.t. addition of $z\in 
\limfunc{span}(\Phi )$%
\begin{equation*}
\varphi (x)=\varphi (\mu _{m}+\Pi _{\limfunc{span}(\Phi )^{\bot }}\left(
x-\mu _{m}\right) +\Pi _{\limfunc{span}(\Phi )}\left( x-\mu _{m}\right)
)=\varphi (\mu _{m}+\Pi _{\limfunc{span}(\Phi )^{\bot }}\left( x-\mu
_{m}\right) )
\end{equation*}%
for every $x$. By the transformation theorem we then have on the one hand%
\begin{eqnarray}
E_{\mu _{m},\sigma _{m}^{2}\Phi _{m}}(\varphi (\cdot )) &=&E_{\mu
_{m},\sigma _{m}^{2}\Phi _{m}}(\varphi (\mu _{m}+\Pi _{\limfunc{span}(\Phi
)^{\bot }}\left( \cdot -\mu _{m}\right) ))=E_{\mu _{m},\sigma _{m}^{2}\Pi _{%
\limfunc{span}(\Phi )^{\bot }}\Phi _{m}\Pi _{\limfunc{span}(\Phi )^{\bot
}}}(\varphi (\cdot ))  \notag \\
&=&E_{\mu _{m},\sigma _{m}^{2}s_{m}D_{m}}(\varphi (\cdot )).  \label{project}
\end{eqnarray}%
On the other hand, by the same invariance property of $\varphi $ 
\begin{equation*}
E_{\mu _{m},\sigma _{m}^{2}s_{m}D_{m}}(\varphi (\cdot ))=E_{\mu _{m},\sigma
_{m}^{2}s_{m}D_{m}}(\varphi (\cdot +z))
\end{equation*}%
holds for every $z\in \limfunc{span}(\Phi )$. Integrating this w.r.t. a
normal distribution $P_{0,\sigma _{m}^{2}s_{m}\Phi }$ (in the variable $z$)
and using the reproductive property of the normal distribution gives%
\begin{equation}
E_{\mu _{m},\sigma _{m}^{2}s_{m}D_{m}}(\varphi (x))=E_{Q_{m}}(\varphi
(x+z))=E_{\mu _{m},\sigma _{m}^{2}s_{m}(D_{m}+\Phi )}(\varphi (x))
\label{augment}
\end{equation}%
where $Q_{m}$ denotes the product of the Gaussian measures $P_{\mu
_{m},\sigma _{m}^{2}s_{m}D_{m}}$ and $P_{0,\sigma _{m}^{2}s_{m}\Phi }$.
Observe that $D+\Phi $ as well as $D_{m}+\Phi $ are positive definite. An
application of Lemma \ref{Lunif_1} now gives

\begin{equation*}
\lim_{m\rightarrow \infty }E_{\mu _{m},\sigma _{m}^{2}s_{m}(D_{m}+\Phi
)}(\varphi )=E_{\nu +\mu _{0},D+\Phi }(\varphi ).
\end{equation*}%
The same argument that has led to (\ref{augment}) now shows that $E_{\nu
+\mu _{0},D+\Phi }(\varphi )=E_{\nu +\mu _{0},D}(\varphi )$. Combining this
with (\ref{project}) completes the proof of the display in the theorem. The
positive definiteness of $D+\Phi $ is obvious as noted earlier in the proof.
\end{proof}

\begin{remark}
\label{rem_LUNIF}(i)\ A remark similar to Remark \ref{rem_Lunif_1}(i) also
applies here. In particular, we typically can expect $E_{\mu _{m},\sigma
_{m}^{2}\Phi _{m}}(\varphi )$ to converge to $1$ in case the norm of $\nu
_{m}^{\ast \ast }$ diverges.

(ii) In the special case where $\mu _{m}\equiv \mu $ it is easy to see,
using Proposition \ref{inv_rej_prob}, that the limit in the above lemma is $%
E_{\mu ,\kappa (D+\Phi )}(\varphi )=E_{\mu ,\kappa D}(\varphi )$ if $\sigma
_{m}^{2}s_{m}\rightarrow \kappa \in (0,\infty )$\ and $\mu \in \mathfrak{M}%
_{1}$, is $E_{\mu _{0},D+\Phi }(\varphi )=E_{\mu _{0},D}(\varphi )$ if $%
\sigma _{m}^{2}s_{m}\rightarrow \infty $\ and $\mu \in \mathfrak{M}_{1}$,
and is $E_{\mu ,D+\Phi }(\varphi )=E_{\mu ,D}(\varphi )$ if $\mu \in 
\mathfrak{M}_{0}$.
\end{remark}

\begin{remark}
\label{rem_scaling_factors}(i) If $s_{m}$ and $s_{m}^{\ast }$ are two
positive scaling factors such that $\Pi _{\limfunc{span}(\Phi )^{\bot }}\Phi
_{m}\Pi _{\limfunc{span}(\Phi )^{\bot }}/s_{m}\rightarrow D$ and $\Pi _{%
\limfunc{span}(\Phi )^{\bot }}\Phi _{m}\Pi _{\limfunc{span}(\Phi )^{\bot
}}/s_{m}^{\ast }\rightarrow D^{\ast }$ with both $D$ and $D^{\ast }$ being
regular on the orthogonal complement of $\limfunc{span}(\Phi )$, then $%
s_{m}/s_{m}^{\ast }$ must converge to a positive finite number, i.e., the
scaling sequence is essentially uniquely determined.

(ii) Typical choices for $s_{m}$ are $s_{m}^{(1)}=\left\Vert \Pi _{\limfunc{%
span}(\Phi )^{\bot }}\Phi _{m}\Pi _{\limfunc{span}(\Phi )^{\bot
}}\right\Vert $ (for some choice of norm) or $s_{m}^{(2)}=\limfunc{tr}(\Pi _{%
\limfunc{span}(\Phi )^{\bot }}\Phi _{m}\Pi _{\limfunc{span}(\Phi )^{\bot }})$%
; note that $s_{m}^{(1)}$ as well as $s_{m}^{(2)}$ are positive, since $\Phi
_{m}$ is positive definite and $\Phi $ is singular. With both choices
convergence of $\Pi _{\limfunc{span}(\Phi )^{\bot }}\Phi _{m}\Pi _{\limfunc{%
span}(\Phi )^{\bot }}/s_{m}$ (at least along suitable subsequences) is
automatic. Furthermore, since for any choice of norm we have $%
c_{1}\left\Vert \Pi _{\limfunc{span}(\Phi )^{\bot }}\Phi _{m}\Pi _{\limfunc{%
span}(\Phi )^{\bot }}\right\Vert \leq \limfunc{tr}(\Pi _{\limfunc{span}(\Phi
)^{\bot }}\Phi _{m}\Pi _{\limfunc{span}(\Phi )^{\bot }})\leq c_{2}\left\Vert
\Pi _{\limfunc{span}(\Phi )^{\bot }}\Phi _{m}\Pi _{\limfunc{span}(\Phi
)^{\bot }}\right\Vert $ for suitable $0<c_{1}\leq c_{2}<\infty $, we have
convergence of $s_{m}^{(1)}/s_{m}^{(2)}$ to a positive finite number (at
least along suitable subsequences). Hence, which of the normalization
factors $s_{m}^{(i)}$ is used in an application of the above lemma,
typically does not make a difference.
\end{remark}

\textbf{Proof of Theorem \ref{TU}: }(1) By the invariance properties of the
rejection probability expressed in Proposition \ref{inv_rej_prob} it
suffices to show for an arbitrary fixed $\mu _{0}\in \mathfrak{M}_{0}$ that%
\begin{equation*}
\sup\limits_{\Sigma \in \mathfrak{C}}E_{\mu _{0},\Sigma }(\varphi )<1
\end{equation*}%
in order to establish the first claim in Part 1. To this end let $\Sigma
_{m}\in \mathfrak{C}$ be a sequence such that $E_{\mu _{0},\Sigma
_{m}}(\varphi )$ converges to $\sup_{\Sigma \in \mathfrak{C}}E_{\mu
_{0},\Sigma }(\varphi )$. Since $\mathfrak{C}$ is assumed to be bounded, we
may assume without loss of generality that $\Sigma _{m}$ converges to a
matrix $\bar{\Sigma}$ (not necessarily in $\mathfrak{C}$). If $\bar{\Sigma}$
is positive definite, it follows from Lemma \ref{Lunif_1} applied to $E_{\mu
_{0},\Sigma _{m}}(\varphi )$ that $\sup_{\Sigma \in \mathfrak{C}}E_{\mu
_{0},\Sigma }(\varphi )=E_{\mu _{0},\bar{\Sigma}}(\varphi )$ (since $\nu =0$%
). But $E_{\mu _{0},\bar{\Sigma}}(\varphi )$ is less than $1$ since $\varphi
\leq 1$ is not $\lambda _{\mathbb{R}^{n}}$-almost everywhere equal to $1$.
If $\bar{\Sigma}$ is singular, then in view of the assumptions of the
theorem we can pass to the subsequence $\Sigma _{m_{i}}$ and then apply
Lemma \ref{LUNIF} to $E_{\mu _{0},\Sigma _{m_{i}}}(\varphi )$ to obtain that 
$\sup_{\Sigma \in \mathfrak{C}}E_{\mu _{0},\Sigma }(\varphi )=E_{\mu _{0},D+%
\bar{\Sigma}}(\varphi )$ (since again $\nu =0$) for a matrix $D$ with the
properties as given in the theorem. But $E_{\mu _{0},D+\bar{\Sigma}}(\varphi
)$ is less than $1$, since $D+\bar{\Sigma}$ is positive definite (as noted
in Lemma \ref{LUNIF}) and since $\varphi \leq 1$ is not $\lambda _{\mathbb{R}%
^{n}}$-almost everywhere equal to $1$. This proves the first claim of Part 1
of the theorem. To prove the second claim in Part 1, observe that for the
same invariance reasons it suffices to show that for an arbitrary fixed $\mu
_{0}\in \mathfrak{M}_{0}$%
\begin{equation*}
\inf_{\Sigma \in \mathfrak{C}}E_{\mu _{0},\Sigma }(\varphi )>0
\end{equation*}%
holds. Now the same argument as before shows that this infimum either equals 
$E_{\mu _{0},\bar{\Sigma}}(\varphi )$ for some positive definite $\bar{\Sigma%
}$, or equals $E_{\mu _{0},D+\bar{\Sigma}}(\varphi )$ for some positive
definite $D+\bar{\Sigma}$. Since $\varphi \geq 0$, but $\varphi $ is not $%
\lambda _{\mathbb{R}^{n}}$-almost everywhere equal to $0$ by assumption, the
result follows.

(2) Let $\mu _{m}\in \mathfrak{M}_{1}$, $0<\sigma _{m}^{2}<\infty $, and $%
\Sigma _{m}\in \mathfrak{C}$ be sequences such that $E_{\mu _{m},\sigma
_{m}^{2}\Sigma _{m}}(\varphi )$ converges to $\inf_{\mu _{1}\in \mathfrak{M}%
_{1}}\inf_{\sigma ^{2}>0}\inf_{\Sigma \in \mathfrak{C}}E_{\mu _{1},\sigma
^{2}\Sigma }(\varphi )$. Since $\mathfrak{C}$ is assumed to be bounded, we
may assume without loss of generality that $\Sigma _{m}$ converges to a
matrix $\bar{\Sigma}$.

Consider first the case where $\bar{\Sigma}$ is positive definite: Set $\nu
_{m}^{\ast }=\Pi _{\left( \mathfrak{M}_{0}-\mu _{0}\right) ^{\bot }}(\mu
_{m}-\mu _{0})/\sigma _{m}$. If this sequence is bounded, we may pass to a
subsequence $m^{\prime }$ such that $\nu _{m^{\prime }}^{\ast }$ converges
to some $\nu $. Applying Lemma \ref{Lunif_1} then shows that $E_{\mu
_{m^{\prime }},\sigma _{m^{\prime }}^{2}\Sigma _{m^{\prime }}}(\varphi )$
converges to $E_{\nu +\mu _{0},\bar{\Sigma}}(\varphi )$, which is positive
since $\varphi \geq 0$ is not $\lambda _{\mathbb{R}^{n}}$-almost everywhere
equal to $0$ and since $\bar{\Sigma}$ is positive definite. If the sequence $%
\nu _{m}^{\ast }$ is unbounded, we may pass to a subsequence $m^{\prime }$
such that $\left\Vert \nu _{m^{\prime }}^{\ast }\right\Vert \rightarrow
\infty $. Since $E_{\mu _{m^{\prime }},\sigma _{m^{\prime }}^{2}\Sigma
_{m^{\prime }}}(\varphi )=E_{\nu _{m^{\prime }}^{\ast }+\mu _{0},\Sigma
_{m^{\prime }}}(\varphi )$ by\ Proposition \ref{inv_rej_prob}, it follows
from assumption (\ref{power_ass}) that $\lim_{m^{\prime }}E_{\mu _{m^{\prime
}},\sigma _{m^{\prime }}^{2}\Sigma _{m^{\prime }}}(\varphi )$ is positive.

Next consider the case where $\bar{\Sigma}$ is singular: Pass to the
subsequence $m_{i}$ mentioned in the theorem and set now $\nu _{m_{i}}^{\ast
\ast }=\Pi _{\left( \mathfrak{M}_{0}-\mu _{0}\right) ^{\bot }}(\mu
_{m_{i}}-\mu _{0})/\left( \sigma _{m_{i}}s_{m_{i}}^{1/2}\right) $. If this
sequence is bounded, we may pass to a subsequence $m_{i}^{\prime }$ of $%
m_{i} $ such that $\nu _{m_{i}^{\prime }}^{\ast \ast }$ converges to some $%
\nu $. Applying Lemma \ref{LUNIF} then shows that $E_{\mu _{m_{i}^{\prime
}},\sigma _{m_{i}^{\prime }}^{2}\Sigma _{m_{i}^{\prime }}}(\varphi )$
converges to $E_{\nu +\mu _{0},D+\bar{\Sigma}}(\varphi )$, which is positive
since $\varphi \geq 0$ is not $\lambda _{\mathbb{R}^{n}}$-almost everywhere
equal to $0$ and since $D+\bar{\Sigma}$ is positive definite. If the
sequence $\nu _{m_{i}}^{\ast \ast }$ is unbounded, we may pass to a
subsequence $m_{i}^{\prime }$ of $m_{i}$ such that $\left\Vert \nu
_{m_{i}^{\prime }}^{\ast \ast }\right\Vert \rightarrow \infty $. Since 
\begin{equation*}
E_{\mu _{m_{i}^{\prime }},\sigma _{m_{i}^{\prime }}^{2}\Sigma
_{m_{i}^{\prime }}}(\varphi )=E_{\mu _{m_{i}^{\prime }},\sigma
_{m_{i}^{\prime }}^{2}s_{m_{i}^{\prime }}\left( D_{m_{i}^{\prime }}+\bar{%
\Sigma}\right) }(\varphi )=E_{\nu _{m_{i}^{\prime }}^{\ast \ast }+\mu
_{0},D_{m_{i}^{\prime }}+\bar{\Sigma}}(\varphi )
\end{equation*}%
by (\ref{project}), (\ref{augment}), and Proposition \ref{inv_rej_prob}, it
follows from assumption (\ref{power_ass}) and positive definiteness of $D+%
\bar{\Sigma}$ that $\lim_{i\rightarrow \infty }E_{\mu _{m_{i}^{\prime
}},\sigma _{m_{i}^{\prime }}^{2}\Sigma _{m_{i}^{\prime }}}(\varphi )$ is
positive. Taken together the preceding arguments establish Part 2 of the
theorem.

(3) To prove the first claim of Part 3 of the theorem observe that we can
find $\mu _{m}\in \mathfrak{M}_{1}$ and $\sigma _{m}^{2}$ with $0<\sigma
_{m}^{2}<\infty $ with $d\left( \mu _{m},\mathfrak{M}_{0}\right) /\sigma
_{m}\geq c$ such that the expression left of the arrow in (\ref{unit_power})
differs from $E_{\mu _{m},\sigma _{m}^{2}\Sigma _{m}}(\varphi )$ only by a
sequence converging to zero. Let $m^{\prime }$ denote an arbitrary
subsequence. We can then find a further subsequence $m_{i}^{\prime }$ such
that the corresponding matrix $D_{m_{i}^{\prime }}$ satisfies the
assumptions of the theorem. Note that the sequence $s_{m_{i}^{\prime }}$
corresponding to $D_{m_{i}^{\prime }}$ necessarily converges to zero. But
then the norm of $\nu _{m_{i}^{\prime }}^{\ast \ast }$ defined above must
diverge since $d\left( \mu _{m_{i}^{\prime }},\mathfrak{M}_{0}\right)
/\sigma _{m}\geq c$ and since $\Pi _{\left( \mathfrak{M}_{0}-\mu _{0}\right)
^{\bot }}$ is the projection onto the orthogonal complement of $\mathfrak{M}%
_{0}-\mu _{0}$. Because%
\begin{equation*}
E_{\mu _{m_{i}^{\prime }},\sigma _{m_{i}^{\prime }}^{2}\Sigma
_{m_{i}^{\prime }}}(\varphi )=E_{\nu _{m_{i}^{\prime }}^{\ast \ast }+\mu
_{0},D_{m_{i}^{\prime }}+\bar{\Sigma}}(\varphi )
\end{equation*}%
in view of (\ref{project}), (\ref{augment}), and Proposition \ref%
{inv_rej_prob}, the result then follows from the assumption that the limit
inferior in (\ref{power_ass}) is equal to $1$, noting that $D_{m_{i}^{\prime
}}+\bar{\Sigma}$ is positive definite and converges to the positive definite
matrix $D+\bar{\Sigma}$.

We next prove the second claim in Part 3. Choose $\mu _{m}\in \mathfrak{M}%
_{1}$ with $d\left( \mu _{m},\mathfrak{M}_{0}\right) \geq c_{m}$ such that
the expression to the left of the arrow in (\ref{unit_power_2}) differs from 
$E_{\mu _{m},\sigma _{m}^{2}\Sigma _{m}}(\varphi )$ only by a sequence
converging to zero. Since 
\begin{equation*}
E_{\mu _{m},\sigma _{m}^{2}\Sigma _{m}}(\varphi )=E_{\nu _{m}^{\ast }+\mu
_{0},\Sigma _{m}}(\varphi )
\end{equation*}%
by Proposition \ref{inv_rej_prob} where $\nu _{m}^{\ast }$ was defined above
and since $\left\Vert \nu _{m}^{\ast }\right\Vert \geq c_{m}/\sigma
_{m}\rightarrow \infty $ clearly holds, the result follows from the
assumption that the limit inferior in (\ref{power_ass}) is equal to $1$.
[Note that we have not made use of condition (\ref{phi_inv_wrt_z}) and the
condition on $\mathfrak{C}$ following (\ref{phi_inv_wrt_z}).] $\blacksquare $

\textbf{Proof of Theorem \ref{TUU}: }By invariance properties of the
rejection probability (cf. Proposition \ref{inv_rej_prob}) it suffices to
show for the particular $\mu _{0}^{\ast }\in \mathfrak{M}_{0}$ appearing in (%
\ref{monotone})\ that for every $\delta $, $0<\delta <1$, there exists $%
k_{0}=k_{0}(\delta )$ such that%
\begin{equation}
\sup\limits_{\Sigma \in \mathfrak{C}}E_{\mu _{0}^{\ast },\Sigma }(\varphi
_{k_{0}})\leq \delta .  \label{equiv_res}
\end{equation}%
For this it suffices to show that $\sup\limits_{\Sigma \in \mathfrak{C}%
}E_{\mu _{0}^{\ast },\Sigma }(\varphi _{k})$ converges to zero for $%
k\rightarrow \infty $. Let $\Sigma _{k}\in \mathfrak{C}$ be a sequence such
that for all $k\geq 1$ 
\begin{equation}
\sup\limits_{\Sigma \in \mathfrak{C}}E_{\mu _{0}^{\ast },\Sigma }(\varphi
_{k})\leq E_{\mu _{0}^{\ast },\Sigma _{k}}(\varphi _{k})+k^{-1}\text{.}
\label{ineq}
\end{equation}%
Since $\mathfrak{C}$ is assumed to be bounded, we can find for every
subsequence $k^{\ast }$ a further subsubsequence $k^{\prime }$ such that $%
\Sigma _{k^{\prime }}$ converges to a matrix $\bar{\Sigma}$ (not necessarily
in $\mathfrak{C}$). Let $\varepsilon >0$ be given. We then distinguish two
cases:

Case 1: $\bar{\Sigma}$ is positive definite. By (\ref{monotone}) we can then
find a $k_{0}^{\prime }$ in the subsequence such that 
\begin{equation*}
E_{\mu _{0}^{\ast },\bar{\Sigma}}(\varphi _{k_{0}^{\prime }})<\varepsilon /2
\end{equation*}%
holds. But then by (\ref{ineq})\ and by the monotonicity expressed in (\ref%
{monotone}) 
\begin{equation}
\sup\limits_{\Sigma \in \mathfrak{C}}E_{\mu _{0}^{\ast },\Sigma }(\varphi
_{k^{\prime }})\leq E_{\mu _{0}^{\ast },\Sigma _{k^{\prime }}}(\varphi
_{k^{\prime }})+k^{\prime -1}\leq E_{\mu _{0}^{\ast },\Sigma _{k^{\prime
}}}(\varphi _{k_{0}^{\prime }})+k^{\prime -1}  \label{upper_bound}
\end{equation}%
holds for all $k^{\prime }\geq k_{0}^{\prime }$. Now Lemma \ref{Lunif_1}
together with Remark \ref{rem_Lunif_1}(ii) may clearly be applied to the
subsequence $k^{\prime }$, showing that $E_{\mu _{0}^{\ast },\Sigma
_{k^{\prime }}}(\varphi _{k_{0}^{\prime }})$ converges to $E_{\mu _{0}^{\ast
},\bar{\Sigma}}(\varphi _{k_{0}^{\prime }})<\varepsilon /2$. But this shows
that 
\begin{equation}
\limsup_{k^{\prime }\rightarrow \infty }\sup\limits_{\Sigma \in \mathfrak{C}%
}E_{\mu _{0}^{\ast },\Sigma }(\varphi _{k^{\prime }})<\varepsilon \text{.}
\label{almost_result}
\end{equation}

Case 2: $\bar{\Sigma}$ is singular. Then we can find a subsequence $%
k_{i}^{\prime }$ of $k^{\prime }$ and normalization constants $%
s_{k_{i}^{\prime }}$ such that the resulting matrices $D_{k_{i}^{\prime }}$
converge to a matrix $D$ with the properties specified in Theorem \ref{TU}.
Because $D+\bar{\Sigma}$ is positive definite, we can in view of (\ref%
{monotone}) find a $k_{i(0)}^{\prime }$ in the subsequence $k_{i}^{\prime }$
such that 
\begin{equation*}
E_{\mu _{0}^{\ast },D+\bar{\Sigma}}(\varphi _{k_{i(0)}^{\prime
}})<\varepsilon /2\text{.}
\end{equation*}%
Now applying Lemma \ref{LUNIF} together with Remark \ref{rem_LUNIF}(ii) to
the subsequence $k_{i}^{\prime }$ shows that $E_{\mu _{0}^{\ast },\Sigma
_{k_{i}^{\prime }}}(\varphi _{k_{i(0)}^{\prime }})$ converges to $E_{\mu
_{0}^{\ast },D+\bar{\Sigma}}(\varphi _{k_{i(0)}^{\prime }})<\varepsilon /2$.
But by (\ref{ineq})\ and (\ref{monotone}) 
\begin{equation*}
\sup\limits_{\Sigma \in \mathfrak{C}}E_{\mu _{0}^{\ast },\Sigma }(\varphi
_{k_{i}^{\prime }})\leq E_{\mu _{0}^{\ast },\Sigma _{k_{i}^{\prime
}}}(\varphi _{k_{i}^{\prime }})+k_{i}^{\prime -1}\leq E_{\mu _{0}^{\ast
},\Sigma _{k_{i}^{\prime }}}(\varphi _{k_{i(0)}^{\prime }})+k_{i}^{\prime -1}
\end{equation*}%
holds for all $i\geq i(0)$. This shows that 
\begin{equation*}
\limsup_{i\rightarrow \infty }\sup\limits_{\Sigma \in \mathfrak{C}}E_{\mu
_{0}^{\ast },\Sigma }(\varphi _{k_{i}^{\prime }})<\varepsilon \text{.}
\end{equation*}%
Taken together we have shown that $\sup\limits_{\Sigma \in \mathfrak{C}%
}E_{\mu _{0}^{\ast },\Sigma }(\varphi _{k})$ must converge to zero along the
original sequence $k$ which proves (\ref{equiv_res}). $\blacksquare $

\section{Appendix: Proofs and Auxiliary Results for Subsection \protect\ref%
{Impclass} \label{App_E}}

\begin{lemma}
\label{aux}Suppose Assumption \ref{AE} holds. Then the sets 
\begin{equation*}
A_{1}=\left\{ y\in \mathbb{R}^{n}\backslash N:\det {\check{\Omega}(y)}%
=0\right\} \text{ \ \ and \ \ }A_{2}=\left\{ y\in \mathbb{R}^{n}\backslash
N:\det {\check{\Omega}(y)}\neq 0\right\}
\end{equation*}%
are invariant under $G(\mathfrak{M})$, the former set being closed in the
relative topology on $\mathbb{R}^{n}\backslash N$. The set 
\begin{equation*}
N^{\ast }=N\cup \left\{ y\in \mathbb{R}^{n}\backslash N:\det {\check{\Omega}%
(y)}=0\right\}
\end{equation*}%
is a closed $\lambda _{\mathbb{R}^{n}}$-null set in $\mathbb{R}^{n}$ that is
invariant under $G(\mathfrak{M})$.
\end{lemma}

\begin{proof}
The invariance of $A_{1}$ and $A_{2}$ follows immediately from the
invariance of $\mathbb{R}^{n}\backslash N$ and the equivariance of ${\check{%
\Omega}(y)}$. The relative closedness of $A_{1}$ is an immediate consequence
of the continuity of ${\check{\Omega}(y)}$ on $\mathbb{R}^{n}\backslash N$.
The invariance of $N^{\ast }$ follows from invariance of $N$ discussed after
Assumption \ref{AE} and the just established invariance of $A_{1}$. Because $%
N$ is a $\lambda _{\mathbb{R}^{n}}$-null set and because ${\check{\Omega}(y)}
$ is $\lambda _{\mathbb{R}^{n}}$-almost everywhere nonsingular on $\mathbb{R}%
^{n}\backslash N$, it follows that $N^{\ast }$ is a $\lambda _{\mathbb{R}%
^{n}}$-null set. Finally, we establish closedness of $N^{\ast }$: let $%
y_{i}\in N^{\ast }$ be a sequence with limit $y_{0}$. If $y_{0}\in N$, we
are done. If $y_{0}\in \mathbb{R}^{n}\backslash N$, by openness of this set
also $y_{i}\in \mathbb{R}^{n}\backslash N$ for all but finitely many $i$
must hold and thus $\det {\check{\Omega}(y}_{i}{)=0}$. But then continuity
of ${\check{\Omega}}$ on $\mathbb{R}^{n}\backslash N$ implies $\det {\check{%
\Omega}(y}_{0}{)}=0$, and hence $y_{0}\in N^{\ast }$.
\end{proof}

\textbf{Proof of Lemma \ref{LWT}: }(1) Follows from the discussion preceding
the lemma and Lemma \ref{aux}.

(2) Follows immediately from the observation that $T$ coincides on the open
set $\mathbb{R}^{n}\backslash N^{\ast }$ with $(R\check{\beta}(y)-r)^{\prime
}\check{\Omega}^{-1}(y)(R\check{\beta}(y)-r)$ which is continuous on this
set by Assumption \ref{AE}.

(3) Since $N^{\ast }$ is invariant under the elements of $G(\mathfrak{M})$,
it is in particular invariant under $G(\mathfrak{M}_{0})$. The result $%
T\left( g\left( y\right) \right) =T\left( y\right) =0$ for $g\in G(\mathfrak{%
M}_{0})$ then follows trivially for $y\in N^{\ast }$. Now suppose $y\in 
\mathbb{R}^{n}\backslash N^{\ast }$. Then also $g_{\alpha ,\mu
_{0}^{(1)},\mu _{0}^{(2)}}(y)\in \mathbb{R}^{n}\backslash N^{\ast }$ for $%
\alpha \neq 0$, $\mu _{0}^{(i)}\in \mathfrak{M}_{0}$ ($i=1,2$) by invariance
of $\mathbb{R}^{n}\backslash N^{\ast }$. The invariance of $T$ then follows
immediately from the equivariance properties of $\check{\beta}$ and $\check{%
\Omega}$ expressed in Assumption \ref{AE}, using that $\mu _{0}^{(i)}\in 
\mathfrak{M}_{0}$ implies $R\gamma ^{(i)}=r$ for uniquely defined vectors $%
\gamma ^{(i)}$ satisfying $\mu _{0}^{(i)}=X\gamma ^{(i)}$.

(4) Set $O=\left\{ y\in \mathbb{R}^{n}:T(y)=C\right\} $ and note that $%
O\subseteq \mathbb{R}^{n}\backslash N^{\ast }$ since $C>0$ by assumption. We
can then write%
\begin{equation*}
O=\bigcup_{y_{2}\in \mathfrak{M}^{\bot }}\left( \left\{ y_{1}\in \mathfrak{M}%
:y_{1}+y_{2}\in \mathbb{R}^{n}\backslash N^{\ast },T(y_{1}+y_{2})=C\right\}
+y_{2}\right) =\bigcup_{y_{2}\in \mathfrak{M}^{\bot }}\left(
O(y_{2})+y_{2}\right) .
\end{equation*}%
Note that $O$ as well as $O(y_{2})$ are clearly measurable sets. By the
already established invariance of $\mathbb{R}^{n}\backslash N^{\ast }$, the
fact that $\mathbb{R}^{n}\backslash N^{\ast }\subseteq \mathbb{R}%
^{n}\backslash N$, and by the equivariance properties of $\check{\beta}$ and 
$\check{\Omega}$ maintained in Assumption \ref{AE}, the set $O(y_{2})$
equals 
\begin{equation*}
\left\{ y_{1}\in \mathfrak{M}:\left( R\left( \check{\beta}(y_{2})+\left(
X^{\prime }X\right) ^{-1}X^{\prime }y_{1}\right) -r\right) ^{\prime }\check{%
\Omega}^{-1}(y_{2})\left( R\left( \check{\beta}(y_{2})+\left( X^{\prime
}X\right) ^{-1}X^{\prime }y_{1}\right) -r\right) =C\right\}
\end{equation*}%
if $y_{2}\in (\mathbb{R}^{n}\backslash N^{\ast })\cap \mathfrak{M}^{\bot }$,
and it is empty if $y_{2}\in N^{\ast }\cap \mathfrak{M}^{\bot }$ (since $C>0$%
). If $y_{2}\in (\mathbb{R}^{n}\backslash N^{\ast })\cap \mathfrak{M}^{\bot
} $, the set $O(y_{2})\subseteq \mathfrak{M}$ is the image of 
\begin{equation*}
\bar{O}(y_{2})=\left\{ \gamma \in \mathbb{R}^{k}:\left( R\left( \check{\beta}%
(y_{2})+\gamma \right) -r\right) ^{\prime }\check{\Omega}^{-1}(y_{2})\left(
R\left( \check{\beta}(y_{2})+\gamma \right) -r\right) =C\right\}
\end{equation*}%
under the invertible linear map $\gamma \mapsto X\gamma $ from $\mathbb{R}%
^{k}$ onto $\mathfrak{M}$. Now $\bar{O}(y_{2})$ is the zero-set of a
multivariate real polynomial (in the components of $\gamma $). The
polynomial does not vanish everywhere on $\mathbb{R}^{k}$ because the
quadratic form making up the polynomial is unbounded on $\mathbb{R}^{k}$
(because $\check{\Omega}^{-1}(y_{2})$ is symmetric and well-defined if $%
y_{2}\in \mathbb{R}^{n}\backslash N^{\ast }$ and because $\limfunc{rank}%
(R)=q $ holds). Consequently, $\bar{O}(y_{2})$ has $k$-dimensional Lebesgue
measure zero and hence $\lambda _{\mathfrak{M}}(O(y_{2}))=0$ for every $%
y_{2}\in (\mathbb{R}^{n}\backslash N^{\ast })\cap \mathfrak{M}^{\bot }$. We
conclude that $\lambda _{\mathfrak{M}}(O(y_{2}))=0$ for every $y_{2}\in 
\mathfrak{M}^{\bot }$.

We now identify $\mathbb{R}^{n}$ with $\mathfrak{M}\times \mathfrak{M}^{\bot
}$ and view Lebesgue measure $\lambda _{\mathbb{R}^{n}}$ on $\mathbb{R}^{n}$
as $\lambda _{\mathfrak{M}}\otimes \lambda _{\mathfrak{M}^{\bot }}$. Hence, $%
y$ is identified with $\left( y_{1},y_{2}\right) \in \mathfrak{M}\times 
\mathfrak{M}^{\bot }$ satisfying $y=y_{1}+y_{2}$. Fubini's Theorem then shows%
\begin{eqnarray*}
\lambda _{\mathbb{R}^{n}}(O) &=&\lambda _{\mathfrak{M}\times \mathfrak{M}%
^{\bot }}(O)=\int\limits_{\mathfrak{M}\times \mathfrak{M}^{\bot }}\mathbf{1}%
_{O}((y_{1},y_{2}))d\lambda _{\mathfrak{M}\times \mathfrak{M}^{\bot
}}(y_{1},y_{2}) \\
&=&\int\limits_{\mathfrak{M}^{\bot }}\int\limits_{\mathfrak{M}}\mathbf{1}%
_{O(y_{2})}(y_{1})d\lambda _{\mathfrak{M}}(y_{1})d\lambda _{\mathfrak{M}%
^{\bot }}(y_{2})=\int\limits_{\mathfrak{M}^{\bot }}\lambda _{\mathfrak{M}%
}(O(y_{2}))d\lambda _{\mathfrak{M}^{\bot }}(y_{2})=0.
\end{eqnarray*}

(5\&6) First observe that $\left\{ y\in \mathbb{R}^{n}\backslash N^{\ast
}:T(y)>C\right\} =\left\{ y\in \mathbb{R}^{n}:T(y)>C\right\} $ holds in view
of $C>0$ and the definition of $T$. By continuity of $T$ on $\mathbb{R}%
^{n}\backslash N^{\ast }$ established in Part 2 and by openness of $\mathbb{R%
}^{n}\backslash N^{\ast }$, the openness of $\left\{ y\in \mathbb{R}%
^{n}\backslash N^{\ast }:T(y)>C\right\} $ and $\left\{ y\in \mathbb{R}%
^{n}\backslash N^{\ast }:T(y)<C\right\} $ follows. It hence suffices to show
that these two sets are non-empty: Choose an arbitrary $y\in \mathbb{R}%
^{n}\backslash N^{\ast }$ and set $y(\gamma )=y+X\gamma $ for $\gamma \in 
\mathbb{R}^{k}$. Then $y(\gamma )\in \mathbb{R}^{n}\backslash N^{\ast }$ by
invariance of $\mathbb{R}^{n}\backslash N^{\ast }$ under $G(\mathfrak{M})$.
Now by the equivariance properties of $\check{\beta}$ and $\check{\Omega}$
expressed in Assumption \ref{AE}%
\begin{equation*}
T(y(\gamma ))=\left( R\gamma +R\check{\beta}(y)-r\right) ^{\prime }\check{%
\Omega}^{-1}(y)\left( R\gamma +R\check{\beta}(y)-r\right) .
\end{equation*}%
Define $\bar{\gamma}=\bar{\beta}-\check{\beta}(y)$ for some $\bar{\beta}$
satisfying $R\bar{\beta}=r$. Then $T(y(\bar{\gamma}))=0<C$ holds showing
that $\left\{ y\in \mathbb{R}^{n}\backslash N^{\ast }:T(y)<C\right\} $ is
non-empty. Finally choose $y\in \mathbb{R}^{n}\backslash N^{\ast }$ and $v$
as in Assumption \ref{omega1}. Choose $\delta $ such that $v=R\delta $. Then
set $\gamma =c\delta +\bar{\beta}-\check{\beta}(y)$ where $\bar{\beta}$ is
as before and $c$ is a real number. Observe that then $T(y(\gamma
))=c^{2}v^{\prime }\check{\Omega}^{-1}(y)v$. Choosing $c$ sufficiently large
shows that $T(y(\gamma ))>C$ can be achieved, establishing that $\left\{
y\in \mathbb{R}^{n}\backslash N^{\ast }:T(y)>C\right\} $ is non-empty.

(7) Let $\mathbf{G}$ be a standard normal $n\times 1$ random vector. Then 
\begin{equation}
P_{\nu _{m}+\mu _{0},\Phi _{m}}(W(C))=\Pr \left( T(\nu _{m}+\mu _{0}+\Phi
_{m}^{1/2}\mathbf{G})-C\geq 0\right) .  \label{prob_w}
\end{equation}%
Set $\gamma _{m}=\left( X^{\prime }X\right) ^{-1}X^{\prime }\nu _{m}$ and $%
\gamma _{0}=\left( X^{\prime }X\right) ^{-1}X^{\prime }\mu _{0}$. Observe
that $R\gamma _{0}=r$ while $\left\Vert R\gamma _{m}\right\Vert \rightarrow
\infty $ as $m\rightarrow \infty $ in view of $\nu _{m}\in \Pi _{\left( 
\mathfrak{M}_{0}-\mu _{0}\right) ^{\bot }}(\mathfrak{M}_{1}-\mu _{0})$ and $%
\left\Vert \nu _{m}\right\Vert \rightarrow \infty $. For $\Phi _{m}^{1/2}%
\mathbf{G}\in \mathbb{R}^{n}\backslash N^{\ast }$ (an event which has
probability $1$ because $N^{\ast }$ is a $\lambda _{\mathbb{R}^{n}}$-null
set and $\Phi _{m}$ is positive-definite) we may use equivariance of $\check{%
\beta}$ and $\check{\Omega}$ and obtain that $T(\nu _{m}+\mu _{0}+\Phi
_{m}^{1/2}\mathbf{G})-C$ coincides on this event with%
\begin{equation}
({R\gamma }_{m}{+R\check{\beta}(\Phi _{m}^{1/2}\mathbf{G})})^{\prime }\check{%
\Omega}^{-1}(\Phi _{m}^{1/2}\mathbf{G})({{R\gamma }_{m}+R\check{\beta}(\Phi
_{m}^{1/2}\mathbf{G})})-C.  \label{quadratic}
\end{equation}%
Observe that $\Phi _{m}^{1/2}\mathbf{G\rightarrow }\Phi ^{1/2}\mathbf{G}$ as 
$m\rightarrow \infty $ with probability $1$. Furthermore, $\check{\beta}$
and $\check{\Omega}^{-1}$ are continuous on $\mathbb{R}^{n}\backslash
N^{\ast }$, a set which has probability $1$ under the law of $\Phi ^{1/2}%
\mathbf{G}$ (since $N^{\ast }$ is a $\lambda _{\mathbb{R}^{n}}$-null set and 
$\Phi $ is positive-definite). From the continuous mapping theorem we
conclude that ${R\check{\beta}(\Phi _{m}^{1/2}\mathbf{G})}$ and $\check{%
\Omega}^{-1}(\Phi _{m}^{1/2}\mathbf{G})$ converge almost surely to ${R\check{%
\beta}(\Phi ^{1/2}\mathbf{G})}$ and $\check{\Omega}^{-1}(\Phi ^{1/2}\mathbf{G%
})$, respectively. Now let $v\in A(\left( \nu _{m}\right) _{m\geq 1})$ and
let $m_{i}$ be a subsequence such that $\left\Vert {R\gamma }%
_{m_{i}}\right\Vert ^{-1}{R\gamma }_{m_{i}}\rightarrow v$. It follows that%
\begin{equation*}
\left[ ({R\gamma }_{m_{i}}{+R\check{\beta}(\Phi _{m_{i}}^{1/2}\mathbf{G})}%
)^{\prime }\check{\Omega}^{-1}(\Phi _{m_{i}}^{1/2}\mathbf{G})({{R\gamma }%
_{m_{i}}+R\check{\beta}(\Phi _{m_{i}}^{1/2}\mathbf{G})})-C\right]
/\left\Vert {R\gamma }_{m_{i}}\right\Vert ^{2}
\end{equation*}%
converges to%
\begin{equation*}
{v}^{\prime }\check{\Omega}^{-1}(\Phi ^{1/2}\mathbf{G})v
\end{equation*}%
with probability $1$. Since $\Pr \left( {v}^{\prime }\check{\Omega}%
^{-1}(\Phi ^{1/2}\mathbf{G})v=0\right) $ by Assumption \ref{omega2}, it
follows that%
\begin{equation*}
\Pr \left( T(\nu _{m_{i}}+\mu _{0}+\Phi _{m_{i}}^{1/2}\mathbf{G})-C\geq
0\right) \rightarrow \Pr \left( {v}^{\prime }\check{\Omega}^{-1}(\Phi ^{1/2}%
\mathbf{G})v\geq 0\right) .
\end{equation*}%
This shows that%
\begin{eqnarray*}
\liminf_{m\rightarrow \infty }P_{\nu _{m}+\mu _{0},\Phi _{m}}(W(C)) &\leq
&\liminf_{i\rightarrow \infty }P_{\nu _{m_{i}}+\mu _{0},\Phi _{m_{i}}}(W(C))
\\
&=&\Pr \left( {v}^{\prime }\check{\Omega}^{-1}(\Phi ^{1/2}\mathbf{G})v\geq
0\right) ,
\end{eqnarray*}%
implying that%
\begin{equation*}
\liminf_{m\rightarrow \infty }P_{\nu _{m}+\mu _{0},\Phi _{m}}(W(C))\leq
\inf_{v\in A(\left( \nu _{m}\right) _{m\geq 1})}\Pr \left( {v}^{\prime }%
\check{\Omega}^{-1}(\Phi ^{1/2}\mathbf{G})v\geq 0\right) .
\end{equation*}%
Conversely, let $m_{i}$ be a subsequence such that 
\begin{equation*}
P_{\nu _{m_{i}}+\mu _{0},\Phi _{m_{i}}}(W(C))\rightarrow
\liminf_{m\rightarrow \infty }P_{\nu _{m}+\mu _{0},\Phi _{m}}(W(C)).
\end{equation*}%
Since the unit ball in $\mathbb{R}^{q}$ is compact, we may assume that $%
\left\Vert {R\gamma }_{m_{i}(j)}\right\Vert ^{-1}{R\gamma }_{m_{i(j)}}$
converges to some $v\in A(\left( \nu _{m}\right) _{m\geq 1})$ along a
suitable subsequence $m_{i(j)}$ . The same arguments as above then show that%
\begin{eqnarray*}
\liminf_{m\rightarrow \infty }P_{\nu _{m}+\mu _{0},\Phi _{m}}(W(C))
&=&\liminf_{j\rightarrow \infty }P_{\nu _{m_{i(j)}}+\mu _{0},\Phi
_{m_{i(j)}}}(W(C)) \\
&=&\Pr \left( {v}^{\prime }\check{\Omega}^{-1}(\Phi ^{1/2}\mathbf{G})v\geq
0\right) \\
&\geq &\inf_{v\in A(\left( \nu _{m}\right) _{m\geq 1})}\Pr \left( {v}%
^{\prime }\check{\Omega}^{-1}(\Phi ^{1/2}\mathbf{G})v\geq 0\right) .
\end{eqnarray*}%
Given Assumption \ref{omega2}, the remaining equalities and inequalities in (%
\ref{power_far_away}) and (\ref{lower_bound_power_far_away}) are now
obvious. $\blacksquare $

\textbf{Proof of Corollary \ref{CW}: }(1) If $z\in \mathbb{R}^{n}\backslash
N^{\ast }$ then $\mu _{0}+z\in \mathbb{R}^{n}\backslash N^{\ast }$ and $T$
is continuous at $\mu _{0}+z$ for every $\mu _{0}\in \mathfrak{M}_{0}$ by
Parts 1 and 2 of Lemma \ref{LWT}. If $T(\mu _{0}^{\ast }+z)>C$ holds, then
by the invariance of $T$ established in Part 3 of Lemma \ref{LWT}, we have $%
T(\mu _{0}+z)=T(\mu _{0}^{\ast }+z)>C$ for every $\mu _{0}\in \mathfrak{M}%
_{0}$. Hence the sufficient conditions in Part 1 of Theorem \ref{inv} are
satisfied and an application of this theorem delivers the result.

(2) Completely analogous to the proof of (1) noting that the invariance of $%
T $ required in Part 3 of Theorem \ref{inv} is clearly satisfied.

(3) Since $N^{\ast }$ is a $\lambda _{\mathbb{R}^{n}}$-null set the test
statistic $T$ is $\lambda _{\mathbb{R}^{n}}$-almost everywhere equal to the
test statistic 
\begin{equation*}
T^{\ast }(y)=%
\begin{cases}
T(y) & y\in \mathbb{R}^{n}\backslash N^{\ast }, \\ 
\infty , & y\in N^{\ast }\text{.}%
\end{cases}%
\end{equation*}%
We verify that the sufficient conditions in Part 1 of Theorem \ref{inv} are
satisfied for $T^{\ast }$. To that end fix $\mu _{0}\in \mathfrak{M}_{0}$
and let $\mathcal{Z}^{\prime }\subseteq \mathcal{Z}$ denote the set of all $%
z $ such that $z\in \mathbb{R}^{n}\backslash N$, $\check{\Omega}(z)=0$, and $%
R\check{\beta}(z)\neq 0$ hold. By invariance of $N$ (cf. discussion after
Assumption \ref{AE}) and equivariance of $\check{\Omega}$ we see that $z\in 
\mathcal{Z}^{\prime }$ implies $\mu _{0}+z\in \mathbb{R}^{n}\backslash N$
and $\check{\Omega}(\mu _{0}+z)=0$, and thus $T^{\ast }(\mu _{0}+z)=\infty
>C $ holds for every $z\in \mathcal{Z}^{\prime }$ by definition of $T^{\ast }
$. We next show that $T^{\ast }$ is lower semicontinuous at $\mu _{0}+z$ for
every $z\in \mathcal{Z}^{\prime }$. Let $y_{m}$ be a sequence converging to $%
\mu _{0}+z$. Since $\mathbb{R}^{n}\backslash N$ is open, we may assume that
this sequence entirely belongs to $\mathbb{R}^{n}\backslash N$. If $\det ${$%
\check{\Omega}$}${(y_{m})}=0$ eventually holds, we are done since then $%
T^{\ast }(y_{m})=\infty $ eventually by construction. By a standard
subsequence argument we may thus assume that $\det ${$\check{\Omega}$}${%
(y_{m})}>0$ eventually holds since $\check{\Omega}$ is nonnegative definite
on $\mathbb{R}^{n}\backslash N$ by assumption. Now note that then%
\begin{equation*}
T^{\ast }(y_{m})=T(y_{m})=(R\check{\beta}(y_{m})-r)^{\prime }\check{\Omega}%
^{-1}(y_{m})(R\check{\beta}(y_{m})-r)\geq \lambda _{max}^{-1}(\check{\Omega}%
(y_{m}))\Vert {R\check{\beta}(y_{m})-r}\Vert ^{2}.
\end{equation*}%
Since $\check{\beta}$ is continuous on $\mathbb{R}^{n}\backslash N$ by
assumption, we have $R\check{\beta}(y_{m})\rightarrow R\check{\beta}(\mu
_{0}+z)=R\check{\beta}(z)+r\neq r$ where we have made use of equivariance of 
$\check{\beta}(z)$ and of $\mu _{0}\in \mathfrak{M}_{0}$. Hence $\Vert {R}${$%
\check{\beta}$}${(y_{m})-r}\Vert \rightarrow \Vert {R}${$\check{\beta}$}${(z)%
}\Vert >0$. Furthermore, $\check{\Omega}$ is continuous on $\mathbb{R}%
^{n}\backslash N$ by assumption, hence $\check{\Omega}(y_{m})\rightarrow 
\check{\Omega}(\mu _{0}+z)=0$. Consequently, $T^{\ast }(y_{m})\rightarrow
\infty $, establishing lower semicontinuity of $T^{\ast }$. We may now apply
Part 1 of Theorem \ref{inv} together with Remark \ref{trivial}(i) to
conclude the proof. $\blacksquare $

\begin{lemma}
\label{lem_100}Let $\check{\beta}$ and $\check{\Omega}$ satisfy Assumption %
\ref{AE}, let $T$ be the test statistic defined in (\ref{DT}), and let $%
W(C)=\left\{ y\in \mathbb{R}^{n}:T(y)\geq C\right\} $ with $0<C<\infty $ be
the rejection region. Let $\Phi _{m}$ be symmetric positive definite $%
n\times n$ matrices such that $\Phi _{m}\rightarrow \Phi $ for $m\rightarrow
\infty $ where $\Phi $ is singular with $l:=\dim \limfunc{span}(\Phi )>0$.
Suppose that for some sequence of positive real numbers $s_{m}$ the matrix $%
D_{m}=\Pi _{\limfunc{span}(\Phi )^{\bot }}\Phi _{m}\Pi _{\limfunc{span}(\Phi
)^{\bot }}/s_{m}$ converges to a matrix $D$, which is regular on $\limfunc{%
span}(\Phi )^{\bot }$, and that $\Pi _{\limfunc{span}(\Phi )^{\bot }}\Phi
_{m}\Pi _{\limfunc{span}(\Phi )}/s_{m}^{1/2}\rightarrow 0$. Suppose further
that $\limfunc{span}(\Phi )\subseteq \mathfrak{M}$. Let $Z$ be a matrix, the
columns of which form a basis for $\limfunc{span}(\Phi )$ and let $%
\boldsymbol{G}$ be a standard normal $n$-vector. Then:

\begin{enumerate}
\item For every $\mu _{0}\in \mathfrak{M}_{0}$, $\gamma \in \mathbb{R}^{l}$, 
$0<\sigma <\infty $ we have%
\begin{equation*}
s_{m}\left[ T\left( \mu _{0}+Z\gamma +\sigma \Phi _{m}^{1/2}\boldsymbol{G}%
\right) -C\right] \overset{d}{\rightarrow }\xi \left( \gamma ,\sigma \right)
\end{equation*}%
for $m\rightarrow \infty $ where the random variable $\xi \left( \gamma
,\sigma \right) $ is given by%
\begin{equation*}
\left( R\hat{\beta}\left( \sigma ^{-1}Z\gamma +\Phi ^{1/2}\boldsymbol{G}%
\right) \right) ^{\prime }\check{\Omega}^{-1}\left( \left( \Phi
^{1/2}+D^{1/2}\right) \boldsymbol{G}\right) \left( R\hat{\beta}\left( \sigma
^{-1}Z\gamma +\Phi ^{1/2}\boldsymbol{G}\right) \right)
\end{equation*}%
for $\left( \Phi ^{1/2}+D^{1/2}\right) \boldsymbol{G}\notin N^{\ast }$,
which is an event that has probability $1$ under the law of $\boldsymbol{G}$%
, and where $\xi \left( \gamma ,\sigma \right) =0$ else.

\item If additionally Assumption \ref{omega2} holds and%
\begin{equation*}
R\hat{\beta}(z)\neq 0\qquad \lambda _{\limfunc{span}(\Phi )}\text{-}a.e.
\end{equation*}%
is satisfied, then 
\begin{equation*}
P_{\mu _{0}+Z\gamma ,\sigma ^{2}\Phi _{m}}\left( W(C)\right) =\Pr \left(
T\left( \mu _{0}+Z\gamma +\sigma \Phi _{m}^{1/2}\boldsymbol{G}\right) \geq
C\right) \rightarrow \Pr \left( \xi \left( \gamma ,\sigma \right) \geq
0\right)
\end{equation*}%
as $m\rightarrow \infty $.
\end{enumerate}
\end{lemma}

\begin{proof}
(1) Observe that $\mu _{0}+Z\gamma \in \mathfrak{M}$, that the columns of $%
\Phi ^{1/2}$ as well of $\Pi _{\limfunc{span}(\Phi )}\Phi _{m}^{1/2}$ belong
to $\mathfrak{M}$, and that $\mathbb{R}^{n}\backslash N^{\ast }$ is
invariant under the group $G(\mathfrak{M})$. Hence, using the equivariance
properties of $\check{\beta}$ and $\check{\Omega}$ expressed in Assumption %
\ref{AE} repeatedly, we obtain that on the event $\left\{ \Phi _{m}^{1/2}%
\boldsymbol{G}\in \mathbb{R}^{n}\backslash N^{\ast }\right\} $%
\begin{eqnarray*}
R\check{\beta}\left( \mu _{0}+Z\gamma +\sigma \Phi _{m}^{1/2}\boldsymbol{G}%
\right) -r &=&R\check{\beta}\left( \mu _{0}+Z\gamma +\sigma \Pi _{\limfunc{%
span}(\Phi )}\Phi _{m}^{1/2}\boldsymbol{G}+\sigma \Pi _{\limfunc{span}(\Phi
)^{\bot }}\Phi _{m}^{1/2}\boldsymbol{G}\right) -r \\
&=&R\left( BZ\gamma +\sigma B\Pi _{\limfunc{span}(\Phi )}\Phi _{m}^{1/2}%
\boldsymbol{G}+\sigma s_{m}^{1/2}\check{\beta}\left( s_{m}^{-1/2}\Pi _{%
\limfunc{span}(\Phi )^{\bot }}\Phi _{m}^{1/2}\boldsymbol{G}\right) \right) \\
&=&\sigma R\left( \sigma ^{-1}BZ\gamma +\boldsymbol{K}_{m}+s_{m}^{1/2}\check{%
\beta}\left( \boldsymbol{L}_{m}\right) \right)
\end{eqnarray*}%
holds, where $B$ is shorthand for $\left( X^{\prime }X\right) ^{-1}X^{\prime
}$, $\boldsymbol{K}_{m}=B\left( \Pi _{\limfunc{span}(\Phi )}\Phi
_{m}^{1/2}-s_{m}^{1/2}\Phi ^{1/2}\right) \boldsymbol{G}$, and $\boldsymbol{L}%
_{m}=\Phi ^{1/2}\boldsymbol{G}+s_{m}^{-1/2}\Pi _{\limfunc{span}(\Phi )^{\bot
}}\Phi _{m}^{1/2}\boldsymbol{G})$. Similarly, we obtain%
\begin{eqnarray*}
\check{\Omega}\left( \mu _{0}+Z\gamma +\sigma \Phi _{m}^{1/2}\boldsymbol{G}%
\right) &=&\sigma ^{2}\check{\Omega}\left( \Phi _{m}^{1/2}\boldsymbol{G}%
\right) =\sigma ^{2}\check{\Omega}\left( \Pi _{\limfunc{span}(\Phi )}\Phi
_{m}^{1/2}\boldsymbol{G}+\Pi _{\limfunc{span}(\Phi )^{\bot }}\Phi _{m}^{1/2}%
\boldsymbol{G}\right) \\
&=&\sigma ^{2}\check{\Omega}\left( \Pi _{\limfunc{span}(\Phi )^{\bot }}\Phi
_{m}^{1/2}\boldsymbol{G}\right) =\sigma ^{2}s_{m}\check{\Omega}(\boldsymbol{L%
}_{m})
\end{eqnarray*}%
on the event $\left\{ \Phi _{m}^{1/2}\boldsymbol{G}\in \mathbb{R}%
^{n}\backslash N^{\ast }\right\} $. Hence, on this event we have%
\begin{eqnarray*}
s_{m}\left[ T\left( \mu _{0}+Z\gamma +\sigma \Phi _{m}^{1/2}\boldsymbol{G}%
\right) -C\right] &=&\left( R\left( \sigma ^{-1}BZ\gamma +\boldsymbol{K}%
_{m}+s_{m}^{1/2}\check{\beta}\left( \boldsymbol{L}_{m}\right) \right)
\right) ^{\prime }\check{\Omega}^{-1}\left( \boldsymbol{L}_{m}\right) \\
&&\times R\left( \sigma ^{-1}BZ\gamma +\boldsymbol{K}_{m}+s_{m}^{1/2}\check{%
\beta}\left( \boldsymbol{L}_{m}\right) \right) -s_{m}C.
\end{eqnarray*}%
Clearly, $\boldsymbol{K}_{m}$ and $\boldsymbol{L}_{m}$ are jointly normal
with mean zero and second moments given by%
\begin{equation*}
\mathbb{E}\left( \boldsymbol{K}_{m}\boldsymbol{K}_{m}^{\prime }\right)
=B\left( \Pi _{\limfunc{span}(\Phi )}\Phi _{m}^{1/2}-s_{m}^{1/2}\Phi
^{1/2}\right) \left( \Pi _{\limfunc{span}(\Phi )}\Phi
_{m}^{1/2}-s_{m}^{1/2}\Phi ^{1/2}\right) ^{\prime }B^{\prime },
\end{equation*}%
\begin{equation*}
\mathbb{E}\left( \boldsymbol{L}_{m}\boldsymbol{L}_{m}^{\prime }\right) =\Phi
+D_{m}+s_{m}^{-1/2}\Pi _{\limfunc{span}(\Phi )^{\bot }}\Phi _{m}^{1/2}\Phi
^{1/2}+s_{m}^{-1/2}\left( \Pi _{\limfunc{span}(\Phi )^{\bot }}\Phi
_{m}^{1/2}\Phi ^{1/2}\right) ^{\prime },
\end{equation*}%
and%
\begin{equation*}
\mathbb{E}\left( \boldsymbol{K}_{m}\boldsymbol{L}_{m}^{\prime }\right)
=B\left( \Pi _{\limfunc{span}(\Phi )}\Phi _{m}^{1/2}-s_{m}^{1/2}\Phi
^{1/2}\right) \left( \Phi ^{1/2}+s_{m}^{-1/2}\Pi _{\limfunc{span}(\Phi
)^{\bot }}\Phi _{m}^{1/2}\right) ^{\prime }.
\end{equation*}%
It is easy to see that $\mathbb{E}\left( \boldsymbol{K}_{m}\boldsymbol{K}%
_{m}^{\prime }\right) $ converges to $B\Phi B^{\prime }$ because $%
s_{m}\rightarrow 0$, while $\mathbb{E}\left( \boldsymbol{L}_{m}\boldsymbol{L}%
_{m}^{\prime }\right) $ converges to $\Phi +D$ because of the following:
Observe that $s_{m}^{-1/2}\Pi _{\limfunc{span}(\Phi )^{\bot }}\Phi
_{m}^{1/2} $ is a (not necessarily symmetric) square root of $D_{m}$, and
hence there exists an orthogonal $n\times n$ matrix $U_{m}$ such that $%
s_{m}^{-1/2}\Pi _{\limfunc{span}(\Phi )^{\bot }}\Phi
_{m}^{1/2}=D_{m}^{1/2}U_{m}$. Let $m^{\prime }$ be an arbitrary subsequence
of $m$. Then we can find a subsequence $m^{\ast }$ of $m^{\prime }$ along
which $U_{m}$ converges to $U$, say. Using $D_{m}\rightarrow D$, we see that
along $m^{\ast }$ the sequence $s_{m}^{-1/2}\Pi _{\limfunc{span}(\Phi
)^{\bot }}\Phi _{m}^{1/2}\Phi ^{1/2}$ converges to $D^{1/2}U\Phi ^{1/2}$. It
remains to show that this limit is zero. By assumption $s_{m}^{-1/2}\Pi _{%
\limfunc{span}(\Phi )^{\bot }}\Phi _{m}\Pi _{\limfunc{span}(\Phi )}$
converges to $0$. By rewriting this sequence as $D_{m}^{1/2}U_{m}\Phi
_{m}^{1/2}\Pi _{\limfunc{span}(\Phi )}$ we see, using $\Phi _{m}\rightarrow
\Phi $, that it converges to $D^{1/2}U\Phi ^{1/2}$ along $m^{\ast }$,
showing that $D^{1/2}U\Phi ^{1/2}=0$.

Furthermore, $\mathbb{E}\left( \boldsymbol{K}_{m}\boldsymbol{L}_{m}^{\prime
}\right) $ converges to $B\Phi $ because%
\begin{eqnarray*}
&&B\left( \Pi _{\limfunc{span}(\Phi )}\Phi _{m}^{1/2}-s_{m}^{1/2}\Phi
^{1/2}\right) \left( s_{m}^{-1/2}\Pi _{\limfunc{span}(\Phi )^{\bot }}\Phi
_{m}^{1/2}\right) ^{\prime } \\
&=&B\left( \Pi _{\limfunc{span}(\Phi )^{\bot }}\Phi _{m}\Pi _{\limfunc{span}%
(\Phi )}/s_{m}^{1/2}\right) ^{\prime }-B\Phi ^{1/2}\Phi _{m}^{1/2}\Pi _{%
\limfunc{span}(\Phi )^{\bot }} \\
&\rightarrow &-B\Phi \Pi _{\limfunc{span}(\Phi )^{\bot }}=-B\left( \Pi _{%
\limfunc{span}(\Phi )^{\bot }}\Phi \right) ^{\prime }=0
\end{eqnarray*}%
where we have made use of the assumption $\Pi _{\limfunc{span}(\Phi )^{\bot
}}\Phi _{m}\Pi _{\limfunc{span}(\Phi )}/s_{m}^{1/2}\rightarrow 0$ and of
symmetry of $\Phi $. Hence we have (cf. Lemma \ref{conc}) that%
\begin{equation*}
\left( 
\begin{array}{c}
\boldsymbol{K}_{m} \\ 
\boldsymbol{L}_{m}%
\end{array}%
\right) \overset{d}{\rightarrow }N\left( 0,\left[ 
\begin{array}{cc}
B\Phi B^{\prime } & B\Phi \\ 
\Phi B^{\prime } & \Phi +D%
\end{array}%
\right] \right) .
\end{equation*}%
Note that this limiting normal distribution is also the joint distribution
of $\boldsymbol{K}=B\Phi ^{1/2}\boldsymbol{G}$ and $\boldsymbol{L}=\left(
\Phi ^{1/2}+D^{1/2}\right) \boldsymbol{G}$. [Observe that $\Phi
^{1/2}+D^{1/2}=\left( \Phi +D\right) ^{1/2}$ since $\Phi D=D\Phi =0$ as $D$
vanishes on $\limfunc{span}(\Phi )$ by construction.] Now consider the map $%
f $ on $\mathbb{R}^{n+k}$ given by $f(x,y)=\left(
f_{1}(x),f_{2}(y),f_{3}(y)\right) $ where $f_{1}(x)=x$ for $x\in \mathbb{R}%
^{k}$, and where $f_{2}(y)=\check{\beta}(y)$, $f_{3}(y)=\check{\Omega}%
^{-1}\left( y\right) $ for $y\in \mathbb{R}^{n}\backslash N^{\ast }$ and are
zero else. Observe that the set of discontinuity points, $F$ say, of $f$ is
contained in $\mathbb{R}^{k}\times N^{\ast }$. But 
\begin{equation}
\Pr \left( \left( \boldsymbol{K},\boldsymbol{L}\right) \in F\right) \leq \Pr
\left( \left( \boldsymbol{K},\boldsymbol{L}\right) \in \mathbb{R}^{k}\times
N^{\ast }\right) =\Pr \left( \boldsymbol{L}\in N^{\ast }\right) =0
\label{zero}
\end{equation}%
because $N^{\ast }$ is a $\lambda _{\mathbb{R}^{n}}$-null set and the
distribution of $\boldsymbol{L}$ is equivalent to Lebesgue measure on $%
\mathbb{R}^{n}$ as $\Phi +D$ is positive definite. This shows that $f\left( 
\boldsymbol{K}_{m},\boldsymbol{L}_{m}\right) $ converges in distribution to $%
f\left( \boldsymbol{K},\boldsymbol{L}\right) $ as $m\rightarrow \infty $. Now%
\begin{eqnarray*}
s_{m}\left[ T\left( \mu _{0}+Z\gamma +\sigma \Phi _{m}^{1/2}\boldsymbol{G}%
\right) -C\right] &=&\left( R\left( \sigma ^{-1}BZ\gamma +f_{1}\left( 
\boldsymbol{K}_{m}\right) +s_{m}^{1/2}f_{2}(\boldsymbol{L}_{m})\right)
\right) ^{\prime }f_{3}\left( \boldsymbol{L}_{m}\right) \\
&&\times R\left( \sigma ^{-1}BZ\gamma +f_{1}\left( \boldsymbol{K}_{m}\right)
+s_{m}^{1/2}f_{2}(\boldsymbol{L}_{m})\right) -s_{m}C
\end{eqnarray*}%
holds everywhere (note that $\boldsymbol{L}_{m}\in \mathbb{R}^{n}\backslash
N^{\ast }$ if and only if $\Phi _{m}^{1/2}\boldsymbol{G}\in \mathbb{R}%
^{n}\backslash N^{\ast }$ by $G(\mathfrak{M})$-invariance of $\mathbb{R}%
^{n}\backslash N^{\ast }$). Because $s_{m}^{1/2}f_{2}(\boldsymbol{L}_{m})$
converges to zero in probability and $s_{m}C\rightarrow 0$ we immediately
see that the random variable in the preceding display converges in
distribution to%
\begin{equation*}
\left( R\left( \sigma ^{-1}BZ\gamma +f_{1}\left( B\Phi ^{1/2}\boldsymbol{G}%
\right) \right) \right) ^{\prime }f_{3}\left( \left( \Phi
^{1/2}+D^{1/2}\right) \boldsymbol{G}\right) R\left( \sigma ^{-1}BZ\gamma
+f_{1}\left( B\Phi ^{1/2}\boldsymbol{G}\right) \right)
\end{equation*}%
which coincides with $\xi \left( \gamma ,\sigma \right) $. Finally, the
claim that $\left\{ \left( \Phi ^{1/2}+D^{1/2}\right) \boldsymbol{G}\in 
\mathbb{R}^{n}\backslash N^{\ast }\right\} $ is a probability $1$ event has
already been established in (\ref{zero}).

(2) This follows from Part 1 if we can establish that $\Pr \left( \xi \left(
\gamma ,\sigma \right) =0\right) =0$. Now observe that $\check{\Omega}%
^{-1}(\left( \Phi ^{1/2}+D^{1/2}\right) \boldsymbol{G})=\check{\Omega}%
^{-1}(D^{1/2}\boldsymbol{G})$ by equivariance and that $\left( \Phi
^{1/2}+D^{1/2}\right) \boldsymbol{G}\in \mathbb{R}^{n}\backslash N^{\ast }$
if and only if $D^{1/2}\boldsymbol{G}\in \mathbb{R}^{n}\backslash N^{\ast }$%
. Hence 
\begin{eqnarray}
&&\Pr \left( \xi \left( \gamma ,\sigma \right) =0\right) =\Pr \left( \xi
\left( \gamma ,\sigma \right) =0,\left( \Phi ^{1/2}+D^{1/2}\right) 
\boldsymbol{G}\in \mathbb{R}^{n}\backslash N^{\ast }\right)  \notag \\
&=&\Pr \left( \left( \hat{\beta}\left( \sigma ^{-1}Z\gamma +\Phi ^{1/2}%
\boldsymbol{G}\right) \right) ^{\prime }R^{\prime }\check{\Omega}%
^{-1}(D^{1/2}\boldsymbol{G})\right.  \notag \\
&&\times \left. R\left( \hat{\beta}\left( \sigma ^{-1}Z\gamma +\Phi ^{1/2}%
\boldsymbol{G}\right) \right) =0,D^{1/2}\boldsymbol{G}\in \mathbb{R}%
^{n}\backslash N^{\ast }\right)  \notag \\
&=&\int \Pr \left( \left( \hat{\beta}\left( \sigma ^{-1}Z\gamma +x\right)
\right) ^{\prime }R^{\prime }\check{\Omega}^{-1}(D^{1/2}\boldsymbol{G}%
)\right.  \notag \\
&&\times \left. R\left( \hat{\beta}\left( \sigma ^{-1}Z\gamma +x\right)
\right) =0,D^{1/2}\boldsymbol{G}\in \mathbb{R}^{n}\backslash N^{\ast
}\right) dP_{0,\Phi }(x)  \notag \\
&=&\int \Pr \left( \left( \hat{\beta}\left( \sigma ^{-1}Z\gamma +x\right)
\right) ^{\prime }R^{\prime }\check{\Omega}^{-1}\left( \left( \Phi
^{1/2}+D^{1/2}\right) \boldsymbol{G}\right) \right.  \notag \\
&&\times \left. R\left( \hat{\beta}\left( \sigma ^{-1}Z\gamma +x\right)
\right) =0,\left( \Phi ^{1/2}+D^{1/2}\right) \boldsymbol{G}\in \mathbb{R}%
^{n}\backslash N^{\ast }\right) dP_{0,\Phi }(x)  \notag \\
&=&\int P_{0,\Phi +D}\left( \left\{ y\in \mathbb{R}^{n}\backslash N^{\ast
}:v(x)^{\prime }\check{\Omega}^{-1}(y)v(x)=0\right\} \right) dP_{0,\Phi }(x)
\label{conditioning100}
\end{eqnarray}%
with $v(x)=R\hat{\beta}\left( \sigma ^{-1}Z\gamma +x\right) $, the third
equality in the preceding display being true since $\Phi ^{1/2}\boldsymbol{G}
$ and $D^{1/2}\boldsymbol{G}$ are independent as 
\begin{equation*}
\mathbb{E}\left( \Phi ^{1/2}\boldsymbol{G}\left( D^{1/2}\boldsymbol{G}%
\right) ^{\prime }\right) =\Phi ^{1/2}D^{1/2}=0.
\end{equation*}%
Now the integrand in the last line of (\ref{conditioning100}) is zero by
Assumption (\ref{omega2}) for every $x$ except when $v(x)=0$. Hence, we are
done if we can establish that $P_{0,\Phi }\left( v(x)=0\right) =0$. Because $%
\limfunc{span}(\Phi )$ equals the span of the columns of $Z$, we can make
the change of variables $x=Zc$ and obtain%
\begin{equation*}
P_{0,\Phi }\left( v(x)=0\right) =P_{0,A}\left( v(Zc)=0\right) =P_{0,A}\left(
R\left( \hat{\beta}\left( Z\left( \sigma ^{-1}\gamma +c\right) \right)
\right) =0\right)
\end{equation*}%
where $A=\left( Z^{\prime }Z\right) ^{-1}Z^{\prime }\Phi Z\left( Z^{\prime
}Z\right) ^{-1}$. Because $A$ is non-singular, this probability is zero if
the event has $\lambda _{\mathbb{R}^{l}}$-measure zero. But 
\begin{equation*}
\lambda _{\mathbb{R}^{l}}\left( \left\{ c:R\hat{\beta}\left( Z\left( \sigma
^{-1}\gamma +c\right) \right) =0\right\} \right) =\lambda _{\limfunc{span}%
(\Phi )}\left( \left\{ z:R\hat{\beta}\left( z\right) =0\right\} \right) =0
\end{equation*}%
by our assumptions.
\end{proof}

\textbf{Proof of Theorem \ref{prop_101}: }Fix $\mu _{0}\in \mathfrak{M}_{0}$
and $\sigma $, $0<\sigma <\infty $. Then for every $\gamma \in \mathbb{R}%
^{l} $ we have%
\begin{equation*}
P_{\mu _{0}+Z\gamma ,\sigma ^{2}\Sigma _{m}}\left( W(C)\right) =\Pr \left(
s_{m}\left[ T\left( \mu _{0}+Z\gamma +\sigma \Sigma _{m}^{1/2}\boldsymbol{G}%
\right) -C\right] \geq 0\right)
\end{equation*}%
which converges to $\Pr \left( \xi \left( \gamma ,\sigma \right) \geq
0\right) $ as shown in the preceding lemma (with $\Sigma _{m}$ and $\bar{%
\Sigma}$ playing the r\^{o}les of $\Phi _{m}$ and $\Phi $, respectively).
Consequently, for every $\gamma \in \mathbb{R}^{l}$%
\begin{equation*}
\inf_{\Sigma \in \mathfrak{C}}P_{\mu _{0}+Z\gamma ,\sigma ^{2}\Sigma }\left(
W(C)\right) \leq \Pr \left( \xi \left( \gamma ,\sigma \right) \geq 0\right) .
\end{equation*}%
But now%
\begin{eqnarray*}
&&\liminf_{M\rightarrow \infty }\inf_{\left\Vert \gamma \right\Vert \geq
M}\Pr \left( \xi \left( \gamma ,\sigma \right) \geq 0\right) \leq
\liminf_{M\rightarrow \infty }\inf_{R\hat{\beta}\left( Z\gamma \right) \neq
0,\left\Vert \gamma \right\Vert \geq M}\Pr \left( \xi \left( \gamma ,\sigma
\right) \geq 0\right) \\
&=&\liminf_{M\rightarrow \infty }\inf_{R\hat{\beta}\left( Z\gamma \right)
\neq 0,\left\Vert \gamma \right\Vert \geq M}\Pr \left( \xi \left( \gamma
,\sigma \right) /\left\Vert \gamma \right\Vert ^{2}\geq 0\right) \\
&\leq &\liminf_{M\rightarrow \infty }\inf_{R\hat{\beta}\left( Z\gamma
\right) \neq 0,\left\Vert \gamma \right\Vert =M}\Pr \left( \xi \left( \gamma
,\sigma \right) /\left\Vert \gamma \right\Vert ^{2}\geq 0\right) \\
&\leq &\inf_{\left\Vert c\right\Vert =1,R\hat{\beta}\left( Zc\right) \neq
0}\liminf_{M\rightarrow \infty }\Pr \left( \bar{\xi}\left( c,M,\sigma
\right) \geq 0\right)
\end{eqnarray*}%
where 
\begin{eqnarray*}
\bar{\xi}\left( c,M,\sigma \right) &=&\left( R\left( \hat{\beta}\left(
Zc\right) +\sigma \hat{\beta}\left( \bar{\Sigma}^{1/2}\boldsymbol{G}\right)
/M\right) \right) ^{\prime }\check{\Omega}^{-1}\left( \left( \bar{\Sigma}%
^{1/2}+D^{1/2}\right) \boldsymbol{G}\right) \\
&&\times R\left( \hat{\beta}\left( Zc\right) +\sigma \hat{\beta}\left( \bar{%
\Sigma}^{1/2}\boldsymbol{G}\right) /M\right)
\end{eqnarray*}%
on the event where $\left( \bar{\Sigma}^{1/2}+D^{1/2}\right) \boldsymbol{G}%
\in \mathbb{R}^{n}\backslash N^{\ast }$ and is zero else. The random
variable $\bar{\xi}\left( c,M,\sigma \right) $ converges in probability to
the random variable $\bar{\xi}\left( c\right) $ as $M\rightarrow \infty $.
Hence 
\begin{equation*}
\liminf_{M\rightarrow \infty }\Pr \left( \bar{\xi}\left( c,M,\sigma \right)
\geq 0\right) =\Pr \left( \bar{\xi}\left( c\right) \geq 0\right)
\end{equation*}%
holds for every $c\in \mathbb{R}^{l}$ satisfying $\left\Vert c\right\Vert =1$
and $R\hat{\beta}\left( Zc\right) \neq 0$, because $\Pr \left( \bar{\xi}%
\left( c\right) =0\right) =0$ for such $c$ in view of Assumption \ref{omega2}
observing that $P_{0,\bar{\Sigma}+D}$ is equivalent to $\lambda _{\mathbb{R}%
^{n}}$ as $\bar{\Sigma}+D$ is nonsingular. This proves that%
\begin{eqnarray*}
&&\liminf_{M\rightarrow \infty }\inf_{\left\Vert \gamma \right\Vert \geq
M}\inf_{\Sigma \in \mathfrak{C}}P_{\mu _{0}+Z\gamma ,\sigma ^{2}\Sigma
}\left( W(C)\right) \leq \inf_{\left\Vert c\right\Vert =1,R\hat{\beta}\left(
Zc\right) \neq 0}\Pr \left( \bar{\xi}\left( c\right) \geq 0\right) \\
&=&\inf_{\left\Vert c\right\Vert =1}\Pr \left( \bar{\xi}\left( c\right) \geq
0\right) =\inf_{c\in \mathbb{R}^{l}}\Pr \left( \bar{\xi}\left( c\right) \geq
0\right) =K_{1},
\end{eqnarray*}%
the first two equalities holding because $\bar{\xi}\left( c\right) \equiv 0$
if $R\hat{\beta}\left( Zc\right) =0$ (and in particular if $c=0$) and
because $\Pr \left( \bar{\xi}\left( c\right) \geq 0\right) $ is homogenous
in $c$. This establishes the first inequality in (\ref{Power_size_ineq})
because the left-most expression in (\ref{Power_size_ineq}) is monotonically
increasing in $M$. Furthermore,%
\begin{equation*}
\sup_{\Sigma \in \mathfrak{C}}P_{\mu _{0},\sigma ^{2}\Sigma }\left(
W(C)\right) \geq P_{\mu _{0},\sigma ^{2}\Sigma _{m}}\left( W(C)\right) ,
\end{equation*}%
and hence we obtain from Lemma \ref{lem_100} that 
\begin{eqnarray*}
&&\sup_{\Sigma \in \mathfrak{C}}P_{\mu _{0},\sigma ^{2}\Sigma }\left(
W(C)\right) \geq \Pr \left( \xi \left( 0,\sigma \right) \geq 0\right) \\
&=&\Pr \left( \left( R\hat{\beta}\left( \bar{\Sigma}^{1/2}\boldsymbol{G}%
\right) \right) ^{\prime }\check{\Omega}^{-1}\left( \left( \bar{\Sigma}%
^{1/2}+D^{1/2}\right) \boldsymbol{G}\right) R\hat{\beta}\left( \bar{\Sigma}%
^{1/2}\boldsymbol{G}\right) \geq 0,\right. \\
&&\left. \left( \bar{\Sigma}^{1/2}+D^{1/2}\right) \boldsymbol{G}\in \mathbb{R%
}^{n}\backslash N^{\ast }\right) .
\end{eqnarray*}%
Now observe that $\check{\Omega}^{-1}\left( \left( \bar{\Sigma}%
^{1/2}+D^{1/2}\right) \boldsymbol{G}\right) =\check{\Omega}^{-1}(D^{1/2}%
\boldsymbol{G})$ by equivariance and that $\left( \bar{\Sigma}%
^{1/2}+D^{1/2}\right) \boldsymbol{G}\in \mathbb{R}^{n}\backslash N^{\ast }$
if and only if $D^{1/2}\boldsymbol{G}\in \mathbb{R}^{n}\backslash N^{\ast }$%
. Then by the same arguments as in (\ref{conditioning100}) we obtain%
\begin{eqnarray*}
\Pr \left( \xi \left( 0,\sigma \right) \geq 0\right) &=&\int \Pr \left(
\left( R\hat{\beta}\left( x\right) \right) ^{\prime }\check{\Omega}%
^{-1}\left( \left( \bar{\Sigma}^{1/2}+D^{1/2}\right) \boldsymbol{G}\right)
\right. \\
&&\times \left. R\left( \hat{\beta}\left( x\right) \right) \geq 0,\left( 
\bar{\Sigma}^{1/2}+D^{1/2}\right) \boldsymbol{G}\in \mathbb{R}^{n}\backslash
N^{\ast }\right) dP_{0,\bar{\Sigma}}(x) \\
&=&\int \Pr \left( \bar{\xi}\left( \gamma \right) \geq 0\right)
dP_{0,A}(\gamma )=K_{2},
\end{eqnarray*}%
the last equality resulting from the variable change $x=Z\gamma $ which is
possible since $\limfunc{span}(\bar{\Sigma})$ equals the space spanned by $Z$%
. Finally, the inequality $K_{1}\leq K_{2}$ is obvious from the definition
of these constants. $\blacksquare $

\textbf{Proof of Theorem \ref{TU_1}: }Define $\varphi =\boldsymbol{1}\left(
W(C)\right) $ and note that invariance of $\varphi $ under $G(\mathfrak{M}%
_{0})$ as well as the fact that $\varphi $ is $\lambda _{\mathbb{R}^{n}}$%
-almost everywhere neither equal to $0$ or $1$ follows from Lemma \ref{LWT}.
Part 1 of Theorem \ref{TU} then implies Part 1 of the theorem. Similarly,
Parts 2 and 3 of the theorem follow from Parts 2 and 3 of Theorem \ref{TU},
respectively, because condition (\ref{power_ass}) follows from Part 7 of
Lemma \ref{LWT} combined with Remark \ref{omega_rem} and because the lower
bound in (\ref{lower_bound_power_far_away}) equals $1$ under the assumptions
of Part 3. To prove Part 4 we use Theorem \ref{TUU}. Choose a sequence $%
C_{k} $, $0<C_{k}<\infty $, \ that diverges monotonically to infinity and
set $\varphi _{k}=\boldsymbol{1}\left( W(C_{k})\right) $. Then (\ref%
{monotone}) is satisfied and the result follows from Theorem \ref{TUU} upon
setting $C(\delta )=C_{k_{0}(\delta )}$. $\blacksquare $

\begin{lemma}
\label{nec}Let $\check{\beta}$ and $\check{\Omega}$ satisfy Assumptions \ref%
{AE} and \ref{omega2}. Let $T$ be the test statistic defined in (\ref{DT})
and let $\mathfrak{C}$ be a covariance model. If there is a $z\in \limfunc{%
span}\left( J(\mathfrak{C})\right) \cap \mathfrak{M}$ with $z\notin 
\mathfrak{M}_{0}-\mu _{0}$ (i.e., with $R\hat{\beta}(z)\neq 0$), then $T$
does not satisfy the invariance condition (\ref{T_inv_wrt_z}).
\end{lemma}

\begin{proof}
Choose $z\in \limfunc{span}\left( J(\mathfrak{C})\right) \cap \mathfrak{M}$
with $z\notin \mathfrak{M}_{0}-\mu _{0}$. Because $\mathfrak{M}$ is a linear
space, we also have $cz\in \limfunc{span}\left( J(\mathfrak{C})\right) \cap 
\mathfrak{M}$ for every $c\in \mathbb{R}$. Now $cz\in \mathfrak{M}$ entails
that $y\in \mathbb{R}^{n}\backslash N^{\ast }$ implies $y+cz\in \mathbb{R}%
^{n}\backslash N^{\ast }$. Using the definition of $T$ and Assumption \ref%
{AE} we obtain%
\begin{equation*}
T\left( y+cz\right) =T\left( y\right) +2c\left( R\check{\beta}\left(
y\right) -r\right) ^{\prime }\check{\Omega}^{-1}\left( y\right) R\hat{\beta}%
(z)+c^{2}\left( R\hat{\beta}(z)\right) ^{\prime }\check{\Omega}^{-1}\left(
y\right) \left( R\hat{\beta}(z)\right)
\end{equation*}%
for every $y\in \mathbb{R}^{n}\backslash N^{\ast }$. Because $R\hat{\beta}%
(z)\neq 0$, we can in view of Assumption \ref{omega2} find an $y\in \mathbb{R%
}^{n}\backslash N^{\ast }$ such that 
\begin{equation*}
\left( R\hat{\beta}(z)\right) ^{\prime }\check{\Omega}^{-1}\left( y\right)
\left( R\hat{\beta}(z)\right) \neq 0
\end{equation*}%
holds. Hence $T\left( y+cz\right) =T\left( y\right) $ cannot hold for the
so-chosen $y$ and all $c\neq 0$. Because $cz\in \limfunc{span}\left( J(%
\mathfrak{C})\right) $, Remark \ref{rem_positive}(i) implies that condition (%
\ref{T_inv_wrt_z}) is not satisfied.
\end{proof}

\textbf{Proof of Proposition \ref{enforce_inv}: }(1) By the assumed
equivariance (invariance, respectively) of $\bar{\theta}$, $\bar{\Omega}$,
and $\bar{N}$ (and hence of $\bar{N}^{\ast }$) w.r.t. the transformations $%
y\mapsto \alpha y+\bar{X}\eta $, the equivariance (invariance, respectively)
of $\bar{\beta}$, ${\bar{\Omega}}$, and $\bar{N}$ required in the original
Assumption \ref{AE} is clearly satisfied. Now choose $z\in J(\mathfrak{C})$
and $y\in \mathbb{R}^{n}$. If $y\in \bar{N}^{\ast }$ then so is $y+z$
because of invariance of $\bar{N}^{\ast }$ and because $z\in J(\mathfrak{C}%
)\subseteq \mathfrak{\bar{M}}$ holds by construction. Hence, $T(y)=0=T(y+z)$
is satisfied in this case. Now let $y\in \mathbb{R}^{n}\backslash \bar{N}%
^{\ast }$ (and hence also $y+z\in \mathbb{R}^{n}\backslash \bar{N}^{\ast }$%
). Note that ${\bar{\Omega}(y)=\bar{\Omega}(y+z)}$ holds by equivariance. It
remains to show that $R\bar{\beta}(y)=R\bar{\beta}(y+z)$. Because $z\in J(%
\mathfrak{C})\subseteq \mathfrak{\bar{M}}$ we have $z=X\gamma +\left( \bar{x}%
_{1},\ldots ,\bar{x}_{p}\right) \delta $ and thus obtain%
\begin{equation}
R\bar{\beta}(y+z)=\left( R,0\right) \bar{\theta}(y+z)=\left( R,0\right)
\left( \bar{\theta}(y)+\left( \gamma ^{\prime },\delta ^{\prime }\right)
^{\prime }\right) =R\bar{\beta}(y)+R\gamma ,  \label{beta}
\end{equation}%
where we have made use of equivariance of $\bar{\theta}$. Now observe that $%
\left( \bar{x}_{1},\ldots ,\bar{x}_{p}\right) \delta \in \limfunc{span}%
\left( J(\mathfrak{C})\cup \left( \mathfrak{M}_{0}-\mu _{0}\right) \right) $
by construction of the $\bar{x}_{i}$. Hence, we can find an element $\mu
_{0}^{\#}\in \mathfrak{M}_{0}$ such that $\left( \bar{x}_{1},\ldots ,\bar{x}%
_{p}\right) \delta -\left( \mu _{0}^{\#}-\mu _{0}\right) \in \limfunc{span}%
\left( J(\mathfrak{C})\right) $. Consequently, we obtain%
\begin{equation*}
z-\left( \left( \bar{x}_{1},\ldots ,\bar{x}_{p}\right) \delta -\left( \mu
_{0}^{\#}-\mu _{0}\right) \right) =X\gamma +\left( \mu _{0}^{\#}-\mu
_{0}\right) .
\end{equation*}%
The left-hand side is obviously an element of $\limfunc{span}\left( J(%
\mathfrak{C})\right) $, while the right-hand side belongs to $\mathfrak{M}$,
implying that the right-hand side is in $\limfunc{span}J(\mathfrak{C})\cap 
\mathfrak{M}$ which is a subset of $\mathfrak{M}_{0}-\mu _{0}$ by
assumption. Because $\mu _{0}^{\#}-\mu _{0}\in \mathfrak{M}_{0}-\mu _{0}$,
we have established that $X\gamma \in \mathfrak{M}_{0}-\mu _{0}$, or in
other words, that $R\gamma =0$.

(2) The very first claim is obvious. If $z\in \limfunc{span}\left( J(%
\mathfrak{C})\right) $ then again we have $z=X\gamma +\left( \bar{x}%
_{1},\ldots ,\bar{x}_{p}\right) \delta $ and $\bar{\theta}\left( z\right)
=\left( \gamma ^{\prime },\delta ^{\prime }\right) ^{\prime }$. Now $R\bar{%
\beta}\left( z\right) =\left( R,0\right) \bar{\theta}\left( z\right)
=R\gamma $ and exactly the same argument as above shows that $R\gamma =0$.
For the last claim note that $\bar{X}\bar{\theta}\left( y\right) =X^{\ast
}\theta ^{\ast }\left( y\right) $ holds because $\bar{X}$ and $X^{\ast }$
span the same space. This equality can be written as%
\begin{equation*}
X\bar{\beta}(y)-X\beta ^{\ast }(y)=\sum_{i=1}^{p}x_{i}^{\ast }\theta
_{k+i}^{\ast }\left( y\right) -\sum_{i=1}^{p}\bar{x}_{i}\bar{\theta}%
_{k+i}\left( y\right) .
\end{equation*}%
Because the right-hand side of the above equation belongs to $\limfunc{span}%
\left( J(\mathfrak{C})\cup \left( \mathfrak{M}_{0}-\mu _{0}\right) \right) $
we can find $\mu _{0}^{\#}\in \mathfrak{M}_{0}$ such that the right-hand
side of%
\begin{equation*}
X\left( \bar{\beta}(y)-\beta ^{\ast }(y)\right) -\left( \mu _{0}^{\#}-\mu
_{0}\right) =\sum_{i=1}^{p}x_{i}^{\ast }\theta _{k+i}^{\ast }\left( y\right)
-\sum_{i=1}^{p}\bar{x}_{i}\bar{\theta}_{k+i}\left( y\right) -\left( \mu
_{0}^{\#}-\mu _{0}\right)
\end{equation*}%
belongs to $\limfunc{span}\left( J(\mathfrak{C})\right) $ while the
left-hand side belongs to $\mathfrak{M}$. Arguing now similarly as in the
proof of Part 1, we conclude that $R\bar{\beta}(y)=R\beta ^{\ast }(y)$. $%
\blacksquare $

\section{Appendix: Properties of AR-Correlation Matrices \label{AR_100}}

\begin{lemma}
\label{AR_1}

\begin{enumerate}
\item Suppose the covariance model $\mathfrak{C}$ contains $\Lambda (\rho
_{m})$ for some sequence $\rho _{m}\in \left( -1,1\right) $ with $\rho
_{m}\rightarrow 1$ ($\rho _{m}\rightarrow -1$, respectively). Then $\limfunc{%
span}\left( e_{+}\right) $ ( $\limfunc{span}\left( e_{-}\right) $,
respectively) is a concentration space of $\mathfrak{C}$.

\item $\mathfrak{C}_{AR(1)}$ has $\limfunc{span}\left( e_{+}\right) $ and $%
\limfunc{span}\left( e_{-}\right) $ as its only concentration spaces.
Consequently, $J(\mathfrak{C}_{AR(1)})=\limfunc{span}(e_{+})\cup \limfunc{%
span}(e_{-})$.

\item If $\rho _{m}\in \left( -1,1\right) $ is a sequence converging to $1$
then $\Sigma _{m}=\Lambda (\rho _{m})$ satisfies $\Sigma _{m}\rightarrow 
\bar{\Sigma}=e_{+}e_{+}^{\prime }$ and $D_{m}=\Pi _{\limfunc{span}\left( 
\bar{\Sigma}\right) ^{\bot }}\Sigma _{m}\Pi _{\limfunc{span}\left( \bar{%
\Sigma}\right) ^{\bot }}/s_{m}\rightarrow D$ as well as $\Pi _{\limfunc{span}%
\left( \bar{\Sigma}\right) ^{\bot }}\Sigma _{m}\Pi _{\limfunc{span}\left( 
\bar{\Sigma}\right) }/s_{m}^{1/2}\rightarrow 0$ where $s_{m}=\limfunc{tr}%
\left( \Pi _{\limfunc{span}\left( \bar{\Sigma}\right) ^{\bot }}\Sigma
_{m}\Pi _{\limfunc{span}\left( \bar{\Sigma}\right) ^{\bot }}\right) $
converges to zero and $D$ is the matrix with $(i,j)$-th element $%
-n\left\vert i-j\right\vert /\sum_{i,j}\left\vert i-j\right\vert $ pre- and
postmultiplied by $\left( I_{n}-n^{-1}e_{+}e_{+}^{\prime }\right) $.
Furthermore, $D$ is regular on $\limfunc{span}\left( \bar{\Sigma}\right)
^{\bot }$.

\item If $\rho _{m}\in \left( -1,1\right) $ is a sequence converging to $-1$
then $\Sigma _{m}=\Lambda (\rho _{m})$ satisfies $\Sigma _{m}\rightarrow 
\bar{\Sigma}=e_{-}e_{-}^{\prime }$ and $D_{m}=\Pi _{\limfunc{span}\left( 
\bar{\Sigma}\right) ^{\bot }}\Sigma _{m}\Pi _{\limfunc{span}\left( \bar{%
\Sigma}\right) ^{\bot }}/s_{m}\rightarrow D$ as well as $\Pi _{\limfunc{span}%
\left( \bar{\Sigma}\right) ^{\bot }}\Sigma _{m}\Pi _{\limfunc{span}\left( 
\bar{\Sigma}\right) }/s_{m}^{1/2}\rightarrow 0$ where $s_{m}=\limfunc{tr}%
\left( \Pi _{\limfunc{span}\left( \bar{\Sigma}\right) ^{\bot }}\Sigma
_{m}\Pi _{\limfunc{span}\left( \bar{\Sigma}\right) ^{\bot }}\right) $
converges to zero and $D$ is the matrix with $(i,j)$-th element $%
n(-1)^{\left\vert i-j\right\vert +1}\left\vert i-j\right\vert
/\sum_{i,j}\left\vert i-j\right\vert $ pre- and postmultiplied by $\left(
I_{n}-n^{-1}e_{-}e_{-}^{\prime }\right) $. Furthermore, $D$ is regular on $%
\limfunc{span}\left( \bar{\Sigma}\right) ^{\bot }$.
\end{enumerate}
\end{lemma}

\begin{proof}
(1) and (2) are obvious.

(3) Because $\Pi _{\limfunc{span}\left( \bar{\Sigma}\right) ^{\bot }}\Sigma
_{m}\Pi _{\limfunc{span}\left( \bar{\Sigma}\right) ^{\bot }}$ is nonnegative
definite, but obviously different from the zero matrix (recall that $n>1$ is
assumed), we see that $s_{m}$ is always positive. Clearly, $\Pi _{\limfunc{%
span}\left( \bar{\Sigma}\right) ^{\bot }}\Sigma _{m}\Pi _{\limfunc{span}%
\left( \bar{\Sigma}\right) ^{\bot }}$ converges to $\Pi _{\limfunc{span}%
\left( \bar{\Sigma}\right) ^{\bot }}\bar{\Sigma}\Pi _{\limfunc{span}\left( 
\bar{\Sigma}\right) ^{\bot }}=0$ and hence $s_{m}\rightarrow 0$. By
l'Hopital's rule the limit of $D_{m}$ can be obtained as the limit of $\Pi _{%
\limfunc{span}\left( \bar{\Sigma}\right) ^{\bot }}\left( d\Lambda /d\rho
\right) (\rho _{m})\Pi _{\limfunc{span}\left( \bar{\Sigma}\right) ^{\bot }}$
divided by the limit of 
\begin{equation*}
\limfunc{tr}\left( \Pi _{\limfunc{span}\left( \bar{\Sigma}\right) ^{\bot
}}\left( d\Lambda /d\rho \right) (\rho _{m})\Pi _{\limfunc{span}\left( \bar{%
\Sigma}\right) ^{\bot }}\right)
\end{equation*}
provided the latter is nonzero. The second limit now equals%
\begin{eqnarray*}
\limfunc{tr}\left( \left( I_{n}-n^{-1}e_{+}e_{+}^{\prime }\right) \left(
d\Lambda /d\rho \right) (1)\left( I_{n}-n^{-1}e_{+}e_{+}^{\prime }\right)
\right) &=&\limfunc{tr}\left( \left( d\Lambda /d\rho \right) (1)\left(
I_{n}-n^{-1}e_{+}e_{+}^{\prime }\right) \right) \\
&=&\limfunc{tr}\left( \left( d\Lambda /d\rho \right) (1)\right) -n^{-1}%
\limfunc{tr}\left( e_{+}^{\prime }\left( d\Lambda /d\rho \right)
(1)e_{+}\right) .
\end{eqnarray*}%
Observe that the $(i,j)$-th element of the matrix $\left( d\Lambda /d\rho
\right) (1)$ is given by $\left\vert i-j\right\vert $. Hence, the above
expression equals%
\begin{equation*}
-n^{-1}\limfunc{tr}\left( e_{+}^{\prime }\left( d\Lambda /d\rho \right)
(1)e_{+}\right) =-n^{-1}\sum_{i,j}\left\vert i-j\right\vert ,
\end{equation*}%
which is clearly nonzero. The first limit exists and equals%
\begin{equation*}
\left( I_{n}-n^{-1}e_{+}e_{+}^{\prime }\right) \left( d\Lambda /d\rho
\right) (1)\left( I_{n}-n^{-1}e_{+}e_{+}^{\prime }\right)
\end{equation*}%
which shows that $D$ is of the form as claimed in the lemma. We next show
that $D$ is regular on $\limfunc{span}\left( \bar{\Sigma}\right) ^{\bot }=%
\limfunc{span}(e_{+})^{\bot }$. This is equivalent to showing that the
equation system%
\begin{equation*}
\begin{array}{c}
\left( d\Lambda /d\rho \right) (1)x+\lambda e_{+}=0 \\ 
e_{+}^{\prime }x=0%
\end{array}%
\end{equation*}%
has $x=0$, $\lambda =0$ as its only solution. We hence need to show that the 
$(n+1)\times (n+1)$ matrix%
\begin{equation*}
A=\left[ 
\begin{array}{cc}
\left( d\Lambda /d\rho \right) (1) & e_{+} \\ 
e_{+}^{\prime } & 0%
\end{array}%
\right]
\end{equation*}%
has rank $n+1$. Let $B$ be the $(n+1)\times (n+1)$ matrix given by%
\begin{equation*}
B=\left[ 
\begin{array}{cc}
B_{11} & 0 \\ 
0 & 1%
\end{array}%
\right]
\end{equation*}%
where the $n\times n$ matrix $B_{11}$ has $1$ everywhere on the main
diagonal, $-1$ everywhere on the first off-diagonal above the main diagonal,
and zeroes elsewhere. Let the $(n+1)\times (n+1)$ matrices $B^{\ast }$ and $%
B^{\ast \ast }$ be given by%
\begin{equation*}
B^{\ast }=\left[ 
\begin{array}{cc}
0 & 1 \\ 
I_{n} & 0%
\end{array}%
\right] ,\qquad B^{\ast \ast }=\left[ 
\begin{array}{cc}
I_{n} & 0 \\ 
f & 1%
\end{array}%
\right] ,
\end{equation*}%
where $f=-\left( n-1,n-2,n-3,\ldots ,1,0\right) $. Observe that $B$, $%
B^{\ast }$, as well as $B^{\ast \ast }$ are non-singular and that%
\begin{equation*}
B^{\ast }BAB^{\ast \ast }=C=\left[ 
\begin{array}{cc}
C_{11} & 0 \\ 
0 & 1%
\end{array}%
\right]
\end{equation*}%
where $C_{11}$ is an $n\times n$ matrix that has $1$ everywhere on and above
the diagonal and $-1$ everywhere below the diagonal. Obviously, $C$ is
nonsingular and hence $A$ is so. Finally, we show that the limit of $\Pi _{%
\limfunc{span}\left( \bar{\Sigma}\right) ^{\bot }}\Sigma _{m}\Pi _{\limfunc{%
span}\left( \bar{\Sigma}\right) }/s_{m}^{1/2}$ equals zero. Because $%
s_{m}\rightarrow 0$, it suffices to show that the limit of $\Pi _{\limfunc{%
span}\left( \bar{\Sigma}\right) ^{\bot }}\Sigma _{m}\Pi _{\limfunc{span}%
\left( \bar{\Sigma}\right) }/s_{m}$ exists and is finite. Now the same
arguments as above show that the latter limit is equal to $\left(
I_{n}-n^{-1}e_{+}e_{+}^{\prime }\right) \left( d\Lambda /d\rho \right)
(1)n^{-1}e_{+}e_{+}^{\prime }$ divided by $-n^{-1}\sum_{i,j}\left\vert
i-j\right\vert $.

(4) For the same reasons as in (3) $s_{m}$ is positive and converges to
zero. By the same argument as in (3) the limit of $D_{m}$ is%
\begin{equation*}
\left[ \left( I_{n}-n^{-1}e_{-}e_{-}^{\prime }\right) \left( d\Lambda /d\rho
\right) (-1)\left( I_{n}-n^{-1}e_{-}e_{-}^{\prime }\right) \right] /\limfunc{%
tr}\left( \left( I_{n}-n^{-1}e_{-}e_{-}^{\prime }\right) \left( d\Lambda
/d\rho \right) (-1)\left( I_{n}-n^{-1}e_{-}e_{-}^{\prime }\right) \right) .
\end{equation*}%
Note that the denominator is equal to%
\begin{equation*}
\limfunc{tr}\left( \left( d\Lambda /d\rho \right) (-1)\right) -n^{-1}%
\limfunc{tr}\left( e_{-}^{\prime }\left( d\Lambda /d\rho \right)
(-1)e_{-}\right) =n^{-1}\sum_{i,j}\left\vert i-j\right\vert \neq 0,
\end{equation*}%
observing that the $(i,j)$-th element of $\left( d\Lambda /d\rho \right)
(-1) $ is given by $(-1)^{\left\vert i-j\right\vert +1}\left\vert
i-j\right\vert $. We next show that $D$ is regular on $\limfunc{span}\left( 
\bar{\Sigma}\right) ^{\bot }=\limfunc{span}(e_{-})^{\bot }$. This is
equivalent to showing that the equation system%
\begin{equation*}
\begin{array}{c}
\left( d\Lambda /d\rho \right) (-1)x+\lambda e_{-}=0 \\ 
e_{-}^{\prime }x=0%
\end{array}%
\end{equation*}%
has $x=0$, $\lambda =0$ as its only solution. We hence need to show that the 
$(n+1)\times (n+1)$ matrix%
\begin{equation*}
A^{\#}=\left[ 
\begin{array}{cc}
\left( d\Lambda /d\rho \right) (-1) & e_{-} \\ 
e_{-}^{\prime } & 0%
\end{array}%
\right]
\end{equation*}%
has rank $n+1$. Note that this is equivalent to establishing that the matrix 
\begin{equation*}
A^{\dag }=\left[ 
\begin{array}{cc}
\left( d\Lambda /d\rho \right) (-1) & (-1)^{n+1}e_{-} \\ 
(-1)^{n+1}e_{-}^{\prime } & 0%
\end{array}%
\right]
\end{equation*}%
is nonsingular. Now note that 
\begin{equation*}
A^{\dag }=-EAE
\end{equation*}%
where $A$ is as in (3) and $E$ is an $(n+1)\times (n+1)$ diagonal matrix
with the $i$-th diagonal element given by $\left( -1\right) ^{i}$. This
proves regularity of $D$ on $\limfunc{span}\left( \bar{\Sigma}\right) ^{\bot
}$. The claim for $\Pi _{\limfunc{span}\left( \bar{\Sigma}\right) ^{\bot
}}\Sigma _{m}\Pi _{\limfunc{span}\left( \bar{\Sigma}\right) }/s_{m}^{1/2}$
is proved as in (3).
\end{proof}

\begin{lemma}
\label{cpar2} For every $\nu \in \lbrack 0,\pi ]$ there exists a sequence $%
\Sigma _{m}\in \mathfrak{C}_{AR(2)}$ converging to $E(\nu )E(\nu )^{\prime }$%
.
\end{lemma}

\begin{proof}
For $\nu =0$ ($\nu =\pi $, respectively) the matrix $E(\nu )E(\nu )^{\prime
} $ equals $e_{+}e_{+}^{\prime }$ ($e_{-}e_{-}^{\prime }$, respectively),
and the result thus follows from Lemma \ref{AR_1}. Hence assume that $\nu
\in (0,\pi )$. Consider for $0<r<1$ the AR(2)-spectral density%
\begin{equation*}
f_{r}(\omega )=\left( 2\pi \right) ^{-1}c\left( r\right) \left\vert 1-2r\cos
\left( \nu \right) \exp (-\iota \omega )+r^{2}\exp (-2\iota \omega
)\right\vert ^{-2}
\end{equation*}%
where 
\begin{equation*}
c\left( r\right) =\left( 1-r^{2}\right) \left( \left( 1+r^{2}\right)
^{2}-4r^{2}\cos ^{2}\left( \nu \right) \right) \left( 1+r^{2}\right) ^{-1}.
\end{equation*}%
Observe that $\int f_{r}\left( \omega \right) d\omega =1$ where the integral
extends over $\left[ -\pi ,\pi \right] $. Hence the $n\times n$ variance
covariance matrix $\Sigma \left( r\right) $ corresponding to $f_{r}$ belongs
to $\mathfrak{C}_{AR(2)}$. Let $\varepsilon >0$ be given and set $A\left(
\varepsilon \right) =\left\{ \omega \in \left[ -\pi ,\pi \right] :\left\vert
\omega -\nu \right\vert \geq \varepsilon \right\} \cup \left\{ \omega \in %
\left[ -\pi ,\pi \right] :\left\vert \omega +\nu \right\vert \geq
\varepsilon \right\} $. Then it is easy to see that%
\begin{equation*}
\sup_{\omega \in A\left( \varepsilon \right) }\left\vert f_{r}(\omega
)\right\vert \rightarrow 0\text{ \ for \ }r\rightarrow 1\text{.}
\end{equation*}%
Consequently, for every $\delta >0$ and every $\varepsilon >0$ there exists
an $0<r\left( \varepsilon ,\delta \right) <1$ such that%
\begin{equation*}
\int\limits_{\left[ -\pi ,\pi \right] \backslash A\left( \varepsilon \right)
}f_{r}\left( \omega \right) d\omega >1-\delta
\end{equation*}%
holds for all $r$ satisfying $r\left( \varepsilon ,\delta \right) <r<1$. In
view of symmetry of $f_{r}$ around $\omega =0$, this shows that for $r$
sufficiently close to $1$ the spectral density $f_{r}$ is arbitrarily small
outside of the union of the neighborhoods $\left\vert \omega -\nu
\right\vert <\varepsilon $ and $\left\vert \omega +\nu \right\vert
<\varepsilon $ and puts mass arbitrarily close to $1/2$ on each one of the
two neighborhoods. A standard argument then shows for every continuous
function $g$ on $\left[ -\pi ,\pi \right] $ that%
\begin{equation*}
\int\limits_{\left[ -\pi ,\pi \right] }g\left( \omega \right) f_{r}\left(
\omega \right) d\omega \rightarrow 0.5g\left( \nu \right) +0.5g\left( -\nu
\right) =\int\limits_{\left[ -\pi ,\pi \right] }g\left( \omega \right)
d\left( 0.5\delta _{\nu }+0.5\delta _{-\nu }\right)
\end{equation*}%
where $\delta _{x}$ denotes unit pointmass at $x$. Specializing to $g\left(
\omega \right) =\exp (-\iota l\omega )$ shows that $\Sigma \left( r\right) $
converges to $E(\nu )E(\nu )^{\prime }$.
\end{proof}

Using the arguments in the above proof it is actually not difficult to show
that the closure of the set of AR(2)-spectral densities in the weak topology
is the class of AR(2)-spectral densities plus all spectral measures of the
form $0.5\delta _{\nu }+0.5\delta _{-\nu }$ for $\nu \in \lbrack 0,\pi ]$.
This result extends in an obvious way to higher-order autoregressive models
and has an appropriate generalization to (multivariate) autoregressive
moving average models, see Theorem 4.1 in \cite{DeiPoe1984}.

\bibliographystyle{ims}
\bibliography{refs}

\end{document}